\documentclass[12pt,reqno]{amsart}

\usepackage[vmargin=1in,hmargin=1.2in]{geometry}
\usepackage{graphicx, latexsym, graphics}
\usepackage{amsfonts, amssymb, amsmath, amsthm, mathtools, mathrsfs}
\usepackage{algorithmic, algorithm}

\usepackage[hyperindex,breaklinks]{hyperref}
\usepackage{bookmark}

\newcommand{\NN}{\mathbb N}

\newcommand{\CC}{\mathbb C}
\newcommand{\RR}{\mathbb R}
\newcommand{\ZZ}{\mathbb Z}

\newcommand{\EE}{\mathcal E}
\newcommand{\DD}{\mathcal D}
\newcommand{\SSS}{\mathcal S}

\newcommand{\supp}{\operatorname{supp}}
\newcommand{\Op}{\operatorname{Op}}
\newcommand{\loc}{\operatorname{loc}}
\newcommand{\comp}{\operatorname{comp}}
\newcommand{\Char}{\operatorname{Char}}
\newcommand{\pr}{\operatorname{pr}}
\newcommand{\phg}{\operatorname{phg}}
\newcommand{\M}{\mathbf{M}}

\theoremstyle{plain}
\newtheorem{theorem}{Theorem}[section]
\newtheorem{proposition}[theorem]{Proposition}
\newtheorem{lemma}[theorem]{Lemma}
\newtheorem{corollary}[theorem]{Corollary}

\theoremstyle{remark}
\newtheorem{remark}[theorem]{Remark}

\theoremstyle{definition}
\newtheorem{definition}[theorem]{Definition}
\newtheorem{example}[theorem]{Example}
\allowdisplaybreaks

\numberwithin{equation}{section}

\begin{document}

\author[S. Pilipovi\' c]{Stevan Pilipovi\' c}
\thanks{The work of S. Pilipovi\'c was partially supported by the Serbian Academy of Sciences and Arts, project, F10.}
\address{Department of Mathematics and Informatics,
University of Novi Sad, Trg Dositeja Obradovi\'{c}a 4, 21000 Novi Sad, Serbia}
\email{stevan.pilipovic@dmi.uns.ac.rs}

\author[B. Prangoski]{Bojan Prangoski}
\thanks{The work of B. Prangoski was partially supported by the bilateral project ``Microlocal analysis and applications'' funded by the Macedonian Academy of Sciences and Arts and the Serbian Academy of Sciences and Arts.}
\address{Department of Mathematics, Faculty of Mechanical
Engineering-Skopje, Ss. Cyril and Methodius University in Skopje, Karposh 2 b.b., 1000 Skopje, Macedonia}
\email{bprangoski@yahoo.com}

\title[Spaces of distributions with Sobolev wave front in a fixed conic set]{Spaces of distributions with Sobolev wave front in a fixed conic set: compactness, pullback by smooth maps and the compensated compactness theorem}



\frenchspacing
\begin{abstract}
We consider the space $\DD'^r_L(M;E)$ of distributional sections of the smooth complex vector bundle $E\rightarrow M$ whose Sobolev wave front set of order $r\in\RR$ lies in the closed conic subset $L$ of $T^*M\backslash0$. We introduce a locally convex topology on it to study the continuity of the pullback by smooth maps and generalise the result of H\"ormander about the pullback on the space of distributions with $\mathcal{C}^{\infty}$ wave front set in $L$. We employ an idea of G\'erard \cite{G1} to extend the Kolmogorov-Riesz compactness theorem to $\DD'^r_L(M;E)$ and we characterise its relatively compact subsets. We study the continuity properties of pseudo-differential operators when acting on $\DD'^r_L(M;E)$, $r\in\RR$, and we generalise the Rellich's lemma. As an application of our results, we extend the microlocal defect measures of G\'erard and Tartar to sequences in $\DD'^0_L(M;E)$ and we show a microlocal variant of the compensated compactness theorem.
\end{abstract}

\maketitle
{\tiny
\tableofcontents
}

\section{Introduction}

The Sobolev wave front set was first introduced and employed by Duistermaat and H\"ormander in \cite{dui-hor} in their analysis of propagation of singularities of solutions of PDEs. Together with the $\mathcal{C}^{\infty}$ wave front set previously introduced by H\"ormander, it became an indispensable tool for classifying the singularities of solutions of PDEs; see \cite{has-mel-vas,mel-vasy-wun,vasy,vasy-wro} and the references therein for various generalisations and their applications. In \cite[Section 8.2]{hor} (see also \cite[Subsection 2.5]{hormander}) H\"ormander introduced a notion of convergence of sequences in the space $\DD'_L(U)$ consisting of all distributions on $U\subseteq \RR^n$ whose $\mathcal{C}^{\infty}$ wave front set is contained in the closed conic subset $L$ of $U\times(\RR^n\backslash\{0\})$ and then showed that the pullback of smooth functions by a smooth map $f:O\rightarrow U$ can be extended to a sequentially continuous map $f^*:\DD'_L(U)\rightarrow \DD'_{f^*L}(O)$ whenever $f^*L$ (the pullback of $L$) does not intersect the set of normals $\mathcal{N}_f$ of $f$. If one writes the natural locally convex topology on $\DD'_L(U)$ that induces this convergence of sequences (see \cite[Chapter 1]{dui-book}), it turns out that the pullback is not continuous: see \cite{D1} for a counterexample. Recently, in \cite{D1,BD}, the authors modified the topology so that the pullback by a smooth map becomes continuous (and not just sequentially continuous). Furthermore, they conducted an extensive analysis of the topological properties of $\DD'_L(U)$ and its dual. We also point out two arXiv preprints \cite{dab-1,dab-2} where the author investigates the topological properties of spaces of distributions whose Sobolev wave front sets of all orders are included in a fix conic set and the article \cite{dap-r-scl} where the authors study the Besov wave front set. An important problem to consider is what happens with the pullback if one considers it as a map on the space $\DD'^r_L(U)$ consisting of distributions that have Sobolev wave front set of order $r\in\RR$ in $L$. Notice that this is a refinement of the result on $\DD'_L(U)$ because of the well-known algebraic identity $\DD'_L(U)=\bigcap_{r\in\RR}\DD'^r_L(U)$. One expects that there will be a loss of regularity since this happens in the simplest examples when considering restrictions of $H^r_{\loc}$ (the local Sobolev space of order $r\in\RR$) to lower dimensional hyperplanes; see \cite[Appendix B]{hor2}. This is one of our primary goals in the article. In fact, we will introduce a locally convex topology on the spaces $\DD'^r_L(M)$, $r\in\RR$, with $M$ a smooth manifold and $L$ a closed conic subset of the cotangent bundle without the zero section $T^*M\backslash0$, and show the following:
\begin{itemize}
\item[$(i)$] The pullback $f^*:\DD'^{r_2}_L(N)\rightarrow \DD'^{r_1}_{f^*L}(M)$, with $f:M\rightarrow N$ smooth and satisfying $L\cap\mathcal{N}_f=\emptyset$, is well-defined and continuous for appropriately chosen $r_2$ and $r_1$. When $f$ has constant rank, we show that the conditions on $r_2$ and $r_1$ are essentially optimal.
\item[$(ii)$] The algebraic identity $\DD'_L(M)=\bigcap_{r\in\RR}\DD'^r_L(M)$ becomes topological when we interpret the right hand side as the projective limit. The algebraic identities $\DD'^r_{\emptyset}(M)=H^r_{\loc}(M)$ and $\DD'^r_{T^*M\backslash0}(M)=\DD'(M)$ also become topological.
\item[$(iii)$] Given a properly supported pseudo-differential operator $A$ of order $r_0$, the linear map $A:\DD'^r_L(M)\rightarrow \DD'^{r-r_0}_L(M)$, which is well-defined by the results of Duistermaat and H\"ormander \cite{dui-hor}, becomes continuous. When $A$ is elliptic (or better yet, non-characteristic in $(T^*M\backslash0)\backslash L$), we show a priori estimates for $A$ which reduce to the well-known a priori estimates for elliptic $A:H^r_{\loc}(M)\rightarrow H^{r-r_0}_{\loc}(M)$.
\end{itemize}
Inspired by ideas from \cite[Chapter 7]{GS} and \cite{BD}, we also identify the strong dual of $\DD'^r_L(M)$: it consists of all compactly supported Sobolev distributions of order $-r$ (technically speaking, they are distribution densities) that have $\mathcal{C}^{\infty}$ wave front set included in $\{(p,\xi)\in T^*M\backslash0\,|\, (p,-\xi)\not\in L\}$. As a consequence of $(iii)$, we also show continuity properties of pseudo-differential operators when acting on these spaces.\\
\indent In \cite{G1}, G\'erard introduced a wave front set associated with a sequence in $L^2_{\loc}$ which contains the directions in the cotangent bundle where the sequence does not behave like a relatively compact subset of $L^2_{\loc}$ (see \cite{ND} for a related concept). The idea in \cite{G1} is to define the complement as the directions where the sequence satisfies the Kolmogorov-Riesz compactness criterium \cite{HH} for $L^2_{\loc}$. In the context of the topology we introduce on $\DD'^r_L(M)$, we realised that this wave front set does something more. We define it for any bounded subset of $\DD'(M)$ and, instead of the compactness criterium for $L^2_{\loc}$, we employ the $H^r_{\loc}$ compactness criterium. Then we show that:
\begin{itemize}
\item[$(iv)$] This wave front set with index $r$ of a bounded set $B$ in $\DD'(M)$ is the closed conic subset $L'$ of $T^*M\backslash0$ such that $B$ is a relatively compact subset of $\DD'^r_{L'}(M)$. In effect, it gives a convenient and computational characterisation of the relatively compact subsets of $\DD'^r_L(M)$.
\end{itemize}
This generalises the Kolmogorov-Riesz compactness theorem to $\DD'^r_L(M)$ and it allows us to show:
\begin{itemize}
\item[$(v)$] A generalisation of the Rellich lemma for the spaces $\DD'^r_L(M)$, $r\in\RR$.
\end{itemize}
We point out that we show these results on general smooth complex vector bundles over manifolds.\\
\indent Our results give a robust parallel between $\DD'^r_L(M)$ and $H^r_{\loc}(M)$ when it comes down to the characterisation of the compact sets and the continuity properties of pseudo-differential operators so one can show microlocal variants of various results by employing similar convergence and continuity arguments as in the local Sobolev case. Important such examples are some of the weak convergence methods employed in the theory of nonlinear PDEs \cite{E} \footnote{For the reader familiar with the theory, we point out that one can always extract convergent subsequence of a relatively compact sequence in $\DD'^r_L(M)$ since we show that the bounded subsets of $\DD'^r_L(M)$ are metrisable although $\DD'^r_L(M)$ in general is not.}. We showcase the theory we developed by employing it to improve two closely related such results: the microlocal defect measures and the compensated compactness theorem.\\
\indent The microlocal defect measures were independently developed by G\'erard \cite{G1,Ger} and Tartar \cite{Tar} who introduced them under the name H-measures. Nowadays they are indispensable tool in the study of linear and nonlinear PDEs; see for example \cite{ant-laz,d-lr-le,d-m,F,G3,Pan,R} and the references therein. We also refer to the works of Lions \cite{lions1,lions2,lions3,lions4} where he introduced the related concentration-compactness method for solving minimisation problems in the calculus of variations. Broadly speaking, for a bounded sequence $(u_k)_{k\in\ZZ_+}$ in $L^2_{\loc}(U)$ which converges to $u\in L^2_{\loc}(U)$ in $\DD'(U)$, the microlocal defect measure associated to $(u_k)_{k\in\ZZ_+}$ is a positive Radon measure on $U\times \mathbb{S}^{n-1}$ whose support contains the directions where $\{u_k\}_{k\in\ZZ_+}$ does not behave like a relatively compact set in $L^2_{\loc}(U)$. Our results naturally lend themselves for generalising this to the case when one only knows that $u,u_k\in\DD'^0_L(M)$, $k\in\ZZ_+$, and $\{u_k\}_{k\in\ZZ_+}$ is bounded in $\DD'^0_L(M)$. In this case, the microlocal defect measure is only defined in the directions outside of $L$ and we show that $u_{k_j}\rightarrow u$ in $\DD'^0_{L'}(M)$ where $L'$ is the union of $L$ and the directions in the support of the microlocal defect measure and $(u_{k_j})_{j\in\ZZ_+}$ is a subsequence of $(u_k)_{k\in\ZZ_+}$.\\
\indent The second place where we are going to apply our results is in showing a microlocal extension of the compensated compactness theorem. The theorem was first proved by Murat \cite{mur1,mur2} and Tartar \cite{Tar-0} (see also \cite{hanouzet,hanou-jol}) and letter improved by G\'erard \cite{Ger}. Without going into details, the theorem states the following. Let $u,u_k\in L^2_{\loc}$, $k\in\ZZ_+$, be as above but with values in $\CC^N$ and let $A$ and $Q$ be properly supported polyhomogeneous matrix valued $\Psi$DOs of orders $r$ and $0$ and principal symbols $a$ and $q$ respectively. If $\{Au_k\}_{k\in\ZZ_+}$ is relatively compact in $H^{-r}_{\loc}$ and $q$ satisfies a weak Legendre-Hadamard condition $q(x,\xi)z\cdot \overline{z}= 0$ for those $(x,\xi)$ and $z\in\CC^N$ which satisfy $a(x,\xi)z=0$, then $Qu_k\cdot \overline{u}_k\rightarrow Qu\cdot \overline{u}$ as distributions. As before, our generalisation is that one can take $u$ and $u_k$ to be in $\DD'^0_L$. We point out that this is not a trivial generalisation since even the products $Qu_k\cdot \overline{u}_k$ and $Qu\cdot \overline{u}$ are meaningless a priori; moreover, one can not employ H\"ormander's definition of products of distributions to defined them because we only have information on the Sobolev wave front sets of $u$ and $u_k$. However, as a byproduct of our results, we show that the sesquilinear map $(v_1,v_2)\mapsto Qv_1\cdot \overline{v}_2$ on smooth functions uniquely extends to a hypocontinuous sesquilinear map on $\DD'^0_L$ if $Q$ is of order $-\infty$ at every point of $L$ and then we show that $Qu_k\cdot \overline{u}_k\rightarrow Qu\cdot \overline{u}$ as distributions. We prove the compensated compactness theorem as well as our generalisation of the microlocal defect measures on general smooth complex vector bundles. We point out that there are variants of the compensated compactness theorem in the $L^p$ setting \cite{bal-f-t-tes,mis-mit,pan2}, but we leave their microlocalisations for future research.\\
\indent The paper is organised as follows. We collect in Section \ref{sec-not-gen-foreveryt} the notations and necessary facts we are going to employ throughout the rest of the article. In Section \ref{Sec-comp}, we introduce the topology on $\DD'^r_L$ and show several key results in open sets of $\RR^n$. Sections \ref{sec-formain-result-onman} and \ref{sec-psido-on-despacesanddual} contain our main results on the spaces $\DD'^r_L$, $r\in\RR$, on general smooth complex vector bundles as well as the continuity properties of pseudo-differential operators when acting on these spaces. Section \ref{application} is devoted to the applications. Finally, in Appendix \ref{app-for-cou-foroptthpulbacsmom} we show the optimality of the loss in regularity in our theorem on the pullback by smooth maps.

\section{Preliminaries}\label{sec-not-gen-foreveryt}

We fix the constants in the Fourier transform as $\mathcal{F}f(\xi)=\int_{\RR^n} e^{-ix \xi}f(x)dx$, $f\in L^1(\RR^n)$. Throughout the article, $B(x,r)$ will stand for the open ball with centre at $x$ and radius $r>0$. As standard, $\langle x\rangle:=(1+|x|^2)^{1/2}$, $x\in\RR^n$. For any open set $O$ in a topological space, $K\subset\subset O$ means that $K$ is compact and $K\subseteq O$. If $B$ is any subset of a topological space, we denote by $\mathbf{1}_B$ the indicator function of $B$, i.e. $\mathbf{1}_B(x)=1$ if $x\in B$ and $\mathbf{1}_B(x)=0$ when $x\not\in B$. For a $\mathcal{C}^1$ map $f:O\rightarrow \RR^n$, with $O$ an open set in $\RR^m$, we denote by $f'(x)$ the derivative of $f$ at $x\in O$ and, when $n=m$, we set $|f'|(x):=|\det(f'(x))|$, $x\in O$. As standard, we denote $\RR_+:=\{t\in\RR\,|\, t>0\}$ and, for $B\subseteq \RR^n$, we set $\RR_+B:=\{tx\in\RR^n\,|\, t>0,\, x\in B\}$; clearly $\RR_+ B$ is a cone in $\RR^n$. Given a smooth manifold $M$, $T^*M\backslash 0$ stands for the cotangent bundle $T^*M$ without the image of the zero section. A subset $L$ of $T^*M\backslash0$ is said to be conic if it satisfies $(p,\xi)\in L$, $\xi\in T^*_pM$, implies $(p,t\xi)\in L$, for all $t>0$. Given a conic subset $L$ of $T^*M\backslash0$, we denote by $L^c$ the complement of $L$ in $T^*M\backslash0$. When $M$ is an open subset $U$ of $\RR^n$ we will always canonically identify $T^*U$ with $U\times\RR^n$ and consequently, $T^*U\backslash0=U\times (\RR^n\backslash\{0\})$ and, for a conic set $L$, $L^c=U\times(\RR^n\backslash\{0\})\backslash L$. Furthermore, for a conic subset $L$ of $T^*M\backslash0$ (or of $U\times(\RR^n\backslash\{0\})$), we denote
$\check{L}:=\{(p,\xi)\in T^*M\backslash0\,|\, (p,-\xi)\in L\}$.\\
\indent Let $U$ be an open set in $\RR^n$. The $\mathcal{C}^{\infty}$ wave front set of $u\in\DD'(U)$ (as standard, we will always abbreviate it as the wave front set) is the closed conic subset $WF(u)$ of $U\times(\RR^m\backslash\{0\})$ defined as follows: $(x,\xi)\in U\times (\RR^n\backslash\{0\})$ does not belong to $WF(u)$ if there are an open cone $V\subseteq \RR^n$ containing $\xi$ and $\varphi\in\DD(U)$ satisfying $\varphi(x)\neq0$ such that $\mathfrak{q}_{\nu;\varphi,V}(u):=\|\langle\cdot\rangle^{\nu} \mathcal{F}(\varphi u)\|_{L^{\infty}(V)}<\infty$, $\nu>0$; see \cite[Section 8.1]{hor}. Given a closed conic subset $L$ of $U\times(\RR^n\backslash\{0\})$, set $\DD'_L(U):=\{u\in\DD'(U)\,|\, WF(u)\subseteq L\}$. Throughout the article $\DD'_L(U)$ will always carry the following locally convex topology introduced in \cite{BD}: the topology defined by all continuous seminorms on $\DD'(U)$ together with all seminorms $\mathfrak{q}_{\nu;\varphi,V}$ where $\varphi\in\DD(U)$ and $V$ is a closed cone in $\RR^n$ such that $(\supp\varphi\times V)\cap L=\emptyset$; we refer to \cite{BD,D1} for the topological properties of $\DD'_L(U)$.\\
\indent We denote by $H^r(\RR^n)$ the standard Sobolev space of order $r\in\RR$ on $\RR^n$, i.e. $H^r(\RR^n):=\{u\in\SSS'(\RR^n)\,|\, \langle\cdot\rangle^r\mathcal{F}u\in L^2(\RR^n)\}$.\\
\indent Given two locally convex spaces $X$ and $Y$ (from now, always abbreviated as l.c.s.) we denote by $\mathcal{L}(X,Y)$ the space of continuous linear operators from $X$ into $Y$ and we denote by $\mathcal{L}_b(X,Y)$ this space equipped with the strong operator topology. When $X=Y$, we simply write $\mathcal{L}(X)$ and $\mathcal{L}_b(X)$. Of course, $X'$ is the dual of $X$ and $X'_b$ stands for $X'$ equipped with the strong dual topology.

\subsection{Distributions on manifolds}\label{subsec-dist-on-manifolds-vecbund}

For later use, we collect standard definitions and classical results about distributions on smooth manifolds; throughout the rest of the section, we employ the Einstein summation convention.\\
\indent Let $M$ be a smooth $m$-dimensional manifold\footnote{Manifolds are always assumed to be second-countable.}. The $\sigma$-algebra of Lebesgue measurable sets on $M$ can be unambiguously defined by declaring a set $X$ to be Lebesgue measurable if $x(X\cap O)$ is Lebesgue measurable for each chart $(O,x)$; the notion of a negligible set (nullset) in $M$ is unambiguous since they are diffeomorphism invariant.\\
\indent Let $(E,\pi_E,M)$ be a smooth complex vector bundle of rank $k$; all the vector bundles throughout the article will be complex and smooth so we will never emphasise this - the only exception to this are the tangent and cotangent vector bundles $TM$ and $T^*M$ which are real and smooth. We denote by $E_p$ the fiber over $p\in M$, i.e. $E_p:=\pi^{-1}_E(\{p\})$, and, for an open set $O$ in $M$, $E_O$ stands for the restriction of $E$ to $O$. For $l\in\NN\cup\{\infty\}$, we denote by $\Gamma^l(E)$ the Fr\'echet spaces of $l$-times continuously differentiable sections, while $\Gamma^l_c(E)$ stands for the space of compactly supported $l$-times continuously differentiable sections equipped with its standard strict $(LF)$-space topology; when $l<\infty$, $\Gamma^l_c(E)$ is in fact a strict $(LB)$-space. For $K\subset\subset M$ and $l\in \NN\cup\{\infty\}$, we denote $\Gamma^l_K(E):=\{\varphi\in\Gamma^l(E)\,|\, \supp\varphi\subseteq K\}$ and recall that it is a closed subspace of both $\Gamma^l(E)$ and of $\Gamma^l_c(E)$ and they induce the same Fr\'echet topology on it; when $l<\infty$, $\Gamma^l_K(E)$ is in fact a Banach space. When $l=\infty$, we will simply denote these spaces by $\Gamma(E)$, $\Gamma_c(E)$ and $\Gamma_K(E)$ respectively. Given another vector bundle $(F,\pi_F, M)$ of rank $k_1$, $(L(E,F),\pi_{L(E,F)},M)$ stands for the smooth complex vector bundle of rank $kk_1$ whose fibres are $L(E,F)_p:=\mathcal{L}(E_p,F_p)$, $p\in M$. We denote by $E'$ the dual bundle to $E$, i.e. $E':=L(E,\CC_M)$ where $\CC_M:=M\times\CC$ is the trivial line bundle. As standard, $DM$ stands for the complex $1$-density bundle over $M$ with fibres $DT_pM$, $p\in M$. With $E$ as before, we denote by $E^{\vee}$ the functional dual bundle over $M$: it is the smooth complex vector bundle with total space $E^{\vee}:=L(E,DM)$. Notice that $(E^{\vee})^{\vee}$ is canonically isomorphic with $E$. The space of distributional sections $\DD'(M;E)$ of $E$ is the strong dual of $\Gamma_c(E^{\vee})$ and the space of distributions $\DD'(M)$ on $M$ is the strong dual of $\Gamma_c(DM)$. Similarly, the space of distributional sections with compact support $\EE'(M;E)$ is the strong dual of $\Gamma(E^{\vee})$ and the space of distributions with compact support $\EE'(M)$ is the strong dual of $\Gamma(DM)$. When $E$ is the trivial $k$-bundle $\CC_M^k:=M\times \CC^k$ over $M$, $(\CC^k_M)^{\vee}$ can be identified with the Whitney sum bundle $\underbrace{DM\oplus \ldots \oplus DM}_k$:
$$
(\CC_M^k)^{\vee}_p\ni T\mapsto (p,T(1,0,\ldots,0),\ldots, T(0,\ldots,0,1))\in \{p\}\times\underbrace{DT_pM\oplus \ldots \oplus DT_pM}_k.
$$
Employing this identification, we have
$$
\DD'(M;\CC_M)=\DD'(M)\,\, \mbox{and, in general,}\,\, \DD'(M;\CC^k_M)=\DD'(M)^k=\underbrace{\DD'(M)\times \ldots \times \DD'(M)}_k;
$$
analogously, $\EE'(M;\CC_M^k)=\EE'(M)^k$. The continuous inclusion $L^p_{\operatorname{loc}}(M;E)\rightarrow \DD'(M;E)$, $1\leq p\leq \infty$, is given by $f\mapsto \langle f,\cdot\rangle$, with $\langle f,\varphi\rangle:=\int_M [f,\varphi]$, $\varphi\in\Gamma_c(E^{\vee})$, where $[f,\varphi]$ stands for the measurable section of $DM$ given by $[f,\varphi]_p:=\varphi_p(f_p)$, a.a. $p\in M$.\\
\indent As standard, we denote by $H^r_{\loc}(M;E)$ and $H^r_{\comp}(M;E)$ the spaces of local and compact Sobolev sections of $E$ of order $r$. We recall that $H^r_{\loc}(M;E)$ is a reflexive Fr\'echet space, $H^r_{\comp}(M;E)$ is a reflexive strict $(LB)$-space, the strong dual of $H^r_{\loc}(M;E)$ is $H^{-r}_{\comp}(M;E^{\vee})$ and the strong dual of $H^r_{\comp}(M;E)$ is $H^{-r}_{\loc}(M;E^{\vee})$. For $K\subset\subset M$, we denote $H^r_K(M;E):=\{u\in H^r_{\loc}(M;E)\,|\, \supp u\subseteq K\}$ and we recall that $H^r_K(M;E)$ is a closed subspace both of $H^r_{\loc}(M;E)$ and of $H^r_{\comp}(M;E)$, both of them induce the same topology on $H^r_K(M;E)$ and with this topology $H^r_K(M;E)$ is a Banach space. Furthermore, $H^0_{\loc}(M;E)=L^2_{\loc}(M;E)$ and $H^0_{\comp}(M;E)=L^2_{\comp}(M;E)$.\\
\indent If $u\in \DD'(M)$ and $(O,x)$ is a chart on $M$ then we define $u_x\in\DD'(x(O))$ by
\begin{equation}\label{ind-tri-man-caseford}
\langle u_x,\phi\rangle:=\langle u, (\phi\circ x)\lambda^x\rangle,\quad \phi\in\DD(x(O)),
\end{equation}
where $\lambda^x$ is the unique section of $DO$ which satisfies $\lambda^x(\frac{\partial}{\partial x^1},\ldots,\frac{\partial}{\partial x^m})=1$; we will always denote this section by $\lambda^x$. Let $(O,x)$ be a chart in $M$ and $\Phi_x:\pi^{-1}_E(O)\rightarrow O\times\CC^k$ a local trivialisation of $E$ over $O$. Denote by $\Psi_x:\pi^{-1}_{DM}(O)\rightarrow O\times \CC$ the local trivialisation of $DM$ over $O$ given by $\Psi_x(z\lambda^x_p)=(p,z)$, $p\in O$, $z\in\CC$. Let $\Phi_{x,p}: E_p\rightarrow \CC^k$ and $\Psi_{x,p}: DT_pM\rightarrow \CC$, $p\in O$, be the induced isomorphisms. We define the frame $(\sigma^1,\ldots,\sigma^k)$ for $E^{\vee}$ over $O$ by $(\sigma^j)_p:=\Psi_{x,p}^{-1}\circ \epsilon^j\circ \Phi_{x,p}$, $p\in O$, $j=1,\ldots,k$, where $\epsilon^j:\CC^k\rightarrow \CC$ is the map $\epsilon^j(z^1,\ldots,z^k)=z^j$; we call $(\sigma^1,\ldots,\sigma^k)$ the induced frame by $\Phi_x$. With its help, for each $u\in \DD'(M;E)$, we can define $u_{\Phi_x}^j\in \DD'(x(O))$, $j=1,\ldots,k$, by
\begin{equation}\label{ind-tri-loc-nestk}
\langle u_{\Phi_x}^j,\phi\rangle:=\langle u,(\phi\circ x) \sigma^j\rangle,\quad \phi\in\DD(x(O)).
\end{equation}
Notice that $\langle u,\varphi\rangle=\langle u_{\Phi_x}^j,\varphi_j\circ x^{-1}\rangle$, $u\in \DD'(M;E)$, $\varphi\in \Gamma_c(E^{\vee}_O)$,\footnote{$(E^{\vee})_O=(E_O)^{\vee}$ and we denote it simply by $E^{\vee}_O$.} where $\varphi=\varphi_j\sigma^j$. If $(U,y)$ is another chart on $M$ that has non-empty intersection with $O$ and over which $E$ locally trivialises via $\Phi_y:\pi_E^{-1}(U)\rightarrow U\times \CC^k$, then
\begin{equation}\label{tra-map-bun-exchkl}
u_{\Phi_y}^j=(\tau^j_l\circ y^{-1})(x\circ y^{-1})^*u_{\Phi_x}^l\quad \mbox{in}\quad \DD'(y(O\cap U)),
\end{equation}
where $\tau=(\tau^j_l)_{j,l}:O\cap U\rightarrow \operatorname{GL}(k,\CC)$ is the transition map: it is the unique smooth map that satisfies $\Phi_y\circ\Phi_x^{-1}(p,z)=(p,\tau(p)z)$, $p\in O\cap U$, $z\in\CC^k$. If $(\rho^1,\ldots,\rho^k)$ is the frame for $E^{\vee}$ over $U$ induced by $\Phi_y$, it holds that
\begin{equation}\label{cha-fra-fun-duabuch}
\rho^j=(|(y\circ x^{-1})'|\circ x)\tau^j_l\sigma^l\quad \mbox{on}\quad O\cap U.
\end{equation}
If $N$ is another manifold and $f:N\rightarrow M$ a smooth map, $f^*E$ stands for the pullback bundle. It is a smooth complex vector bundle of rank $k$ over $N$ with total space $f^*E:=\bigcup_{p\in N} \{p\}\times E_{f(p)}$ and equipped with the topology induced from $N\times E$. The fibre over $p$ is $\{p\}\times E_{f(p)}$, the projection is $\pi_{f^*E}(p,e)=p$, $e\in E_{f(p)}$, and the local trivialisations are defined as follows. For each local trivialisation $\Phi:\pi^{-1}_E(O)\rightarrow O\times \CC^k$ of $E$ define a local trivialisation $f^*\Phi:\pi^{-1}_{f^*E}(f^{-1}(O))\rightarrow f^{-1}(O)\times \CC^k$, $f^*\Phi(p,e)=(p,\Phi_{f(p)}e)$. The pullback map $f^*:\Gamma^l(E)\rightarrow \Gamma^l(f^*E)$, $f^*\varphi(p):=(p,\varphi\circ f(p))$, $p\in N$, is well-defined and continuous for all $l\in\NN\cup\{\infty\}$.

\subsection{Pseudo-differential operators}

As standard, for $r\in\RR$, $0\leq \delta\leq \rho\leq 1$ and $\delta<1$, we denote by $S^r_{\rho,\delta}(\RR^{2n})$ the H\"ormander class of global symbols on $\RR^{2n}$ \cite{hormander1}: it is the Fr\'echet space of all $a\in\mathcal{C}^{\infty}(\RR^{2n})$ which satisfy $\sup_{x,\xi\in\RR^n}\langle\xi\rangle^{-r+\rho|\alpha|-\delta|\beta|}|\partial^{\alpha}_{\xi}\partial^{\beta}_x a(x,\xi)|<\infty$, $\alpha,\beta\in\NN^n$. The pseudo-differential operator with symbol $a\in S^r_{\rho,\delta}(\RR^{2n})$ is defined by
\begin{equation}\label{def-ofp-ope-onthewholer}
\Op(a)\varphi(x):=\frac{1}{(2\pi)^n}\int_{\RR^n}e^{ix\xi}a(x,\xi)\mathcal{F}\varphi(\xi) d\xi,\quad \varphi\in\SSS(\RR^n).
\end{equation}
It is a continuous operator on $\SSS(\RR^n)$ and it extends to a continuous operator on $\SSS'(\RR^n)$. Furthermore, when $r=0$, $\Op(a)$ is continuous on $L^2(\RR^n)$ and the mapping $S^0_{\rho,\delta}(\RR^{2n})\rightarrow \mathcal{L}_b(L^2(\RR^n))$, $a\mapsto \Op(a)$, is continuous \cite[Theorem 2.5.1, p. 110, and Theorem 2.3.18, p. 100]{lernerB}. With only a few exceptions, we will mostly employ the case when $\rho=1$ and $\delta=0$ and we will always denote this space by $S^r(\RR^{2n})$ for short.\\
\indent We collect the notations and facts we need from the local theory of pseudo-differential operators; we refer to \cite{dui-hor,hormander,hor2} for the complete account. If $U$ is an open set in $\RR^n$, the local symbol space $S^r_{\loc}(U\times\RR^n)$ is the Fr\'echet space of all $a\in\mathcal{C}^{\infty}(U\times\RR^n)$ which satisfy
\begin{equation}\label{est-for-sym-spaceonrwhols}
\sup_{x\in K,\, \xi\in\RR^n}\langle\xi\rangle^{-r+|\alpha|}|\partial^{\alpha}_{\xi}\partial^{\beta}_x a(x,\xi)|<\infty,\quad \alpha,\beta\in\NN^n,\, K\subset\subset U
\end{equation}
(we will only use the variant when $\rho=1$ and $\delta=0$). Furthermore, $S^{-\infty}_{\loc}(U\times\RR^n):=\bigcap_{r\in\RR}S^r_{\loc}(U\times\RR^n)$ equipped with its natural Fr\'echet space topology. More generally, if $V$ is an open cone in $\RR^n$, we denote by $S^r_{\loc}(U\times V)$ the space of all $a\in\mathcal{C}^{\infty}(U\times V)$ which satisfy \eqref{est-for-sym-spaceonrwhols} but on $K\times V'$ for every $K\subset\subset U$ and every cone $V'\subseteq V$ such that $V'\cap \mathbb{S}^{n-1}\subset\subset V$; when $V=\RR^n$, this space coincides with the above definition of $S^r_{\loc}(U\times\RR^n)$. As before, set $S^{-\infty}_{\loc}(U\times V):=\bigcap_{r\in\RR}S^r_{\loc}(U\times V)$. For $r\in\RR\cup\{-\infty\}$ and $K\subset\subset U$, we denote $S^r_K(U\times\RR^n):=\{a\in S^r_{\loc}(U\times\RR^n)\,|\, \supp a\subseteq K\times\RR^n\}$; it is a closed subspace of $S^r_{\loc}(U\times\RR^n)$ (and of $S^r(\RR^{2n})$). Furthermore, we denote $S^r_c(U\times\RR^n):=\bigcup_{K\subset\subset U}S^r_K(U\times\RR^n)$ ($S^r_c(U\times\RR^n)$ can be equipped with a natural $(LF)$-space topology but we will not need this fact); clearly $S^r_c(U\times\RR^n)\subseteq S^r(\RR^{2n})$. For $a\in S^r_{\loc}(U\times\RR^n)$, the operator $\Op(a)$ is defined as in \eqref{def-ofp-ope-onthewholer} but with $U$ in place of $\RR^n_x$ and $\varphi\in\DD(U)$. Then $\Op(a):\DD(U)\rightarrow \mathcal{C}^{\infty}(U)$ is well-defined and continuous and it extends to a well-defined and continuous mapping $\Op(a):\EE'(U)\rightarrow \DD'(U)$. Every continuous operator $T:\EE'(U)\rightarrow \mathcal{C}^{\infty}(U)$ is called a regularising operator; in view of the Schwartz kernel theorem, these are the operators whose kernels are smooth on $U\times U$. The space of regularising operators on $U$ is denoted by $\Psi^{-\infty}(U)$. A pseudo-differential operator of order $r\in\RR\cup\{-\infty\}$ is an operator $A:\DD(U)\rightarrow \mathcal{C}^{\infty}(U)$ of the form $A=\Op(a)+T$ with $a\in S^r_{\loc}(U\times\RR^n)$ and $T\in\Psi^{-\infty}(U)$; $a$ is called the symbol of $A$ and is unique modulo $S^{-\infty}_{\loc}(U\times\RR^n)$. The kernel of any $\Psi$DO is always smooth outside of the diagonal. The space of pseudo-differential operators on $U$ of order $r$ is denoted by $\Psi^r(U)$. For $r\in\RR$, the operator $A\in\Psi^r(U)$ is said to be polyhomogeneous of order $r$ if some (or, equivalently any) symbol $a$ of $A$ has an asymptotic expansion $a\sim \sum_{j=0}^{\infty} a_j$, where $a_j\in S^{r-j}_{\loc}(U\times\RR^n)$, $j\in\NN$, are positively homogeneous of degree $r-j$ when $|\xi|>1$, i.e. $a_j(x,t\xi)=t^{r-j}a_j(x,\xi)$, $x\in U$, $|\xi|>1$, $t>1$. The space of polyhomogeneous $\Psi$DOs of order $r\in\RR$ on $U$ is denote by $\Psi^r_{\phg}(U)$. The space of polyhomogeneous symbols of order $r$, i.e. the symbols that have asymptotic expansion as above, is denoted by $S^r_{\loc,\phg}(U\times\RR^n)$.\\
\indent Let $M$ be a smooth manifold of dimension $m$. A continuous operator $A:\DD(M)\rightarrow \mathcal{C}^{\infty}(M)$ is said to be pseudo-differential operator of order $r\in\RR\cup\{-\infty\}$ if its kernel is smooth outside of the diagonal in $M\times M$ and for every $p\in M$ there is a coordinate chart $(O,x)$ containing $p$ such that the operator $A_x:\DD(x(O))\rightarrow \mathcal{C}^{\infty}(x(O))$, $A_x(\phi):= A(\phi\circ x)\circ x^{-1}$, belongs to $\Psi^r(x(O))$; when this is the case, one can show this holds for all charts on $M$ and this definition of $\Psi$DO coincides with the one above when $M$ is an open subset of $\RR^m$. The space of all $\Psi$DOs of order $r\in\RR\cup\{-\infty\}$ on $M$ is denoted by $\Psi^r(M)$. Finally, the space of symbols of order $r\in\RR\cup\{-\infty\}$ on $M$ is denoted by $S^r_{\loc}(T^*M)$ and it consists of all $a\in\mathcal{C}^{\infty}(T^*M)$ such that for every chart $(O,x)$ it holds that $(\kappa^{-1})^*a\in S^r_{\loc}(x(O)\times\RR^m)$ where $\kappa$ is the chart induced total local trivialisation of $T^*M$ over $O$:
\begin{equation}\label{loc-tri-ofb-undimatotcotbusks}
\kappa:\pi_{T^*M}^{-1}(O)\rightarrow x(O)\times\RR^m,\quad \kappa(p,\xi_j dx^j|_p):=(x(p),\xi_1,\ldots,\xi_m).
\end{equation}
When $M$ is an open subset of $\RR^m$, this definition coincides with the one above.\\
\indent Let $E$ and $F$ be two vector bundles over $M$ of rank $k$ and $k'$ respectively. A continuous operator $A:\Gamma_c(E)\rightarrow\Gamma(F)$ is said to be pseudo-differential operator of order $r\in\RR\cup\{-\infty\}$ if its kernel is smooth outside of the diagonal in $M\times M$ and if for every $p\in M$ there is a coordinate chart $(O,x)$ containing $p$ over which both $E$ and $F$ locally trivialise via $\Phi:\pi_E^{-1}(O)\rightarrow O\times \CC^k$ and $\Phi':\pi_F^{-1}(O)\rightarrow O\times \CC^{k'}$ and such that the operators $A_{\Phi,\Phi',j}^l:\DD(O)\rightarrow \mathcal{C}^{\infty}(O)$, $A_{\Phi,\Phi',j}^l(\varphi):= s'^l(A(\varphi e_j))$, $j=1,\ldots,k$, $l=1,\ldots,k'$, belong to $\Psi^r(O)$ where $(e_1,\ldots,e_k)$ is the local frame for $E$ over $O$ induced by $\Phi$ and $(s'^1,\ldots,s'^{k'})$ is the local frame for the dual bundle $F'$ induced by $\Phi'$. When this is the case, one can show this holds for all charts on $M$ over which both $E$ and $F$ locally trivialise and for any local trivialisations of $E$ and $F$. This definition of a $\Psi$DO coincides with the one above on manifolds when $E=F=\CC_M$. The space of all pseudo-differential operators of order $r\in\RR\cup\{-\infty\}$ between the bundles $E$ and $F$ is denoted by $\Psi^r(M;E,F)$. Every $A\in\Psi^r(M;E,F)$ extends to a continuous operator $A:\EE'(M;E)\rightarrow\DD'(M;F)$. When $A$ is properly supported (i.e., both projections from the support of the kernel of $A$ in $M\times M$ to $M$ are proper maps), $A:\Gamma_c(E)\rightarrow \Gamma_c(F)$ is well-defined and continuous and it extends to a well-defined and continuous operator $\Gamma(E)\rightarrow\Gamma(F)$, $\EE'(M;E)\rightarrow \EE'(M;F)$ and $\DD'(M;E)\rightarrow \DD'(M;F)$; furthermore, it also extends to a well-defined and continuous operator $H^{r'}_{\loc}(M;E)\rightarrow H^{r'-r}_{\loc}(M;F)$ and $H^{r'}_{\comp}(M;E)\rightarrow H^{r'-r}_{\comp}(M;F)$, $r'\in\RR$.\\
\indent For $r\in\RR\cup\{-\infty\}$, the space of symbols $S^r_{\loc}(T^*M;E,F)$ consists of all $a\in \Gamma(\pi_{T^*M}^*L(E,F))$ such that for every chart $(O,x)$ over which $E$ and $F$ locally trivialise via $\Phi$ and $\Phi'$ as above, the functions $a^l_{\Phi,\Phi',j}: T^*O\rightarrow\CC$, $j=1,\ldots,k$, $l=1,\ldots,k'$, defined by $a_{|O}=a^l_{\Phi,\Phi',j} \pi_{T^*M}^*(e'^j\otimes s_l)$ belong to $S^r_{\loc}(T^*O)$ where $(e'^1,\ldots,e'^k)$ and $(s_1,\ldots,s_{k'})$ are the local frames for $E'$ and $F$ over $O$ induced by $\Phi$ and $\Phi'$ and $(\pi_{T^*M}^*(e'^j\otimes s_l))_{j,l}$ is the pullback local frame for $\pi_{T^*M}^*L(E,F)$ over $T^*O$.\\
\indent Let $A\in\Psi^r(M;E,F)$. For every chart $(O,x)$ on $M$ over which $E$ and $F$ locally trivialise, there are $\widetilde{a}^l_j\in S^r_{\loc}(x(O)\times \RR^m)$ and $\widetilde{T}^l_j\in\Psi^{-\infty}(x(O))$, $j=1,\ldots,k$, $l=1,\ldots,k'$, such that
$$
(A\varphi)_{|O}=\Op(\widetilde{a}^l_j)(\varphi^j\circ x^{-1})\circ x\,s_l+\widetilde{T}^l_j(\varphi^j\circ x^{-1})\circ x\,s_l,\quad \varphi\in\Gamma_c(E_O),\, \varphi=\varphi^je_j,
$$
where $(e_1,\ldots,e_k)$ and $(s_1,\ldots,s_{k'})$ are the local frames over $O$ for $E$ and $F$ induced by their local trivialisations respectively. We call $\{\widetilde{a}^l_j\}_{l,j}$ symbols of the coordinate representation of $A$; for fixed chart and local trivialisations of $E$ and $F$, $\{\widetilde{a}^l_j\}_{l,j}$ are unique modulo $S^{-\infty}_{\loc}(x(O)\times\RR^m)$. The principal symbol of order $r\in\RR\cup\{-\infty\}$ is a surjective linear map $\boldsymbol{\sigma}^r:\Psi^r(M;E,F)\rightarrow S^r_{\loc}(T^*M;E,F)/S^{r-1}_{\loc}(T^*M;E,F)$ whose kernel is $\Psi^{r-1}(M;E,F)$. With $\{\widetilde{a}^l_j\}_{l,j}$ and $(\pi_{T^*M}^*(e'^j\otimes s_l))_{j,l}$ as above and $\kappa$ given by \eqref{loc-tri-ofb-undimatotcotbusks}, $a:= \kappa^*(\widetilde{a}^l_j) \pi_{T^*M}^*(e'^j\otimes s_l)\in S^r_{\loc}(T^*O;E_O,F_O)$ and $a\in\boldsymbol{\sigma}^r(A)|_{T^*O}$. If $A$ is properly supported and $A'\in\Psi^{r'}(M;F,H)$ is properly supported with $r'\in\RR\cup\{-\infty\}$ and $H$ a vector bundle over $M$, then $A'A\in\Psi^{r+r'}(M;E,H)$, $A'A$ is properly supported and $\boldsymbol{\sigma}^{r'+r}(A'A)=\boldsymbol{\sigma}^{r'}(A')\boldsymbol{\sigma}^r(A)$.\\
\indent Let $r\in\RR$. We say that $A\in\Psi^r(M;E,F)$ is of order $r'\in\RR\cup\{-\infty\}$, $r'\leq r$, at $(p,\xi)\in T^*M\backslash0$ if there are a chart $(O,x)$ about $p$ over which $E$ and $F$ trivialise and an open cone $V\subseteq \RR^m$ containing $\left(\xi(\frac{\partial}{\partial x^1}|_p),\ldots,\xi(\frac{\partial}{\partial x^m}|_p)\right)$ such that $\widetilde{a}^l_j\in S^{r'}_{\operatorname{loc}}(x(O)\times V)$, $j=1,\ldots,k$, $l=1\ldots, k'$, where $\{\widetilde{a}^l_j\}_{l,j}$ are symbols of the coordinate representation of $A$. The definition is independent of local coordinates, local trivialisations of $E$ and $F$ and of the choice of $\{\widetilde{a}^l_j\}_{l,j}$. We say that $A$ is of order $r'$ in a conic subset $L$ of $T^*M\backslash0$ if $A$ is of order $r'$ at every point of $L$; if $L=T^*M\backslash 0$, then $A\in\Psi^{r'}(M;E,F)$. We say that $A\in\Psi^r(M;E,F)$ is polyhomogeneous of order $r$ if for every point there is a chart $(O,x)$ over which both $E$ and $F$ trivialise such that $\{\widetilde{a}^l_j\}_{l,j}\subseteq S^r_{\loc,\phg}(x(O)\times\RR^m)$ where $\{\widetilde{a}^l_j\}_{l,j}$ are symbols of the coordinate representation of $A$. When this is the case, this is valid on all charts on $M$ over which both $E$ and $F$ locally trivialise and for any local trivialisations of $E$ and $F$. The space of polyhomogeneous $\Psi$DOs of order $r$ is denote by $\Psi^r_{\phg}(M;E,F)$. If $A$ and $A'\in\Psi^{r'}_{\phg}(M;F,H)$ are properly supported, then $A'A\in\Psi^{r'+r}_{\phg}(M;E,H)$.\\
\indent Assume now that $E$ and $F$ have the same rank $k$ and $r\in\RR$. The operator $A\in\Psi^r(M;E,F)$ is said to be elliptic if there is $b\in S^{-r}_{\operatorname{loc}}(T^*M;F,E)$ such that $ba-\operatorname{I}_E\in S^{-1}_{\loc}(M;E,E)$ and $ab-\operatorname{I}_F\in S^{-1}_{\loc}(M;F,F)$ for some (or, equivalently all) $a\in\boldsymbol{\sigma}^r(A)$; here $\operatorname{I}_E\in\Gamma(\pi_{T^*M}^*L(E,E))$ is given by $\operatorname{I}_E(p,\xi):=((p,\xi),\operatorname{Id})\in\{(p,\xi)\}\times\mathcal{L}(E_p,E_p)$ and $\operatorname{I}_F$ is defined analogously. If in addition $A$ is properly supported, then there exists a properly supported $A'\in\Psi^{-r}(M;F,E)$, called a parametrix for $A$, such that $A'A-\operatorname{Id}\in\Psi^{-\infty}(M;E,E)$ and $AA'-\operatorname{Id}\in\Psi^{-\infty}(M;F,F)$ and $A'$ is unique modulo $\Psi^{-\infty}(M;F,E)$. The operator $A\in\Psi^r(M;E,F)$ is said to be non-characteristic at $(p,\xi)\in T^*M\backslash 0$ if there is $b\in S^{-r}_{\operatorname{loc}}(T^*M;F,E)$ such that both $ba-\operatorname{I}_E$ and $ab-\operatorname{I}_F$ are of order $-1$ in a conic neighbourhood of $(p,\xi)$, for some (or, equivalently all) $a\in\boldsymbol{\sigma}^r(A)$. The set of characteristic points of $A$ is denoted by $\Char A$; it is a closed conic subset of $T^*M\backslash 0$. The operator $A$ is elliptic if and only if $\Char A=\emptyset$. We will need the following result; its proof is analogous to the proof of \cite[Theorem 18.1.24', p. 88]{hor2} and we omit it.

\begin{lemma}\label{lemma-for-parmetrixalsmforsonlyindirc}
Let $A\in\Psi^r(M;E,F)$ be properly supported and $(p,\xi)\not\in\Char A$. Then there is a properly supported $A'\in\Psi^{-r}(M;F,E)$ such that both $A'A-\operatorname{Id}$ and $AA'-\operatorname{Id}$ are of order $-\infty$ in a conic neighbourhood of $(p,\xi)$.
\end{lemma}

\section{Spaces of distributions with Sobolev wave front in a fixed conic set: topological properties, characterisation of compact sets and pullback by smooth maps}\label{Sec-comp}

Let $U$ be an open subset of $\RR^n$. Following H\"ormander \cite{hor1}, for $u\in\DD'(U)$, we define the Sobolev wave front set $WF^r(u)$ of order $r\in\RR$ as follows. The point $(x,\xi)\in U\times (\RR^n\backslash\{0\})$ does not belong $WF^r(u)$ if there are an open cone $V\subseteq \RR^n$ containing $\xi$ and $\varphi\in\DD(U)$ satisfying $\varphi(x)\neq0$ such that $\|\langle\cdot\rangle^r \mathcal{F}(\varphi u)\|_{L^2(V)}<\infty$; $WF^r(u)$ is a closed conic subset of $U\times (\RR^n\backslash\{0\})$ (see \cite[Definition 8.2.5, p. 188, and Proposition 8.2.6, p. 189]{hor1}). The original definition is due to Duistermaat and H\"ormander \cite[p. 201]{dui-hor} and is via pseudo-differential operators but this amounts to the same thing in view of \cite[Proposition 8.2.6, p. 189]{hor1}. Given a closed conic subset $L$ of $U\times(\RR^n\backslash\{0\})$, we define
\begin{equation}\label{def-spa-with-wfinconcsl}
\DD'^r_L(U):=\{u\in\DD'(U)\,|\, WF^r(u)\subseteq L\}.
\end{equation}
Notice that if $L=\emptyset$, then \eqref{def-spa-with-wfinconcsl} is $H^r_{\loc}(U)$ and when $L=U\times(\RR^n\backslash\{0\})$, \eqref{def-spa-with-wfinconcsl} is the whole $\DD'(U)$. Our goal in this section is to introduce a useful locally convex topology on $\DD'^r_L(U)$ which will make the pullback by smooth maps continuous, the above identities topological and which will be compatible with the topology on $\DD'_L(U)$ form \cite{BD} in the sense that $\DD'_L(U)=\bigcap_{r\in\RR}\DD'^r_L(U)$ topologically (the right hand side should be read as a projective limit; see Proposition \ref{hor-spa-for-fixdsmwavfrsw} below).

\subsection{The Sobolev compactness wave front set and the topology of \texorpdfstring{$\DD'^r_L(U)$}{D'rL(U)}}

Let $B$ be a bounded subset of $\EE'(U)$; notice that $\bigcup_{u\in B}\supp u$ is relatively compact in $U$. For $r\in\RR$, we define the set $\Sigma_c^r(B)\subseteq \RR^n\backslash\{0\}$ as follows. The point $\eta\in\RR^n\backslash\{0\}$ \textit{does not} belongs to $\Sigma_c^r(B)$ if there is an open cone $V\subseteq \RR^n\backslash\{0\}$ containing $\eta$ such that
\begin{equation}
\sup_{u\in B} \int_{\xi\in V,\, |\xi|>R}|\mathcal{F}u(\xi)|^2\langle\xi\rangle^{2r}d\xi\rightarrow 0,\quad \mbox{as}\quad R\rightarrow \infty.
\end{equation}
Clearly, $\Sigma^r_c(B)$ is a closed cone in $\RR^n\backslash\{0\}$. Employing classical arguments as in the proof of \cite[Lemma 8.1.1, p. 253]{hor}, it is straightforward to show the following result (see the proof of Lemma \ref{lem-for-cha-wfofset-comwfl} below for a similar type of argument).

\begin{lemma}\label{lemma-for-inc-of-sig}
With $r$ and $B$ as above, it holds that $\Sigma^r_c(\varphi B)\subseteq \Sigma^r_c(B)$, $\varphi\in\DD(U)$.
\end{lemma}

Let now $B$ be a bounded subset of $\DD'(U)$. For every $x\in U$, we define
\begin{equation}
\Sigma_{c,x}^r(B):=\bigcap_{\varphi\in\DD(U),\, \varphi(x)\neq0}\Sigma^r_c(\varphi B).
\end{equation}
Clearly, $\Sigma_{c,x}^r(B)$ is a closed cone in $\RR^n\backslash\{0\}$. The compactness of the unit sphere together with Lemma \ref{lemma-for-inc-of-sig} yield the following result (cf. \cite[p. 253]{hor}).

\begin{lemma}\label{lemma-for-clos-wavef-se}
Let $r\in\RR$, $x\in U$ and $B$ a bounded subset of $\DD'(U)$. If $\Sigma_{c,x}^r(B)\subseteq V$ for some open cone $V\subseteq \RR^n\backslash\{0\}$, then there exists an open neighbourhood $U'\subseteq U$ of $x$ such that $\Sigma^r_c(\varphi B)\subseteq V$, $\varphi\in \DD(U')$.
\end{lemma}

\begin{definition}\label{def-com-wav-forsetdefonbds}
Let $r\in\RR$ and let $B$ be a bounded subset of $\DD'(U)$. The \textit{Sobolev compactness wave front set of order $r$} of $B$ is
\begin{equation*}
WF^r_c(B):=\{(x,\xi)\in U\times (\RR^n\backslash\{0\})\, |\, \xi\in\Sigma^r_{c,x}(B)\}.
\end{equation*}
\end{definition}

When $r=0$, this was introduced by G\'erard \cite[Definition 1.1]{G1} as a bookkeeping device to keep track of the frequencies where a weakly convergent sequence in $L^2_{\loc}(U)$ fails to have a convergent subsequence. We will show that $WF^r_c(\cdot)$ is an effective tool for characterising the relatively compact subsets of $\DD'^r_L(U)$ once we introduce the locally convex topology on it. For the moment, we point out that $WF^r_c(B)$ is a closed conic subset of $U\times(\RR^n\backslash\{0\})$ in view of Lemma \ref{lemma-for-clos-wavef-se}, and, if $B$ has only one element $u$ (i.e., $B=\{u\}$), then $WF^r_c(B)$ is exactly the Sobolev wave front set $WF^r(u)$ of $u$ of order $r$; in this case we will also simply write $\Sigma^r_x(u)$. Furthermore, Lemma \ref{lemma-for-inc-of-sig} implies that $WF^r_c(a_1B_1+a_2B_2)\subseteq WF^r_c(B_1)\cup WF^r_c(B_2)$ for any $a_1,a_2\in\mathcal{C}^{\infty}(U)$ and any bounded subsets $B_1$ and $B_2$ of $\DD'(U)$; if in addition $B_1\subseteq B_2$ then $WF^r_c(B_1)\subseteq WF^r_c(B_2)$.

\begin{lemma}\label{lem-for-cha-wfofset-comwfl}
Let $B$ be a bounded subset of $\DD'(U)$ and let $L$ be a closed conic subset of $U\times (\RR^n\backslash\{0\})$. Then $WF^r_c(B)\subseteq L$ if and only if for every $\varphi\in\DD(U)$ and every closed cone $V\subseteq \RR^n$ satisfying $(\supp\varphi\times V)\cap L=\emptyset$ it holds that
\begin{equation}\label{cond-for-wfset-comto-be-equss}
\sup_{u\in B}\int_{V,\, |\xi|>R} |\mathcal{F}(\varphi u)(\xi)|^2\langle \xi\rangle^{2r} d\xi\rightarrow 0,\quad \mbox{as}\quad R\rightarrow\infty.
\end{equation}
\end{lemma}

\begin{proof} If $L=U\times (\RR^n\backslash\{0\})$, then when $\varphi\in\DD(U)\backslash\{0\}$ it holds that $V\subseteq \{0\}$ and the claim in the lemma is trivial. Assume that $L^c\neq \emptyset$. If \eqref{cond-for-wfset-comto-be-equss} is satisfied for every $\varphi\in\DD(U)$ and $V$ as in the lemma, then clearly $WF^r_c(B)\subseteq L$. Assume now that $WF^r_c(B)\subseteq L$ and let $\varphi\in\DD(U)\backslash\{0\}$ and the closed cone $V\subseteq \RR^n$, $V\backslash\{0\}\neq \emptyset$, are such that $(\supp\varphi\times V)\cap L=\emptyset$. Set $K:=\supp\varphi$. For each $(y,\eta)\in K\times V$ there is $\varphi_{(y,\eta)}\in\DD(U)$ with $\varphi_{(y,\eta)}(y)\neq0$ such that $\eta\not\in\Sigma^r_c(\varphi_{(y,\eta)} B)$. In view of Lemma \ref{lemma-for-inc-of-sig}, we can assume that $\varphi_{(y,\eta)}$ is nonnegative and $\varphi_{(y,\eta)}(y)=1$. There is an open cone $V'_{(y,\eta)}\subseteq \RR^n\backslash\{0\}$ containing $\eta$ such that
$$
\sup_{u\in B}\int_{V'_{(y,\eta)},\, |\xi|>R} |\mathcal{F}(\varphi_{(y,\eta)} u)(\xi)|^2\langle \xi\rangle^{2r} d\xi\rightarrow 0,\quad \mbox{as}\quad R\rightarrow\infty.
$$
Pick an open cone $V_{(y,\eta)}\subseteq \RR^n\backslash\{0\}$ such that $\eta\in V_{(y,\eta)}$ and $\overline{V_{(y,\eta)}}\backslash\{0\}\subseteq V'_{(y,\eta)}$. Since $V\cap\mathbb{S}^{n-1}$ is compact, we infer that for each $y\in K$ there are a finite number of open cones $V_{y,l}:=V_{(y,\eta^{(l)})}$, $l=1,\ldots, k_y$, whose union covers $V$. For each $y\in K$, set $\varphi_y:=\varphi_{(y,\eta^{(1)})}\cdot\ldots\cdot\varphi_{(y,\eta^{(k_y)})}$ and notice that $\varphi_y\in\DD(U)$ and $\varphi_y(y)=1$. For every $y\in K$, denote by $O_y$ the open set $\{x\in U\,|\, \varphi_y(x)>1/2\}$. As $K$ is compact, there are a finite number of open sets $O_j:=O_{y^{(j)}}$, $j=1,\ldots, m$, whose union covers $K$. Set $\varphi_j:=\varphi_{y^{(j)}}$, $j=1,\ldots,m$, and $\psi=\varphi_1+\ldots+\varphi_m$. Clearly $\psi\in\DD(U)$ and $\psi>1/2$ on a neighbourhood of $K$. Hence $\psi_j:=(\varphi/\psi)\varphi_j\in\DD(U)$ and $\varphi=\psi_1+\ldots+\psi_m$. For each $j\in \{1,\ldots,m\}$, set $k_j:=k_{y^{(j)}}$ and additionally $V_{j,l}:=V_{y^{(j)},l}$, $V'_{j,l}:=V'_{(y^{(j)},\eta^{(l)})}$ and $\varphi_{j,l}:=\varphi_{(y^{(j)},\eta^{(l)})}$, $l=1,\ldots,k_j$. We infer
\begin{equation*}
\left(\int_{V,\, |\xi|>R} |\mathcal{F}(\varphi u)(\xi)|^2\langle \xi\rangle^{2r} d\xi\right)^{1/2}\\
\leq \sum_{j=1}^m \sum_{l=1}^{k_j}\left(\int_{V_{j,l},\, |\xi|>R} |\mathcal{F}(\psi_j u)(\xi)|^2\langle \xi\rangle^{2r} d\xi\right)^{1/2}.
\end{equation*}
We show that for every $j\in\{1,\ldots,m\}$ and $l\in\{1,\ldots,k_j\}$,
\begin{equation}\label{sho-for-dec-sum-int-for-lemcharc}
\sup_{u\in B}\int_{V_{j,l},\, |\xi|\geq R} |\mathcal{F}(\psi_j u)(\xi)|^2\langle \xi\rangle^{2r} d\xi\rightarrow 0,\quad \mbox{as}\quad R\rightarrow\infty,
\end{equation}
which will complete the proof. Fix such $j$ and $l$. Set\footnote{Here and throughout the rest of the article we employ the principle of vacuous (empty) products; i.e., $\prod_{j=1}^0q_j=\prod_{j\in\emptyset} q_j=1$.} $\chi_{j,l}:=(\varphi/\psi)\prod_{l'\neq l}\varphi_{j,l'}\in \DD(U)$ and notice that $\psi_j=\chi_{j,l}\varphi_{j,l}$. There is $c\in(0,1)$ such that
\begin{equation}\label{ine-for-cones-subset-ofothe}
\{\eta\in\RR^n\,|\,\exists \xi\in V_{j,l}\,\, \mbox{such that}\,\,|\xi-\eta|\leq c|\xi|\}\subseteq V'_{j,l}.
\end{equation}
Since $\mathcal{F}(\psi_j u)=(2\pi)^{-n}\mathcal{F}(\chi_{j,l})*\mathcal{F}(\varphi_{j,l}u)$, the Minkowski integral inequality gives
\begin{align*}
&\left(\int_{V_{j,l},\, |\xi|\geq R} |\mathcal{F}(\psi_j u)(\xi)|^2\langle \xi\rangle^{2r} d\xi\right)^{1/2}\\
&\leq \frac{1}{(2\pi)^n}\int_{\RR^n_{\eta}}|\mathcal{F}\chi_{j,l}(\eta)|\left(\int_{V_{j,l},\, |\xi|\geq R} |\mathcal{F}(\varphi_{j,l} u)(\xi-\eta)|^2\langle \xi\rangle^{2r}d\xi\right)^{1/2}d\eta\\
&\leq  \frac{2^{|r|}}{(2\pi)^n}\int_{\RR^n_{\eta}}|\mathcal{F}\chi_{j,l}(\eta)|\langle\eta\rangle^{|r|}\left(\int_{V_{j,l},\, |\xi|\geq R} |\mathcal{F}(\varphi_{j,l} u)(\xi-\eta)|^2\langle \xi-\eta\rangle^{2r}d\xi\right)^{1/2}d\eta\\
&\leq 2^{|r|}(2\pi)^{-n}(I_1+I_2),
\end{align*}
with
\begin{align*}
I_1&:=\int_{\RR^n_{\eta}}|\mathcal{F}\chi_{j,l}(\eta)|\langle\eta\rangle^{|r|}\left(\int_{\substack{\xi\in V_{j,l}\\ |\xi|\geq \max\{|\eta|/c,R\}}} |\mathcal{F}(\varphi_{j,l} u)(\xi-\eta)|^2\langle \xi-\eta\rangle^{2r}d\xi\right)^{1/2}d\eta,\\
I_2&:=\int_{\RR^n_{\eta}}|\mathcal{F}\chi_{j,l}(\eta)|\langle\eta\rangle^{|r|}\left(\int_{\substack{\xi\in V_{j,l}\\ R\leq |\xi|<|\eta|/c}} |\mathcal{F}(\varphi_{j,l} u)(\xi-\eta)|^2\langle \xi-\eta\rangle^{2r}d\xi\right)^{1/2}d\eta.
\end{align*}
We change variables in the inner integral in $I_1$ to obtain
\begin{align*}
I_1&= \int_{\RR^n_{\eta}}|\mathcal{F}\chi_{j,l}(\eta)|\langle\eta\rangle^{|r|}\left(\int_{\substack{\xi\in V_{j,l}-\{\eta\}\\ |\xi+\eta|\geq \max\{|\eta|/c,R\}}} |\mathcal{F}(\varphi_{j,l} u)(\xi)|^2\langle \xi\rangle^{2r}d\xi\right)^{1/2}d\eta\\
&\leq \int_{\RR^n_{\eta}}|\mathcal{F}\chi_{j,l}(\eta)|\langle\eta\rangle^{|r|}\left(\int_{V'_{j,l},\,|\xi|\geq(1-c)R} |\mathcal{F}(\varphi_{j,l} u)(\xi)|^2\langle \xi\rangle^{2r}d\xi\right)^{1/2}d\eta\\
&= \|\langle\cdot \rangle^{|r|}\mathcal{F}\chi_{j,l}\|_{L^1(\RR^n)} \left(\int_{V'_{j,l},\,|\xi|\geq(1-c)R} |\mathcal{F}(\varphi_{j,l} u)(\xi)|^2\langle \xi\rangle^{2r}d\xi\right)^{1/2},
\end{align*}
where in the inequality we employed the fact
\begin{equation}\label{inc-for-ine-for-int-iskrt}
\{\xi\in V_{j,l}-\{\eta\}\,|\, |\xi+\eta|\geq \max\{|\eta|/c,R\}\}\subseteq \{\xi\in V'_{j,l}\,|\, |\xi|\geq (1-c)R\}.
\end{equation}
To verify it, first notice that \eqref{ine-for-cones-subset-ofothe} gives $\{\xi\in V_{j,l}-\{\eta\}\,|\, |\xi+\eta|\geq |\eta|/c\}\subseteq V'_{j,l}$. When $\xi$ belong to the left-hand side of \eqref{inc-for-ine-for-int-iskrt}, we have $|\xi|\geq (c^{-1}-1)|\eta|$ and consequently
$$
|\xi|\geq R-|\eta|\geq R-c|\xi|/(1-c),\quad \mbox{whence}\quad |\xi|\geq (1-c)R,
$$
which proves \eqref{inc-for-ine-for-int-iskrt}. To bound $I_2$, notice that since $B$ is bounded in $\DD'(U)$ there are $C'_1,s>0$ such that $|\mathcal{F}(\varphi_{j,l} u)(\xi)|\leq C'_1\langle \xi\rangle^s$, $\xi\in\RR^n$, $u\in B$. Hence
\begin{align*}
I_2&\leq C'_1\int_{|\eta|>cR}|\mathcal{F}\chi_{j,l}(\eta)|\langle\eta\rangle^{|r|}\left(\int_{\substack{\xi\in V_{j,l}\\ R\leq|\xi|<|\eta|/c}} \langle \xi-\eta\rangle^{2(s+|r|)}d\xi\right)^{1/2}d\eta\\
&\leq C'_1 (1+c^{-1})^{s+|r|}c^{-n}\|\langle\cdot\rangle^{-n}\|_{L^2(\RR^n)}\int_{|\eta|>cR}| \mathcal{F}\chi_{j,l}(\eta)|\langle\eta\rangle^{2|r|+s+n}d\eta.
\end{align*}
These bounds for $I_1$ and $I_2$ imply the validity of \eqref{sho-for-dec-sum-int-for-lemcharc} and the proof of the lemma is complete.
\end{proof}

Applying the lemma when $B$ is a singleton, we deduce the following corollary.

\begin{corollary}\label{lemma-for-sem-wav-fr-set-spa}
The distribution $u\in\DD'(U)$ belongs to $\DD'^r_L(U)$ if and only if for every $\varphi\in\DD(U)$ and every closed cone $V\subseteq \RR^n$ satisfying $(\supp \varphi \times V)\cap L=\emptyset$ it holds that
\begin{equation*}
\int_V |\mathcal{F}(\varphi u)(\xi)|^2\langle \xi\rangle^{2r} d\xi<\infty.
\end{equation*}
\end{corollary}

With the help of the corollary, we define the following set of seminorms on $\DD'^r_L(U)$:
\begin{equation*}
\mathfrak{p}_{r;\varphi,V}(u):=\left(\int_V |\mathcal{F}(\varphi u)(\xi)|^2\langle \xi\rangle^{2r} d\xi\right)^{1/2},
\end{equation*}
where $\varphi\in \DD(U)$ and $V$ are as in Corollary \ref{lemma-for-sem-wav-fr-set-spa}. We equip $\DD'^r_L(U)$ with the locally convex topology induced by all continuous seminorms on $\DD'(U)$ together with all seminorms $\mathfrak{p}_{r;\varphi,V}$ where $\varphi$ and $V$ are as in Corollary \ref{lemma-for-sem-wav-fr-set-spa}. Clearly, $H^r_{\loc}(U)\subseteq\DD'^r_L(U)\subseteq \DD'(U)$ continuously.

\begin{remark}
If $L=U\times (\RR^n\backslash\{0\})$, then $\DD'^r_L(U)=\DD'(U)$ topologically since in this case $V$ can only be $\{0\}$ or $\emptyset$ when $\varphi$ is not the zero function and hence $\mathfrak{p}_{r;\varphi,V}=0$.\\
\indent When $L=\emptyset$, $\DD'^r_{\emptyset}(U)= H^r_{\loc}(U)$ topologically since one can take $V=\RR^n$ in this case.
\end{remark}

All of the important topological properties of $\DD'^r_L(U)$ will follow from the following proposition. To set the stage, we need the following objects. Assume that $L^c\neq \emptyset$. Pick a countable dense subset $\{(x^{(j)},\xi^{(j)})\}_{j\in\ZZ_+}$ of $L^c$ and denote $\omega^{(j)}:=\xi^{(j)}/|\xi^{(j)}|\in\mathbb{S}^{n-1}$, $j\in\ZZ_+$. Set\footnote{Here we employ $\operatorname{dist}(x,\emptyset)=\infty$, for any element $x$.}
$$
s_j:=\min\{1,\operatorname{dist}(x^{(j)},\partial U),\operatorname{dist}((x^{(j)},\omega^{(j)}),L)\}>0.
$$
Clearly, $(\overline{B(x^{(j)},s_j/2)}\times \overline{B(\omega^{(j)},s_j/2)})\cap L=\emptyset$ and $B(x^{(j)},s_j)\subseteq U$, $j\in\ZZ_+$. For each $j,k\in\ZZ_+$, we set:
\begin{gather*}
O_{j,k}:=B(x^{(j)},s_j/(5k)),\quad O'_{j,k}:=B(x^{(j)},s_j/(4k)),\quad O''_{j,k}:=B(x^{(j)},s_j/(3k));\\
V_{j,k}:=\RR_+ B(\omega^{(j)},s_j/(5k)),\quad V'_{j,k}:=\RR_+ B(\omega^{(j)},s_j/(4k)),\quad V''_{j,k}:=\RR_+ B(\omega^{(j)},s_j/(3k)).
\end{gather*}
Notice that
$$
\overline{O_{j,k}}\times \overline{V_{j,k}}\subseteq O'_{j,k}\times (V'_{j,k}\cup\{0\}),\quad \overline{O'_{j,k}}\times \overline{V'_{j,k}}\subseteq O''_{j,k}\times (V''_{j,k}\cup\{0\}),\quad (\overline{O''_{j,k}}\times \overline{V''_{j,k}})\cap L=\emptyset.
$$
For each $j,k\in\ZZ_+$, pick $\phi_{j,k}\in\DD(\RR^n)$ and $\widetilde{\phi}_{j,k}\in\mathcal{C}^{\infty}(\RR^n)$ which satisfy the following conditions:
\begin{itemize}
\item[$(a)$] $0\leq \phi_{j,k}\leq 1$ and $0\leq \widetilde{\phi}_{j,k}\leq 1$;
\item[$(b)$] $\phi_{j,k}=1$ on $\overline{O_{j,k}}$ and $\supp\phi_{j,k}\subseteq O'_{j,k}$;
\item[$(c)$] $\widetilde{\phi}_{j,k}=1$ on $\overline{V}_{j,k}\backslash B(0,2)$ and $\supp\widetilde{\phi}_{j,k}\subseteq V'_{j,k}\backslash \overline{B(0,1)}$;
\item[$(d)$] $\|\partial^{\alpha} \widetilde{\phi}_{j,k}\|_{L^{\infty}(\RR^n)}<\infty$, $\alpha\in\NN^n$.
\end{itemize}

\begin{proposition}\label{pro-for-top-imbedingthforc}
Let $L$ be a closed conic subset of $U\times(\RR^n\backslash\{0\})$ satisfying $L^c\neq \emptyset$. Let $O_{j,k}$, $O'_{j,k}$, $O''_{j,k}$, $V_{j,k}$, $V'_{j,k}$, $V''_{j,k}$, $\phi_{j,k}$ and $\widetilde{\phi}_{j,k}$ be as above. Then the mapping
\begin{gather*}
\mathcal{I}:\DD'^r_L(U)\rightarrow \DD'(U)\times (L^2(\RR^n)^{\ZZ_+\times \ZZ_+}),\quad \mathcal{I}(u)=(u,\mathbf{f}_u),\,\, \mbox{where}\\
\mathbf{f}_u(j,k):=\langle\cdot\rangle^r \widetilde{\phi}_{j,k}\mathcal{F}(\phi_{j,k} u),\, j,k\in\ZZ_+,
\end{gather*}
is a well-defined topological imbedding with closed image.
\end{proposition}

\begin{proof} The proof that $\mathcal{I}$ is a well-defined continuous injection is straightforward and we omit it (cf. Corollary \ref{lemma-for-sem-wav-fr-set-spa}). We now prove that $\mathcal{I}$ is open mapping onto its image. It suffices to show that for each $\varphi\in\DD(U)\backslash\{0\}$ and every closed cone $V\subseteq \RR^n$, $V\backslash\{0\}\neq \emptyset$, satisfying $(\supp \varphi \times V)\cap L=\emptyset$, there is a finite $J\subseteq \ZZ_+\times \ZZ_+$, a continuous seminorm $\mathfrak{p}$ on $\DD'(U)$ and $C>0$ such that
\begin{equation}\label{ine-sem-for-openmap=forthi}
\mathfrak{p}_{r;\varphi,V}(u)\leq C\mathfrak{p}(u)+C\sum_{(j,k)\in J}\|\langle \cdot\rangle^r\widetilde{\phi}_{j,k}\mathcal{F}(\phi_{j,k} u)\|_{L^2(\RR^n)},\quad u\in\DD'^r_L(U).
\end{equation}
Fix such $\varphi$ and $V$. There is $0<\varepsilon<1$ such that
$$
\big((\supp\varphi+\overline{B(0,\varepsilon)})\times ((V\cap \mathbb{S}^{n-1})+\overline{B(0,\varepsilon)})\big)\cap L=\emptyset\quad \mbox{and}\quad \supp\varphi+\overline{B(0,\varepsilon)}\subseteq U.
$$
Define the closed cones $V'$ and $V''$ by
$$
V'':=\RR_+\big((V\cap \mathbb{S}^{n-1})+\overline{B(0,\varepsilon)}\big)\cup\{0\}\quad \mbox{and}\quad V':=\RR_+\big((V\cap \mathbb{S}^{n-1})+\overline{B(0,\varepsilon/15)}\big)\cup\{0\}
$$
and the compact sets $K''$, $K'$ and $K$ by
$$
K'':=\supp\varphi+\overline{B(0,\varepsilon)},\quad K':=\supp\varphi+\overline{B(0,\varepsilon/15)}\quad \mbox{and}\quad K:=\supp\varphi.
$$
Notice that
$$
K\times V\subseteq K'\times V'\subseteq K''\times V''\subseteq U\times V'' \quad\mbox{and}\quad (K''\times V'')\cap L=\emptyset.
$$
Set $\widetilde{J}:=\{(j,k)\in\ZZ_+\times \ZZ_+\,|\,  k\geq 15/\varepsilon,\, (O_{j,k}\times B(\omega^{(j)},s_j/(5k)))\cap (K\times V)\neq\emptyset\}$. We claim that
\begin{equation}\label{inc-cl-for-the-sse}
\widetilde{J}\neq \emptyset\quad\mbox{and}\quad K\times (V\backslash\{0\})\subseteq \bigcup_{(j,k)\in \widetilde{J}} O_{j,k}\times V_{j,k}\subseteq\bigcup_{(j,k)\in \widetilde{J}} \overline{O'_{j,k}}\times \overline{V'_{j,k}}\subseteq K'\times V'.
\end{equation}
We show $\widetilde{J}\neq \emptyset$ and the first inclusion simultaneously. Let $(x,\xi)\in K\times (V\backslash\{0\})$ and set $\omega:=\xi/|\xi|\in V\cap \mathbb{S}^{n-1}$ and $s:=\min\{1,\operatorname{dist}(x,\partial U), \operatorname{dist}((x,\omega),L)\}>0$. Pick any integer $k\geq 15/\varepsilon$ (the condition $k\geq 15/\varepsilon$ is important for the last inclusion in \eqref{inc-cl-for-the-sse} and the subsequent part of the proof) and set $m:=10k+4$. Since $(\overline{B(x,s/m)}\times \overline{B(\omega,s/m)})\cap L=\emptyset$, there is $j\in\ZZ_+$ such that $|x-x^{(j)}|<s/m$ and $|\omega-\xi^{(j)}|<s/m$. The second inequality yields
$$
|1-|\xi^{(j)}||=||\omega|-|\xi^{(j)}||\leq |\omega-\xi^{(j)}|<s/m
$$
which implies $|\xi^{(j)}|> 3/4$. It also gives $||\xi^{(j)}|^{-1}-1|< s/(m|\xi^{(j)}|)<4s/(3m)$. Since
\begin{align*}
\operatorname{dist}((x,\omega),L)&\leq |(x,\omega)-(x^{(j)},\omega^{(j)})|+\operatorname{dist}((x^{(j)},\omega^{(j)}),L)\quad \mbox{and}\\
\operatorname{dist}(x,\partial U)&\leq |x-x^{(j)}|+\operatorname{dist}(x^{(j)},\partial U),
\end{align*}
we infer
\begin{align*}
s&\leq |x-x^{(j)}|+|\omega-\omega^{(j)}|+s_j\leq \frac{s}{m}+\left|\omega-\frac{\omega}{|\xi^{(j)}|}\right|+\left|\omega^{(j)}-\frac{\omega}{|\xi^{(j)}|}\right|+s_j\\
&= \frac{s}{m} +\left|1-\frac{1}{|\xi^{(j)}|}\right|+\frac{|\xi^{(j)}-\omega|}{|\xi^{(j)}|}+s_j< \frac{4s}{m}+s_j
\end{align*}
and hence $s<ms_j/(m-4)$. Consequently,
$$
\left|\frac{\omega}{|\xi^{(j)}|}-\frac{\xi^{(j)}}{|\xi^{(j)}|}\right|<\frac{s}{m|\xi^{(j)}|}<\frac{2s}{m}<\frac{2s_j}{m-4}= \frac{s_j}{5k}.
$$
We deduce $\omega/|\xi^{(j)}|\in B(\omega^{(j)},s_j/5k)\cap V$. We also have $|x-x^{(j)}|<s/m<s_j/(5k)$ and hence $x\in O_{j,k}\cap K$. Thus, $\widetilde{J}\neq\emptyset$. The above considerations also immediately yield $(x,\xi)\in O_{j,k}\times V_{j,k}$ which shows the first inclusion in \eqref{inc-cl-for-the-sse}.\\
\indent To prove the last inclusion in \eqref{inc-cl-for-the-sse}, let $(x,\xi)\in \overline{O'_{j,k}}\times\overline{V'_{j,k}}$ for some $(j,k)\in\widetilde{J}$. Then $|x-x^{(j)}|\leq s_j/(4k)$ and there are $t\geq 0$ and $\xi'$ such that $\xi=t \xi'$ and $|\xi'-\omega^{(j)}|\leq s_j/(4k)$. As $(j,k)\in\widetilde{J}$, there are $\widetilde{x}\in K$ and $\widetilde{\xi}\in V\backslash\{0\}$ such that $|\widetilde{x}-x^{(j)}|<s_j/(5k)$ and $|\widetilde{\xi}-\omega^{(j)}|<s_j/(5k)$. Similarly as before, the last inequality immediately gives $|\widetilde{\xi}|> 1/2$. These inequalities also imply $|x-\widetilde{x}|\leq s_j/(4k)+s_j/(5k)<s_j/(2k)\leq \varepsilon/30$ and hence $x\in K'$. Similarly, $|\xi'-\widetilde{\xi}|<\varepsilon/30$ and hence $|\xi'/|\widetilde{\xi}|-\widetilde{\xi}/|\widetilde{\xi}||< \varepsilon/(30|\widetilde{\xi}|)< \varepsilon/15$. We infer $\xi'\in V'$. Consequently $\xi\in V'$ and the proof of the last inclusion in \eqref{inc-cl-for-the-sse} is complete.\\
\indent Now, a standard compactness argument yields the existence of a finite $J_0\subseteq \widetilde{J}$ such that
$$
K\times (V\backslash\{0\})\subseteq \bigcup_{(j,k)\in J_0} O_{j,k}\times V_{j,k}\subseteq\bigcup_{(j,k)\in \widetilde{J}} \overline{O'_{j,k}}\times \overline{V'_{j,k}}\subseteq K'\times V'.
$$
Set $\phi:=\sum_{(j,k)\in J_0}\phi_{j,k}\in \DD(\bigcup_{(j,k)\in J_0} O'_{j,k})$ and notice that $\phi\geq 1$ on $\bigcup_{(j,k)\in J_0}\overline{O_{j,k}}$. Hence $\varphi/\phi\in \DD(U)$. We also denote $\psi_{j,k}:=(\varphi/\phi)\phi_{j,k}\in \DD(O'_{j,k})$, $(j,k)\in J_0$, and notice that $\sum_{(j,k)\in J_0}\psi_{j,k}=\varphi$. As $V\backslash\{0\}\subseteq \bigcup_{(j,k)\in J_0} V_{j,k}$, we infer
\begin{align}
\mathfrak{p}_{r;\varphi,V}(u)&\leq \|\langle\cdot\rangle^r\mathcal{F}(\varphi u)\|_{L^2(B(0,2))}+\sum_{(j,k)\in J_0}\|\langle\cdot\rangle^r \widetilde{\phi}_{j,k}\mathcal{F}(\varphi u)\|_{L^2(\RR^n\backslash B(0,2))}\nonumber \\
&\leq \|\langle\cdot\rangle^r\mathcal{F}(\varphi u)\|_{L^2(B(0,2))}+\sum_{\substack{(j,k)\in J_0\\ (j',k')\in J_0}}\|\langle\cdot\rangle^r \widetilde{\phi}_{j,k}\mathcal{F}(\psi_{j',k'} u)\|_{L^2(\RR^n\backslash B(0,2))}.\label{sum-for-bou-ftoplimbed}
\end{align}
Let $(j,k),(j',k')\in J_0$ be arbitrary but fixed. The last inclusion in \eqref{inc-cl-for-the-sse} implies that there is $m\in\ZZ_+$ such that $(x^{(m)},\xi^{(m)})\in O_{j',k'}\times B(\omega^{(j)},s_j/(5k))$; hence
\begin{equation}\label{ine-for-xxi-for-bel}
|x^{(j')}-x^{(m)}|< s_{j'}/(5k')\leq \varepsilon/75,\quad |\omega^{(j)}-\xi^{(m)}|< s_{j}/(5k)\leq \varepsilon/75.
\end{equation}
The second bound yields $|1-|\xi^{(m)}||\leq |\omega^{(j)}-\xi^{(m)}|\leq \varepsilon/75$ and hence
\begin{equation}\label{ine-for-dif-ome-onsphdi}
|\omega^{(j)}-\omega^{(m)}|\leq |\omega^{(j)}-\xi^{(m)}|+\left|\xi^{(m)}-\frac{\xi^{(m)}}{|\xi^{(m)}|}\right|\leq \frac{\varepsilon}{75} +||\xi^{(m)}|-1|\leq \frac{2\varepsilon}{75}.
\end{equation}
We claim that
\begin{equation}\label{dist-tosetL}
\operatorname{dist}((x^{(m)},\xi^{(m)}),L)\geq \delta,\quad \mbox{with}\quad \delta:=67\varepsilon/150\in (0,1).
\end{equation}
Assume that \eqref{dist-tosetL} does not hold. There is $(y,\eta)\in L$ such that $|(x^{(m)},\xi^{(m)})-(y,\eta)|< \delta$ and hence
\begin{align*}
\left|(x^{(j')},\omega^{(j)})-(y,\eta)\right| \leq \left|(x^{(j')},\omega^{(j)})-(x^{(m)},\xi^{(m)})\right| +\left|(x^{(m)},\xi^{(m)})-(y,\eta)\right|< 2\varepsilon/75 +\delta.
\end{align*}
Consequently, $|x^{(j')}-y|\leq 2\varepsilon/75+\delta$ and $|\omega^{(j)}-\eta|\leq 2\varepsilon/75+\delta$. Since $x^{(j')}\in K'$, the first inequality implies that $y\in K''$. Pick $\widetilde{\xi}\in B(\omega^{(j)},s_j/(5k))\cap V$ (such $\widetilde{\xi}$ exists by the way we defined $\widetilde{J}$). Similarly as before, one easily shows that $|\widetilde{\xi}|>1/2$. As $|\widetilde{\xi}-\eta|\leq |\widetilde{\xi}-\omega^{(j)}|+|\omega^{(j)}-\eta|\leq \varepsilon/25+\delta$, we infer
$$
\left|\frac{\widetilde{\xi}}{|\widetilde{\xi}|}-\frac{\eta}{|\widetilde{\xi}|}\right|\leq \frac{\varepsilon}{25|\widetilde{\xi}|} +\frac{\delta}{|\widetilde{\xi}|}<\frac{2\varepsilon}{25} +\frac{67\varepsilon}{75}<\varepsilon.
$$
Thus, $\eta\in V''$. We deduce $(y,\eta)\in (K''\times V'')\cap L$ which is a contradiction. Hence, \eqref{dist-tosetL} holds true. Analogously, assuming $\operatorname{dist}(x^{(m)},\partial U)<\delta$ leads to $\operatorname{dist}(x^{(j')},\partial U)<\varepsilon/2$ which is in contradiction with $K''\subseteq U$; whence $\operatorname{dist}(x^{(m)},\partial U)\geq \delta$. Consequently $s_m\geq \delta$. We show that
\begin{equation}\label{two-inc-for-dec-int-imprs}
O'_{j',k'}\subseteq O_{m,1}\quad \mbox{and}\quad V'_{j,k}\subseteq V_{m,2}.
\end{equation}
The first inclusion follows from the following (cf. \eqref{ine-for-xxi-for-bel})
$$
x\in O'_{j',k'}\,\,\, \Rightarrow \,\,\, |x-x^{(m)}|\leq |x-x^{(j')}|+|x^{(j')}-x^{(m)}|< s_{j'}/(4k)+\varepsilon/75<\varepsilon/30< s_m/5.
$$
To verify the second inclusion, it suffices to show that $B(\omega^{(j)},s_j/(4k))\subseteq B(\omega^{(m)}, s_m/10)$. Let $\xi\in B(\omega^{(j)},s_j/(4k))$. In view of \eqref{ine-for-dif-ome-onsphdi}, we infer
$$
|\xi-\omega^{(m)}|\leq |\xi-\omega^{(j)}|+|\omega^{(j)}-\omega^{(m)}|< s_j/(4k)+2\varepsilon/75\leq \varepsilon/60+2\varepsilon/75 <s_m/10,
$$
which yields that $\xi\in B(\omega^{(m)}, s_m/10)$. The inclusions \eqref{two-inc-for-dec-int-imprs} imply $\widetilde{\phi}_{j,k}=\widetilde{\phi}_{j,k}\widetilde{\phi}_{m,2}$ on $\RR^n\backslash B(0,2)$ and $\psi_{j',k'}=\psi_{j',k'}\phi_{m,1}$ on $\RR^n$. We infer
$$
\|\langle\cdot\rangle^r \widetilde{\phi}_{j,k}\mathcal{F}(\psi_{j',k'} u)\|_{L^2(\RR^n\backslash B(0,2))} \leq \|\langle\cdot\rangle^r \widetilde{\phi}_{m,2}\mathcal{F}(\psi_{j',k'} \phi_{m,1}u)\|_{L^2(\RR^n\backslash B(0,2))},\quad u\in \DD'^r_L(U).
$$
Since $\overline{V'_{m,2}}\backslash\{0\}\subseteq V_{m,1}$, there is $0<c<1$ such that
\begin{equation}\label{sub-sfc-int-inefin}
\{\eta\in\RR^n\,|\, \exists \xi\in V'_{m,2}\,\, \mbox{such that}\,\, |\xi-\eta|\leq c|\xi|\}\subseteq V_{m,1}.
\end{equation}
Let $u\in \DD'^r_L(U)$ be arbitrary but fixed. As $\mathcal{F}(\psi_{j',k'} \phi_{m,1} u)=(2\pi)^{-n}\mathcal{F}\psi_{j',k'}*\mathcal{F}(\phi_{m,1} u)$, the Minkowski integral inequality and a change of variables give
\begin{align*}
\|&\langle\cdot\rangle^r \widetilde{\phi}_{m,2}\mathcal{F}(\psi_{j',k'} \phi_{m,1}u)\|_{L^2(\RR^n\backslash B(0,2))}\\
&\leq \frac{1}{(2\pi)^n}\int_{\RR^n_{\eta}}|\mathcal{F}\psi_{j',k'}(\eta)|\left(\int_{V'_{m,2}} |\mathcal{F}(\phi_{m,1} u)(\xi-\eta)|^2\langle \xi\rangle^{2r} d\xi\right)^{1/2}d\eta\\
&\leq \frac{2^{|r|}}{(2\pi)^n} \int_{\RR^n_{\eta}}|\mathcal{F}\psi_{j',k'}(\eta)|\langle \eta\rangle^{|r|}\left(\int_{\xi\in V'_{m,2}-\{\eta\}} |\mathcal{F}(\phi_{m,1} u)(\xi)|^2\langle \xi\rangle^{2r} d\xi\right)^{1/2}d\eta\\
&\leq 2^{|r|}(2\pi)^{-n}(I_1+I_2),
\end{align*}
where
\begin{align*}
I_1&:= \int_{\RR^n_{\eta}}|\mathcal{F}\psi_{j',k'}(\eta)|\langle \eta\rangle^{|r|}\left(\int_{\substack{\xi\in V'_{m,2}-\{\eta\}\\ |\xi+\eta|\geq |\eta|/c}} |\mathcal{F}(\phi_{m,1} u)(\xi)|^2\langle \xi\rangle^{2r} d\xi\right)^{1/2}d\eta,\\
I_2&:=\int_{\RR^n_{\eta}}|\mathcal{F}\psi_{j',k'}(\eta)|\langle \eta\rangle^{|r|}\left(\int_{\substack{\xi\in V'_{m,2}-\{\eta\}\\ |\xi+\eta|< |\eta|/c}} |\mathcal{F}(\phi_{m,1} u)(\xi)|^2\langle \xi\rangle^{2r} d\xi\right)^{1/2}d\eta.
\end{align*}
To estimate $I_1$ notice that \eqref{sub-sfc-int-inefin} yields $\{\xi\in V'_{m,2}-\{\eta\}\,|\, |\xi+\eta|\geq |\eta|/c\}\subseteq V_{m,1}$ and hence
\begin{align*}
I_1&\leq \|\langle \cdot\rangle^{|r|}\mathcal{F}\psi_{j',k'}\|_{L^1(\RR^n)} \left(\int_{V_{m,1}} |\mathcal{F}(\phi_{m,1} u)(\xi)|^2\langle \xi\rangle^{2r} d\xi\right)^{1/2}\\
&\leq \|\langle \cdot\rangle^{|r|}\mathcal{F}\psi_{j',k'}\|_{L^1(\RR^n)}\left(\|\langle\cdot\rangle^r \mathcal{F}(\phi_{m,1} u)\|_{L^2(B(0,2))} + \|\langle\cdot\rangle^r\widetilde{\phi}_{m,1}\mathcal{F}(\phi_{m,1} u)\|_{L^2(\RR^n)}\right).
\end{align*}
Since the mapping $\DD'(U)\rightarrow \mathcal{C}^{\infty}(\RR^n)$, $u\mapsto\mathcal{F}(\phi_{m,1} u)$, is continuous, there is a continuous seminorm $\widetilde{\mathfrak{p}}$ on $\DD'(U)$ such that
$$
I_1\leq C_1\widetilde{\mathfrak{p}}(u)+C_2\|\langle\cdot\rangle^r\widetilde{\phi}_{m,1}\mathcal{F}(\phi_{m,1} u)\|_{L^2(\RR^n)}.
$$
To estimate $I_2$, we make the following\\
\\
\noindent\textbf{Claim.} Let $f$ be a nonnegative measurable function on $\RR^n$ which satisfies $\langle\cdot\rangle^{\nu} f\in L^1(\RR^n)$, for all $\nu>0$. For each $r\in\RR$, $t>0$ and $\chi\in\DD(U)$, the mapping
\begin{equation}\label{sem-cont-ondp-for-ine-c}
\mathfrak{p}:\DD'(U)\rightarrow [0,\infty),\quad \mathfrak{p}(u)=\int_{\RR^n_{\eta}}f(\eta)\left(\int_{\substack{\xi\in \RR^n\\ |\xi+\eta|< t|\eta|}} |\mathcal{F}(\chi u)(\xi)|^2\langle \xi\rangle^{2r} d\xi\right)^{1/2}d\eta,
\end{equation}
is a well-defined continuous seminorm on $\DD'(U)$.\\
\\
\noindent We defer its proof for later and continue with the proof of the proposition. The claim implies that $I_2$ is bounded from above by a continuous seminorm on $\DD'(U)$ of the same type as \eqref{sem-cont-ondp-for-ine-c}. In view of \eqref{sum-for-bou-ftoplimbed} we conclude the validity of \eqref{ine-sem-for-openmap=forthi}; notice that the term $\|\langle\cdot\rangle^r\mathcal{F}(\varphi u)\|_{L^2(B(0,2))}$ in \eqref{sum-for-bou-ftoplimbed} is a continuous seminorm on $\DD'(U)$ (since $\DD'(U)\rightarrow \mathcal{C}^{\infty}(\RR^n)$, $u\mapsto\mathcal{F}(\varphi u)$, is continuous). This completes the proof that $\mathcal{I}$ is a topological imbedding.\\
\indent We now show that the range of $\mathcal{I}$ is closed. Let $(u_{\mu})_{\mu\in\Lambda}$ be a net in $\DD'^r_L(U)$ such that $\mathcal{I}(u_{\mu})=(u_{\mu},\textbf{f}_{u_{\mu}})$ converges to $(u,\mathbf{g})\in \DD'(U)\times (L^2(\RR^n)^{\ZZ_+\times \ZZ_+})$ in the topology of $\DD'(U)\times (L^2(\RR^n)^{\ZZ_+\times \ZZ_+})$. Hence $u_{\mu}\rightarrow u$ in $\DD'(U)$ and, for each $(j,k)\in\ZZ_+\times \ZZ_+$, $\langle\cdot\rangle^r\widetilde{\phi}_{j,k}\mathcal{F}(\phi_{j,k} u_{\mu})\rightarrow g_{j,k}$ in $L^2(\RR^n)$. The former implies that $\widetilde{\phi}_{j,k}\mathcal{F}(\phi_{j,k} u_{\mu})\rightarrow \widetilde{\phi}_{j,k}\mathcal{F}(\phi_{j,k} u)$ in $\mathcal{C}^{\infty}(\RR^n)$ and consequently $\langle\cdot\rangle^r\widetilde{\phi}_{j,k}\mathcal{F}(\phi_{j,k} u)=g_{j,k}\in L^2(\RR^n)$. It is straightforward to show that this implies $WF^r(u)\subseteq L$ (cf. the proof of the first inclusion in \eqref{inc-cl-for-the-sse}) and thus $u\in \DD'^r_L(U)$. This completes the proof that the range of $\mathcal{I}$ is closed.\\
\\
\noindent \textbf{Proof of Claim.} Set $\widetilde{K}:=\supp\chi$. We show that $\mathfrak{p}$ is well-defined and bounded on bounded subsets of $\DD'(U)$. Let $B$ be a bounded subset of $\DD'(U)$. Hence $B$ is equicontinuous (as $\DD'(U)$ is barrelled) and consequently also equicontinuous as a subset of $\mathcal{L}(\DD_{\widetilde{K}},\CC)$. Whence, there are $C>0$ and $k\in\ZZ_+$ such that $|\langle u,\varphi\rangle|\leq C\sup_{|\alpha|\leq k}\|\partial^{\alpha}\varphi\|_{L^{\infty}(\RR^n)}$, $\varphi\in \DD_{\widetilde{K}}$, $u\in B$. This implies
\begin{equation}\label{equ-ine0for-est-equiconmap-fordpp}
|\mathcal{F}(\chi u)(\xi)|=|\langle u,\chi e^{-i\,\cdot\, \xi}\rangle|\leq C2^k\langle \xi\rangle^k\sup_{|\alpha|\leq k} \|\partial^{\alpha}\chi\|_{L^{\infty}(\RR^n)},\quad \xi\in\RR^n,\, u\in B.
\end{equation}
For each $\eta\in\RR^n$, $\{\xi\in\RR^n\,|\,|\xi+\eta|<t|\eta|\}\subseteq B(0,(1+t)|\eta|)$ and thus
\begin{equation*}
\mathfrak{p}(u)\leq C_1\int_{\RR^n_{\eta}}f(\eta)\left(\int_{\xi\in B(0,(1+t)|\eta|)} \langle \xi\rangle^{2(|r|+k)} d\xi\right)^{1/2}d\eta\leq  C_2\|\langle \cdot\rangle^{|r|+k+n} f\|_{L^1(\RR^n)}.
\end{equation*}
Hence, $\mathfrak{p}$ is a well-defined seminorm on $\DD'(U)$ and it is bounded on bounded subsets of $\DD'(U)$. Since $\DD'(U)$ is bornological, $\mathfrak{p}$ is continuous and the proof is complete.
\end{proof}

\begin{remark}\label{rem-for-cov-ofcomconsubbytheses}
Employing the same arguments as in the proof of the first inclusion in \eqref{inc-cl-for-the-sse}, one shows that $L^c=\bigcup_{j\in\ZZ_+} O_{j,k}\times V_{j,k}$ for each fixed $k\in\ZZ_+$.
\end{remark}

\begin{remark}\label{rem-for-hyp-conmultmapfirsk}
Arguing as in Lemma \ref{lem-for-cha-wfofset-comwfl} and employing similar reasoning as in the proof of the above claim, one shows that the bilinear map $\mathcal{C}^{\infty}(U)\times\DD'^r_L(U)\rightarrow \DD'^r_L(U)$, $(\chi,u)\mapsto \chi u$, is well-defined and hypocontinuous.
\end{remark}

\begin{corollary}\label{rem-for-sem-refd}
Let $L$ be a closed conic subset of $U\times(\RR^n\backslash\{0\})$. Then $\DD'^r_L(U)$ is complete, semi-reflexive and a strictly webbed space in the sense of De Wilde.
\end{corollary}

\begin{proof}
The fact that $\DD'^r_L(U)$ is complete is an immediate consequence of the proposition and it is semi-reflexive in view of \cite[Theorem 5 and Theorem 6, p. 299]{kothe1}. Furthermore, \cite[Theorem 1, p. 61, and Theorem 6, p. 62]{kothe2} implies that $\DD'^r_L(U)$ is a strictly webbed space in the sense of De Wilde ($\DD'(U)$ is strictly webbed in view of \cite[Theorem 13, p. 64]{kothe2}).
\end{proof}

The fact that $\DD'^r_L(U)$ is a strictly webbed space allows one to employ De Wilde's closed graph and open mapping theorems (see \cite[Chapter 35]{kothe2}) when considering maps to and from $\DD'^r_L(U)$. As a consequence of Proposition \ref{pro-for-top-imbedingthforc}, we now derive the following important characterisation of the relatively compact subsets of $\DD'^r_L(U)$ which we announced at the beginning of the subsection.

\begin{corollary}\label{car-of-comset-by-wavefrse}
Let $L$ be a closed conic subset of $U\times(\RR^n\backslash\{0\})$ and let $B$ be a bounded subset of $\DD'(U)$. Then $B$ is a relatively compact subset $\DD'^r_L(U)$ if and only if $WF^r_c(B)\subseteq L$.
\end{corollary}

\begin{remark}
The corollary can be viewed as a generalisation of the Kolmogorov-Riesz compactness theorem \cite{HH}. Indeed, taking $L=\emptyset$, the corollary together with Lemma \ref{lem-for-cha-wfofset-comwfl} immediately give the following well known characterisation of relatively compact subsets of $H^r_{\loc}(U)$. A bounded subset $B$ of $\DD'(U)$ is a relatively compact subset of $H^r_{\loc}(U)$ if and only if
$$
\lim_{R\rightarrow\infty}\sup_{u\in B}\int_{|\xi|>R}\langle\xi\rangle^{2r}|\mathcal{F}(\varphi u)(\xi)|^2d\xi=0,\quad \mbox{for all}\,\,\varphi\in\DD(U).
$$
\end{remark}

\begin{proof}[Proof of Corollary \ref{car-of-comset-by-wavefrse}] The claim is trivial when $L=U\times(\RR^n\backslash \{0\})$ (since $\DD'(U)$ is Montel). Assume that $L^c\neq\emptyset$. If $B$ is relatively compact in $\DD'^r_L(U)$, Proposition \ref{pro-for-top-imbedingthforc} implies that for each $(j,k)\in\ZZ_+\times \ZZ_+$, $\{\langle\cdot\rangle^r\widetilde{\phi}_{j,k}\mathcal{F}(\phi_{j,k} u)\,|\, u\in B\}$ is relatively compact in $L^2(\RR^n)$ and the Kolmogorov-Riesz compactness theorem \cite[Theorem 5]{HH} yields
\begin{equation*}
\sup_{u\in B} \int_{\xi\in V_{j,k},\, |\xi|>R} |\mathcal{F}(\phi_{j,k} u)(\xi)|^2\langle\xi\rangle^{2r}d\xi \rightarrow 0,\quad \mbox{as}\quad R\rightarrow\infty,\quad \mbox{for each}\,\, j,k\in\ZZ_+.
\end{equation*}
It is straightforward to verify that the latter implies $WF^r_c(B)\subseteq L$ (cf. Remark \ref{rem-for-cov-ofcomconsubbytheses}).\\
\indent Assume now $WF^r_c(B)\subseteq L$; hence $B\subseteq \DD'^r_L(U)$. In view of Lemma \ref{lem-for-cha-wfofset-comwfl}, we infer
\begin{equation*}
\sup_{u\in B} \int_{|\xi|>R} \widetilde{\phi}_{j,k}(\xi)^2 |\mathcal{F}(\phi_{j,k} u)(\xi)|^2\langle\xi\rangle^{2r}d\xi \rightarrow 0,\quad \mbox{as}\quad R\rightarrow\infty,\quad \mbox{for each}\,\, j,k\in\ZZ_+.
\end{equation*}
We claim that this implies that
\begin{equation}\label{rel-com-set-for-comargofawfs}
\{\langle\cdot\rangle^r\widetilde{\phi}_{j,k} \mathcal{F}(\phi_{j,k} u)\,|\, u\in B\}\,\, \mbox{is relatively compact in}\,\, L^2(\RR^n)\,\, \mbox{for each}\,\, j,k\in\ZZ_+.
\end{equation}
Once we show \eqref{rel-com-set-for-comargofawfs} the claim in the corollary follows from Tychonoff's theorem and Proposition \ref{pro-for-top-imbedingthforc} ($B$ is relatively compact in $\DD'(U)$ since $\DD'(U)$ is Montel). Since $B$ is bounded in $\DD'(U)$, we immediately deduce that
$$
\sup_{u\in B} \int_{|\xi|\leq R} \widetilde{\phi}_{j,k}(\xi)^2 |\mathcal{F}(\phi_{j,k} u)(\xi)|^2\langle\xi\rangle^{2r}d\xi <\infty,\quad \mbox{for all}\,\, R>0,\,\,j,k\in\ZZ_+,
$$
and hence the sets in \eqref{rel-com-set-for-comargofawfs} are bounded in $L^2(\RR^n)$. In view of the Kolmogorov-Riesz compactness theorem \cite[Theorem 5]{HH}, to verify \eqref{rel-com-set-for-comargofawfs} it suffices to show that for each $j,k\in\ZZ_+$ it holds that
\begin{equation}\label{wha-sho-wfo-con-map-comims}
\sup_{u\in B}\|\langle\cdot+\eta\rangle^r\widetilde{\phi}_{j,k}(\cdot+\eta) \mathcal{F}(\phi_{j,k} u)(\cdot+\eta)- \langle\cdot\rangle^r\widetilde{\phi}_{j,k} \mathcal{F}(\phi_{j,k} u)\|_{L^2(\RR^n)}\rightarrow 0,\quad \mbox{as}\quad \eta\rightarrow 0.
\end{equation}
We Taylor expand $\langle \xi+\eta\rangle^r\widetilde{\phi}_{j,k}(\xi+\eta) \mathcal{F}(\phi_{j,k} u)(\xi+\eta)$ at $\xi$ up to order $0$ to infer
\begin{align*}
|&\langle \xi+\eta\rangle^r\widetilde{\phi}_{j,k}(\xi+\eta) \mathcal{F}(\phi_{j,k} u)(\xi+\eta)-\langle \xi\rangle^r\widetilde{\phi}_{j,k}(\xi) \mathcal{F}(\phi_{j,k} u)(\xi)|\\
&\leq |\eta|\sum_{l=1}^n\int_0^1 |r(\xi_l+t\eta_l)\langle\xi+t\eta\rangle^{r-2}\widetilde{\phi}_{j,k}(\xi+t\eta) \mathcal{F}(\phi_{j,k} u)(\xi+t\eta)\\
&{}\qquad\qquad\qquad+ \langle \xi+t\eta\rangle^r \partial_l\widetilde{\phi}_{j,k}(\xi+t\eta)\mathcal{F}(\phi_{j,k} u)(\xi+t\eta)\\
&{}\qquad\qquad\qquad+ \langle\xi+t\eta\rangle^r\widetilde{\phi}_{j,k}(\xi+t\eta)\partial_l\mathcal{F}(\phi_{j,k} u)(\xi+t\eta)|dt.
\end{align*}
Set $\phi_{j,k;l}(x):=x_l\phi_{j,k}(x)$, $l=1,\ldots,n$. We infer
\begin{align*}
|&\langle \xi+\eta\rangle^r\widetilde{\phi}_{j,k}(\xi+\eta) \mathcal{F}(\phi_{j,k} u)(\xi+\eta)-\langle \xi\rangle^r\widetilde{\phi}_{j,k}(\xi) \mathcal{F}(\phi_{j,k} u)(\xi)|\\
&\leq n|r||\eta|\int_0^1 \langle\xi+t\eta\rangle^{r-1}\widetilde{\phi}_{j,k}(\xi+t\eta)|\mathcal{F}(\phi_{j,k} u)(\xi+t\eta)|dt\\
&{}\quad+|\eta|\sum_{l=1}^n\int_0^1\langle \xi+t\eta\rangle^r|\partial_l\widetilde{\phi}_{j,k}(\xi+t\eta)||\mathcal{F}(\phi_{j,k} u)(\xi+t\eta)|dt\\
&{}\quad+|\eta|\sum_{l=1}^n\int_0^1\langle \xi+t\eta\rangle^r\widetilde{\phi}_{j,k}(\xi+t\eta) |\mathcal{F}(\phi_{j,k;l} u)(\xi+t\eta)|dt.
\end{align*}
Employing the Minkowski integral inequality, a change of variables and the property $(d)$ of $\widetilde{\phi}_{j,k}$ we obtain
\begin{align*}
\|&\langle\cdot+\eta\rangle^r\widetilde{\phi}_{j,k}(\cdot+\eta) \mathcal{F}(\phi_{j,k} u)(\cdot+\eta)- \langle\cdot\rangle^r\widetilde{\phi}_{j,k} \mathcal{F}(\phi_{j,k} u)\|_{L^2(\RR^n)}\\
&\leq n|r||\eta|\int_0^1\left(\int_{\RR^n}\langle\xi+t\eta\rangle^{2r-2}\widetilde{\phi}_{j,k}(\xi+t\eta)^2|\mathcal{F}(\phi_{j,k} u)(\xi+t\eta)|^2d\xi\right)^{1/2}dt\\
&{}\quad +|\eta|\sum_{l=1}^n\int_0^1\left(\int_{\RR^n}\langle \xi+t\eta\rangle^{2r}|\partial_l\widetilde{\phi}_{j,k}(\xi+t\eta)|^2|\mathcal{F}(\phi_{j,k} u)(\xi+t\eta)|^2 d\xi\right)^{1/2}dt\\
&{}\quad +|\eta|\sum_{l=1}^n\int_0^1\left(\int_{\RR^n}\langle \xi+t\eta\rangle^{2r}\widetilde{\phi}_{j,k}(\xi+t\eta)^2|\mathcal{F}(\phi_{j,k;l} u)(\xi+t\eta)|^2 d\xi\right)^{1/2}dt\\
&\leq C|\eta|\|\langle\cdot\rangle^r\mathcal{F}(\phi_{j,k} u)\|_{L^2(\overline{V'_{j,k}})}+|\eta|\sum_{l=1}^n \|\langle\cdot\rangle^r\mathcal{F}(\phi_{j,k;l} u)\|_{L^2(\overline{V'_{j,k}})}.
\end{align*}
Lemma \ref{lem-for-cha-wfofset-comwfl} together with the fact that $B$ is bounded in $\DD'(U)$ implies that
$$
\sup_{u\in B}\|\langle\cdot\rangle^r\mathcal{F}(\phi_{j,k} u)\|_{L^2(\overline{V'_{j,k}})}<\infty\quad \mbox{and}\quad \sup_{u\in B}\|\langle\cdot\rangle^r\mathcal{F}(\phi_{j,k;l} u)\|_{L^2(\overline{V'_{j,k}})}<\infty,\, l=1,\ldots,n.
$$
Consequently, the above estimates verify \eqref{wha-sho-wfo-con-map-comims} and the proof is complete.
\end{proof}

We end this subsection with the following result on the existence of a sequence of smoothing operators on $\DD'^r_L(U)$ that approximate the identity operator on $\DD'^r_L(U)$.

\begin{proposition}\label{seq-den-comsmf}
Let $\psi_j\in\DD(U)$, $j\in\ZZ_+$, be such that $0\leq \psi_j\leq 1$, $\psi_j=1$ on the compact $\{x\in U\,|\, |x|\leq j,\, \operatorname{dist}(x,\partial U)\geq 3/j\}$ and $\supp\psi_j\subseteq \{x\in U\,|\, \operatorname{dist}(x,\partial U)> 2/j\}$. Let $\chi$ be a nonnegative function in $\DD(\RR^n)$ satisfying $\supp\chi\subseteq \{x\in\RR^n\,|\, |x|\leq 1\}$ and $\int_{\RR^n} \chi(x) dx=1$ and set $\chi_j(x):=j^n\chi(jx)$, $x\in\RR^n$, $j\in\ZZ_+$. For each $j\in\ZZ_+$, the operators
\begin{equation}\label{ope-app-ide-simc}
P_j:\DD'(U)\rightarrow\DD(U),\quad P_ju=\chi_j*(\psi_j u),
\end{equation}
are well-defined and continuous. Furthermore, for each $r\in\RR$ and each closed conic subset $L$ of $U\times (\RR^n\backslash\{0\})$, $\{P_j\}_{j\in\ZZ_+}$ is a bounded subset of $\mathcal{L}_b(\DD'^r_L(U))$ and $P_j\rightarrow \operatorname{Id}$ in $\mathcal{L}_p(\DD'^r_L(U))$ (the index $p$ stands for the topology of precompact convergence). In particular, $\DD(U)$ is sequentially dense in $\DD'^r_L(U)$.
\end{proposition}

\begin{proof} Clearly, \eqref{ope-app-ide-simc} are well-defined and continuous. It is a well-known fact that $P_j\rightarrow \operatorname{Id}$ in $\mathcal{L}_{\sigma}(\DD'(U))$ (the index $\sigma$ stands for the topology of simple convergence). The Banach–Steinhaus theorem \cite[Theorem 4.5, p. 85]{Sch} together with the fact that $\DD'(U)$ is Montel implies that the convergence holds in $\mathcal{L}_b(\DD'(U))$ and consequently $\{P_j\}_{j\in\ZZ_+}$ is bounded in $\mathcal{L}_b(\DD'(U))$. This proves the claim when $L=U\times (\RR^n\backslash\{0\})$.\\
\indent Assume that $L^c\neq\emptyset$. Because of the above, to prove the boundedness of $\{P_j\}_{j\in\ZZ_+}$ in $\mathcal{L}_b(\DD'^r_L(U))$ it remains to show that for every bounded set $B$ in $\DD'^r_L(U)$ and every seminorm $\mathfrak{p}_{r;\varphi,V}$, it holds that $\sup_{j\in\ZZ_+}\sup_{u\in B}\mathfrak{p}_{r;\varphi,V}(P_j u)<\infty$. Let $B$ be a bounded subset of $\DD'^r_L(U)$ and let $\varphi\in\DD(U)\backslash\{0\}$ and the closed cone $V\subseteq \RR^n$, $V\backslash\{0\}\neq\emptyset$, are such that $(\supp\varphi\times V)\cap L=\emptyset$. Employing a standard compactness argument, one can find $\widetilde{\varphi}\in\DD(U)$ satisfying $\widetilde{\varphi}=1$ on a neighbourhood of $\supp\varphi$ and a closed cone $\widetilde{V}\subseteq \RR^n$ with $V\backslash\{0\}\subseteq \operatorname{int} \widetilde{V}$ such that $(\supp\widetilde{\varphi}\times \widetilde{V})\cap L=\emptyset$. There is $j_0\in\ZZ_+$ such that $u_j:=\chi_j*(\widetilde{\varphi} u)\in\DD(U)$, for all $u\in B$, $j\geq j_0$. There is $j'\geq j_0$ such that $\psi_j=\widetilde{\varphi}=1$ on $\supp\varphi +\{x\in\RR^n\,|\, |x|\leq 2/j\}$, for all $j\geq j'$. Consequently,
$$
\varphi P_ju-\varphi u_j=\varphi(\chi_j*((\psi_j-\widetilde{\varphi})u))=0,\quad j\geq j',\,\, u\in B.
$$
Hence, it suffices to show that $\sup_{j\geq j'}\sup_{u\in B}\mathfrak{p}_{r;\varphi,V}(u_j)<\infty$. Since $|\mathcal{F}\chi_j(\xi)|\leq 1$, $\xi\in\RR^n$, $j\in\ZZ_+$, we infer (for $j\geq j'$)
\begin{equation}\label{ine-for-ni-witresj}
|\mathcal{F}(\varphi u_j)(\xi)|=(2\pi)^{-n}|(\mathcal{F}\varphi)*(\mathcal{F}(\chi_j) \mathcal{F}(\widetilde{\varphi} u))(\xi)|\leq \frac{1}{(2\pi)^n}\int_{\RR^n}|\mathcal{F}\varphi(\eta)||\mathcal{F}(\widetilde{\varphi} u)(\xi-\eta)| d\eta.
\end{equation}
There is $c\in(0,1)$ such that
\begin{equation}\label{sub-con-for}
\{\eta\in\RR^n\,|\, \exists \xi\in V\backslash\{0\}\,\, \mbox{such that}\,\, |\xi-\eta|\leq c|\xi|\}\subseteq \operatorname{int}\widetilde{V}.
\end{equation}
In the same way as in the proof of Proposition \ref{pro-for-top-imbedingthforc} (and Lemma \ref{lem-for-cha-wfofset-comwfl}) we employ the Minkowski integral inequality together with \eqref{ine-for-ni-witresj} to infer $\mathfrak{p}_{r;\varphi,V}(u_j)\leq 2^{|r|}(2\pi)^{-n}(I_1+I_2)$ with
\begin{align*}
I_1&:=\int_{\RR^n_{\eta}}|\mathcal{F}\varphi(\eta)|\langle\eta\rangle^{|r|}\left(\int_{\substack{\xi\in V\\ |\xi|\geq |\eta|/c}} |\mathcal{F}(\widetilde{\varphi} u)(\xi-\eta)|^2\langle \xi-\eta\rangle^{2r}d\xi\right)^{1/2}d\eta,\\
I_2&:=\int_{\RR^n_{\eta}}|\mathcal{F}\varphi(\eta)|\langle\eta\rangle^{|r|}\left(\int_{\substack{\xi\in V\\|\xi|<|\eta|/c}} |\mathcal{F}(\widetilde{\varphi} u)(\xi-\eta)|^2\langle \xi-\eta\rangle^{2r}d\xi\right)^{1/2}d\eta.
\end{align*}
Again, as in the proof of Proposition \ref{pro-for-top-imbedingthforc}, \eqref{sub-con-for} implies
\begin{equation}\label{for-con-ine-est-w-eneeforbsk}
\{\xi\in V-\{\eta\}\,|\, |\xi+\eta|\geq |\eta|/c\}\subseteq \widetilde{V}
\end{equation}
which in turn yields that $I_1\leq \|\langle\cdot\rangle^{|r|}\mathcal{F}\varphi\|_{L^1(\RR^n)} \mathfrak{p}_{r;\widetilde{\varphi},\widetilde{V}}(u)$. Since $B$ is bounded in $\DD'(U)$ (as it is bounded in $\DD'^r_L(U)$), there are $C,l>0$ such that $|\mathcal{F}(\widetilde{\varphi} u)(\xi)|\leq C\langle \xi\rangle^l$, $\xi\in\RR^n$, $u\in B$. This immediately implies that $I_2$ is uniformly bounded for all $u\in B$. We conclude $\sup_{j\geq j'}\sup_{u\in B}\mathfrak{p}_{r;\varphi,V}(u_j)<\infty$ which completes the proof of the boundedness of $\{P_j\}_{j\in\ZZ_+}$ in $\mathcal{L}_b(\DD'^r_L(U))$.\\
\indent It remains to prove that $P_j\rightarrow \operatorname{Id}$ in $\mathcal{L}_p(\DD'^r_L(U))$. Since $P_j\rightarrow \operatorname{Id}$ in $\mathcal{L}_b(\DD'(U))$, it suffices to show that for every precompact subset $A$ of $\DD'^r_L(U)$, every $\varphi\in\DD(U)\backslash\{0\}$ and every closed cone $V\subseteq \RR^n$, $V\backslash\{0\}\neq \emptyset$, satisfying $(\supp \varphi\times V)\cap L=\emptyset$, it holds that $\sup_{u\in A}\mathfrak{p}_{r;\varphi,V}(P_ju-u)\rightarrow 0$, as $j\rightarrow \infty$. Fix such $A$, $\varphi$ and $V$. Let $\widetilde{\varphi}$, $\widetilde{V}$, $c$ and $j'$ be as in the first part of the proof. Again, we denote $u_j:=\chi_j*(\widetilde{\varphi}u)\in\DD(U)$, $j\geq j'$, and point out that $\varphi P_j u=\varphi u_j$, $j\geq j'$, $u\in A$. Similarly as above, we have $\mathfrak{p}_{r;\varphi,V}(P_ju-u)\leq 2^{|r|}(2\pi)^{-n}(I_1+ I_2)$ with
\begin{align*}
I_1&:=\int_{\RR^n_{\eta}}|\mathcal{F}\varphi(\eta)|\langle\eta\rangle^{|r|}\left(\int_{\substack{\xi\in V\\ |\xi|\geq |\eta|/c}} |\mathcal{F}\chi_j(\xi-\eta)-1|^2|\mathcal{F}(\widetilde{\varphi} u)(\xi-\eta)|^2\langle \xi-\eta\rangle^{2r}d\xi\right)^{1/2}d\eta,\\
I_2&:=\int_{\RR^n_{\eta}}|\mathcal{F}\varphi(\eta)|\langle\eta\rangle^{|r|}\left(\int_{\substack{\xi\in V\\ |\xi|<|\eta|/c}} |\mathcal{F}\chi_j(\xi-\eta)-1|^2|\mathcal{F}(\widetilde{\varphi} u)(\xi-\eta)|^2\langle \xi-\eta\rangle^{2r}d\xi\right)^{1/2}d\eta.
\end{align*}
There are $C_1,l>0$ such that $|\mathcal{F}(\widetilde{\varphi}u)(\xi)|\leq C_1\langle \xi\rangle^l$, $\xi\in\RR^n$, $u\in A$. Let $\varepsilon>0$ be arbitrary but fixed. In view of Corollary \ref{car-of-comset-by-wavefrse}, $WF^r_c(A)\subseteq L$ and hence Lemma \ref{lem-for-cha-wfofset-comwfl} yields that there is $R>1$ such that
\begin{equation}
\sup_{u\in A} \left(\int_{\widetilde{V},\, |\xi|\geq R} |\mathcal{F}(\widetilde{\varphi}u)(\xi)|^2\langle \xi\rangle^{2r}d\xi\right)^{1/2} \leq \varepsilon/(2^{|r|+3}\|\langle \cdot\rangle^{|r|}\mathcal{F}\varphi\|_{L^1(\RR^n)}).
\end{equation}
We choose $R$ large enough so that we additionally have
$$
\int_{|\eta|\geq Rc/(c+1)}|\mathcal{F}\varphi(\eta)|\langle \eta\rangle^{2|r|+l+n}d\eta\leq \varepsilon\cdot\left(2^{|r|+3}C_1(1+1/c)^{|r|+l+n}\|\langle\cdot\rangle^{-n}\|_{L^2(\RR^n)}\right)^{-1}.
$$
Since $\mathcal{F}\chi_j=\mathcal{F}\chi(\cdot/j)$ and $\mathcal{F}\chi(0)=1$, dominated convergence implies that there is $j'_1\geq j'$ such that
$$
\|(\mathcal{F}\chi_j-1)\langle\cdot\rangle^{|r|+l}\|_{L^2(B(0,R))}\leq \varepsilon/(C_1 2^{|r|+2}\|\langle\cdot\rangle^{|r|}\mathcal{F}\varphi\|_{L^1(\RR^n)}),\quad j\geq j'_1.
$$
Let $j\geq j'_1$ and $u\in A$ be arbitrary. To estimate $I_1$, we change variables in the inner integral and employ \eqref{for-con-ine-est-w-eneeforbsk} to obtain (recall, $\|\mathcal{F}\chi_j\|_{L^{\infty}(\RR^n)}=1$)
\begin{align*}
I_1&\leq \int_{\RR^n_{\eta}}|\mathcal{F}\varphi(\eta)|\langle\eta\rangle^{|r|}\left(\int_{\widetilde{V}} |\mathcal{F}\chi_j(\xi)-1|^2|\mathcal{F}(\widetilde{\varphi} u)(\xi)|^2\langle \xi\rangle^{2r}d\xi\right)^{1/2}d\eta\\
&\leq 2\|\langle\cdot\rangle^{|r|}\mathcal{F}\varphi\|_{L^1(\RR^n)}\left(\int_{\widetilde{V},\, |\xi|\geq R} |\mathcal{F}(\widetilde{\varphi} u)(\xi)|^2\langle \xi\rangle^{2r}d\xi\right)^{1/2}\\
&{}\quad+C_1\|\langle\cdot\rangle^{|r|}\mathcal{F}\varphi\|_{L^1(\RR^n)}\left(\int_{B(0,R)} |\mathcal{F}\chi_j(\xi)-1|^2\langle \xi\rangle^{2(|r|+l)}d\xi\right)^{1/2}\\
&\leq \varepsilon/2^{|r|+1}.
\end{align*}
We estimate $I_2$ as follows:
\begin{align*}
I_2&\leq \int_{\RR^n_{\eta}}|\mathcal{F}\varphi(\eta)|\langle\eta\rangle^{|r|}\left(\int_{|\xi+\eta|<|\eta|/c} |\mathcal{F}\chi_j(\xi)-1|^2|\mathcal{F}(\widetilde{\varphi} u)(\xi)|^2\langle \xi\rangle^{2r}d\xi\right)^{1/2}d\eta\\
&\leq \int_{\RR^n_{\eta}}|\mathcal{F}\varphi(\eta)|\langle\eta\rangle^{|r|}\left(\int_{|\xi|<(1+1/c)|\eta|} |\mathcal{F}\chi_j(\xi)-1|^2|\mathcal{F}(\widetilde{\varphi} u)(\xi)|^2\langle \xi\rangle^{2r}d\xi\right)^{1/2}d\eta\\
&\leq \int_{\RR^n_{\eta}}|\mathcal{F}\varphi(\eta)|\langle\eta\rangle^{|r|}\left(\int_{|\xi|<\min\{(1+1/c)|\eta|,R\}} |\mathcal{F}\chi_j(\xi)-1|^2|\mathcal{F}(\widetilde{\varphi} u)(\xi)|^2\langle \xi\rangle^{2r}d\xi\right)^{1/2}d\eta\\
&{}\quad+ \int_{\RR^n_{\eta}}|\mathcal{F}\varphi(\eta)|\langle\eta\rangle^{|r|}\left(\int_{R\leq |\xi|<(1+1/c)|\eta|} |\mathcal{F}\chi_j(\xi)-1|^2|\mathcal{F}(\widetilde{\varphi} u)(\xi)|^2\langle \xi\rangle^{2r}d\xi\right)^{1/2}d\eta\\
&\leq C_1\|\langle \cdot\rangle^{|r|}\mathcal{F}\varphi\|_{L^1(\RR^n)}\left(\int_{B(0,R)} |\mathcal{F}\chi_j(\xi)-1|^2\langle \xi\rangle^{2(|r|+l)}d\xi\right)^{1/2}\\
&{}\quad+ 2C_1 \int_{|\eta|\geq Rc/(c+1)}|\mathcal{F}\varphi(\eta)|\langle\eta\rangle^{|r|}\left(\int_{R\leq |\xi|<(1+1/c)|\eta|} \frac{\langle \xi\rangle^{2(|r|+l+n)}}{\langle\xi\rangle^{2n}}d\xi\right)^{1/2}d\eta\\
&\leq \frac{\varepsilon}{2^{|r|+2}}+2C_1(1+1/c)^{|r|+l+n}\|\langle\cdot\rangle^{-n}\|_{L^2(\RR^n)}\int_{|\eta|\geq Rc/(c+1)}|\mathcal{F}\varphi(\eta)|\langle\eta\rangle^{2|r|+l+n}d\eta\leq \frac{\varepsilon}{2^{|r|+1}}.
\end{align*}
Combining these estimates for $I_1$ and $I_2$ we deduce $\sup_{u\in A}\mathfrak{p}_{r;\varphi,V}(P_ju-u)\leq \varepsilon$, $j\geq j'_1$, which completes the proof of the proposition.
\end{proof}

\subsection{The dual of \texorpdfstring{$\DD'^r_L(U)$}{D'rL(U)}}\label{sub-sec-for-dualityee}

Notice that (cf. Proposition \ref{seq-den-comsmf})
\begin{equation}\label{inc-for-d'lrford}
H^r_{\loc}(U)\subseteq \DD'^r_L(U)\subseteq \DD'(U)\quad \mbox{continuously and densely}
\end{equation}
and hence the dual of $\DD'^r_L(U)$ is a space of distributions on $U$. In fact \eqref{inc-for-d'lrford} gives
\begin{equation}\label{equ-for-dua-ofthedspacins}
\DD(U)\subseteq (\DD'^r_L(U))'_b\subseteq H^{-r}_{\comp}(U)\quad \mbox{continuously}.
\end{equation}
Our goal is to identify the space $(\DD'^r_L(U))'_b$. For this purpose, given an open conic subset $W$ of $U\times(\RR^n\backslash\{0\})$, we define the space
\begin{equation*}
\EE'^r_W(U):=\{u\in H^r_{\comp}(U)\,|\, WF(u)\subseteq W\}.
\end{equation*}
To introduce a locally convex topology on it, we consider the following auxiliary space. Let $L$ be a closed conic subset of $U\times (\RR^n\backslash\{0\})$ and $K$ a compact subset of $U$ satisfying $\pr_1(L)\subseteq K$, where $\pr_1$ stands for the projection on the first variable. We define
\begin{equation*}
\EE'^r_{L;K}(U):=\{u\in H^r_K(U)\,|\, WF(u)\subseteq L\}.\footnote{The reason for the condition $\pr_1(L)\subseteq K$ is the fact $\pr_1(WF(u))=\operatorname{sing}\supp u\subseteq \supp u$.}
\end{equation*}
The distribution $u\in H^r_K(U)$ belongs to $\EE'^r_{L;K}(U)$ if and only if for every $\varphi\in\DD(U)$ and every closed cone $\emptyset\neq V\subseteq \RR^n$ satisfying $(\supp \varphi \times V)\cap L=\emptyset$ it holds that (cf. \cite[Section 8.1]{hor})
\begin{equation*}
\mathfrak{q}_{\nu;\varphi,V}(u):=\sup_{\xi\in V}\langle\xi \rangle^{\nu}|\mathcal{F}(\varphi u)(\xi)|<\infty,\quad \mbox{for all}\,\, \nu>0.
\end{equation*}
We equip $\EE'^r_{L;K}(U)$ with the locally convex topology induced by the norm on $H^r_K(U)$ together with all seminorms $\mathfrak{q}_{\nu;\varphi,V}$ where $\varphi$, $V$ and $\nu$ are as above. Clearly, $\DD_K\subseteq \EE'^r_{L;K}(U)\subseteq H^r_K(U)$ continuously.

\begin{proposition}\label{pro-for-esp-closincinimbinprdsp}
Let $K\subset\subset U$ and let $L$ be a closed conic subset of $U\times(\RR^n\backslash\{0\})$ satisfying $\pr_1(L)\subseteq K$. Let $O_{j,k}$, $O'_{j,k}$, $O''_{j,k}$, $V_{j,k}$, $V'_{j,k}$, $V''_{j,k}$, $\phi_{j,k}$ and $\widetilde{\phi}_{j,k}$ be as in Proposition \ref{pro-for-top-imbedingthforc}. Then the mapping
\begin{gather*}
\EE'^r_{L;K}(U)\rightarrow H^r_K(U)\times (\SSS(\RR^n)^{\ZZ_+\times \ZZ_+}),\quad u\mapsto(u,\mathbf{f}_u),\,\, \mbox{where}\\
\mathbf{f}_u(j,k):=\widetilde{\phi}_{j,k}\mathcal{F}(\phi_{j,k} u),\, j,k\in\ZZ_+,
\end{gather*}
is a well-defined topological imbedding with closed image. Consequently, $\EE'^r_{L;K}(U)$ is a reflexive Fr\'echet space.
\end{proposition}

\begin{proof} Clearly, the map is injective. The continuity follows from the estimate
$$
\|\langle\cdot\rangle^{\nu}\partial^{\alpha}(\widetilde{\phi}_{j,k}\mathcal{F}(\phi_{j,k} u))\|_{L^{\infty}(\RR^n)}\leq C\sum_{\beta\leq \alpha}\mathfrak{q}_{\nu;\phi_{j,k,\beta},V'_{j,k}}(u),\quad\mbox{where}\quad \phi_{j,k,\beta}(x)=x^{\beta}\phi_{j,k}(x).
$$
The fact that the map is a topological imbedding with closed image can be shown in the same way as in the proof of Proposition \ref{pro-for-top-imbedingthforc} (to bound the appropriate variant of $I_2$, one employs the inequality $|\mathcal{F}(\psi u)(\xi)|\leq C\langle \xi\rangle^{|r|}\|\langle\cdot\rangle^{|r|}\mathcal{F}\psi\|_{L^2(\RR^n)}\|u\|_{H^r(\RR^n)}$, $\xi\in\RR^n$, $\psi\in\DD(U)$, $u\in H^r_K(\RR^n)$). From this, we immediately deduce that $\EE'^r_{L;K}(U)$ is a reflexive Fr\'echet space since the codomain of the map is a countable topological product of reflexive Fr\'echet spaces.
\end{proof}

Notice that
\begin{equation}\label{con-inc-for-setdualofwfssp}
\EE'^r_{L;K}(U)\subseteq \EE'^r_{\widetilde{L};\widetilde{K}}(U)\,\, \mbox{continuously if}\,\, L\subseteq \widetilde{L}\,\, \mbox{and}\,\, K\subseteq \widetilde{K}.
\end{equation}
Given an open conic subset $W$ of $U\times(\RR^n\backslash\{0\})$, denote by $\mathfrak{W}$ the set of all pairs $(L,K)$, where $K$ is a compact subset of $U$ and $L$ a closed conic subset of $U\times(\RR^n\backslash\{0\})$ satisfying $L\subseteq W$ and $\pr_1(L)\subseteq K$. Then $(\mathfrak{W},\leq)$ becomes a directed set with order $(L,K)\leq(\widetilde{L},\widetilde{K})$ if $L\subseteq \widetilde{L}$ and $K\subseteq \widetilde{K}$. Notice that $\EE'^r_W(U)=\bigcup_{(L,K)\in\mathfrak{W}} \EE'^r_{L;K}(U)$. We define the locally convex topology on $\EE'^r_W(U)$ by
$$
\EE'^r_W(U)=\lim_{\substack{\longrightarrow\\ (L,K)\in\mathfrak{W}}} \EE'^r_{L;K}(U),
$$
where the linking mappings in the inductive limit are the continuous inclusions \eqref{con-inc-for-setdualofwfssp}; notice that the inductive limit topology on $\EE'^r_W(U)$ is indeed Hausdorff since $\EE'^r_{L;K}(U)\subseteq H^r_{\comp}(U)$ continuously. For any sequence $(L_j,K_j)_{j\in\ZZ_+}\subseteq \mathfrak{W}$ which satisfies
\begin{equation}\label{inc-set-com-exchscoses}
L_j\subseteq \operatorname{int} L_{j+1},\,\, K_j\subseteq \operatorname{int} K_{j+1},\,\, \bigcup_{j\in\ZZ_+}L_j=W,\,\, \bigcup_{j\in\ZZ_+} K_j=U,
\end{equation}
it holds that
$$
\EE'^r_W(U)=\lim_{\substack{\longrightarrow\\ j\rightarrow \infty}} \EE'^r_{L_j;K_j}(U)\quad\mbox{topologically}
$$
(one can always find such sequence, cf. \cite[p. 1354]{BD}; when $W=\emptyset$, we can take $L_j=\emptyset$, $j\in\ZZ_+$). Consequently, $\EE'^r_W(U)$ is an $(LF)$-space and thus it is barrelled and bornological. Furthermore, we have the following continuous inclusions:
\begin{equation}\label{inc-for-dual-ofd'l}
\DD(U)\subseteq \EE'^r_W(U)\subseteq H^r_{\comp}(U)\quad\mbox{and}\quad \EE'^{r_1}_{W_1}(U)\subseteq \EE'^{r_2}_{W_2}(U),\, r_1\geq r_2,\, W_1\subseteq W_2,
\end{equation}

\begin{remark}\label{rem-for-ide-spawithorddspinrsetck}
Since the continuous inclusion $\DD_K\subseteq \EE'^r_{\emptyset;K}(U)$ is a bijection, the open mapping theorem for Fr\'echet spaces yields that $\EE'^r_{\emptyset;K}(U)=\DD_K$ topologically. Consequently, $\EE'^r_{\emptyset}(U)=\DD(U)$ topologically.\\
\indent Similarly, since the continuous inclusion $\EE'^r_{K\times (\RR^n\backslash\{0\});K}(U)\subseteq H^r_K(U)$ is a bijection, the open mapping theorem for Fr\'echet spaces shows that $\EE'^r_{K\times(\RR^n\backslash\{0\});K}(U)=H^r_K(U)$ topologically. Consequently, $\EE'^r_{U\times(\RR^n\backslash\{0\})}(U)=H^r_{\comp}(U)$ topologically.
\end{remark}

\begin{remark}\label{rem-for-con-ofmulonespaceofonesk}
The bilinear map $\mathcal{C}^{\infty}(U)\times \EE'^r_{L;K}(U)\rightarrow\EE'^r_{L;K}(U)$, $(\chi,u)\mapsto \chi u$, is well-defined and continuous. Indeed, the fact that it is well-defined is trivial and a standard closed graph argument implies that it is separately continuous since both spaces are Fr\'echet. Now, \cite[Theorem 1, p. 158]{kothe2} verifies its continuity. Consequently, \cite[Theorem 5, p. 159]{kothe2} shows that $\mathcal{C}^{\infty}(U)\times \EE'^r_W(U)\rightarrow\EE'^r_W(U)$, $(\chi,u)\mapsto \chi u$, is hypocontinuous.
\end{remark}

The space $\EE'^r_W(U)$ satisfies similar approximation result to Proposition \ref{seq-den-comsmf} for $\DD'^r_L(U)$.

\begin{proposition}\label{lem-for-den-ofdine'lddd}
Let $\chi_j$, $j\in\ZZ_+$, be as in Proposition \ref{seq-den-comsmf}. For every $K,\widetilde{K}\subseteq U$ satisfying $K\subset\subset \operatorname{int}\widetilde{K}$ and $\widetilde{K}\subset\subset U$ there is $j_0\in\ZZ_+$ such that $\chi_j*u\in\DD_{\widetilde{K}}$ for all $u\in\EE'^r_{L;K}(U)$, $j\geq j_0$ and $L$ a closed conic subset of $U\times(\RR^n\backslash\{0\})$ satisfying $\pr_1(L)\subseteq K$. Furthermore $\chi_j*u\rightarrow u$ in $\EE'^r_{L;\widetilde{K}}(U)$ for all $u\in\EE'^r_{L;K}(U)$.\\
\indent If $P_j$, $j\in\ZZ_+$, are the operators from Proposition \ref{seq-den-comsmf}, then $P_j\rightarrow \operatorname{Id}$ in $\mathcal{L}_p(\EE'^r_W(U))$ for any open conic subset $W$ of $U\times (\RR^n\backslash\{0\})$. In particular, $\DD(U)$ is sequentially dense in $\EE'^r_W(U)$.
\end{proposition}

\begin{proof} The proof of the first part is similar to the proof of Proposition \ref{seq-den-comsmf} and we omit it. This implies $P_j\rightarrow \operatorname{Id}$ in $\mathcal{L}_{\sigma}(\EE'^r_W(U))$ and, since $\EE'^r_W(U)$ is barrelled, the Banach-Steinhaus theorem \cite[Theorem 4.5, p. 85]{Sch} yields that the convergence holds in $\mathcal{L}_p(\EE'^r_W(U))$.
\end{proof}

We need the following preparatory lemma for the result on the duality (see the preliminaries for the meaning of $\check{L})$.

\begin{lemma}\label{lem-for-inc-ofeinddualsis}
Let $\emptyset\neq\widetilde{K}\subset\subset U$ and let $L$ and $\widetilde{L}$ be closed conic subsets of $U\times (\RR^n\backslash\{0\})$ satisfying $\widetilde{L}\subseteq \check{L}^c$ and $\pr_1(\widetilde{L})\subseteq \widetilde{K}$.
\begin{itemize}
\item[$(i)$] There are open subsets $U_1,\ldots,U_k$ of $U$, closed cones $V_1,\ldots,V_k$ in $\RR^n$ and $\phi_j\in\DD(U_j)$, $j=1,\ldots,k$, which satisfy the following conditions:
    \begin{gather}
    \widetilde{K}\subseteq \bigcup_{j=1}^k U_j,\quad (U_j\times (-V_j))\cap L=\emptyset\,\, \mbox{and}\,\,(U_j\times (\RR^n\backslash V_j))\cap \widetilde{L}=\emptyset,\,\, j=1,\ldots,k,\label{con-onc-for-eqsopessv}\\
    \phi_1^2+\ldots+\phi_k^2=1\,\, \mbox{on a neighbourhood of}\,\, \widetilde{K}\label{con-for-fun-paruniforintcon}.
    \end{gather}
\item[$(ii)$] If $U_j$, $V_j$ and $\phi_j$, $j=1,\ldots,k$, are as in $(i)$, then $\mathcal{F}(\phi_j u) \mathcal{F}^{-1}(\phi_jv)\in L^1(\RR^n)$, $j=1,\ldots,k$, for all $u\in\DD'^r_L(U)$ and $v\in\EE'^{-r}_{\widetilde{L};\widetilde{K}}(U)$, and, for each $v\in\EE'^{-r}_{\widetilde{L};\widetilde{K}}(U)$, the linear functional
\begin{equation}\label{lin-fun-for-dualityofewithdrs}
\DD'^r_L(U)\rightarrow\CC,\quad u\mapsto\sum_{j=1}^k\int_{\RR^n}\mathcal{F}(\phi_j u)(\xi)\mathcal{F}^{-1}(\phi_j v)(\xi) d\xi,
\end{equation}
is continuous. Furthermore, when $v$ varies in a bounded subset of $\EE'^{-r}_{\widetilde{L};\widetilde{K}}(U)$, \eqref{lin-fun-for-dualityofewithdrs} becomes an equicontinuous subset of $(\DD'^r_L(U))'$.
\end{itemize}
\end{lemma}

\begin{proof} To show $(i)$, assume first $L^c\neq \emptyset$. For $x_0\in\pr_1(\widetilde{L})$, the compactness of the unit sphere implies that there are an open set $U'\subseteq U$ that contains $x_0$ and an open cone $V'\subseteq \RR^n$ such that $(U'\times (-\overline{V'}))\cap L=\emptyset$ and $\{\xi\in\RR^n\,|\, (x_0,\xi)\in \widetilde{L}\}\subseteq V'$. We claim that there is an open neighbourhood $U_0\subseteq U'$ of $x_0$ such that $(U_0\times (\RR^n\backslash V'))\cap \widetilde{L}=\emptyset$. To verify that this is true, assume the contrary. Then there are $x^{(j)}\in U'$ and $\eta^{(j)}\in \mathbb{S}^{n-1}\cap (\RR^n\backslash V')$, $j\in\ZZ_+$, such that $x^{(j)}\rightarrow x_0$ and $(x^{(j)},\eta^{(j)})\in \widetilde{L}$, $j\in\ZZ_+$. In view of the compactness of the unit sphere, there is a subsequence $(\eta^{(j_k)})_{k\in\ZZ_+}$ which converges to some $\eta_0\in\mathbb{S}^{n-1}\cap(\RR^n\backslash V')$. Hence $(x_0,\eta_0)\in \widetilde{L}$ and thus $\eta_0\in V'$ which is a contradiction. We showed that there are an open neighbourhood $U_0\subseteq U$ of $x_0$ and a closed cone $V_0$ in $\RR^n$ such that
\begin{equation}\label{con-onc-nei-forsplks}
(U_0\times (-V_0))\cap L=\emptyset\quad \mbox{and}\quad (U_0\times (\RR^n\backslash V_0))\cap \widetilde{L}=\emptyset.
\end{equation}
When $x_0\in \widetilde{K}\backslash \pr_1(\widetilde{L})$, we can take $V_0:=\{0\}$ and $U_0:=U\backslash \pr_1(\widetilde{L})$ for \eqref{con-onc-nei-forsplks} to be satisfied. Now, a compactness argument verifies the existence of $U_j$ and $V_j$, $j=1,\ldots,k$, as in $(i)$ which satisfy \eqref{con-onc-for-eqsopessv}. The existence of $\phi_j\in\DD(U_j)$, $j=1,\ldots,k$, which satisfy \eqref{con-for-fun-paruniforintcon} is straightforward. In the case when $L^c=\emptyset$, \eqref{con-onc-for-eqsopessv} is satisfied by taking $k=1$, $U_1\subseteq U$ an open neighbourhood of $\widetilde{K}$ and $V_1=\{0\}$; in this case, for \eqref{con-for-fun-paruniforintcon}, we take $\phi_1\in \DD(U_1)$ which equals $1$ on a neighbourhood of $\widetilde{K}$.\\
\indent To prove $(ii)$, let $B$ be a bounded subset of $\EE'^{-r}_{\widetilde{L};\widetilde{K}}(U)$ and let $U_j$, $V_j$ and $\phi_j$, $j=1,\ldots,k$, be as in $(i)$. For $v\in B$ and $u\in\DD'^r_L(U)$, we estimate as follows (cf. \eqref{con-onc-for-eqsopessv}):
\begin{align*}
\int_{\RR^n}&|\mathcal{F}(\phi_j u)(\xi)||\mathcal{F}^{-1}(\phi_j v)(\xi)|d\xi\\
&\leq \int_{-V_j}|\mathcal{F}(\phi_j u)(\xi)||\mathcal{F}(\phi_j v)(-\xi)|d\xi+\int_{\RR^n\backslash (-V_j)}|\mathcal{F}(\phi_j u)(\xi)||\mathcal{F}(\phi_j v)(-\xi)|d\xi\\
&\leq \mathfrak{p}_{r;\phi_j,-V_j}(u)\sup_{v\in B}\|\langle\cdot\rangle^{-r}\mathcal{F}(\phi_j v)\|_{L^2(\RR^n)}+\sup_{v\in B}\int_{\RR^n\backslash (-V_j)}|\mathcal{F}(\phi_j u)(\xi)||\mathcal{F}(\phi_j v)(-\xi)|d\xi.
\end{align*}
We claim that
$$
\DD'(U)\rightarrow[0,\infty),\quad f\mapsto\sup_{v\in B}\int_{\RR^n\backslash (-V_j)}|\mathcal{F}(\phi_j f)(\xi)||\mathcal{F}(\phi_j v)(-\xi)|d\xi,
$$
is a continuous seminorm on $\DD'(U)$. This can be proven by showing that it is bounded on bounded subsets on $\DD'(U)$ by employing similar technique as in the prove of the Claim in Proposition \ref{pro-for-top-imbedingthforc} ($\sup_{v\in B}\sup_{\xi\in\RR^n\backslash V_j}|\mathcal{F}(\phi_j v)(\xi)|\langle\xi\rangle^{\nu}<\infty$, for all $\nu>0$, in view of \eqref{con-onc-for-eqsopessv} and the boundedness of $B$); whence, it is continuous since $\DD'(U)$ is bornological. Now, all claims in $(ii)$ immediately follow from the above estimate.
\end{proof}

\begin{remark}
Of course, we can always take $\phi_1,\ldots,\phi_k$ to be nonnegative.
\end{remark}

Now, we are ready to show that the strong dual of $\DD'^r_L(U)$ is $\EE'^{-r}_{\check{L}^c}(U)$. In the first part of the proof, we employ some of the ideas used in \cite[Proposition 7.6, p. 80]{GS} and \cite[Lemma 3 and Proposition 7]{BD} where the dual of $\DD'_L(U)$ was identified. However, the ideas from these results can not be applied in the second part of the proof: here we employ pseudo-differential operator techniques and $L^2$-estimates.

\begin{theorem}\label{the-for-dua-fordwithespacewithwfs}
Let $L$ be a closed conic subset of $U\times (\RR^n\backslash\{0\})$ and $r\in\RR$.
\begin{itemize}
\item[$(i)$] The bilinear map $\DD(U)\times \mathcal{C}^{\infty}(U)\rightarrow \CC$, $(\varphi,\psi)\mapsto\langle \varphi,\psi\rangle=\int_U\varphi(x)\psi(x)dx$, uniquely extends to a separately continuous bilinear mapping $\EE'^{-r}_{\check{L}^c}(U)\times \DD'^r_L(U)\rightarrow \CC$, given by
    \begin{equation}\label{equ-for-dua-fornsped}
    \langle v,u\rangle=\sum_{j=1}^k\int_{\RR^n}\mathcal{F}(\phi_j u)(\xi)\mathcal{F}^{-1}(\phi_j v)(\xi) d\xi,\quad u\in\DD'^r_L(U),\, v\in\EE'^{-r}_{\widetilde{L};\widetilde{K}}(U),
    \end{equation}
    with $\phi_j$, $j=1,\ldots,k$, as in Lemma \ref{lem-for-inc-ofeinddualsis} $(i)$.
\item[$(ii)$] It holds that $(\DD'^r_L(U))'_b=\EE'^{-r}_{\check{L}^c}(U)$ topologically and the duality $\langle \EE'^{-r}_{\check{L}^c}(U), \DD'^r_L(U)\rangle$ is given by \eqref{equ-for-dua-fornsped}.
\end{itemize}
\end{theorem}

\begin{proof} We first address $(i)$. In view of Lemma \ref{lem-for-inc-ofeinddualsis} $(ii)$ and the density of $\DD(U)$ in $\DD'^r_L(U)$, the right-hand side in \eqref{equ-for-dua-fornsped} does not depend on any of the choices made in Lemma \ref{lem-for-inc-ofeinddualsis} $(i)$, i.e. on $\widetilde{L}$, $\widetilde{K}$, $U_j$, $V_j$ and $\phi_j$; whence $\EE'^{-r}_{\check{L}^c}(U)\times \DD'^r_L(U)\rightarrow \CC$ is well-defined. Notice that when $v\in\DD(U)$, the right-hand side in \eqref{equ-for-dua-fornsped} is exactly the duality $\langle \DD(U),\DD'(U)\rangle$ and thus coincides with $\int_U v(x)u(x)dx$ when $u\in\mathcal{C}^{\infty}(U)$. Lemma \ref{lem-for-inc-ofeinddualsis} $(ii)$ verifies the continuity of $\DD'^r_L(U)\rightarrow \CC$, $u\mapsto\langle v,u\rangle$, for each fixed $v$. When $u$ is fixed, Lemma \ref{lem-for-inc-ofeinddualsis} $(ii)$ shows that $\EE'^{-r}_{\widetilde{L};\widetilde{K}}(U)\rightarrow \CC$, $v\mapsto\langle v,u\rangle$, maps bounded sets into bounded sets and hence it is continuous since $\EE'^{-r}_{\widetilde{L};\widetilde{K}}(U)$ is Fr\'echet space (hence bornological). This immediately implies the continuity of $\EE'^{-r}_{\check{L}^c}(U)\rightarrow \CC$, $v\mapsto\langle v,u\rangle$, when $u$ is fixed and the proof of $(i)$ is complete.\\
\indent We turn our attention to $(ii)$. In view of \eqref{inc-for-dual-ofd'l} and Proposition \ref{lem-for-den-ofdine'lddd}, we infer
\begin{equation}\label{con-inl-for-espfordinsss}
H^r_{\loc}(U)\subseteq (\EE'^{-r}_{\check{L}^c}(U))'_b\subseteq \DD'(U)\quad \mbox{continuously}.
\end{equation}
We are going to show
\begin{gather}
(\EE'^{-r}_{\check{L}^c}(U))'=\DD'^r_L(U)\,\, \mbox{and}\,\, \EE'^{-r}_{\check{L}^c}(U)=(\DD'^r_L(U))'\,\, \mbox{algebraically and}\label{idenfotrduals}\\
\mbox{the identity mappings}\,\, (\EE'^{-r}_{\check{L}^c}(U))'_b\rightarrow\DD'^r_L(U)\,\, \mbox{and}\,\, \EE'^{-r}_{\check{L}^c}(U)\rightarrow (\DD'^r_L(U))'_b\,\, \mbox{are continuous.}\label{iden-for-dualpariiconts}
\end{gather}
This would immediately give the continuous inclusions $\EE'^{-r}_{\check{L}^c}(U)\rightarrow (\DD'^r_L(U))'_b\rightarrow ((\EE'^{-r}_{\check{L}^c}(U))'_b)'_b$, which, in view of the fact that the evaluation map into the strong bidual $\EE'^{-r}_{\check{L}^c}(U)\rightarrow ((\EE'^{-r}_{\check{L}^c}(U))'_b)'_b$ is a topological imbedding since $\EE'^{-r}_{\check{L}^c}(U)$ is barrelled, would yield that $\EE'^{-r}_{\check{L}^c}(U)\rightarrow (\DD'^r_L(U))'_b$ is a topological imbedding\footnote{Here we employed the fact that if the continuous injections between topological spaces $f:X_1\rightarrow X_2$ and $g:X_2\rightarrow X_3$ are such that $g\circ f:X_1\rightarrow X_3$ is a topological imbedding, then so is $f$.} and hence $\EE'^{-r}_{\check{L}^c}(U)=(\DD'^r_L(U))'_b$ topologically. Thus we need to show \eqref{idenfotrduals} and \eqref{iden-for-dualpariiconts}.\\
\indent \underline{STEP 1: $\EE'^{-r}_{\check{L}^c}(U)=(\DD'^r_L(U))'$ and the continuity of $\EE'^{-r}_{\check{L}^c}(U)\rightarrow (\DD'^r_L(U))'_b$.} The validity of $\EE'^{-r}_{\check{L}^c}(U)\subseteq (\DD'^r_L(U))'_b$ immediately follows from $(i)$. Its continuity is a consequence of Lemma \ref{lem-for-inc-ofeinddualsis} $(ii)$; indeed, the latter shows that the inclusion $\EE'^{-r}_{\widetilde{L};\widetilde{K}}(U)\subseteq (\DD'^r_L(U))'_b$ maps bounded sets into bounded sets and hence it is continuous (as $\EE'^{-r}_{\widetilde{L};\widetilde{K}}(U)$ is Fr\'echet) which, in turn, shows the continuity of $\EE'^{-r}_{\check{L}^c}(U)\subseteq (\DD'^r_L(U))'_b$. To verify that $\EE'^{-r}_{\check{L}^c}(U)=(\DD'^r_L(U))'$ as sets, let $v\in(\DD'^r_L(U))'$. In view of \eqref{equ-for-dua-ofthedspacins}, it suffices to show $WF(v)\subseteq \check{L}^c$. There is $C>0$, a bounded subset $B_0$ of $\DD(U)$ and $\psi_1,\ldots,\psi_k\in\DD(U)$ and closed cones $V_1,\ldots, V_k$ in $\RR^n$ satisfying $(\supp\psi_j\times V_j)\cap L=\emptyset$, $j=1,\ldots,k$, such that
$$
|\langle v,u\rangle|\leq C\sum_{j=1}^k\mathfrak{p}_{r;\psi_j,V_j}(u)+C\sup_{\varphi\in B_0}|\langle u,\varphi\rangle|,\quad u\in\DD'^r_L(U).
$$
We are going to show that $WF(v)\subseteq \bigcup_{j=1}^k\supp\psi_j\times (-V_j\backslash\{0\})$; as the latter is a subset of $\check{L}^c$, this will complete the proof of $\EE'^{-r}_{\check{L}^c}(U)=(\DD'^r_L(U))'$. Let $\psi\in\DD(U)$ and the closed cone $V$ in $\RR^n$ be such that $(\supp\psi\times V)\cap (\bigcup_{j=1}^k\supp\psi_j\times (-V_j\backslash\{0\}))=\emptyset$. For $\xi\in \RR^n$, define $\phi_{\xi}:=e^{-i\xi\,\cdot}\psi\in\DD(U)$. Since $\mathcal{F}(\psi v)(\xi)= \langle v,\phi_{\xi}\rangle$, $\xi\in\RR^n$, for $\nu>0$, we infer
$$
\sup_{\xi\in V}\langle \xi\rangle^{\nu}|\mathcal{F}(\psi v)(\xi)|\leq C\sum_{j=1}^k\sup_{\xi\in V}\langle\xi\rangle^{\nu}\mathfrak{p}_{r;\psi_j,V_j}(\phi_{\xi})+C\sup_{\xi\in V}\sup_{\varphi\in B_0}\langle\xi\rangle^{\nu}|\langle \phi_{\xi},\varphi\rangle|.
$$
Assume that $V\backslash\{0\}\neq \emptyset$. As $\langle\phi_{\xi},\varphi\rangle=\mathcal{F}(\psi\varphi)(\xi)$, we infer $\sup_{\xi\in V}\sup_{\varphi\in B_0}\langle\xi\rangle^{\nu}|\langle \phi_{\xi},\varphi\rangle|<\infty$. We only need to bound $\mathfrak{p}_{r;\psi_j,V_j}(\phi_{\xi})$ when $V_j\backslash\{0\}\neq \emptyset$. Notice that it is zero if $\supp\psi\cap\supp\psi_j=\emptyset$. When $\supp\psi\cap\supp\psi_j\neq\emptyset$, we have $V\cap(-V_j\backslash\{0\})=\emptyset$ and thus there is $c>0$ such that $|\xi/|\xi|-\eta/|\eta||>c$, $\xi\in V\backslash\{0\}$, $\eta\in -V_j\backslash\{0\}$. For such $\xi$ and $\eta$, we have
$$
|\xi-\eta|\geq \left|\frac{\xi}{|\xi|}-\frac{\eta}{|\eta|}\right||\xi|-\left|\frac{\eta}{|\eta|}-\frac{\eta}{|\xi|}\right||\xi|> c|\xi|-||\xi|-|\eta||\geq c|\xi|-|\xi-\eta|,
$$
and thus $|\xi-\eta|>c|\xi|/2$; analogously, $|\xi-\eta|>c|\eta|/2$. Consequently,
\begin{align*}
\sup_{\xi\in V}\langle\xi\rangle^{\nu}\mathfrak{p}_{r;\psi_j,V_j}=\sup_{\xi\in V}\left(\int_{-V_j}\langle\xi\rangle^{2\nu}\langle \eta\rangle^{2r}|\mathcal{F}(\psi\psi_j)(\xi-\eta)|^2d\eta\right)^{1/2}<\infty
\end{align*}
(as $\mathcal{F}(\psi\psi_j)\in\SSS(\RR^n)$). We deduce $\sup_{\xi\in V}\langle \xi\rangle^{\nu}|\mathcal{F}(\psi v)(\xi)|<\infty$. Since this trivially holds in the case when $V\backslash\{0\}=\emptyset$, the proof of $\EE'^{-r}_{\check{L}^c}(U)=(\DD'^r_L(U))'$ is complete.\\
\indent \underline{STEP 2: $(\EE'^{-r}_{\check{L}^c}(U))'=\DD'^r_L(U)$ and the continuity of $(\EE'^{-r}_{\check{L}^c}(U))'_b\rightarrow\DD'^r_L(U)$.} Since $\DD'^r_L(U)$ is semi-reflexive (cf. Corollary \ref{rem-for-sem-refd}), STEP 1 verifies the algebraic inclusion $\DD'^r_L(U)\subseteq (\EE'^{-r}_{\check{L}^c}(U))'$. It remains to show $(\EE'^{-r}_{\check{L}^c}(U))'_b\subseteq \DD'^r_L(U)$ continuously. This is trivial when $L^c=\emptyset$. Assume that $L^c\neq\emptyset$. We estimate the seminorm $\mathfrak{p}_{r;\psi,V}$ of the distributions in $(\EE'^{-r}_{\check{L}^c}(U))'$, where $\psi\in\DD(U)\backslash\{0\}$ and $V$ is a closed cone in $\RR^n$ such that $(\supp\psi\times V)\cap L=\emptyset$ and $V\backslash\{0\}\neq \emptyset$. Pick an open cone $V'\subseteq \RR^n$ such that $V\backslash\{0\}\subseteq V'$ and $(\supp\psi\times \overline{V'})\cap L=\emptyset$. We are going to show that for all $u\in(\EE'^{-r}_{\check{L}^c}(U))'$, the seminorm $\mathfrak{p}_{r;\psi,V'}(u)$ is bounded from above by a continuous seminorm on $(\EE'^{-r}_{\check{L}^c}(U))'_b$ of $v$. Since $\mathfrak{p}_{r;\psi,V}\leq\mathfrak{p}_{r;\psi,V'}$, this would imply $(\EE'^{-r}_{\check{L}^c}(U))'_b\subseteq \DD'^r_L(U)$ continuously; in view of \eqref{con-inl-for-espfordinsss} for every continuous seminorm $\mathfrak{p}$ on $\DD'(U)$, $\mathfrak{p}(u)$ is bounded by a continuous seminorm on $(\EE'^{-r}_{\check{L}^c}(U))'_b$ of $u$.\\
\indent Denoting $\widetilde{K}:=\supp\psi$ and $\widetilde{L}:=\supp\psi\times (-\overline{V'}\backslash\{0\})$, we infer $\widetilde{L}\subseteq \check{L}^c$ and $\pr_1(\widetilde{L})=\widetilde{K}$. Since $V'$ is open, for $u\in(\EE'^{-r}_{\check{L}^c}(U))'$, we have
\begin{equation}\label{ide-for-nor-nl2forsemind}
\mathfrak{p}_{r;\psi,V'}(u)=\sup\left\{\left|\int_{V'}\mathcal{F}(\psi u)(\xi)\chi(\xi) d\xi\right|\,\Bigg|\, \chi\in\DD(V'),\, \|\langle \cdot\rangle^{-r}\chi\|_{L^2(V')}\leq 1\right\}.
\end{equation}
For $\chi\in\DD(V')$ satisfying $\|\langle\cdot\rangle^{-r}\chi\|_{L^2(V')}\leq 1$, we infer
\begin{equation}\label{ide-for-new-parinfofoursf}
\left|\int_{V'}\mathcal{F}(\psi u)(\xi)\chi(\xi) d\xi\right|=\left|\int_{\RR^n}\mathcal{F}(\psi u)(\xi)\chi(\xi) d\xi\right|=|\langle u,\psi\mathcal{F}\chi\rangle|.
\end{equation}
Notice that $\psi\mathcal{F}\chi\in\EE'^{-r}_{\widetilde{L};\widetilde{K}}(U)$. We claim that $\widetilde{B}:=\{\psi\mathcal{F}\chi\,|\, \chi\in\DD(V'),\,\|\langle\cdot\rangle^{-r}\chi\|_{L^2(V')}\leq 1\}$ is a bounded subset of $\EE'^{-r}_{\widetilde{L};\widetilde{K}}(U)$. We have
\begin{align*}
\|\psi\mathcal{F}\chi\|_{H^{-r}(\RR^n)}^2&=\int_{\RR^n}|\mathcal{F}(\psi\mathcal{F}\chi)(\xi)|^2\langle\xi\rangle^{-2r}d\xi
= \int_{\RR^n}\left|\mathcal{F}\left(\langle D\rangle^{-r}(\psi \langle D\rangle^r \mathcal{F}(\langle\cdot\rangle^{-r}\chi))\right)(\xi)\right|^2d\xi\\
&=(2\pi)^n\left\|\langle D\rangle^{-r}(\psi \langle D\rangle^r \mathcal{F}(\langle\cdot\rangle^{-r}\chi))\right\|_{L^2(\RR^n)}^2.
\end{align*}
Since $\langle D\rangle^{-r} \psi\langle D\rangle^r$ is a pseudo-differential operator with symbol in $S^0(\RR^{2n})$, it is continuous on $L^2(\RR^n)$. Hence,
$$
\|\psi\mathcal{F}\chi\|_{H^{-r}(\RR^n)}\leq C'\|\mathcal{F}(\langle\cdot\rangle^{-r}\chi)\|_{L^2(\RR^n)}\leq C',\quad \chi\in\DD(V'),\, \|\langle \cdot\rangle^{-r}\chi\|_{L^2(V')}\leq 1.
$$
Let $\varphi\in\DD(U)\backslash\{0\}$ and the closed cone $V_0\subseteq \RR^n$, $V_0\backslash\{0\}\neq\emptyset$, be such that $(\supp\varphi\times V_0)\cap \widetilde{L}=\emptyset$. If $\supp\varphi\cap\supp\psi=\emptyset$, then $\mathfrak{q}_{\nu;\varphi;V_0}(\widetilde{B})=0$, for all $\nu>0$. Assume that $\supp\varphi\cap\supp\psi\neq\emptyset$. Then $V_0\cap (-\overline{V'}\backslash\{0\})=\emptyset$. Similarly as before, this implies that there is $c'>0$ such that $|\xi-\eta|>c'|\xi|$ and $|\xi-\eta|>c'|\eta|$, $\xi\in V_0\backslash\{0\}$, $\eta\in-\overline{V'}\backslash\{0\}$. For $\nu>0$, we have
\begin{align*}
\sup_{\xi\in V_0}\langle \xi\rangle^{\nu}|\mathcal{F}(\varphi\psi\mathcal{F}\chi)(\xi)|&=\sup_{\xi\in V_0}\left|\int_{\RR^n}\langle\xi\rangle^{\nu}\mathcal{F}(\varphi\psi)(\xi-\eta)\chi(-\eta) d\eta\right|\\
&\leq C_1\int_{-V'}\langle\eta\rangle^{-r-n}|\chi(-\eta)| d\eta\leq C_2,
\end{align*}
for all $\chi\in\DD(V')$ satisfying $\|\langle\cdot\rangle^{-r}\chi\|_{L^2(V')}\leq 1$. We conclude that $\widetilde{B}$ is a bounded subset of $\EE'^{-r}_{\widetilde{L},\widetilde{K}}(U)$. In view of \eqref{ide-for-nor-nl2forsemind} and \eqref{ide-for-new-parinfofoursf}, we deduce that $\mathfrak{p}_{r;\psi,V'}(u)= \sup_{\phi\in \widetilde{B}}|\langle u,\phi\rangle|$. This completes the proof of STEP 2 since the right-hand side is a continuous seminorm on $(\EE'^{-r}_{\check{L}^c}(U))'_b$.\\
\indent The fact that the duality $\langle \EE'^{-r}_{\check{L}^c}(U), \DD'^r_L(U)\rangle$ is given by \eqref{equ-for-dua-fornsped} follows from $(i)$. The proof of the theorem is complete.
\end{proof}

\subsection{Pullback by smooth maps}

Following H\"ormander \cite[Section 8.2]{hor}, \cite[Subsection 2.5]{hormander}, given a smooth map $f:M\rightarrow N$ between the manifolds $M$ and $N$, we denote by $\mathcal{N}_f$ the following conic subset of $T^*N$:
\begin{equation}\label{nor-for-map-mani-cas-nrs}
\mathcal{N}_f:=\{(f(p),\eta)\in T^*N\,|\, \eta\in T^*_{f(p)}N,\, df^*_p\eta=0\in T^*_pM\}.
\end{equation}
If $L$ is a closed conic subset of $T^*N\backslash 0$ which satisfies $L\cap \mathcal{N}_f=\emptyset$, then
\begin{equation}\label{pul-bac-for-con-sub-ctb}
f^*L:=\{(p,df^*_p\eta)\in T^*M\,|\, (f(p),\eta)\in L\}
\end{equation}
is a closed conic subset of $T^*M\backslash 0$. Given a smooth map $g:\widetilde{M}\rightarrow M$ and diffeomorphisms $\iota_1:M\rightarrow \widetilde{M}$ and $\iota_2:N\rightarrow \widetilde{N}$, it is straightforward to check that
\begin{equation}\label{equ-for-con-subwithdifandordmass}
(f\circ g)^*L=g^*f^*L,\quad \mathcal{N}_{\iota_2\circ f}=(\iota^{-1}_2)^*\mathcal{N}_f\quad \mbox{and}\quad \mathcal{N}_{f\circ\iota^{-1}_1}=\mathcal{N}_f.
\end{equation}
When $M$ and $N$ are two open subsets $O$ and $U$ of $\RR^m$ and $\RR^n$ respectively, employing the canonical identifications, \eqref{nor-for-map-mani-cas-nrs} and \eqref{pul-bac-for-con-sub-ctb} boil down to
\begin{align*}
\mathcal{N}_f&=\{(f(x),\eta)\in U\times \RR^n\,|\, x\in O,\,{}^tf'(x)\eta=0\},\\
f^*L&=\{(x,{}^tf'(x)\eta)\in O\times \RR^m\,|\, (f(x),\eta)\in L\}.
\end{align*}
We are now ready to state and prove the result on the pullback by smooth maps.

\begin{theorem}\label{the-pul-bac-for-smcrmdiff}
Let $O$ and $U$ be open subsets of $\RR^m$ and $\RR^n$ respectively, let $f:O\rightarrow U$ be a smooth map and let $L$ be a closed conic subset of $U\times (\RR^n\backslash\{0\})$ satisfying $L\cap \mathcal{N}_f=\emptyset$. The pullback $f^*:\mathcal{C}^{\infty}(U)\rightarrow \mathcal{C}^{\infty}(O)$, $f^*(u)=u\circ f$, uniquely extends to a well defined and continuous mapping $f^*:\DD'^{r_2}_L(U)\rightarrow \DD'^{r_1}_{f^*L}(O)$ when $r_2-r_1>n/2$ and $r_2>n/2$.\\
\indent If $f$ has constant rank $k\geq 1$, then this is valid when $r_2-r_1\geq (n-k)/2$ and $r_2>(n-k)/2$. When $f$ is a submersion, $f^*:\DD'^{r_2}_L(U)\rightarrow \DD'^{r_1}_{f^*L}(O)$ is well-defined and continuous even when $r_2\geq r_1$. Consequently, if $f$ is a diffeomorphism, then $f^*:\DD'^r_L(U)\rightarrow \DD'^r_{f^*L}(O)$ is a topological isomorphism for each $r\in\RR$.
\end{theorem}

\begin{remark}
We consider the trivial case when $f$ has constant rank $0$ in Remark \ref{rem-for-zer-rankmappulbc} below. Before we prove the theorem, we point out the following:\\
\indent $(i)$ When $L=\emptyset$ the theorem states that the pullback is well-defined and continuous map $f^*:H^{r_2}_{\loc}(U)\rightarrow H^{r_1}_{\loc}(O)$ with $r_1$ and $r_2$ as in the theorem. If $m<n$, $O=U\cap \RR^m$ viewed as an open subset of $\RR^m$ and $f:O\rightarrow U$ the canonical imbedding, then $f^*$ is just restriction and the theorem claims that $f^*:H^{r_2}_{\loc}(U)\rightarrow H^{r_1}_{\loc}(O)$ is well-defined and continuous when $r_2\geq r_1+(n-m)/2$ and $r_2>(n-m)/2$. Thus, we can view this case as a local version of the Sobolev imbedding theorem for restrictions to lower dimensional hyperplanes \cite[Theorem 4.12, p. 85]{adams}; this case is also shown in \cite[Appendix B, p. 476]{hor2}.\\
\indent $(ii)$ In the constant rank case, the theorem can not be much improved. In Appendix \ref{app-for-cou-foroptthpulbacsmom} we give examples of maps with constant rank where the theorem fails if $r_2<(n-k)/2$ or $r_2-r_1<(n-k)/2$. The only open problem is the case when $r_2= (n-k)/2$ and $r_2-r_1\geq (n-k)/2$ which is equivalent to $r_2=(n-k)/2$ and $r_1\leq0$.
\end{remark}

\begin{proof}[Proof of Theorem \ref{the-pul-bac-for-smcrmdiff}] Throughout the proof, for $y\in U$, we denote $L_y:=\{\eta\in\RR^n\backslash\{0\}\,|\, (y,\eta)\in L\}$. Notice that $L_y$ is a closed cone in $\RR^n\backslash\{0\}$ (which may be empty!). We employ analogous notations for closed conic subsets of $O\times (\RR^m\backslash\{0\})$.\\
\indent We first make the following observations for a general smooth map $f:O\rightarrow U$ satisfying $L\cap \mathcal{N}_f=\emptyset$. Let $x_0\in O$ be arbitrary but fixed and set $y_0:=f(x_0)\in U$. Let $G$ be an open cone in $\RR^m\backslash\{0\}$ such that $(f^* L)_{x_0}={}^t f'(x_0)L_{y_0}\subseteq G$. The compactness of $\mathbb{S}^{n-1}\cap L_{y_0}$ implies that there are closed cones $V'$ and $V$ in $\RR^n\backslash\{0\}$ such that $L_{y_0}\subseteq\operatorname{int}V'\subseteq V'\subseteq \operatorname{int} V$ and ${}^tf'(x_0) V\subseteq G$; consequently, ${}^tf'(x_0)\eta\neq 0$, $\eta\in V$ (when $f'(x_0)=0$ the condition $L\cap\mathcal{N}_f=\emptyset$ implies $L_{y_0}=\emptyset$ and the above is satisfied with $V=V'=\emptyset$). There is an open neighbourhood $U_0\subseteq U$ of $y_0$ such that $\bigcup_{y\in U_0} L_y\subseteq \operatorname{int} V'$. To see that this is true, assume the contrary. Then there is a sequence $(y^{(j)})_{j\in\ZZ_+}$ which tends to $y_0$ such that for each $j\in\ZZ_+$ there is $\eta^{(j)}\in L_{y^{(j)}}\backslash\operatorname{int} V'$ and $|\eta^{(j)}|=1$. The compactness of $\mathbb{S}^{n-1}$ implies that there is a subsequence $(\eta^{(j_k)})_{k\in\ZZ_+}$ which converges to some $\eta\in\mathbb{S}^{n-1}\backslash\operatorname{int} V'$. Since $(y^{(j_k)},\eta^{(j_k)})\in L$, $k\in\ZZ_+$, we infer $\eta\in L_{y_0}\backslash\operatorname{int} V'$ which is a contradiction since the latter set is empty. Consequently, the open neighbourhood $U_0$ exists. The continuity of $(x,\eta)\mapsto {}^t f'(x)\eta$ together with the compactness of $V\cap \mathbb{S}^{n-1}$ yields that there is a relatively compact open neighbourhood $O_0$ of $x_0$ such that $\overline{O_0}\subseteq O$ and ${}^tf'(x)\eta \in G$, $x\in \overline{O_0}$, $\eta\in V$. We take $O_0$ small enough so that $f(\overline{O_0})\subseteq U_0$. The fact ${}^tf'(x)\eta \in G$, $x\in \overline{O_0}$, $\eta\in V$, together with the continuity of the function $(x,\xi,\eta)\mapsto {}^t f'(x)\eta-\xi$ implies that there is $\varepsilon>0$ such that $|{}^t f'(x)\eta-\xi|\geq \varepsilon$ on the compact set $\overline{O_0} \times \{(\xi,\eta)\in (\RR^m\backslash G)\times (V\cup\{0\}) \,|\, |\xi|+|\eta|=1\}$. Consequently
\begin{equation}\label{ine-for-staphamethhor}
|{}^t f(x)\eta-\xi|\geq \varepsilon(|\xi|+|\eta|),\quad x\in\overline{O_0},\, \xi\in \RR^m\backslash G,\, \eta\in V\cup\{0\}.
\end{equation}
Given $\varphi\in\DD(O)$, we define
\begin{align*}
&I_{\varphi}:\RR^n\rightarrow \CC,\quad I_{\varphi}(\eta):=\int_O e^{i f(x)\eta} \varphi(x)dx,\\
&\widetilde{I}_{\varphi}:\RR^m\times \RR^n\rightarrow \CC,\quad \widetilde{I}_{\varphi}(\xi,\eta):=\int_O e^{i (f(x)\eta-x\xi)} \varphi(x)dx.
\end{align*}
Clearly $I_{\varphi}\in\DD_{L^{\infty}}(\RR^n)$, $\widetilde{I}_{\varphi}\in\DD_{L^{\infty}}(\RR^{m+n})$ and $\widetilde{I}_{\varphi}(0,\eta)=I_{\varphi}(\eta)$, $\eta\in\RR^n$. For $u\in \DD(U)$ and $\varphi\in\DD(O)$ it holds that
\begin{align*}
\langle f^*u,\varphi\rangle=\frac{1}{(2\pi)^n} \int_O\int_{\RR^n} e^{i f(x)\eta} \mathcal{F}u(\eta) \varphi(x)d\eta dx=\frac{1}{(2\pi)^n}\int_{\RR^n} \mathcal{F}u(\eta) I_{\varphi}(\eta)d\eta.
\end{align*}
\indent\underline{CASE 1: $f$ is smooth and satisfies $L\cap\mathcal{N}_f=\emptyset$.} We show that $f^*:\mathcal{C}^{\infty}(U)\rightarrow \mathcal{C}^{\infty}(O)$ uniquely extends to a well-defined continuous mapping $f^*:\DD'^{r_2}_L(U)\rightarrow\DD'^{r_1}_{f^*L}(O)$, when $r_2-r_1>n/2$ and $r_2>n/2$. Let $\varphi\in \DD(O)\backslash\{0\}$ and the closed cone $\emptyset\neq G_1\subseteq \RR^m$ be such that $(\supp\varphi\times G_1)\cap f^*L=\emptyset$ (when $f^*L=\RR^m\backslash\{0\}$, we take $G_1=\{0\}$). Set $G:=\RR^m\backslash G_1$. Then $G$ is an open cone in $\RR^m\backslash\{0\}$ and ${}^tf'(x)L_{f(x)}\subseteq G$, $x\in\supp\varphi$. We apply the above construction for this $G$ and each $x\in \supp\varphi$ to obtain the open neighbourhoods $O_x$ and $U_x$ of $x$ and $f(x)$ respectively having the above properties. As $\supp\varphi$ is compact there are finitely many such $O_j$ $j=1,\ldots,l$, whose union covers $\supp\varphi$. We denote by $U_j$, $j=1,\ldots,l$, the corresponding subsets of $U$ and by $V'_j$ and $V_j$, $j=1,\ldots,l$, the corresponding closed cones in $\RR^n\backslash\{0\}$ from the above construction. Let $\psi_j\in\DD(O_j)$, $0\leq \psi_j\leq 1$, $j=1,\ldots, l$, be such that $\sum_{j=1}^l\psi_j=1$ on a neighbourhood of $\supp\varphi$. Pick $\phi_j\in\DD(U_j)$ such that $\phi_j=1$ on a neighbourhood of $f(\overline{O_j})$, $j=1,\ldots,l$. For $u\in\mathcal{C}^{\infty}(U)$, we infer
\begin{align}
\mathcal{F}(\varphi f^*u)(\xi)&= \sum_{j=1}^l \langle f^*(\phi_ju), e^{-i\,\cdot\,\xi} \psi_j\varphi\rangle\label{equ-for-part-uni-nowrtsh-frtpl}\\
&= \frac{1}{(2\pi)^n} \sum_{j=1}^l\int_{\RR^n}\mathcal{F}(\phi_j u)(\eta)\widetilde{I}_{\psi_j\varphi}(\xi,\eta)d\eta= \frac{1}{(2\pi)^n}\sum_{j=1}^l(I_{1;j}(\xi)+I_{2;j}(\xi)),\nonumber
\end{align}
with
$$
I_{1;j}(\xi):=\int_{V_j}\mathcal{F}(\phi_j u)(\eta)\widetilde I_{\psi_j\varphi}(\xi,\eta)d\eta,\quad I_{2;j}(\xi):=\int_{\RR^n\backslash V_j}\mathcal{F}(\phi_j u)(\eta)\widetilde{I}_{\psi_j\varphi}(\xi,\eta)d\eta.
$$
Hence,
\begin{equation}\label{sum-for-sem-nor-for-wfsespa}
\mathfrak{p}_{r_1;\varphi, G_1}(f^*u)\leq \sum_{j=1}^l\|\langle \cdot\rangle^{r_1}I_{1;j}\|_{L^2(G_1)}+\sum_{j=1}^l\|\langle \cdot\rangle^{r_1}I_{2;j}\|_{L^2(\RR^m)}.
\end{equation}
In view of \eqref{ine-for-staphamethhor}, the stationary phase method \cite[Theorem 7.7.1, p. 216]{hor} verifies that for every $N>0$ there is $C_N>0$ such that $|\widetilde{I}_{\psi_j\varphi}(\xi,\eta)|\leq C_N(1+|\xi|+|\eta|)^{-N}$, $\xi\in G_1$, $\eta\in V_j$, $j=1,\ldots,l$. Hence, employing the same technique as in the proof of the Claim in the proof of Proposition \ref{pro-for-top-imbedingthforc}, one shows that
$$
\DD'(U)\rightarrow[0,\infty),\, u\mapsto \left(\int_{G_1}\left(\int_{V_j}\langle \xi\rangle^{r_1}|\mathcal{F}(\phi_j u)(\eta)||\widetilde I_{\psi_j\varphi}(\xi,\eta)|d\eta\right)^2d\xi\right)^{1/2},\,\, j=1,\ldots,l,
$$
are continuous seminorms on $\DD'(U)$. Hence, the first sum in \eqref{sum-for-sem-nor-for-wfsespa} is bounded by a continuous seminorm on $\DD'(U)$ of $u$. To estimate the second sum in \eqref{sum-for-sem-nor-for-wfsespa}, first notice that
$$
|I_{2;j}(\xi)|\leq \|\langle \cdot\rangle^{r_2}\mathcal{F}(\phi_j u)\|_{L^2(\RR^n\backslash\operatorname{int}V_j)}\left(\int_{\RR^n\backslash V_j} \langle\eta\rangle^{-2r_2}|\widetilde{I}_{\psi_j\varphi}(\xi,\eta)|^2d\eta\right)^{1/2}
$$
and hence,
\begin{equation}\label{int-for-spl-fortwpa}
\|\langle \cdot\rangle^{r_1}I_{2;j}\|_{L^2(\RR^m)}\leq \mathfrak{p}_{r_2;\phi_j,\RR^n\backslash\operatorname{int} V_j}(u) \left(\int_{\RR^m\times\RR^n}\langle \xi\rangle^{2r_1}\langle\eta\rangle^{-2r_2}|\widetilde I_{\psi_j\varphi}(\xi,\eta)|^2 d\xi d\eta\right)^{1/2}.
\end{equation}
Notice that $\widetilde{I}_{\psi_j\varphi}(\xi,\eta)=\mathcal{F}(e^{if(\cdot)\eta} \psi_j\varphi)(\xi)$, $\xi\in\RR^m$, $\eta\in\RR^n$. If $r_1\leq 0$ then the last integral in \eqref{int-for-spl-fortwpa} is bounded by
$$
\left(\int_{\RR^n}\langle\eta\rangle^{-2r_2}\|\mathcal{F}(e^{if(\cdot)\eta}\psi_j\varphi)\|^2_{L^2(\RR^m)}d\eta\right)^{1/2}= (2\pi)^{m/2}\|\psi_j\varphi\|_{L^2(\RR^m)}\|\langle\cdot\rangle^{-r_2}\|_{L^2(\RR^n)}
$$
which is finite since $r_2>n/2$. Assume now that $r_1>0$. As standard, denote $\lfloor r_1\rfloor=\max\{k\in\ZZ\,|\, k\leq r_1\}$ and pick $k\in\NN$ and $l\in\ZZ_+$ so that $r_1\leq\lfloor r_1\rfloor+k/l$ and $r_2-\lfloor r_1\rfloor-k/l>n/2$. We estimate as follows:
\begin{align*}
\langle \xi\rangle^{2lr_1}&|\mathcal{F}(e^{if(\cdot)\eta}\psi_j\varphi)(\xi)|^{2l}\\
&\leq (1+|\xi|^2)^{l\lfloor r_1\rfloor +k} |\mathcal{F}(e^{if(\cdot)\eta}\psi_j\varphi)(\xi)|^{2l}\\
&= |\mathcal{F}(e^{if(\cdot)\eta}\psi_j\varphi)(\xi)|^{2(l-1)}\sum_{|\alpha|\leq l\lfloor r_1\rfloor+k}\frac{(l\lfloor r_1\rfloor+k)!}{(l\lfloor r_1\rfloor+k-|\alpha|)!\alpha!} |\mathcal{F}(\partial^{\alpha}(e^{if(\cdot)\eta}\psi_j\varphi))(\xi)|^2\\
&\leq C'_1|\mathcal{F}(e^{if(\cdot)\eta}\psi_j\varphi)(\xi)|^{2(l-1)} \sum_{|\alpha|\leq l\lfloor r_1\rfloor+k}\sum_{|\beta|\leq|\alpha|}\langle\eta\rangle^{2|\beta|}|\mathcal{F}(e^{if(\cdot)\eta}\varphi_{\alpha,\beta,j})(\xi)|^2,
\end{align*}
for some $\varphi_{\alpha,\beta,j}\in \DD(O_j)$. Hence
\begin{multline*}
\langle \xi\rangle^{r_1}|\mathcal{F}(e^{if(\cdot)\eta}\psi_j\varphi)(\xi)|\\
\leq C'^{1/(2l)}_1\langle\eta\rangle^{\lfloor r_1\rfloor+k/l} \sum_{|\alpha|\leq l\lfloor r_1\rfloor+k}\sum_{|\beta|\leq|\alpha|}|\mathcal{F}(e^{if(\cdot)\eta}\varphi_{\alpha,\beta,j})(\xi)|^{1/l} |\mathcal{F}(e^{if(\cdot)\eta}\psi_j\varphi)(\xi)|^{(l-1)/l}.
\end{multline*}
We employ the H\"older inequality with $p=l$ and $q=l/(l-1)$ to infer
\begin{align*}
\|&\langle \cdot\rangle^{r_1}\mathcal{F}(e^{if(\cdot)\eta}\psi_j\varphi)\|_{L^2(\RR^m)}\\
&\leq C'^{1/(2l)}_1\langle\eta\rangle^{\lfloor r_1\rfloor+k/l}  \sum_{|\alpha|\leq l\lfloor r_1\rfloor+k}\sum_{|\beta|\leq|\alpha|}\|\mathcal{F}(e^{if(\cdot)\eta}\varphi_{\alpha,\beta,j})\|^{1/l}_{L^2(\RR^m)} \|\mathcal{F}(e^{if(\cdot)\eta}\psi_j\varphi)\|^{(l-1)/l}_{L^2(\RR^m)}\\
&= (2\pi)^{m/2}C'^{1/(2l)}_1\langle\eta\rangle^{\lfloor r_1\rfloor+k/l} \|\psi_j\varphi\|^{(l-1)/l}_{L^2(\RR^m)} \sum_{|\alpha|\leq l\lfloor r_1\rfloor+k}\sum_{|\beta|\leq|\alpha|}\|\varphi_{\alpha,\beta,j}\|^{1/l}_{L^2(\RR^m)} \leq C'_2\langle\eta\rangle^{\lfloor r_1\rfloor+k/l}
\end{align*}
(with the obvious modifications when $l=1$ and thus $q=\infty$). Consequently, the last integral in \eqref{int-for-spl-fortwpa} is bounded by
\begin{align*}
\left(\int_{\RR^n}\langle\eta\rangle^{-2r_2}\|\langle \cdot\rangle^{r_1} \mathcal{F}(e^{if(\cdot)\eta}\psi_j\varphi)\|^2_{L^2(\RR^m)} d\eta\right)^{1/2}\leq C'_2\left(\int_{\RR^n}\langle\eta\rangle^{-2r_2+2\lfloor r_1\rfloor+2k/l}d\eta\right)^{1/2}<\infty.
\end{align*}
Employing these bounds in \eqref{int-for-spl-fortwpa}, the inequality \eqref{sum-for-sem-nor-for-wfsespa} immediately gives
$$
\mathfrak{p}_{r_1;\varphi, G_1}(f^*u)\leq \mathfrak{p}(u)+C'\sum_{j=1}^l\mathfrak{p}_{r_2;\phi_j,\RR^n\backslash\operatorname{int} V_j}(u),\quad u\in\mathcal{C}^{\infty}(U),
$$
where $\mathfrak{p}$ is a continuous seminorm of $\DD'(U)$; the summands are continuous seminorms on $\DD'^{r_2}_L(U)$ since $(U_j\times (\RR^n\backslash\operatorname{int}V_j))\cap L=\emptyset$, $j=1,\ldots,l$.\\
\indent Next we show similar bounds for $\mathfrak{p}(f^*u)$ where $\mathfrak{p}$ is an arbitrary continuous seminorm on $\DD'(O)$. Once we show this, we can deduce that $f^*:\mathcal{C}^{\infty}(U)\rightarrow \mathcal{C}^{\infty}(O)$ is continuous when $\mathcal{C}^{\infty}(U)$ and $\mathcal{C}^{\infty}(O)$ are equipped with the topologies induced by $\DD'^{r_2}_L(U)$ and $\DD'^{r_1}_{f^*L}(O)$ respectively, which, in view of Proposition \ref{seq-den-comsmf}, yields that $f^*$ uniquely extends to a well-defined and continuous map $f^*:\DD'^{r_2}_L(U)\rightarrow \DD'^{r_1}_{f^*L}(O)$, when $r_2-r_1>n/2$ and $r_2>n/2$. Let $\mathfrak{p}$ be a continuous seminorm on $\DD'(O)$; without loss in generality we can assume that $\mathfrak{p}=\sup_{\chi\in B}|\langle\cdot,\chi\rangle|$ for some bounded subset $B$ of $\DD(O)$. There exists $K\subset\subset O$ such that $B$ is a bounded subset of $\DD_K$. In the same way as above, we find open sets $O_j\subseteq O$, $j=1,\ldots,l$, with compact closures in $O$ which cover $K$ and corresponding open set $U_j\subseteq U$ and closed cones $V'_j$, $V_j$ in $\RR^n\backslash\{0\}$, $j=1,\ldots,l$ (apply the above construction with $G=\RR^m\backslash\{0\}$). As above, $\psi_j\in\DD(O_j)$, $0\leq \psi_j\leq 1$, $j=1,\ldots,l$, denotes a partition of unity on a neighbourhood of $K$, and $\phi_j\in\DD(U_j)$ is such that $\phi_j=1$ on a neighbourhood of $f(\overline{O_j})$, $j=1,\ldots,l$. For $u\in\mathcal{C}^{\infty}(U)$ and $\chi\in B$, we have
\begin{align*}
|\langle f^*u,\chi\rangle|&\leq\sum_{j=1}^l|\langle f^*(\phi_j u),\psi_j\chi\rangle|\leq \frac{1}{(2\pi)^n}\sum_{j=1}^l\int_{\RR^n}|\mathcal{F}(\phi_j u)(\eta)||I_{\psi_j\chi}(\eta)| d\eta\\
&\leq \sum_{j=1}^l\int_{V_j}|\mathcal{F}(\phi_j u)(\eta)||I_{\psi_j\chi}(\eta)| d\eta+\sum_{j=1}^l\int_{\RR^n\backslash V_j}|\mathcal{F}(\phi_j u)(\eta)||I_{\psi_j\chi}(\eta)| d\eta.
\end{align*}
Employing \eqref{ine-for-staphamethhor} with $\xi=0$ in the stationary phase method \cite[Theorem 7.7.1, p. 216]{hor}, one verifies that for every $N>0$ there is $C_N>0$ such that
\begin{equation}\label{ine-stafor-nep-srt}
|I_{\psi_j\chi}(\eta)|\leq C_N\langle \eta\rangle^{-N},\quad \eta\in V_j,\,\chi\in B.
\end{equation}
Hence, by employing the same technique as in the proof of the Claim in the proof of Proposition \ref{pro-for-top-imbedingthforc}, one shows that for each $j\in\{1,\ldots,l\}$,
$$
\DD'(U)\rightarrow[0,\infty),\quad u\mapsto \sup_{\chi\in B}\int_{V_j}|\mathcal{F}(\phi_j u)(\eta)||I_{\psi_j\chi}(\eta)| d\eta,
$$
is a continuous seminorm on $\DD'(U)$. Furthermore, notice that
\begin{align*}
\int_{\RR^n\backslash V_j}|\mathcal{F}(\phi_j u)(\eta)||I_{\psi_j\chi}(\eta)| d\eta\leq \mathfrak{p}_{r_2;\phi_j,\RR^n\backslash\operatorname{int}V_j}(u) \left(\int_{\RR^n}|I_{\psi_j\chi}(\eta)|^2\langle \eta\rangle^{-2r_2} d\eta\right)^{1/2}
\end{align*}
and the very last integral is uniformly bounded for all $\chi\in B$ since $r_2>n/2$ and $\sup_{\chi\in B}\|I_{\psi_j\chi}\|_{L^{\infty}(\RR^n)}<\infty$. This shows that $\sup_{\chi\in B}|\langle f^*u,\chi\rangle|$ is bounded by a continuous seminorm on $\DD'^{r_2}_L(U)$ and the proof of CASE 1 is complete.\\
\indent \underline{CASE 2: $f$ is a local diffeomorphism.} \footnote{In this case $\mathcal{N}_f=f(O)\times\{0\}$ and thus $L\cap\mathcal{N}_f=\emptyset$ for any $L$.} We show that $f^*$ uniquely extends to a well-defined and continuous mapping $f^*:\DD'^r_L(U)\rightarrow\DD'^r_{f^*L}(O)$ for each $r\in\RR$. Notice that $n=m$. By employing partitions of unity, one can easily show that $f^*$ uniquely extends to a continuous mapping $f^*:\DD'(U)\rightarrow \DD'(O)$ (cf. \cite[Subsection 5.2]{D1}). Hence, it suffices to provide bounds for $\mathfrak{p}_{r;\varphi, G_1}(f^*u)$ where $\varphi\in \DD(O)\backslash\{0\}$ and the closed cone $\emptyset\neq G_1\subseteq \RR^m$ are such that $(\supp\varphi\times G_1)\cap f^*L=\emptyset$. We proceed the same as in CASE 1 to obtain \eqref{sum-for-sem-nor-for-wfsespa} (of course, with $r$ in place of $r_1$), but now in the construction of $O_j$, $j=1,\ldots,l$, we make them sufficiently small so that $f$ is a diffeomorphism from an open neighbourhood of $\overline{O_j}$ onto an open subset of $U_j$. As before, the first sum in \eqref{sum-for-sem-nor-for-wfsespa} is bounded by a continuous seminorm on $\DD'(U)$ of $u$. We estimate the second sum as follows. Fix $j\in\{1,\ldots,l\}$. We are going to show that
\begin{equation}\label{equ-for-con-when-f-diffeops}
\left|\int_{\RR^m} I_{2;j}(\xi)\chi(\xi)d\xi\right|\leq C\|\langle\cdot\rangle^{-r}\chi\|_{L^2(\RR^m)}\mathfrak{p}_{r;\phi_j,\RR^m\backslash\operatorname{int} V_j}(u),\quad \chi\in\DD(\RR^m);
\end{equation}
this immediately gives $\|\langle\cdot\rangle^rI_{2;j}\|_{L^2(\RR^m)}\leq C\mathfrak{p}_{r;\phi_j,\RR^m\backslash\operatorname{int} V_j}(u)$ which completes the proof of the existence, continuity and uniqueness of the extension $f^*:\DD'^r_L(U)\rightarrow\DD'^r_{f^*L}(O)$. For simpler notations, set $\varphi_j:=\psi_j\varphi\in\DD(O_j)$. Let $\chi\in\DD(\RR^m)$ and set $\widetilde{\chi}:=\langle\cdot\rangle^{-r}\chi\in\DD(\RR^m)$. Notice that
\begin{align*}
\left|\int_{\RR^m} I_{2;j}(\xi)\chi(\xi)d\xi\right|&\leq\int_{\RR^m\backslash\operatorname{int} V_j}|\mathcal{F}(\phi_j u)(\eta)|\left|\int_{O_j\times \RR^m}e^{if(x)\eta-ix\xi}\varphi_j(x)\chi(\xi)dx d\xi\right|d\eta\\
&=\int_{\RR^m\backslash \operatorname{int} V_j}|\mathcal{F}(\phi_j u)(\eta)|\left|\int_{O_j} e^{if(x)\eta}\varphi_j(x)\langle D\rangle^r\mathcal{F}\widetilde{\chi}(x)dx\right|d\eta.
\end{align*}
Pick $\widetilde{\psi}_j\in\DD(O_j)$ such that $\widetilde{\psi}_j=1$ on a neighbourhood of $\supp\psi_j$ and write $\varphi_j\langle D\rangle^r= \varphi_j\langle D\rangle^r\widetilde{\psi}_j+\varphi_j\langle D\rangle^r(1-\widetilde{\psi}_j)$. There are $a_j,a'_j\in S^r_c(O_j\times \RR^m)$ so that $\Op(a_j)=\varphi_j\langle D\rangle^r\widetilde{\psi}_j$ and $\Op(a'_j)=\varphi_j\langle D\rangle^r(1-\widetilde{\psi}_j)$. By construction, the kernel of $\Op(a_j)$ has compact support in $O_j\times O_j$. Since pseudo-differential operators have kernels that are smooth outside of the diagonal, $\Op(a'_j)\in\Psi^{-\infty}(O_j)$; consequently, $a'_j\in S^{-\infty}_{\loc}(O_j\times\RR^m)$ which yields $a'_j\in S^{-\infty}_c(O_j\times\RR^m)$. Choose $\widetilde{\psi}'_j\in\DD(O_j)$ so that $\widetilde{\psi}'_j=1$ on a neighbourhood of $\supp\widetilde{\psi}_j$ and notice that
\begin{align}
\Bigg|&\int_{\RR^m} I_{2;j}(\xi)\chi(\xi)d\xi\Bigg|\nonumber\\
&\leq \mathfrak{p}_{r;\phi_j;\RR^m\backslash\operatorname{int}V_j}(u) \left(\int_{\RR^m}\langle\eta\rangle^{-2r}\left|\int_{O_j} e^{if(x)\eta}\Op(a_j)(\widetilde{\psi}'_j\mathcal{F}\widetilde{\chi})(x)dx\right|^2d\eta\right)^{1/2}\nonumber\\
&{}\quad+\mathfrak{p}_{r;\phi_j;\RR^m\backslash\operatorname{int}V_j}(u)\left(\int_{\RR^m}\langle\eta\rangle^{-2r}\left|\int_{O_j} e^{if(x)\eta}\Op(a'_j)(\mathcal{F}\widetilde{\chi})(x)dx\right|^2d\eta\right)^{1/2}.\label{ter-for-coe-wit-conopeest}
\end{align}
In the first term we change variables $y=f(x)$. In view of \cite[Theorem 18.1.17, p. 81]{hor2}, $\Op(a_j)(\widetilde{\psi}'_j\mathcal{F}\widetilde{\chi})\circ f^{-1}=\Op(\widetilde{a}_j)((\widetilde{\psi}'_j\mathcal{F}\widetilde{\chi})\circ f^{-1})$ with $\widetilde{a}_j\in S^r_c(f(O_j)\times \RR^m)$. We infer
\begin{align*}
&\left(\int_{\RR^m}\langle\eta\rangle^{-2r}\left|\int_{O_j} e^{if(x)\eta}\Op(a_j)(\widetilde{\psi}'_j\mathcal{F}\widetilde{\chi})(x)dx\right|^2d\eta\right)^{1/2}\\
&= \left(\int_{\RR^m}\langle\eta\rangle^{-2r}\left|\int_{f(O_j)} e^{iy\eta}\Op(\widetilde{a}_j)((\widetilde{\psi}'_j\mathcal{F}\widetilde{\chi})\circ f^{-1})(y)|f^{-1\,'}(y)|dy\right|^2d\eta\right)^{1/2}\\
&=(2\pi)^m\left(\int_{\RR^m}\langle\eta\rangle^{-2r}\left|\mathcal{F}^{-1} \left(\Op(\widetilde{a}_j)((\widetilde{\psi}'_j\mathcal{F}\widetilde{\chi})\circ f^{-1})|f^{-1\,'}|\right)(\eta)\right|^2d\eta\right)^{1/2}\\
&=(2\pi)^{m/2}\left\|\langle D\rangle^{-r}\left(|f^{-1\,'}|\Op(\widetilde{a}_j)((\widetilde{\psi}'_j\mathcal{F}\widetilde{\chi})\circ f^{-1})\right)\right\|_{L^2(\RR^m)}\\
&\leq C_1 \|(\widetilde{\psi}'_j\mathcal{F}\widetilde{\chi})\circ f^{-1}\|_{L^2(\RR^m)}\leq C_2\|\mathcal{F}\widetilde{\chi}\|_{L^2(\RR^m)}=(2\pi)^{m/2}C_2\|\langle \cdot\rangle^{-r}\chi\|_{L^2(\RR^m)},
\end{align*}
where the second to last inequality follows from the fact that $\langle D\rangle^{-r}(|f^{-1\,'}|\Op(\widetilde{a}_j))$ is a $\Psi$DO with symbol in $S^0(\RR^{2m})$ (cf. \cite[Theorem 18.1.17, p. 81]{hor2}) and hence continuous on $L^2(\RR^m)$. It remains to show a similar estimate for the last integral in \eqref{ter-for-coe-wit-conopeest}. Pick $k\in\ZZ_+$ such that $4k+2r>m$. We change variables $y=f(x)$ and infer
\begin{align*}
&\left(\int_{\RR^m}\langle\eta\rangle^{-2r}\left|\int_{O_j} e^{if(x)\eta}\Op(a'_j)(\mathcal{F}\widetilde{\chi})(x)dx\right|^2d\eta\right)^{1/2}\\
&= \left(\int_{\RR^m}\langle\eta\rangle^{-2r-4k}\left|\int_{f(O_j)} e^{iy\eta}(\operatorname{Id}-\Delta)^k\left(|f^{-1\, '}(y)| \Op(a'_j)(\mathcal{F}\widetilde{\chi})(f^{-1}(y))\right)dy\right|^2d\eta\right)^{1/2}\\
&\leq C'_1\left\|(\operatorname{Id}-\Delta)^k\left(|f^{-1\, '}| \Op(a'_j)(\mathcal{F}\widetilde{\chi})\circ f^{-1}\right)\right\|_{L^{\infty}(f(O_j))}\\
&\leq C'_2 \sup_{|\alpha|\leq 2k}\|\partial^{\alpha}\Op(a'_j)(\mathcal{F}\widetilde{\chi})\|_{L^{\infty}(\RR^m)}\leq C'_3\|\mathcal{F}\widetilde{\chi}\|_{L^2(\RR^m)}=(2\pi)^{m/2}C'_3\|\langle\cdot\rangle^{-r}\chi\|_{L^2(\RR^m)},
\end{align*}
where the last inequality follows from the fact that $\partial^{\alpha}\Op(a'_j)$ is a pseudo-differential operator with symbol in $S^{-\infty}_c(\RR^{2m})$ for all $\alpha\in\NN^m$. This completes the proof of \eqref{equ-for-con-when-f-diffeops}.\\
\indent \underline{CASE 3: $f$ has constant rank $k\geq 1$ and satisfies $L\cap\mathcal{N}_f=\emptyset$.} Let $r_1,r_2\in\RR$ be such that $r_2-r_1\geq (n-k)/2$ and $r_2>(n-k)/2$; when $f$ is a submersion, we only assume that $r_2\geq r_1$. For $x\in\RR^m$, we denote $x=(x',x'')$, with $x'\in\RR^k$ and $x''\in\RR^{m-k}$. Similarly, for $\eta\in\RR^n$, we denote $\eta=(\eta',\eta''')$, with $\eta'\in\RR^k$ and $\eta'''\in\RR^{n-k}$. Furthermore, when it is important but not clear from the context, we will denote by $0_l$ the zero in $\RR^l$, $l\in\ZZ_+$. By the constant rank theorem \cite[Theorem 4.12, p. 81]{lee}, for each $x^{(0)}\in O$ there are open neighbourhoods $O_0\subseteq O$ of $x^{(0)}$ and $U_0\subseteq U$ of $f(x^{(0)})$ and diffeomorphisms $\kappa:O_0\rightarrow \widetilde{O}_0$ and $\iota:U_0\rightarrow \widetilde{U}_0$ satisfying $\kappa(x^{(0)})=0\in \widetilde{O}_0$ and $\iota(f(x^{(0)}))=0\in\widetilde{U}_0$ such that $f(O_0)\subseteq U_0$ and
\begin{equation}\label{for-ofm-contrankth-for-newmm}
\hat{f}_0(x):=\iota\circ f_{|O_0}\circ\kappa^{-1}(x)=(x',0_{n-k}),\quad x=(x',x'')\in \widetilde{O}_0;
\end{equation}
of course, when $f$ is a submersion then $k=n$ and $\hat{f}_0(x)=x'$, $x=(x',x'')\in \widetilde{O}_0$. We make the following\\
\\
\noindent \textbf{Claim.} Let $\widetilde{L}$ be a closed conic subset of $\widetilde{U}_0\times(\RR^n\backslash\{0\})$ which satisfies $\widetilde{L}\cap \mathcal{N}_{\hat{f}_0}=\emptyset$. Then the map $\hat{f}_0^*:\mathcal{C}^{\infty}(\widetilde{U}_0)\rightarrow\mathcal{C}^{\infty}(\widetilde{O}_0)$ is continuous when $\mathcal{C}^{\infty}(\widetilde{U}_0)$ and $\mathcal{C}^{\infty}(\widetilde{O}_0)$ are equipped with the topologies induced by $\DD'^{r_2}_{\widetilde{L}}(\widetilde{U}_0)$ and $\DD'^{r_1}_{\hat{f}^*_0\widetilde{L}}(\widetilde{O}_0)$ respectively.\\
\\
\noindent Before we prove the claim, we show how CASE 3 follows from it. Let $\varphi\in \DD(O)\backslash\{0\}$ and the closed cone $\emptyset\neq G_1\subseteq \RR^m$ be such that $(\supp\varphi\times G_1)\cap f^*L=\emptyset$. Arguing as in CASE 1, one can find open sets $O_1,\ldots, O_l\subseteq O$ each with compact closure in $O$, open sets $U_1,\ldots,U_l\subseteq U$ and diffeomorphisms $\kappa_j:O_j\rightarrow\widetilde{O}_j$ and $\iota_j:U_j\rightarrow \widetilde{U}_j$, $j=1,\ldots,l$, such that $\supp\varphi\subseteq \bigcup_{j=1}^l O_j$, $f(\overline{O_j})\subseteq U_j$ and $\hat{f}_j:=\iota_j\circ f_{|O_j}\circ\kappa_j^{-1}:\widetilde{O}_j\rightarrow \widetilde{U}_j$ is given by \eqref{for-ofm-contrankth-for-newmm}. Let $\psi_j\in\DD(O_j)$, $0\leq \psi_j\leq 1$, $j=1,\ldots, l$, be a partition of unity on a neighbourhood of $\supp\varphi$ and let $\phi_j\in\DD(U_j)$ be such that $\phi_j=1$ on a neighbourhood of $f(\overline{O_j})$. Set $\varphi_j:=\psi_j\varphi\in\DD(O_j)$, $f_j:=f_{|O_j}:O_j\rightarrow U_j$ and $L_j:=L\cap (U_j\times(\RR^n\backslash\{0\}))$, $j=1,\ldots,l$. For $u\in\mathcal{C}^{\infty}(U)$, we have (since $\phi_j u\in \DD(U_j)$)
$$
\mathcal{F}(\varphi f^* u)=\sum_{j=1}^l \mathcal{F}(\varphi_j f^*(\phi_j u))=\sum_{j=1}^l \mathcal{F}(\varphi_j \kappa_j^*\hat{f}_j^*\iota_j^{-1\,*}(\phi_j u))
$$
and consequently $\mathfrak{p}_{r_1;\varphi, G_1}(f^*u)\leq \sum_{j=1}^l\mathfrak{p}_{r_1;\varphi_j, G_1}(\kappa_j^*\hat{f}_j^*\iota_j^{-1\,*}(\phi_j u))$. Fix $j\in\{1,\ldots,l\}$. In view of CASE 2, $\kappa_j^*$ and $\iota_j^{-1\,*}$ uniquely extend to well-defined and continuous mappings $\kappa_j^*:\DD'^{r_1}_{\hat{f}_j^*\iota_j^{-1\,*}L_j}(\widetilde{O}_j)\rightarrow \DD'^{r_1}_{f^*_jL_j}(O_j)$ (notice that $L_j\cap \mathcal{N}_{f_j}=\emptyset$ and $\hat{f}_j^*\iota_j^{-1\, *}L_j=\kappa_j^{-1\, *}f_j^*L_j$) and $\iota_j^{-1\,*}:\DD'^{r_2}_{L_j}(U_j)\rightarrow\DD'^{r_2}_{\iota_j^{-1\, *}L_j}(\widetilde{U}_j)$. The Claim yields that $\hat{f}_j^*$ uniquely extends to a well-defined and continuous mapping $\hat{f}_j^*:\DD'^{r_2}_{\iota_j^{-1\, *}L_j}(\widetilde{U}_j)\rightarrow \DD'^{r_1}_{\hat{f}_j^*\iota_j^{-1\,*}L_j}(\widetilde{O}_j)$ ($\iota_j^{-1\,*}L_j\cap \mathcal{N}_{\hat{f}_j}=\emptyset$ in view of \eqref{equ-for-con-subwithdifandordmass}). Consequently, there are $\chi_{j'}\in \DD(U_j)$ and closed cones $V_{j'}\subseteq \RR^n$, $j'=1,\ldots, k'$ ($k'$ depends on $j$) satisfying $(\supp\chi_{j'}\times V_{j'})\cap L_j=\emptyset$, $j'=1,\ldots,k'$, a continuous seminorm $\mathfrak{p}$ on $\DD'(U_j)$ (whence, continuous on $\DD'(U)$ as well) and $C>0$ such that
\begin{equation}\label{sem-forabcc-lks}
\mathfrak{p}_{r_1;\varphi_j, G_1}(\kappa_j^*\hat{f}_j^*\iota_j^{-1\,*}(\phi_j u))\leq C\mathfrak{p}(\phi_j u)+C\sum_{j'=1}^{k'} \mathfrak{p}_{r_2;\chi_{j'},V_{j'}}(\phi_j u);
\end{equation}
notice that each term on the right-hand side is a continuous seminorm on $\DD'^{r_2}_L(U)$ of $u$ (as $\mathfrak{p}_{r_2;\chi_{j'},V_{j'}}(\phi_j u)=\mathfrak{p}_{r_2;\chi_{j'}\phi_j,V_{j'}}(u)$). It remains to show similar bounds for $\mathfrak{p}(f^*u)$, where $\mathfrak{p}$ is an arbitrary continuous seminorm on $\DD'(O)$. Without loss of generality, we can assume that $\mathfrak{p}=\sup_{\chi\in B}|\langle \cdot, \chi\rangle|$ for some bounded subset $B$ of $\DD'(O)$. There is a compact subset $K$ of $O$ such that $B$ is a bounded subset of $\DD_K$. Now, as before, one applies a partition of unity together with the Claim and CASE 2 to show analogous bounds for $\mathfrak{p}(u)$, $u\in\mathcal{C}^{\infty}(O)$, as in \eqref{sem-forabcc-lks}. This completes the proof of CASE 3 and the theorem.\\
\\
\noindent \textbf{Proof of Claim.} Notice that
\begin{align}
\mathcal{N}_{\hat{f}_0}&=\{((x',0_{n-k}),(0_k,\eta'''))\in \widetilde{U}_0\times\RR^n\,|\, \exists x''\in\RR^{m-k},\, (x',x'')\in\widetilde{O}_0\},\\
\hat{f}_0^*\widetilde{L}&=\{((x',x''),(\eta',0_{m-k}))\in \widetilde{O}_0\times(\RR^m\backslash \{0\})\,|\, \exists\eta'''\in\RR^{n-k},\, ((x',0_{n-k}),(\eta',\eta'''))\in\widetilde{L}\}.\label{equ-for-puba-of-coni-setfr}
\end{align}
Let $G_0$ be a closed cone in $\RR^m$ and let $\varphi\in\DD(\widetilde{O}_0)\backslash\{0\}$ be such that $(\supp\varphi\times G_0) \cap \hat{f}^*_0\widetilde{L}=\emptyset$; since our goal is to estimate $\mathfrak{p}_{r_1;\varphi,G_0}(\hat{f}^*_0u)$, we can assume that $G_0\backslash\{0\}\neq \emptyset$. For $\xi=(\xi',\xi'')\in G_0\cap\mathbb{S}^{m-1}$, a standard compactness argument shows that there are $\varepsilon_{\xi}>0$ and an open set $O_{\xi}\subseteq \widetilde{O}_0$ with compact closure in $\widetilde{O}_0$ such that $\supp\varphi\subseteq O_{\xi}$ and $(\overline{O_{\xi}}\times G_{\xi})\cap \hat{f}^*_0\widetilde{L}=\emptyset$ where $G_{\xi}\subseteq \RR^m$ is the open cone $\RR_+ (B(\xi',\varepsilon_{\xi})\times B(\xi'',\varepsilon_{\xi}))$; furthermore, when $\xi''\neq 0$, we can take $\varepsilon_{\xi}<|\xi''|/6$. Of course, when $k=m$, $G_{\xi}= \RR_+B(\xi,\varepsilon_{\xi})$. We employ another compactness argument to find open cones
$$
G_j:=\RR_+(B(\xi^{(j)\, '},\varepsilon_j)\times B(\xi^{(j)\,''},\varepsilon_j))\quad \mbox{and}\quad \widetilde{G}_j:=\RR_+(B(\xi^{(j)\, '},3\varepsilon_j)\times B(\xi^{(j)\,''},3\varepsilon_j))
$$
with $\xi^{(j)}=(\xi^{(j)\,'},\xi^{(j)\,''})\in G_0\cap\mathbb{S}^{m-1}$ and $\varepsilon_j\in(0,1/6)$, $j=1,\ldots, s$, and an open set $\widetilde{O}_1\subseteq \widetilde{O}_0$ with compact closure in $\widetilde{O}_0$ such that $\supp\varphi\times G_0\subseteq \widetilde{O}_1\times \bigcup_{j=1}^s G_j$ and $(\overline{\widetilde{O}_1}\times \bigcup_{j=1}^s \overline{\widetilde{G}_j})\cap \hat{f}^*_0\widetilde{L}=\emptyset$. Furthermore, when $\xi^{(j)\,''}\neq 0$, it holds that $\varepsilon_j<|\xi^{(j)\,''}|/6$. Again, when $k=m$, $G_j= \RR_+ B(\xi,\varepsilon_j)$ and $\widetilde{G}_j= \RR_+ B(\xi,3\varepsilon_j)$. Write $\{1,\ldots,s\}=J_1\cup J_2$ where $J_1$ contains all indexes $j$ such that $\xi^{(j)\,''}\neq0$ and $J_2=\{1,\ldots,s\}\backslash J_1$; when $k=m$, we set $J_1=\emptyset$. If $J_1\neq \emptyset$, there is $0<\varepsilon<1/2$ such that $\bigcup_{j\in J_1}\widetilde{G}_j\subseteq \{(\xi',\xi'')\in\RR^m\,|\, |\xi''|>\varepsilon|\xi'|\}=:\widetilde{G}_0$; when $J_1=\emptyset$, we set $\widetilde{G}_0:=\emptyset$. In view of \eqref{equ-for-puba-of-coni-setfr}, $(\widetilde{O}_1\times\overline{\widetilde{G}_0})\cap \hat{f}^*_0\widetilde{L}=\emptyset$. When $J_2\neq \emptyset$, for each $j\in J_2$, define the open cones $G'_j\subseteq \widetilde{G}'_j\subseteq \RR^k$ as $G'_j:=\RR_+ B(\xi^{(j)\,'},\varepsilon_j)$ and $\widetilde{G}'_j:=\RR_+ B(\xi^{(j)\,'},3\varepsilon_j)$. Then $G'_j\times\RR^{m-k}$ and $\widetilde{G}'_j\times\RR^{m-k}$ are open cones in $\RR^m$ such that $G_j\subseteq G'_j\times \RR^{m-k}$ and, in view of \eqref{equ-for-puba-of-coni-setfr}, we have $(\hat{f}_0(\overline{\widetilde{O}_1})\times \overline{\widetilde{G}'_j}\times \RR^{n-k})\cap \widetilde{L}=\emptyset$, $j\in J_2$. Another compactness argument implies that there is an open set $U_1\subseteq \widetilde{U}_0$ such that $\hat{f}(\overline{\widetilde{O}_1})\subseteq U_1$ and $(U_1\times \overline{\widetilde{G}'_j}\times \RR^{n-k})\cap \widetilde{L}=\emptyset$. When $k<n$, since $\hat{f}_0(\supp\varphi)\times (\{0_k\}\times\RR^{n-k})\subseteq \mathcal{N}_{\hat{f}_0}$ and $\mathcal{N}_{\hat{f}_0}\cap\widetilde{L}=\emptyset$, an analogous compactness argument shows that there are an open set $\widetilde{U}_1\subseteq \widetilde{U}_0$ satisfying $\hat{f}_0(\supp\varphi)\subseteq \widetilde{U}_1$ and points $\eta^{(1)\,'''},\ldots,\eta^{(l)\,'''}\in\mathbb{S}^{n-k-1}$ defining the open cones $V_j:=\RR_+(B(0_k,\varepsilon'_j)\times B(\eta^{(j)\,'''},\varepsilon'_j))\subseteq \RR^n$ with some $\varepsilon'_j\in(0,1/2)$, $j=1,\ldots,l$, such that $\hat{f}_0(\supp\varphi)\times (\{0_k\}\times(\RR^{n-k}\backslash\{0_{n-k}\}))\subseteq\widetilde{U}_1\times(\bigcup_{j=1}^l V_j)$ and $(\widetilde{U}_1\times(\bigcup_{j=1}^l V_j))\cap\widetilde{L}=\emptyset$. It is straightforward to verify that there is $C_0>1$ such that the open cone $\widetilde{V}_0:=\{(\eta',\eta''')\in\RR^n\,|\, |\eta'''|>C_0|\eta'|\}$ satisfies $\widetilde{U}_1\times(\overline{\widetilde{V}_0}\backslash\{0_n\})\subseteq \widetilde{U}_1\times(\bigcup_{j=1}^l V_j)$.\\
\indent Pick $\phi\in\DD(U_1)$ such that $\phi=1$ on $\hat{f}_0(\supp\varphi)$; when $k<n$ we choose $\phi$ such that it also satisfies $\supp\phi\subseteq \widetilde{U}_1\cap U_1$. For $u\in\mathcal{C}^{\infty}(\widetilde{U}_0)$ we compute
\begin{align}
\mathcal{F}(\varphi\hat{f}_0^*u)(\xi)&=\mathcal{F}(\varphi\hat{f}_0^*(\phi u))(\xi)=\frac{1}{(2\pi)^n}\int_{\RR^n} \int_{\widetilde{O}_1}e^{i\hat{f}_0(x)\eta-ix\xi} \varphi(x)\mathcal{F}(\phi u)(\eta) dxd\eta\nonumber\\
&=\frac{1}{(2\pi)^n}\int_{\RR^n}\mathcal{F}(\phi u)(\eta) \int_{\widetilde{O}_1}e^{-ix(\xi'-\eta',\xi'')} \varphi(x) dxd\eta\nonumber\\
&=\frac{1}{(2\pi)^n}\int_{\RR^n}\mathcal{F}(\phi u)(\eta) \mathcal{F}\varphi(\xi'-\eta',\xi'') d\eta\label{equ-for-ano-est-for-int-cla-int-the}\\
&=\frac{1}{(2\pi)^k}\int_{\RR^k}v(\xi'-\eta') \mathcal{F}\varphi(\eta',\xi'') d\eta',\label{equ-for-fut-estforintinefortt}
\end{align}
where we denoted $v(\eta'):=(2\pi)^{-n+k}\int_{\RR^{n-k}}\mathcal{F}(\phi u)(\eta',\eta''')d\eta'''$; when $k=n$, $\eta'=\eta$ and we set $v:=\mathcal{F}(\phi u)$. We infer
\begin{multline}\label{for-est-nor-tpfortnnes}
\mathfrak{p}_{r_1;\varphi, G_0}(\hat{f}_0^*u)\leq \left(\int_{\widetilde{G}_0} |\mathcal{F}(\varphi\hat{f}_0^*(\phi u))(\xi)|^2\langle \xi\rangle^{2r_1}d\xi\right)^{1/2}\\
+\sum_{j\in J_2} \left(\int_{G'_j\times\RR^{m-k}} |\mathcal{F}(\varphi\hat{f}_0^*(\phi u))(\xi)|^2\langle \xi\rangle^{2r_1}d\xi\right)^{1/2}.
\end{multline}
We first consider the integral over $\widetilde{G}_0$. We only look at the case $k<m$, since when $k=m$, $\widetilde{G}_0=\emptyset$. Notice that for every $N>0$ there is $C_N>0$ such that
\begin{equation}\label{est-for-fou-imafcnt}
|\mathcal{F}\varphi(\xi'-\eta',\xi'')|\leq C_N\langle\xi'\rangle^{-N}\langle\xi''\rangle^{-N}\langle\eta'\rangle^{-N},\quad (\xi',\xi'')\in\widetilde{G}_0,\, \eta'\in\RR^k.
\end{equation}
When $k=n$, employing this bound together with \eqref{equ-for-ano-est-for-int-cla-int-the} and arguing as in the proof of the Claim in the proof of Proposition \ref{pro-for-top-imbedingthforc}, it is straightforward to show that
$$
\DD'(\widetilde{U}_0)\rightarrow[0,\infty),\quad u\mapsto \left(\int_{\widetilde{G}_0} |\mathcal{F}(\varphi\hat{f}_0^*(\phi u))(\xi)|^2\langle \xi\rangle^{2r_1}d\xi\right)^{1/2},
$$ is a continuous seminorm on $\DD'(\widetilde{U}_0)$. When $k<n$, we employ \eqref{equ-for-ano-est-for-int-cla-int-the} and write
\begin{align}
\Bigg(\int_{\widetilde{G}_0} &|\mathcal{F}(\varphi\hat{f}_0^*(\phi u))(\xi)|^2\langle \xi\rangle^{2r_1}d\xi\Bigg)^{1/2}\nonumber\\
&\leq \left(\int_{\widetilde{G}_0} \left(\int_{\widetilde{V}_0}|\mathcal{F}(\phi u)(\eta)| |\mathcal{F}\varphi(\xi'-\eta',\xi'')|\langle \xi\rangle^{r_1}d\eta\right)^2d\xi\right)^{1/2}\label{int-for-est-nep-vortsklerps}\\
&{}\quad+ \left(\int_{\widetilde{G}_0} \left(\int_{\RR^n\backslash\widetilde{V}_0}|\mathcal{F}(\phi u)(\eta)| |\mathcal{F}\varphi(\xi'-\eta',\xi'')|\langle \xi\rangle^{r_1}d\eta\right)^2d\xi\right)^{1/2}.\nonumber
\end{align}
Employing \eqref{est-for-fou-imafcnt}, one again shows that
$$
\DD'(\widetilde{U}_0)\rightarrow[0,\infty),\quad u\mapsto \left(\int_{\widetilde{G}_0} \left(\int_{\RR^n\backslash\widetilde{V}_0}|\mathcal{F}(\phi u)(\eta)| |\mathcal{F}\varphi(\xi'-\eta',\xi'')|\langle \xi\rangle^{r_1}d\eta\right)^2d\xi\right)^{1/2},
$$
is a continuous seminorm on $\DD'(\widetilde{U}_0)$. To estimate the term in \eqref{int-for-est-nep-vortsklerps}, we employ \eqref{est-for-fou-imafcnt} with $N=|r_1|+r_2+n+m+1$ and infer
\begin{align*}
&\left(\int_{\widetilde{G}_0} \left(\int_{\widetilde{V}_0}|\mathcal{F}(\phi u)(\eta)| |\mathcal{F}\varphi(\xi'-\eta',\xi'')|\langle \xi\rangle^{r_1}d\eta\right)^2d\xi\right)^{1/2}\\
&\leq C'_1\int_{\widetilde{V}_0}|\mathcal{F}(\phi u)(\eta)|\langle\eta'\rangle^{-r_2-n-1}d\eta\\
&\leq C'_2\mathfrak{p}_{r_2;\phi,\overline{\widetilde{V}_0}}(u)\left(\int_{\RR^n}\langle\eta\rangle^{-2r_2}\langle \eta'\rangle^{-2r_2-2n-2}d\eta\right)^{1/2}\leq C'_3\mathfrak{p}_{r_2;\phi,\overline{\widetilde{V}_0}}(u),
\end{align*}
where the last inequality follows from $r_2>(n-k)/2$; notice that $\mathfrak{p}_{r_2;\phi,\overline{\widetilde{V}_0}}$ is a continuous seminorm on $\DD'^{r_2}_{\widetilde{L}}(\widetilde{U}_0)$ by the way we defined $\widetilde{V}_0$ and $\phi$. We showed that the first term in \eqref{for-est-nor-tpfortnnes} is bounded by a continuous seminorm on $\DD'^{r_2}_{\widetilde{L}}(\widetilde{U}_0)$ of $u$. Next, we bound each of the summands in the second term in \eqref{for-est-nor-tpfortnnes}. Fix $j\in J_2$. We employ \eqref{equ-for-ano-est-for-int-cla-int-the} and estimate as follows:
\begin{align*}
&\left(\int_{G'_j\times\RR^{m-k}} |\mathcal{F}(\varphi\hat{f}_0^*(\phi u))(\xi)|^2\langle \xi\rangle^{2r_1}d\xi\right)^{1/2}\\
&\leq \left(\int_{G'_j\times\RR^{m-k}} \left(\int_{\widetilde{G}'_j\times\RR^{n-k}}|\mathcal{F}(\phi u)(\eta)||\mathcal{F}\varphi(\xi'-\eta',\xi'')|\langle \xi\rangle^{r_1}d\eta\right)^2d\xi\right)^{1/2}\\
&{}\quad+\left(\int_{G'_j\times\RR^{m-k}} \left(\int_{(\RR^k\backslash \widetilde{G}'_j)\times\RR^{n-k}}|\mathcal{F}(\phi u)(\eta)||\mathcal{F}\varphi(\xi'-\eta',\xi'')|\langle \xi\rangle^{r_1}d\eta\right)^2d\xi\right)^{1/2}.
\end{align*}
Denote the two terms by $I'_j$ and $I''_j$ respectively. We first estimate $I'_j$. For the moment, when $k<n$, denote $r'_1:=\max\{r_1,0\}$ and, if $k=n$, set $r'_1:=r_1$. We employ H\"older's inequality in the inner integral to obtain
\begin{multline*}
I'_j\leq \Bigg(\int_{G'_j\times\RR^{m-k}} \left(\int_{\widetilde{G}'_j\times\RR^{n-k}}\frac{|\mathcal{F}\varphi(\xi'-\eta',\xi'')|\langle \xi\rangle^{r_1}}{\langle \eta\rangle^{2r_2-r'_1}}d\eta\right)\\
\cdot\left(\int_{\widetilde{G}'_j\times\RR^{n-k}}|\mathcal{F}(\phi u)(\eta)|^2\langle \eta\rangle^{2r_2-r'_1}|\mathcal{F}\varphi(\xi'-\eta',\xi'')|\langle \xi\rangle^{r_1}d\eta\right)d\xi\Bigg)^{1/2}.
\end{multline*}
When $k<n$, we employ spherical coordinates and the fact $2r_2-r'_1-n+k>0$ (this follows from $r_2>(n-k)/2$ and $r_2-r_1\geq (n-k)/2$) to infer
\begin{align*}
\int_{\RR^{n-k}}\langle\eta\rangle^{-2r_2+r'_1}d\eta'''&\leq C''_1\int_0^{\infty}\frac{\rho^{n-k-1}d\rho}{(1+|\eta'|+\rho)^{2r_2-r'_1}} \leq C''_1\int_{1+|\eta'|}^{\infty}\frac{d\rho}{\rho^{2r_2-r'_1-n+k+1}}\\
&\leq C''_2\langle\eta'\rangle^{-2r_2+r'_1+n-k}
\end{align*}
and, as $r_2-r'_1\geq (n-k)/2$ (again, this follows from $r_2>(n-k)/2$ and $r_2-r_1\geq (n-k)/2$), we deduce
$$
\int_{\widetilde{G}'_j\times\RR^{n-k}}\frac{|\mathcal{F}\varphi(\xi'-\eta',\xi'')|\langle \xi\rangle^{r_1}}{\langle \eta\rangle^{2r_2-r'_1}}d\eta\leq C''_3\int_{\RR^k}\frac{|\mathcal{F}\varphi(\xi'-\eta',\xi'')|\langle \xi'-\eta'\rangle^{|r_1|}\langle\xi''\rangle^{|r_1|}}{\langle \eta'\rangle^{2r_2-2r'_1-n+k}}d\eta'\leq C''_4
$$
for all $\xi\in\RR^m$. Notice that this bound also holds when $k=n$, since we assume $r_2\geq r_1$ in this case. Consequently
\begin{align*}
I'_j&\leq \sqrt{C''_4}\left(\int_{\widetilde{G}'_j\times\RR^{n-k}}|\mathcal{F}(\phi u)(\eta)|^2\langle \eta\rangle^{2r_2-r'_1}\int_{\RR^m}|\mathcal{F}\varphi(\xi'-\eta',\xi'')|\langle \xi\rangle^{r_1}d\xi d\eta \right)^{1/2}\\
&\leq C''_5\left(\int_{\widetilde{G}'_j\times\RR^{n-k}}|\mathcal{F}(\phi u)(\eta)|^2\langle \eta\rangle^{2r_2-r'_1}\int_{\RR^m}|\mathcal{F}\varphi(\xi)|\langle \xi'\rangle^{|r_1|}\langle\eta'\rangle^{r'_1}\langle\xi''\rangle^{|r_1|}d\xi d\eta \right)^{1/2}\\
&\leq C''_6 \mathfrak{p}_{r_2;\phi,\overline{\widetilde{G}'_j}\times\RR^{n-k}}(u);
\end{align*}
notice that $\mathfrak{p}_{r_2;\phi,\overline{\widetilde{G}'_j}\times\RR^{n-k}}$ is a continuous seminorm on $\DD'^{r_2}_{\widetilde{L}}(\widetilde{U}_0)$ by the way we defined $\widetilde{G}'_j\times\RR^{n-k}$.\footnote{The bound for $I'_j$ can be also derived by employing the Schur test with weights; the above arguments are essentially the proof of the test. We did not apply it merely to avoid additional unnecessary notational complexity which would come from writing the kernel and the weights.} We now turn our attention to $I''_j$. Let $\eta'\in(\RR^k\backslash\widetilde{G}'_j)\backslash\{0_k\}$ and $\xi'\in G'_j$ be arbitrary. Then $|\eta'/|\eta'|-\xi^{(j)\,'}|\geq 3\varepsilon_j$ and there is $\lambda>0$ such that $|\xi'/\lambda-\xi^{(j)\,'}|<\varepsilon_j$. Since $|\xi^{(j)\,'}|=1$ (as $\xi^{(j)\,''}=0$ and $(\xi^{(j)\,'},\xi^{(j)\,''})\in\mathbb{S}^{m-1}$), we have
$$
\left|1-\frac{|\xi'|}{\lambda}\right|\leq \left|\xi^{(j)'\,}-\frac{\xi'}{\lambda}\right|<\varepsilon_j,\quad \mbox{hence}\quad \left|\frac{\xi'}{|\xi'|}-\xi^{(j)\,'}\right|\leq \left|\frac{\xi'}{|\xi'|}-\frac{\xi'}{\lambda}\right|+\left|\frac{\xi'}{\lambda}-\xi^{(j)\,'}\right|<2\varepsilon_j.
$$
Consequently, $|\eta'/|\eta'|-\xi'/|\xi'||>\varepsilon_j$ which gives
$$
\left|\frac{\eta'}{|\eta'|}-\frac{\xi'}{|\eta'|}\right|\geq \left|\frac{\eta'}{|\eta'|}-\frac{\xi'}{|\xi'|}\right| -\left|\frac{\xi'}{|\xi'|}-\frac{\xi'}{|\eta'|}\right|> \varepsilon_j-\frac{||\eta'|-|\xi'||}{|\eta'|}\geq \varepsilon_j-\left|\frac{\eta'}{|\eta'|}-\frac{\xi'}{|\eta'|}\right|.
$$
Thus $\langle\eta'-\xi'\rangle> (\varepsilon_j/2)\langle\eta'\rangle$ and similarly $\langle\eta'-\xi'\rangle> (\varepsilon_j/2)\langle\xi'\rangle$; notice that these inequalities are valid even when $\eta'=0$. These inequalities show that for every $N>0$ there is $C_N>0$ such that
\begin{equation}\label{est-for-par-int-nespfhk}
|\mathcal{F}\varphi(\xi'-\eta',\xi'')|\leq C_N\langle\xi'\rangle^{-N}\langle\xi''\rangle^{-N}\langle\eta'\rangle^{-N},\quad \xi'\in G'_j,\,\xi''\in\RR^{m-k},\, \eta'\in\RR^k\backslash\widetilde{G}'_j.
\end{equation}
When $k=n$, these bounds yield that the integral $I''_j$ is a continuous seminorm on $\DD'(\widetilde{U}_0)$ of $u$. Assume that $k<n$. Write
\begin{multline*}
I''_j\leq \left(\int_{G'_j\times\RR^{m-k}} \left(\int_{((\RR^k\backslash \widetilde{G}'_j)\times\RR^{n-k})\backslash \widetilde{V}_0}|\mathcal{F}(\phi u)(\eta)||\mathcal{F}\varphi(\xi'-\eta',\xi'')|\langle \xi\rangle^{r_1}d\eta\right)^2d\xi\right)^{1/2}\\
+\left(\int_{G'_j\times\RR^{m-k}} \left(\int_{((\RR^k\backslash \widetilde{G}'_j)\times\RR^{n-k})\cap \widetilde{V}_0}|\mathcal{F}(\phi u)(\eta)||\mathcal{F}\varphi(\xi'-\eta',\xi'')|\langle \xi\rangle^{r_1}d\eta\right)^2d\xi\right)^{1/2}.
\end{multline*}
In view of \eqref{est-for-par-int-nespfhk}, the first integral is a continuous seminorm on $\DD'(\widetilde{U}_0)$ of $u$. It remains to bound the second integral; denote it for simplicity by $I'''_j$. For ease in writing, we denote the closed cone $((\RR^k\backslash \widetilde{G}'_j)\times\RR^{n-k})\cap \overline{\widetilde{V}_0}$ by $\widetilde{V}_1$. Notice that, by construction, $\mathfrak{p}_{r_2;\phi,\widetilde{V}_1}$ is a continuous seminorm on $\DD'^{r_2}_{\widetilde{L}}(\widetilde{U}_0)$. We infer
\begin{equation*}
I'''_j\leq \mathfrak{p}_{r_2;\phi,\widetilde{V}_1}(u) \left(\int_{G'_j\times\RR^{m-k}} \int_{\widetilde{V}_1}|\mathcal{F}\varphi(\xi'-\eta',\xi'')|^2\langle \xi\rangle^{2r_1}\langle\eta\rangle^{-2r_2}d\eta d\xi\right)^{1/2}= C''\mathfrak{p}_{r_2;\phi,\widetilde{V}_1}(u),
\end{equation*}
where we employed \eqref{est-for-par-int-nespfhk} with $N=|r_1|+n+m+1$ and the assumption $r_2>(n-k)/2$. This completes the proof that $\mathfrak{p}_{r_1;\varphi,G_0}(\hat{f}_0^*u)$ is bounded by a continuous seminorm on $\DD'^{r_2}_{\widetilde{L}}(\widetilde{U}_0)$ of $u$. To finish the proof of the claim, it remains to show such bounds for $\mathfrak{p}(\hat{f}_0^*u)$ where $\mathfrak{p}$ is an arbitrary continuous seminorm on $\DD'(\widetilde{O}_0)$. Without loss of generality, we can assume that $\mathfrak{p}=\sup_{\chi\in B}|\langle\cdot,\chi\rangle|$ for some bounded subset $B$ of $\DD(\widetilde{O}_0)$. There exists a compact subset $K$ of $\widetilde{O}_0$ such that $B$ is a bounded subset of $\DD_K$. When $k<n$, in the same way as above, one can find an open set $\widetilde{U}_K\subseteq\widetilde{U}_0$ satisfying $\hat{f}_0^*(K)\subseteq \widetilde{U}_K$ and $C_K>1$ such that $\widetilde{V}_K:=\{(\eta',\eta''')\in\RR^n\,|\, |\eta'''|>C_K|\eta'|\}$ satisfies $(\widetilde{U}_K\times \overline{\widetilde{V}_K})\cap \widetilde{L}=\emptyset$. Pick $\phi_0\in\DD(\widetilde{U}_0)$ such that $\phi_0=1$ on $\hat{f}_0^*(K)$; when $k<n$ we choose $\phi_0$ so that $\supp\phi_0\subseteq \widetilde{U}_K$. For $\chi \in B$, we denote $\widetilde{\chi}(x'):=\int_{\RR^{m-k}}\chi(x',x'')dx''$, $x'\in\RR^k$; when $k=m$ we set $\widetilde{\chi}:=\chi$. For every $N>0$ there is $C_N>0$ such that
\begin{equation}\label{bou-for-est-for-newpeesitr}
|\mathcal{F}^{-1}\widetilde{\chi}(\eta')|\leq C_N\langle\eta'\rangle^{-N},\quad \eta'\in\RR^k,\, \chi\in B.
\end{equation}
We infer
\begin{align*}
\sup_{\chi\in B}|\langle \hat{f}_0^*u,\chi\rangle|&=\sup_{\chi\in B}|\langle \hat{f}_0^*(\phi_0 u),\chi\rangle|=\sup_{\chi\in B}\frac{1}{(2\pi)^n}\left|\int_{\RR^n}\mathcal{F}(\phi_0 u)(\eta)\int_{\widetilde{O}_0}e^{ix'\eta'}\chi(x)dxd\eta\right|\\
&\leq \sup_{\chi\in B}\int_{\RR^n}|\mathcal{F}(\phi_0 u)(\eta)||\mathcal{F}^{-1}\widetilde{\chi}(\eta')|d\eta.
\end{align*}
When $k=n$, the bound \eqref{bou-for-est-for-newpeesitr} shows that the very last quantity is a continuous seminorm on $\DD'(\widetilde{U}_0)$ of $u$. Assume that $k<n$ and write
\begin{multline*}
\sup_{\chi\in B}|\langle \hat{f}_0^*u,\chi\rangle| \leq \sup_{\chi\in B}\int_{\RR^n\backslash \widetilde{V}_K}|\mathcal{F}(\phi_0 u)(\eta)||\mathcal{F}^{-1}\widetilde{\chi}(\eta')|d\eta\\
+\sup_{\chi\in B}\int_{\widetilde{V}_K}|\mathcal{F}(\phi_0 u)(\eta)||\mathcal{F}^{-1}\widetilde{\chi}(\eta')|d\eta.
\end{multline*}
The bound \eqref{bou-for-est-for-newpeesitr} shows that the first quantity is a continuous seminorm on $\DD'(\widetilde{U}_0)$ of $u$. We estimate the second quantity as follows:
\begin{equation*}
\sup_{\chi\in B}\int_{\widetilde{V}_K}|\mathcal{F}(\phi_0 u)(\eta)||\mathcal{F}^{-1}\widetilde{\chi}(\eta')|d\eta\leq \mathfrak{p}_{r_2;\phi_0,\overline{\widetilde{V}_K}}(u) \sup_{\chi\in B}\left(\int_{\RR^n}|\mathcal{F}^{-1}\widetilde{\chi}(\eta')|^2\langle\eta\rangle^{-2r_2}d\eta\right)^{1/2}.
\end{equation*}
The integral is uniformly bounded for all $\chi\in B$ in view of \eqref{bou-for-est-for-newpeesitr} and the assumption $r_2>(n-k)/2$. This completes the proof of the claim.
\end{proof}

\begin{remark}\label{rem-for-zer-rankmappulbc}
If $f:O\rightarrow U$ has constant rank $k=0$ then $f^*$ uniquely extends to a continuous mapping $f^*:\DD'^r_L(U)\rightarrow \mathcal{C}^{\infty}(O)$ when $L\cap \mathcal{N}_f=\emptyset$ and $r>n/2$. To verify this, we first point out that the connected components of $O$ are open and at most countably many. The assumption on $f$ implies that it is constant on every connected component of $O$. Hence, there are at most countably many pairwise disjoint open sets $O_j\subseteq O$, $j\in\Lambda$, whose union is $O$ and distinct points $y_j\in U$, $j\in \Lambda$, such that $f(x)=y_j$, $x\in O_j$, $j\in\Lambda$ (some of the $O_j$'s may be unions of some of the components of $O$!). Notice that $\mathcal{N}_f=(\bigcup_{j\in \Lambda}\{y_j\})\times\RR^n$. Let $K\subset\subset O$. Denote $\Lambda_K:=\{j\in\Lambda\,|\, O_j\cap K\neq \emptyset\}$; clearly, $\Lambda_K$ is finite. Choose $\phi_j\in\DD(U)$, $j\in\Lambda_K$, such that $\phi_j=1$ on a neighbourhood of $y_j$ and $(\supp\phi_j\times\RR^n)\cap L=\emptyset$. For $u\in\mathcal{C}^{\infty}(U)$ we have $f^*u=\sum_{j\in\Lambda}u(y_j)\mathbf{1}_{O_j}$ and consequently, for $l\in\NN$, we infer
\begin{align*}
\sup_{|\alpha|\leq l}\|\partial^{\alpha}(f^*u)\|_{L^{\infty}(K)}&\leq\sum_{j\in\Lambda_K}|\phi_j(y_j)u(y_j)|\leq \sum_{j\in\Lambda_K}\frac{1}{(2\pi)^n}\int_{\RR^n} |\mathcal{F}(\phi_j u)(\eta)|d \eta\\
&\leq \|\langle\cdot\rangle^{-r}\|_{L^2(\RR^n)}\sum_{j\in\Lambda_K}\mathfrak{p}_{r;\phi_j,\RR^n}(u),
\end{align*}
which shows the claim. When $u\in \DD'^r_L(U)$, the condition $L\cap \mathcal{N}_f=\emptyset$ implies that for every $j\in\Lambda$ there is an open neighbourhood $U_j\subseteq U$ of $y_j$ such that $u\in\mathcal{C}(U_j)$; indeed, for each $j\in\Lambda$ there is $\phi_j\in\DD(U)$ as above such that $\phi_j u\in H^r(\RR^n)\subseteq \mathcal{C}(\RR^n)$. Whence, in view of the above and Proposition \ref{seq-den-comsmf}, we deduce $f^*u=\sum_{j\in\Lambda}u(y_j)\mathbf{1}_{O_j}$, $u\in\DD'^r_L(U)$.
\end{remark}

Applying the theorem with $L:= WF^{r_2}(u)$ for $u\in\DD'(U)$ satisfying $WF^{r_2}(u)\cap\mathcal{N}_f=\emptyset$, we deduce the following result.

\begin{corollary}\label{cor-for-dif-inv-wavsetformacss}
Let $O$ and $U$ be open subsets of $\RR^m$ and $\RR^n$ respectively and let $f:O\rightarrow U$ be a smooth map. Given $r_1,r_2\in\RR$ which satisfy $r_2-r_1>n/2$ and $r_2>n/2$, it holds that
\begin{equation}\label{equ-inc-wav-stfor-dif-newth-plstk}
WF^{r_1}(f^*u)\subseteq f^*WF^{r_2}(u),\quad \mbox{for all}\,\, u\in\DD'(U)\,\, \mbox{satisfying}\,\, WF^{r_2}(u)\cap \mathcal{N}_f=\emptyset.
\end{equation}
If $f$ has constant rank $k\geq 1$, then \eqref{equ-inc-wav-stfor-dif-newth-plstk} holds true when $r_2-r_1\geq (n-k)/2$ and $r_2>(n-k)/2$. When $f$ is a submersion, \eqref{equ-inc-wav-stfor-dif-newth-plstk} holds true for all $r_2\geq r_1$. In particular, if $f$ is a diffeomorphism then $WF^r(f^*u)= f^*WF^r(u)$ for all $u\in\DD'(U)$ and $r\in\RR$.
\end{corollary}

When $B$ is a bounded subset of $\DD'(U)$ satisfying $WF^{r_2}_c(B)\cap\mathcal{N}_f=\emptyset$, we can apply the theorem with $L:= WF^{r_2}_c(B)$ together with Corollary \ref{car-of-comset-by-wavefrse}, to deduce the following result.

\begin{corollary}\label{inv-of-wav-fro-set-com-on-difss}
Let $O$ and $U$ be open subsets of $\RR^m$ and $\RR^n$ respectively, let $f:O\rightarrow U$ be a smooth map and let $B$ be a bounded subset of $\DD'(U)$. Given $r_1,r_2\in\RR$ which satisfy $r_2-r_1>n/2$ and $r_2>n/2$, the following statement holds true:
\begin{itemize}
\item[$(*)$] if $WF^{r_2}_c(B)\cap \mathcal{N}_f=\emptyset$ then $B$ is a relatively compact subset of $\DD'^{r_2}_{WF^{r_2}_c(B)}(U)$, $f^*B$ is a relatively compact subset of $\DD'^{r_1}_{f^*WF^{r_2}_c(B)}(O)$ and $WF^{r_1}_c(f^*B)\subseteq f^*WF^{r_2}_c(B)$.
\end{itemize}
If $f$ has constant rank $k\geq1$, then the statement $(*)$ holds true when $r_2-r_1\geq (n-k)/2$ and $r_2>(n-k)/2$. When $f$ is a submersion, $(*)$ holds true for all $r_2\geq r_1$. In particular, if $f$ is a diffeomorphism then $WF^r_c(f^*B)= f^*WF^r_c(B)$ for all $r\in\RR$ and all bounded subsets $B$ of $\DD'(U)$.
\end{corollary}

Finally, we point out that one can consider a variant of Theorem \ref{the-pul-bac-for-smcrmdiff} for the space $\EE'^r_W(U)$, but we will not need such general facts. However, we will need the special case when $f$ is a diffeomorphism.

\begin{corollary}\label{cor-for-dif-invofespacf}
Let $r\in\RR$ and let $f:O\rightarrow U$ be a diffeomorphism between the open sets $O$ and $U$ in $\RR^n$. For any closed conic subset $L$ of $U\times(\RR^n\backslash\{0\})$ and any compact subset $K$ of $U$ satisfying $\pr_1(L)\subseteq K$, the pullback $f^*:\DD'(U)\rightarrow \DD'(O)$ restricts to a topological isomorphism $f^*:\EE'^r_{L;K}(U)\rightarrow \EE'^r_{f^*L;f^{-1}(K)}(O)$. Consequently, it also restricts to a topological isomorphism $f^*:\EE'^r_W(U)\rightarrow \EE'^r_{f^*W}(O)$ for any open conic subset $W$ of $U\times(\RR^n\backslash\{0\})$.
\end{corollary}

\begin{proof} Since $f^*(H^r_K(U))=H^r_{f^{-1}(K)}(O)$ as sets, \cite[Theorem 8.2.4, p. 263]{hor} implies that $f^*(\EE'^r_{L;K}(U))= \EE'^r_{f^*L;f^{-1}(K)}(O)$ as sets. As $f^*:\DD'(U)\rightarrow \DD'(O)$ is continuous, the closed graph and the open mapping theorems for Fr\'echet spaces imply that $f^*:\EE'^r_{L;K}(U)\rightarrow \EE'^r_{f^*L;f^{-1}(K)}(O)$ is a topological isomorphism. The last part is an immediate consequence of this.
\end{proof}

\section{The spaces \texorpdfstring{$\DD'^r_L$}{D'rL} and \texorpdfstring{$\EE'^r_W$}{E'rW} on manifolds and vector bundles}\label{sec-formain-result-onman}

The diffeomorphism invariance from Theorem \ref{the-pul-bac-for-smcrmdiff} and Corollary \ref{cor-for-dif-invofespacf} allows us to define $\DD'^r_L$ and $\EE'^r_W$ on smooth manifolds. Our goal in this section is to show a duality result for these spaces and to characterise the relatively compact subsets of $\DD'^r_L$ as in the Euclidean case. At the very end, we will show more general version of Theorem \ref{the-pul-bac-for-smcrmdiff} for pullback by smooth maps of distributional sections of vector bundles.\\
\indent From now and throughout the rest of the article, we will always employ the Einstein summation convention. Furthermore, $M$ will always stand for a smooth $m$-dimensional manifold and we will consistently apply the notations from Subsection \ref{subsec-dist-on-manifolds-vecbund}.

\subsection{The Sobolev compactness wave front set and the topology of \texorpdfstring{$\DD'^r_L$}{D'rL} on manifolds and vector bundles}

We start by recalling the Sobolev wave front set of order $r\in \RR$ of $u\in\DD'(M)$. The definition we are going to give is in the same spirit as in the Euclidean case; as we pointed out in Section \ref{Sec-comp}, the original definition of Duistermaat and H\"ormander \cite[p. 201]{dui-hor} is via pseudo-differential operators but one can easily convince oneself that they are the same by going to a chart and noticing that there they coincide in view of \cite[Proposition 8.2.6, p. 189]{hor1}. Pick a chart $(O,x)$ about $p\in M$ and, for $u\in \DD'(M)$, denote
$$
\Sigma^r_p(u):=\{\xi_jdx^j|_p\in T^*_pM\,|\, (\xi_1,\ldots,\xi_m)\in\Sigma^r_{x(p)}(u_x)\},
$$
where $u_x$ is the distribution defined in \eqref{ind-tri-man-caseford}. Corollary \ref{cor-for-dif-inv-wavsetformacss} verifies that $\Sigma^r_p(u)$ does not depend on the chart $(O,x)$ that contains $p$. The Sobolev wave front set of order $r$ of $u\in \DD'(M)$ is defined by
\begin{equation}\label{wav-fro-set-dis-manif}
WF^r(u):=\{(p,\xi)\in T^*M\backslash 0\,|\, \xi\in \Sigma^r_p(u)\}.
\end{equation}
If $\pi_E:E\rightarrow M$ is a vector bundle of rank $k$, then for a chart $(O,x)$ about $p\in M$ over which $E$ has a local trivialisation $\Phi_x:\pi^{-1}_E(O)\rightarrow O\times \CC^k$ and $u\in \DD'(M;E)$ define
$$
\Sigma^r_p(u):=\{\xi_jdx^j|_p\in T^*_pM\,|\, (\xi_1,\ldots,\xi_m)\in\cup_{j=1}^k\Sigma^r_{x(p)}(u_{\Phi_x}^j)\},
$$
where $u_{\Phi_x}^1,\ldots, u_{\Phi_x}^k\in \DD'(x(O))$ are the distributions defined in \eqref{ind-tri-loc-nestk}. Employing Corollary \ref{cor-for-dif-inv-wavsetformacss}, one can show that $\Sigma^r_p(u)$ does not depend on the chart $(O,x)$ nor on the local trivialisation $\Phi_x$. The Sobolev wave front set $WF^r(u)$ of $u\in\DD'(M;E)$ is defined as in \eqref{wav-fro-set-dis-manif}. Notice that $\Sigma^r_p(u)$ is a closed cone in $T^*_pM\backslash\{0\}$ and $WF^r(u)$ is a closed conic subset of $T^*M\backslash 0$.\\
\indent Let $B$ be a bounded subset of $\DD'(M)$ and $(O,x)$ a chart on $M$. Denote $B_x:=\{u_x\in\DD'(x(O))\,|\, u\in B\}$ and notice that $B_x$ is bounded in $\DD'(x(O))$. For a chart $(O,x)$ about $p\in M$, set
$$
\Sigma^r_{c,p}(B):=\{\xi_jdx^j|_p\in T^*_pM\,|\, (\xi_1,\ldots,\xi_m)\in\Sigma^r_{c,x(p)}(B_x)\}.
$$
In view of Corollary \ref{inv-of-wav-fro-set-com-on-difss}, $\Sigma^r_{c,p}(B)$ does not depend on the chart $(O,x)$ that contains $p$ and, similarly as above, we define the \textit{Sobolev compactness wave front set of order $r$} of $B$ by
\begin{equation}\label{wav-fro-set-dis-manif1}
WF^r_c(B):=\{(p,\xi)\in T^*M\backslash 0\,|\, \xi\in \Sigma^r_{c,p}(B)\}.
\end{equation}
Of course, this coincides with Definition \ref{def-com-wav-forsetdefonbds} when $M$ is an open subset of $\RR^m$.\\
\indent Similarly, if $E$ is a $k$-vector bundle over $M$ and $B$ a bounded subset of $\DD'(M;E)$, for any chart $(O,x)$ over which $E$ has a local trivialisation $\Phi_x:\pi^{-1}_E(O)\rightarrow O\times \CC^k$, we denote $B_{\Phi_x}^j:=\{u_{\Phi_x}^j\in\DD'(x(O))\,|\, u\in B\}$, $j=1,\ldots,k$; of course, $B_{\Phi_x}^1,\ldots, B_{\Phi_x}^k$ are bounded in $\DD'(x(O))$. For any such chart $(O,x)$ about $p\in M$, we define
$$
\Sigma^r_{c,p}(B):=\{\xi_jdx^j|_p\in T^*_pM\,|\, (\xi_1,\ldots,\xi_m)\in\cup_{j=1}^k\Sigma^r_{c,x(p)}(B_{\Phi_x}^j)\}.
$$
As before, employing Corollary \ref{inv-of-wav-fro-set-com-on-difss}, one can show that $\Sigma^r_{c,p}(B)$ does not depend on the chart $(O,x)$ that contains $p$ nor on the local trivialisation $\Phi_x$. We define the \textit{Sobolev compactness wave front set of order $r$} of $B$ by \eqref{wav-fro-set-dis-manif1}. Of course, $\Sigma^r_{c,p}(B)$ is a closed cone in $T^*_pM\backslash\{0\}$ and $WF^r_c(B)$ is a closed conic subset of $T^*M\backslash 0$. We point out that if $B=\{u^{(1)},\ldots, u^{(l)}\}$ then $WF^r_c(B)=WF^r(u^{(1)})\cup\ldots\cup WF^r(u^{(l)})$. Furthermore, as in the Euclidean case, $WF^r_c(a_1B_1+a_2B_2)\subseteq WF^r_c(B_1)\cup WF^r_c(B_2)$ for any $a_1,a_2\in\mathcal{C}^{\infty}(M)$ and any bounded subsets $B_1$ and $B_2$ of $\DD'(M;E)$ (or, of $\DD'(M)$); if in addition $B_1\subseteq B_2$ then $WF^r_c(B_1)\subseteq WF^r_c(B_2)$.\\
\indent Let $L$ be a closed conic subset of $T^*M\backslash 0$ and $E$ a $k$-vector bundle over $M$. For $r\in\RR$, we define
\begin{align*}
\DD'^r_L(M)&:=\{u\in\DD'(M)\,|\, WF^r(u)\subseteq L\}\quad \mbox{and}\\
\DD'^r_L(M;E)&:=\{u\in\DD'(M;E)\,|\, WF^r(u)\subseteq L\};
\end{align*}
of course, when $M$ is an open set in $\RR^m$, $\DD'^r_L(M)$ coincides as a set with $\DD'^r_L(M)$ as defined in \eqref{def-spa-with-wfinconcsl} in the Euclidean setting. Our goal is to equip $\DD'^r_L(M)$ and $\DD'^r_L(M;E)$ with locally convex topologies. Let $(O,x)$ be a chart on $M$. Let $\varphi\in\DD(O)$ and let $V$ be a closed cone in $\RR^m$ such that
\begin{equation}\label{equ-for-emp-intse}
\{(p,\xi_jdx^j|_p)\in T^*O\,|\, p\in \supp \varphi,\, (\xi_1,\ldots,\xi_m)\in V\}\cap L=\emptyset;
\end{equation}
notice that the first set is the inverse image of $\supp\varphi\times V$ under the chart induced local trivialisation of $T^*M$ over $O$ (i.e., $\pi^{-1}_{T^*M}(O)\rightarrow O\times \RR^m$, $(p,\xi_jdx^j|_p)\mapsto (p,\xi_1,\ldots,\xi_m)$). Then, when $u\in \DD'^r_L(M)$ we have $(\supp (\varphi\circ x^{-1}) \times V)\cap WF^r(u_x)=\emptyset$ and Corollary \ref{lemma-for-sem-wav-fr-set-spa} yields that
\begin{equation}\label{sem-for-topofddmlewithbninc}
\mathfrak{p}^x_{r;\varphi,V}(u):=\left(\int_V |\mathcal{F}((\varphi\circ x^{-1})u_x)(\xi)|^2\langle \xi\rangle^{2r} d\xi\right)^{1/2}<\infty,\quad u\in\DD'^r_L(M).
\end{equation}
We equip $\DD'^r_L(M)$ with the locally convex topology induced by all continuous seminorms on $\DD'(M)$ together with all seminorms $\mathfrak{p}^x_{r;\varphi,V}$ for all charts $(O,x)$ and $\varphi\in \DD(O)$ and $V\subseteq \RR^m$ as above. It is straightforward to check that if $u\in \DD'(M)$ satisfies $\mathfrak{p}^x_{r;\varphi,V}(u)<\infty$ for all $\mathfrak{p}^x_{r;\varphi,V}$ as above then $u\in \DD'^r_L(M)$.\\
\indent Analogously, given a chart $(O,x)$ over which $E$ trivialises via $\Phi_x:\pi^{-1}_E(O)\rightarrow O\times \CC^k$, a function $\varphi\in\DD(O)$ and a closed cone $V\subseteq \RR^m$ which satisfy \eqref{equ-for-emp-intse}, it holds that $(\supp (\varphi\circ x^{-1}) \times V)\cap WF^r(u_{\Phi_x}^j)=\emptyset$, $j=1,\ldots,k$, $u\in\DD'^r_L(M;E)$. Consequently, Corollary \ref{lemma-for-sem-wav-fr-set-spa} gives
\begin{equation}\label{sem-for-bun-valdiswavsincsstrs}
\mathfrak{p}^{\Phi_x}_{r;\varphi,V}(u):=\max_{1\leq j\leq k}\left(\int_V |\mathcal{F}((\varphi\circ x^{-1})u_{\Phi_x}^j)(\xi)|^2\langle \xi\rangle^{2r} d\xi\right)^{1/2}<\infty,\quad u\in\DD'^r_L(M;E).
\end{equation}
We equip $\DD'^r_L(M;E)$ with the locally convex topology induced by all continuous seminorms on $\DD'(M;E)$ together with all seminorms $\mathfrak{p}^{\Phi_x}_{r;\varphi,V}$ for all charts $(O,x)$ over which $E$ locally trivialises via $\Phi_x:\pi^{-1}_E(O)\rightarrow O\times \CC^k$ and all $\varphi\in \DD(O)$ and $V\subseteq \RR^m$ as above. Again, it is straightforward to check that if $u\in \DD'(M;E)$ satisfies $\mathfrak{p}^{\Phi_x}_{r;\varphi,V}(u)<\infty$ for all $\mathfrak{p}^{\Phi_x}_{r;\varphi,V}$ as above then $u\in\DD'^r_L(M;E)$.

\begin{remark}\label{rem-for-con-csifd}
Notice that $\DD'^r_{T^*M\backslash0}(M)=\DD'(M)$ and $\DD'^r_{T^*M\backslash0}(M;E)=\DD'(M;E)$ topologically. Also, it is straightforward to verify that $\DD'^r_{\emptyset}(M)=H^r_{\loc}(M)$ and $\DD'^r_{\emptyset}(M;E)=H^r_{\loc}(M;E)$ topologically.
\end{remark}

The following result will allow us to transfer the topological properties we showed in the Euclidean case to the case of manifolds and vector bundles. From now, we only state and prove the claims in the vector bundle case since the manifold case can be viewed as a special case of it (cf. Subsection \ref{subsec-dist-on-manifolds-vecbund}); we will keep both cases in the definitions for better clarity.

\begin{proposition}\label{lem-for-top-imebdofmapscs}
Let $\pi_E:E\rightarrow M$ be a vector bundle of rank $k$ and let $L$ be a closed conic subset of $T^*M\backslash0$. Let $\{(O_{\mu},x_{\mu})\}_{\mu\in\Lambda}$ be a family of coordinate charts on $M$ which cover $M$ such that $E$ locally trivialises over each $O_{\mu}$ via $\Phi_{x_{\mu}}:\pi^{-1}_E(O_{\mu})\rightarrow O_{\mu}\times \CC^k$, $\mu\in\Lambda$. Let $\kappa_{\mu}$, $\mu\in\Lambda$, be the total local trivialisation of $T^*M$ over $O_{\mu}$:
\begin{equation}\label{tri-cot-bun-coordindms}
\kappa_{\mu}: \pi^{-1}_{T^*M}(O_{\mu})\rightarrow x_{\mu}(O_{\mu})\times\RR^m,\quad \kappa_{\mu}(p,\xi_jdx^j|_p)=(x_{\mu}(p),\xi_1,\ldots,\xi_m).
\end{equation}
Then for each bounded subset $B$ of $\DD'(M;E)$, $WF^r_c(B)\subseteq L$ if and only if for all $\mu\in \Lambda$ and for all $\varphi\in\DD(O_{\mu})$ and closed cones $V\subseteq \RR^m$ which satisfy \eqref{equ-for-emp-intse}, it holds that
\begin{equation}\label{est-for-wfs-of-dis-bumva}
\sup_{u\in B}\max_{1\leq j\leq k}\int_{V,\,|\xi|>R} |\mathcal{F}((\varphi\circ x_{\mu}^{-1})u_{\Phi_{x_{\mu}}}^j)(\xi)|^2\langle \xi\rangle^{2r} d\xi\rightarrow 0,\quad \mbox{as}\quad R\rightarrow \infty.
\end{equation}
Furthermore, for each $\mu\in\Lambda$, $L_{\mu}:=\kappa_{\mu}(\pi^{-1}_{T^*M}(O_{\mu})\cap L)$ is a closed conic subset of $x_{\mu}(O_{\mu})\times (\RR^m\backslash\{0\})$ and the map
\begin{gather}
\DD'^r_L(M;E)\rightarrow \prod_{\mu\in\Lambda}\DD'^r_{L_{\mu}}(x_{\mu}(O_{\mu}))^k,\,\, u\mapsto\mathbf{f}_u,\quad \mbox{where}\label{map-imb-lcf-incprodick}\\
\mathbf{f}_u(\mu):=(u_{\Phi_{x_{\mu}}}^1,\ldots,u_{\Phi_{x_{\mu}}}^k),\, \mu\in\Lambda,\nonumber
\end{gather}
is a well-defined topological imbedding whose image is closed and complemented in $\prod_{\mu\in\Lambda}\DD'^r_{L_{\mu}}(x_{\mu}(O_{\mu}))^k$.
\end{proposition}

\begin{proof} The proof that $WF^r_c(B)\subseteq L$ is equivalent to \eqref{est-for-wfs-of-dis-bumva} is straightforward and we omit it (cf. Lemma \ref{lem-for-cha-wfofset-comwfl}). For the proof of the second part, we denote $u_{\mu}^j:=u_{\Phi_{x_{\mu}}}^j$ and $\Phi_{\mu}:=\Phi_{x_{\mu}}$, $\mu\in\Lambda$, $u\in\DD'(M;E)$, $j=1,\ldots,k$. The fact that $L_{\mu}$ is a closed conic subset of $x_{\mu}(O_{\mu})\times(\RR^m\backslash\{0\})$ is trivial. The map \eqref{map-imb-lcf-incprodick} is well-defined in view of the first part (applied to a singleton) and the proof of its continuity is straightforward (recall that $\DD'(M;E)\rightarrow \DD'(x_{\mu}(O_{\mu}))$, $u\mapsto u_{\mu}^j$, is continuous). Let $(\varphi_{\mu})_{\mu\in \Lambda}$ be a smooth partition of unity subordinated to $(O_{\mu})_{\mu\in\Lambda}$. A standard argument employing this partition of unity immediately yields the injectivity of \eqref{map-imb-lcf-incprodick}.\\
\indent To show that \eqref{map-imb-lcf-incprodick} is a topological imbedding whose image is complemented in $\prod_{\mu\in\Lambda}\DD'^r_{L_{\mu}}(x_{\mu}(O_{\mu}))^k$, we consider the map
\begin{gather}
\mathcal{R}:\prod_{\mu\in\Lambda}\DD'(x_{\mu}(O_{\mu}))^k\rightarrow \DD'(M;E),\,\, \mathcal{R}(\mathbf{f})= u_{\mathbf{f}},\quad \mbox{where}\label{map-bet-spa-disvecvalofprowiskkll}\\
\langle u_{\mathbf{f}},\psi\rangle:=\sum_{\mu\in\Lambda} \langle\mathbf{f}(\mu)^j,(\varphi_{\mu}\psi_{\mu,j})\circ x_{\mu}^{-1}\rangle,\,\, \psi\in\Gamma_c(E^{\vee}),\, \psi_{|O_{\mu}}=\psi_{\mu,j}\sigma_{\mu}^j,\nonumber
\end{gather}
and $(\sigma_{\mu}^1,\ldots,\sigma_{\mu}^k)$ is the local frame for $E^{\vee}$ over $O_{\mu}$ induced by $\Phi_{\mu}$. It is straightforward to check that the map is well-defined and continuous. We claim that $\mathcal{R}$ restricts to a well-defined and continuous map
\begin{equation}\label{res-map-t-c-for-newresinals}
\mathcal{R}:\prod_{\mu\in\Lambda}\DD'^r_{L_{\mu}}(x_{\mu}(O_{\mu}))^k\rightarrow \DD'^r_L(M;E).
\end{equation}
Let $(O,x)$ be a chart on $M$ over which $E$ locally trivialises via $\Phi_x:\pi^{-1}_E(O)\rightarrow O\times \CC^k$. Denote by $\kappa$ the total local trivialisation of $T^*M$ over $O$:
\begin{equation}\label{tot-loc-tri-cotbunovercoordpa}
\kappa:\pi^{-1}_{T^*M}(O)\rightarrow x(O)\times \RR^m,\quad \kappa(p,\xi_jdx^j|_p)=(x(p),\xi_1,\ldots,\xi_m).
\end{equation}
Let $\varphi\in \DD(O)\backslash\{0\}$ and the closed cone $V\subseteq \RR^m$ satisfy \eqref{equ-for-emp-intse}. Let $\Lambda_0\subseteq \Lambda$ be the finite set for which it holds that $\supp\varphi_{\mu}\cap \supp\varphi\neq \emptyset$, $\mu\in\Lambda_0$. Let $\mathbf{f}\in \prod_{\mu\in\Lambda}\DD'^r_{L_{\mu}}(x_{\mu}(O_{\mu}))^k$ and set $v_{\mu}^j:=\mathbf{f}(\mu)^j\in \DD'^r_{L_{\mu}}(x_{\mu}(O_{\mu}))$, $\mu\in\Lambda$, $j\in\{1,\ldots,k\}$. In view of \eqref{cha-fra-fun-duabuch}, we have
\begin{equation*}
\sigma^j=(|(x\circ x_{\mu}^{-1})'|\circ x_{\mu})\tau_{\mu,l}^j \sigma_{\mu}^l\quad \mbox{on}\quad O_{\mu}\cap O,
\end{equation*}
where $(\sigma^1,\ldots,\sigma^k)$ is the local frame for $E^{\vee}$ over $O$ induced by $\Phi_x$ and $\tau_{\mu}=(\tau_{\mu,l}^j)_{j,l}:O_{\mu}\cap O\rightarrow\operatorname{GL}(k,\CC)$ is the transition map given by $\Phi_x\circ\Phi_{\mu}^{-1}(p,z)=(p,\tau_{\mu}(p)z)$, $p\in O_{\mu}\cap O$, $z\in\CC^k$. Set $\phi_{\xi}:=e^{-i\,\cdot\, \xi}(\varphi\circ x^{-1})\in\DD(x(O))$, $\xi\in\RR^m$, and notice that
\begin{align*}
\mathcal{F}((\varphi\circ x^{-1}) (u_{\mathbf{f}})^j_{\Phi_x})(\xi)&=\langle (u_{\mathbf{f}})^j_{\Phi_x},\phi_{\xi}\rangle=\langle u_{\mathbf{f}},(\phi_{\xi}\circ x) \sigma^j\rangle\\
&=\sum_{\mu\in\Lambda_0} \langle v_{\mu}^l, ((\varphi_{\mu}\tau_{\mu,l}^j)\circ x_{\mu}^{-1}) (\phi_{\xi}\circ x\circ x_{\mu}^{-1})|(x\circ x_{\mu}^{-1})'|\rangle\\
&=\sum_{\mu\in\Lambda_0} \langle (x_{\mu}\circ x^{-1})^*v_{\mu}^l, ((\varphi_{\mu}\tau_{\mu,l}^j)\circ x^{-1}) \phi_{\xi}\rangle\\
&=\sum_{\mu\in\Lambda_0}\mathcal{F}\left(((\varphi\varphi_{\mu}\tau_{\mu,l}^j)\circ x^{-1})(x_{\mu}\circ x^{-1})^*v_{\mu}^l\right)(\xi).
\end{align*}
This implies that
\begin{equation}\label{est-sem-nor-onthebspaformaprtl}
\mathfrak{p}^{\Phi_x}_{r;\varphi,V}(u_{\mathbf{f}})\leq \max_{1\leq j\leq k}\sum_{\mu\in\Lambda_0}\sum_{l=1}^k \mathfrak{p}_{r;(\varphi\varphi_{\mu}\tau_{\mu,l}^j)\circ x^{-1},V}((x_{\mu}\circ x^{-1})^*v_{\mu}^l).
\end{equation}
For each $\mu\in\Lambda_0$, we set
$$
L_{\mu,O}:=\kappa_{\mu}(\pi^{-1}_{T^*M}(O_{\mu}\cap O)\cap L)\quad \mbox{and}\quad \widetilde{L}_{\mu,O}:=\kappa(\pi^{-1}_{T^*M}(O_{\mu}\cap O)\cap L).
$$
Of course, $L_{\mu,O}$ and $\widetilde{L}_{\mu,O}$ are closed conic subsets of $x_{\mu}(O_{\mu}\cap O)\times (\RR^m\backslash \{0\})$ and $x(O_{\mu}\cap O)\times (\RR^m\backslash\{0\})$ respectively. It is straightforward to check that $(x_{\mu}\circ x^{-1})^*L_{\mu,O}=\widetilde{L}_{\mu,O}$, whence Theorem \ref{the-pul-bac-for-smcrmdiff} yields that
$$
(x_{\mu}\circ x^{-1})^*:\DD'^r_{L_{\mu,O}}(x_{\mu}(O_{\mu}\cap O))\rightarrow \DD'^r_{\widetilde{L}_{\mu,O}}(x(O_{\mu}\cap O))
$$
is well-defined and continuous. Hence, for each $\mu\in\Lambda_0$ and $j,l\in\{1,\ldots,k\}$, there are $C>0$, a continuous seminorm $\mathfrak{p}$ on $\DD'(x_{\mu}(O_{\mu}\cap O))$, $\widetilde{\phi}_{1},\ldots,\widetilde{\phi}_{J}\in\DD(x_{\mu}(O_{\mu}\cap O))$ and closed cones $\widetilde{V}_{1},\ldots,\widetilde{V}_{J}\subseteq \RR^m$ satisfying $(\supp\widetilde{\phi}_{j'}\times \widetilde{V}_{j'})\cap L_{\mu,O}=\emptyset$, $j'=1,\ldots,J$, (of course, these depend on $\mu$, $j$ and $l$) such that
\begin{align*}
\mathfrak{p}_{r;(\varphi\varphi_{\mu}\tau_{\mu,l}^j)\circ x^{-1},V}((x_{\mu}\circ x^{-1})^*v_{\mu}^l)&\leq C\mathfrak{p}(v_{\mu}^l)+ C\sum_{j'=1}^{J}\mathfrak{p}_{r;\widetilde{\phi}_{j'}, \widetilde{V}_{j'}}(v_{\mu}^l)\\
&\leq C\widetilde{\mathfrak{p}}(v_{\mu}^l)+ C\sum_{j'=1}^{J}\mathfrak{p}_{r;\widetilde{\phi}_{j'}, \widetilde{V}_{j'}}(v_{\mu}^l),
\end{align*}
where $\widetilde{\mathfrak{p}}$ is a continuous seminorm on $\DD'(x_{\mu}(O_{\mu}))$ (since the restriction mapping $\DD'(x_{\mu}(O_{\mu}))\rightarrow \DD'(x_{\mu}(O_{\mu}\cap O))$ is continuous). In view of \eqref{est-sem-nor-onthebspaformaprtl}, this implies that \eqref{res-map-t-c-for-newresinals} is well-defined and continuous (the required bounds for $\mathfrak{p}(u_{\mathbf{f}})$, $\mathfrak{p}$ a continuous seminorm on $\DD'(M;E)$, follow from the continuity of \eqref{map-bet-spa-disvecvalofprowiskkll}). It is straightforward to check that $\mathcal{R}(\mathbf{f}_u)=u$, $u\in\DD'^r_L(M;E)$; i.e., $\mathcal{R}$ is a left inverse of \eqref{map-imb-lcf-incprodick}. This immediately implies that \eqref{map-imb-lcf-incprodick} is a topological imbedding and that its image is complemented in $\prod_{\mu\in\Lambda}\DD'^r_{L_{\mu}}(x_{\mu}(O_{\mu}))^k$. The image of \eqref{map-imb-lcf-incprodick} is closed since complemented subspaces always are.
\end{proof}

\begin{remark}
In the proof, we showed that for any partition of unity subordinated to $(O_{\mu})_{\mu\in\Lambda}$, the mapping \eqref{res-map-t-c-for-newresinals} is well-defined continuous left inverse of \eqref{map-imb-lcf-incprodick}. The analogous mapping for $\DD'^r_L(M)$ is
\begin{gather*}
\prod_{\mu\in\Lambda}\DD'^r_{L_{\mu}}(x_{\mu}(O_{\mu}))\rightarrow \DD'^r_L(M),\,\, \mathbf{f}\mapsto u_{\mathbf{f}},\quad \mbox{where}\\
\langle u_{\mathbf{f}},\psi\rangle:=\sum_{\mu\in\Lambda} \langle\mathbf{f}(\mu),(\varphi_{\mu}\psi_{\mu})\circ x_{\mu}^{-1}\rangle,\,\, \psi\in\Gamma_c(DM),\, \psi_{|O_{\mu}}=\psi_{\mu}\lambda^{x_{\mu}}.
\end{gather*}
\end{remark}

\begin{remark}\label{char-for-topofwfsetsobdefofn}
The proposition immediately implies that $u\in \DD'(M;E)$ (resp., $u\in\DD'(M)$) belongs to $\DD'^r_L(M;E)$ (resp., $\DD'^r_L(M)$) if and only if the seminorms \eqref{sem-for-bun-valdiswavsincsstrs} (resp., \eqref{sem-for-topofddmlewithbninc}) are finite when $(O,x)$ runs through the charts of a fixed atlas of $M$ such that $E$ locally trivialises over each of the charts. Similarly, the topology of $\DD'^r_L(M;E)$ (resp., $\DD'^r_L(M)$) is given by the continuous seminorms on $\DD'(M;E)$ (resp., $\DD'(M)$) together with the seminorms \eqref{sem-for-bun-valdiswavsincsstrs} (resp., \eqref{sem-for-topofddmlewithbninc}) when $(O,x)$ runs through the charts of a fixed atlas of $M$ such that $E$ locally trivialises over each of the charts. In particular, this shows that when $M$ is an open set in $\RR^m$, the topology on $\DD'^r_L(M)$ is the same as the one we defined in Section \ref{Sec-comp}.
\end{remark}

\begin{remark}\label{hyp-con-rem-formanbundlcaseofcon}
As a consequence of the proposition and Remark \ref{rem-for-hyp-conmultmapfirsk}, we infer that $\mathcal{C}^{\infty}(M)\times \DD'^r_L(M;E)\rightarrow\DD'^r_L(M;E)$, $(\chi, u)\mapsto \chi u$, is well-defined and hypocontinuous.
\end{remark}

Since we can take $\{(O_{\mu}, x_{\mu})\}_{\mu\in\Lambda}$ to be countable, the proposition together with Corollary \ref{rem-for-sem-refd}, \cite[Theorem 5 and Theorem 6, p. 299]{kothe1} and \cite[Theorem 1, p. 61, and Theorem 6, p. 62]{kothe2} immediately yield the following result.

\begin{corollary}
The space $\DD'^r_L(M;E)$ is complete, semi-reflexive and strictly webbed (in the sense of De Wilde).
\end{corollary}

Employing the above proposition together with Corollary \ref{car-of-comset-by-wavefrse} and Tychonoff's theorem, we obtain the characterisation of the relatively compact sets in $\DD'^r_L(M;E)$ we announced in the introduction.

\begin{corollary}\label{cor-for-relcom-sub-wafe-fronchar}
Let $L$ be a closed conic subset of $T^*M\backslash 0$ and let $B$ be a bounded subset of $\DD'(M;E)$. Then $WF^r_c(B)\subseteq L$ if and only if $B$ is a relatively compact subset of $\DD'^r_L(M;E)$.
\end{corollary}

\begin{corollary}
The bounded subsets of $\DD'^r_L(M;E)$ are metrisable when equipped with the induced topology. Consequently, if $\{u_j\}_{j\in\ZZ_+}$ is relatively compact in $\DD'^r_L(M;E)$, then there exists a subsequence $(u_{j_k})_{k\in\ZZ_+}$ which converges in $\DD'^r_L(M;E)$.
\end{corollary}

\begin{proof} It suffices to show the claim for $\DD'^r_L(O)$, for $O$ an open set in $\RR^m$, because of Proposition \ref{lem-for-top-imebdofmapscs} (by taking the cover $\{(O_{\mu},x_{\mu})\}_{\mu\in\Lambda}$ to be countable). The claim for $\DD'^r_L(O)$ immediately follows from Proposition \ref{pro-for-top-imbedingthforc} in view of \cite[Theorem 1.7, p. 128]{Sch} since the weak and strong topologies coincide on the bounded subsets of $\DD'(O)$.
\end{proof}

\begin{proposition}\label{res-for-den-ope-map-thagivdes}
Let $\pi_E:E\rightarrow M$ be a vector bundle of rank $k$ and $\{(O_{\mu},x_{\mu})\}_{\mu\in\ZZ_+}$ a countable family of coordinate charts on $M$ which cover $M$. Assume that $E$ locally trivialises over each $O_{\mu}$ via $\Phi_{x_{\mu}}:\pi^{-1}_E(O_{\mu})\rightarrow O_{\mu}\times \CC^k$, $\mu\in\ZZ_+$, and let $(s_{\mu,1},\ldots,s_{\mu,k})$, $\mu\in\ZZ_+$, be the local frame of $E$ over $O_{\mu}$ induced by $\Phi_{x_{\mu}}$, $\mu\in\ZZ_+$. Let $(\varphi_{\mu})_{\mu\in \ZZ_+}$ be a smooth partition of unity subordinated to $(O_{\mu})_{\mu\in\ZZ_+}$. For each $\mu\in\ZZ_+$, let $(\widetilde{P}_{\mu,n})_{n\in\ZZ_+}$ be a sequence of operators in $\mathcal{L}(\DD'(x_{\mu}(O_{\mu})),\DD(x_{\mu}(O_{\mu})))$ which satisfies the properties as in the conclusion of Proposition \ref{seq-den-comsmf}. Then the operators
\begin{equation}\label{ope-tha-giv-dens-indfprevlls}
P_n:\DD'(M;E)\rightarrow \Gamma_c(E),\quad P_n(u):=\sum_{\mu=1}^n \varphi_{\mu} (\widetilde{P}_{\mu,n}u_{\Phi_{x_{\mu}}}^j)\circ x_{\mu} s_{\mu,j},\quad n\in\ZZ_+,
\end{equation}
are well-defined, continuous and $P_n\rightarrow \operatorname{Id}$ in $\mathcal{L}_b(\DD'(M;E))$. Furthermore, for any closed conic subset $L$ of $T^*M\backslash 0$, the set $\{P_n\}_{n\in\ZZ_+}$ is bounded in $\mathcal{L}_b(\DD'^r_L(M;E))$ and $P_n\rightarrow \operatorname{Id}$ in $\mathcal{L}_p(\DD'^r_L(M;E))$. In particular, $\Gamma_c(E)$ is sequentially dense in $\DD'^r_L(M;E)$.
\end{proposition}

\begin{proof} For simpler notation, denote $u_{\mu}^j:=u_{\Phi_{x_{\mu}}}^j$ and $\Phi_{\mu}:=\Phi_{x_{\mu}}$, $\mu\in\ZZ_+$, $u\in\DD'(M;E)$, $j=1,\ldots,k$. Let $(\sigma^1_{\mu},\ldots,\sigma^k_{\mu})$, $\mu\in\ZZ_+$, be the frame induced by $\Phi_{\mu}$ on $E^{\vee}$ over $O_{\mu}$; notice that $\sigma^j_{\mu}(s_{\mu,l})=\delta^j_l\lambda^{x_{\mu}}$ on $O_{\mu}$. Denote by $\kappa_{\mu}$ the coordinate induced total local trivialisation \eqref{tri-cot-bun-coordindms} of $T^*M$ over $O_{\mu}$ and set $L_{\mu}:=\kappa_{\mu}(\pi^{-1}_{T^*M}(O_{\mu})\cap L)$. Clearly \eqref{ope-tha-giv-dens-indfprevlls} is well-defined and its continuity follows immediately from the continuity of $\widetilde{P}_{\mu,j}:\DD'(x_{\mu}(O_{\mu}))\rightarrow \DD(x_{\mu}(O_{\mu}))$, $\mu,j\in\ZZ_+$. We show that for each fixed $u\in\DD'(M;E)$, $P_nu\rightarrow u$ in $\DD'(M;E)$. Let $B$ be a bounded subset of $\Gamma_c(E^{\vee})$. There is a compact subset $K$ of $M$ such that $B$ is a bounded subset of $\Gamma_K(E^{\vee})$. There is $n_0\in\ZZ_+$ such that $\supp\varphi_{\mu}\cap K=\emptyset$, $\mu>n_0$. For $\psi\in B$, write $\psi_{|O_{\mu}}=\psi_{\mu,j}\sigma^j_{\mu}$. When $n\geq n_0+1$, we have
\begin{align*}
\langle P_nu-u,\psi\rangle&=\sum_{\mu=1}^n\langle (\widetilde{P}_{\mu,n}u_{\mu}^j)\circ x_{\mu} s_{\mu,j},\varphi_{\mu}\psi_{\mu,j'}\sigma^{j'}_{\mu}\rangle-\sum_{\mu=1}^{n_0}\langle u,\varphi_{\mu}\psi\rangle\\
&=\sum_{\mu=1}^{n_0} \langle\widetilde{P}_{\mu,n}u_{\mu}^j-u_{\mu}^j,(\varphi_{\mu}\psi_{\mu,j})\circ x_{\mu}^{-1}\rangle.
\end{align*}
Since $\{(\varphi_{\mu}\psi_{\mu,j})\circ x_{\mu}^{-1}\,|\, \psi\in B\}$ is a bounded subset of $\DD(x_{\mu}(O_{\mu}))$, we deduce $\sup_{\psi\in B}|\langle P_nu-u,\psi\rangle|\rightarrow0$ as $n\rightarrow \infty$. Hence $P_n\rightarrow \operatorname{Id}$ in the topology of simple convergence on $\mathcal{L}(\DD'(M;E))$. The Banach-Steinhaus theorem \cite[Theorem 4.5, p. 85]{Sch} together with the fact that $\DD'(M;E)$ is Montel now yield that $P_n\rightarrow \operatorname{Id}$ in $\mathcal{L}_b(\DD'(M;E))$. The latter implies that $\{P_n\}_{n\in\ZZ_+}$ is bounded in $\mathcal{L}_b(\DD'(M;E))$. Hence, given a closed conic subset $L$ of $T^*M\backslash 0$, in order to prove that $\{P_n\}_{n\in\ZZ_+}$ is bounded in $\mathcal{L}_b(\DD'^r_L(M;E))$ it suffices to show that the seminorms \eqref{sem-for-bun-valdiswavsincsstrs} are uniformly bounded when $u$ varies in a bounded subset of $\DD'^r_L(M;E)$. Let $B$ be a bounded subset of $\DD'^r_L(M;E)$, let $(O,x)$ be a coordinate chart on $M$ over which $E$ locally trivialises via $\Phi_x:\pi_E^{-1}(O)\rightarrow O\times \CC^k$ and let $\varphi\in\DD(O)$ and the closed cone $V\subseteq \RR^m$ satisfy \eqref{equ-for-emp-intse}. We denote by $\tau_{\mu}=(\tau_{\mu,l}^j)_{j,l}:\pi_E^{-1}(O_{\mu}\cap O)\rightarrow \operatorname{GL}(k,\CC)$ the transition map given by $\Phi_x\circ\Phi_{\mu}^{-1}(p,z)=(p,\tau_{\mu}(p)z)$, $p\in O_{\mu}\cap O$, $z\in\CC^k$. Define $\phi_{\xi}:=e^{-i\, \cdot\,\xi}(\varphi\circ x^{-1})$, $\xi\in\RR^m$; clearly $\phi_{\xi}\in\DD(x(O))$, $\xi\in\RR^m$. There is $n_0\in\ZZ_+$ such that $\supp\varphi_{\mu}\cap \supp\varphi=\emptyset$, $\mu>n_0$. For $n\geq n_0+1$ and $u\in B$, we employ \eqref{cha-fra-fun-duabuch} to infer
\begin{align*}
\mathcal{F}((\varphi\circ x^{-1}) (P_nu)^j_{\Phi_x})(\xi)&=\langle (P_n u)^j_{\Phi_x},\phi_{\xi}\rangle=\sum_{\mu=1}^n\langle (\widetilde{P}_{\mu,n}u_{\mu}^t)\circ x_{\mu} s_{\mu,t},(\phi_{\xi}\circ x)\varphi_{\mu} \sigma^j\rangle\\
&=\sum_{\mu=1}^{n_0}\langle (\widetilde{P}_{\mu,n}u_{\mu}^t)\circ x_{\mu} s_{\mu,t}, (\phi_{\xi}\circ x)(|(x\circ x^{-1}_{\mu})'|\circ x_{\mu})\varphi_{\mu}\tau^j_{\mu,l}\sigma^l_{\mu}\rangle\\
&=\sum_{\mu=1}^{n_0}\langle \widetilde{P}_{\mu,n}u_{\mu}^l, (\phi_{\xi}\circ x\circ x^{-1}_{\mu})((\varphi_{\mu}\tau^j_{\mu,l})\circ x^{-1}_{\mu})|(x\circ x^{-1}_{\mu})'|\rangle\\
&=\sum_{\mu=1}^{n_0}\langle (x_{\mu}\circ x^{-1})^*\widetilde{P}_{\mu,n}u_{\mu}^l, \phi_{\xi}((\varphi_{\mu}\tau^j_{\mu,l})\circ x^{-1})\rangle\\
&=\sum_{\mu=1}^{n_0}\mathcal{F}\left(((\varphi\varphi_{\mu}\tau^j_{\mu,l})\circ x^{-1})(x_{\mu}\circ x^{-1})^*\widetilde{P}_{\mu,n}u_{\mu}^l\right)(\xi).
\end{align*}
We deduce
$$
\mathfrak{p}^{\Phi_x}_{r;\varphi,V}(P_nu)\leq \max_{1\leq j\leq k} \sum_{\mu=1}^{n_0}\sum_{l=1}^k\mathfrak{p}_{r;(\varphi\varphi_{\mu}\tau^j_{\mu,l})\circ x^{-1}, V}((x_{\mu}\circ x^{-1})^*\widetilde{P}_{\mu,n}u_{\mu}^l).
$$
With $\kappa$ as in \eqref{tot-loc-tri-cotbunovercoordpa}, we set $L_{\mu,O}:=\kappa_{\mu}(\pi^{-1}_{T^*M}(O_{\mu}\cap O)\cap L)$ and $\widetilde{L}_{\mu,O}:=\kappa(\pi^{-1}_{T^*M}(O_{\mu}\cap O)\cap L)$. Then $(x_{\mu}\circ x^{-1})^*L_{\mu,O}=\widetilde{L}_{\mu,O}$ and Theorem \ref{the-pul-bac-for-smcrmdiff} yields that $(x_{\mu}\circ x^{-1})^*:\DD'^r_{L_{\mu,O}}(x_{\mu}(O_{\mu}\cap O))\rightarrow \DD'^r_{\widetilde{L}_{\mu,O}}(x(O_{\mu}\cap O))$ is continuous. Since the map $\DD'^r_L(M;E)\rightarrow \DD'^r_{L_{\mu}}(x_{\mu}(O_{\mu}))$, $u\mapsto u^l_{\mu}$, is continuous, the set $\{\widetilde{P}_{\mu,n}u_{\mu}^l\,|\, u\in B,\,n\in\ZZ_+\}$ is bounded in $\DD'^r_{L_{\mu}}(x_{\mu}(O_{\mu}))$ and consequently in $\DD'^r_{L_{\mu,O}}(x_{\mu}(O_{\mu}\cap O))$ as well. The above now implies that $\sup_{n\geq n_0+1}\sup_{u\in B} \mathfrak{p}^{\Phi_x}_{r;\varphi,L}(P_nu)<\infty$ which completes the proof for the boundedness of $\{P_n\}_{n\in\ZZ_+}$ in $\mathcal{L}_b(\DD'^r_L(M;E))$. It remains to show that $P_n\rightarrow \operatorname{Id}$ in $\mathcal{L}_p(\DD'^r_L(M;E))$. Since we show that the convergence holds in $\mathcal{L}_b(\DD'(M;E))$, it suffices to show that for each precompact subset $B$ of $\DD'^r_L(M;E)$, $\sup_{u\in B}\mathfrak{p}^{\Phi_x}_{r;\varphi,V}(P_nu-u)\rightarrow 0$ where $\varphi$ and $V$ are as above. Similarly as above, for all $u\in B$ and $n\geq n_0+1$, we have
$$
\mathfrak{p}^{\Phi_x}_{r;\varphi,V}(P_nu-u)\leq \max_{1\leq j\leq k} \sum_{\mu=1}^{n_0}\sum_{l=1}^k\mathfrak{p}_{r;(\varphi\varphi_{\mu}\tau^j_{\mu,l})\circ x^{-1}, V}((x_{\mu}\circ x^{-1})^*(\widetilde{P}_{\mu,n}u_{\mu}^l-u_{\mu}^l)).
$$
Employing the same reasoning as before, one shows that $\sup_{u\in B}\mathfrak{p}^{\Phi_x}_{r;\varphi,V}(P_nu-u)\rightarrow 0$ which completes the proof.
\end{proof}

\begin{remark}\label{con-ofm-for-confasec}
If for each $\mu\in\ZZ_+$, $\{\widetilde{P}_{\mu,n}\}_{n\in\ZZ_+}$ are defined as in Proposition \ref{seq-den-comsmf}, then it is straightforward to verify that for each $u\in \Gamma^0(E)$, $P_nu\rightarrow u$ in $\Gamma^0(E)$.
\end{remark}

\begin{remark}
The density we just proved immediately shows that the following continuous inclusions are also dense:
\begin{gather}
H_{\loc}^r(M)\subseteq \DD'^r_L(M)\subseteq \DD'(M),\quad H_{\loc}^r(M;E)\subseteq \DD'^r_L(M;E)\subseteq \DD'(M;E);\label{inc-for-dif-nincfduals}\\
\DD'^{r_2}_L(M)\subseteq \DD'^{r_1}_L(M)\quad \mbox{and}\quad \DD'^{r_2}_L(M;E)\subseteq \DD'^{r_1}_L(M;E),\quad \mbox{when}\,\, r_2\geq r_1.\label{map-inc-for-dspasobw}
\end{gather}
\end{remark}

\subsection{The dual of \texorpdfstring{$\DD'^r_L(M;E)$}{D'rL(M;E)}}

Our next goal is to find the strong dual of $\DD'^r_L(M)$ and of $\DD'^r_L(M;E)$. As in the Euclidean case, for an open conic subset $W$ of $T^*M\backslash0$ and $r\in\RR$ we define
\begin{align*}
\EE'^r_W(M)&:=\{u\in H^r_{\comp}(M)\,|\, WF(u)\subseteq W\}\quad\mbox{and}\\
\EE'^r_W(M;E)&:=\{u\in H^r_{\comp}(M;E)\,|\, WF(u)\subseteq W\}.
\end{align*}
Additionally, for a closed conic subset $L$ of $T^*M\backslash0$ and a compact set $K\subseteq M$ satisfying $\pi_{T^*M}(L)\subseteq K$, we define
\begin{align*}
\EE'^r_{L;K}(M)&:=\{u\in H^r_K(M)\,|\, WF(u)\subseteq L\}\quad \mbox{and}\\
\EE'^r_{L;K}(M;E)&:=\{u\in H^r_K(M;E)\,|\, WF(u)\subseteq L\}.
\end{align*}
Of course, when $M$ is an open subset of $\RR^m$, $\EE'^r_{L;K}(M)$ and $\EE'^r_W(M)$ coincide as sets with their Euclidean counterparts we defined in Subsection \ref{sub-sec-for-dualityee}. As in the Euclidean case, we first define locally convex topologies on $\EE'^r_{L;K}(M)$ and $\EE'^r_{L;K}(M;E)$ and then define the topology on $\EE'^r_W(M)$ and $\EE'^r_W(M;E)$ as inductive limits of these spaces.\\
\indent Let $(O,x)$ be a chart on $M$. Let $\varphi\in\DD(O)$ and let $\emptyset\neq V\subseteq \RR^m$ be a closed cone such that \eqref{equ-for-emp-intse} holds true. When $u\in \EE'^r_{L;K}(M)$, we have $(\supp (\varphi\circ x^{-1}) \times V)\cap WF(u_x)=\emptyset$ and hence
\begin{equation}\label{sem-for-topofddmlewithbninc1111}
\mathfrak{q}^x_{\nu;\varphi,V}(u):=\sup_{\xi\in V}\langle \xi\rangle^{\nu}|\mathcal{F}((\varphi\circ x^{-1})u_x)(\xi)|<\infty,\quad u\in\EE'^r_{L;K}(M),\, \nu>0.
\end{equation}
We equip $\EE'^r_{L;K}(M)$ with the locally convex topology induced by (any) norm on $H^r_K(M)$ together with all seminorms $\mathfrak{q}^x_{\nu;\varphi,V}$, for all $\nu>0$, all charts $(O,x)$ and all $\varphi\in \DD(O)$ and $\emptyset\neq V\subseteq \RR^m$ as above. Analogously, given a chart $(O,x)$ over which $E$ trivialises via $\Phi_x:\pi^{-1}_E(O)\rightarrow O\times \CC^k$, a function $\varphi\in\DD(O)$ and a closed cone $\emptyset\neq V\subseteq \RR^m$ which satisfy \eqref{equ-for-emp-intse}, it holds that $(\supp (\varphi\circ x^{-1}) \times V)\cap WF(u_{\Phi_x}^j)=\emptyset$, $j=1,\ldots,k$, $u\in\EE'^r_{L;K}(M;E)$, and thus
\begin{equation}\label{sem-for-bun-valdiswavsincsstrs1122}
\mathfrak{q}^{\Phi_x}_{\nu;\varphi,V}(u):=\max_{1\leq j\leq k}\sup_{\xi\in V}\langle \xi\rangle^{\nu}|\mathcal{F}((\varphi\circ x^{-1})u_{\Phi_x}^j)(\xi)|<\infty,\quad u\in\EE'^r_{L;K}(M;E),\, \nu>0.
\end{equation}
We equip $\EE'^r_{L;K}(M;E)$ with the locally convex topology induced by (any) norm on $H^r_K(M;E)$ together with all seminorms $\mathfrak{q}^{\Phi_x}_{\nu;\varphi,V}$, for all $\nu>0$, all charts $(O,x)$ over which $E$ locally trivialises via $\Phi_x:\pi^{-1}_E(O)\rightarrow O\times \CC^k$ and all $\varphi\in \DD(O)$ and $\emptyset\neq V\subseteq \RR^m$ as above.

\begin{proposition}\label{lem-for-top-imebdofmapscs11}
Let $\pi_E:E\rightarrow M$ be a vector bundle of rank $k$, $K\subset\subset M$ and $L$ a closed conic subset of $T^*M\backslash0$ which satisfy $\pi_{T^*M}(L)\subseteq K$. Let $\{(O_{\mu},x_{\mu})\}_{\mu\in\Lambda}$ be a finite family of coordinate charts on $M$ which cover $K$ such that $E$ locally trivialises over $O_{\mu}$ via $\Phi_{x_{\mu}}:\pi^{-1}_E(O_{\mu})\rightarrow O_{\mu}\times \CC^k$, $\mu\in\Lambda$. Let $\kappa_{\mu}$, $\mu\in\Lambda$, be the total local trivialisation of $T^*M$ over $O_{\mu}$ as in \eqref{tri-cot-bun-coordindms}. Let $\varphi_{\mu}\in\mathcal{C}^{\infty}(M)$, $\mu\in\Lambda$, be nonnegative and such that $\supp\varphi_{\mu}\subseteq O_{\mu}$ and $\sum_{\mu\in\Lambda}\varphi_{\mu}=1$ on a neighbourhood of $K$. For each $\mu\in\Lambda$, $K_{\mu}:=x_{\mu}(K\cap \supp\varphi_{\mu})$ is a compact subset of $x_{\mu}(O_{\mu})$ and $L_{\mu}:=\kappa_{\mu}(\pi^{-1}_{T^*M}(\supp\varphi_{\mu})\cap L)$ is a closed conic subset of $x_{\mu}(O_{\mu})\times (\RR^m\backslash\{0\})$ which satisfies $\pr_1(L_{\mu})\subseteq K_{\mu}$. The distribution $u\in H^r_K(M;E)$ belongs to $\EE'^r_{L;K}(M;E)$ if and only if $\mathfrak{q}^{\Phi_{x_{\mu}}}_{\nu;\varphi,V}(u)<\infty$ for all $\nu>0$, $\mu\in \Lambda$ and for all $\varphi\in\DD(O_{\mu})$ and closed cones $V\subseteq \RR^m$ which satisfy \eqref{equ-for-emp-intse}. Furthermore, the map
\begin{gather}
\EE'^r_{L;K}(M;E)\rightarrow \prod_{\mu\in\Lambda}\EE'^r_{L_{\mu};K_{\mu}}(x_{\mu}(O_{\mu}))^k,\,\, u\mapsto\mathbf{f}_u,\quad \mbox{where}\label{map-imb-lcf-incprodick1111}\\
\mathbf{f}_u(\mu):=((\varphi_{\mu}\circ x_{\mu}^{-1})u_{\Phi_{x_{\mu}}}^1,\ldots,(\varphi_{\mu}\circ x_{\mu}^{-1})u_{\Phi_{x_{\mu}}}^k),\, \mu\in\Lambda,\nonumber
\end{gather}
is a well-defined topological imbedding whose image is closed and complemented in $\prod_{\mu\in\Lambda}\EE'^r_{L_{\mu};K_{\mu}}(x_{\mu}(O_{\mu}))^k$.
\end{proposition}

\begin{proof} The proof is similar to the proof of Proposition \ref{lem-for-top-imebdofmapscs}. The only notable differences are that instead of Theorem \ref{the-pul-bac-for-smcrmdiff}, one now employs Corollary \ref{cor-for-dif-invofespacf} and, instead of \eqref{res-map-t-c-for-newresinals}, one now shows that
\begin{gather}
\prod_{\mu\in\Lambda}\EE'^r_{L_{\mu};K_{\mu}}(x_{\mu}(O_{\mu}))^k\rightarrow \EE'^r_{L;K}(M;E),\,\, \mathbf{f}\mapsto u_{\mathbf{f}},\quad \mbox{where}\label{map-lef-inv-fortheedualofdcomosl}\\
\langle u_{\mathbf{f}},\psi\rangle:=\sum_{\mu\in\Lambda} \langle\mathbf{f}(\mu)^j,\psi_{\mu,j}\circ x_{\mu}^{-1}\rangle,\,\, \psi\in\Gamma_c(E^{\vee}),\, \psi_{|O_{\mu}}=\psi_{\mu,j}\sigma_{\mu}^j,\nonumber
\end{gather}
is a well-defined continuous left inverse of \eqref{map-imb-lcf-incprodick1111}.
\end{proof}

\begin{remark}\label{char-for-topofwfsetsobdefofn11}
The proposition implies that the topology of $\EE'^r_{L;K}(M;E)$ (resp., $\EE'^r_{L;K}(M)$) is given by (any) norm on $H^r_K(M;E)$ (resp., $H^r_K(M)$) together with the seminorms \eqref{sem-for-bun-valdiswavsincsstrs1122} (resp., \eqref{sem-for-topofddmlewithbninc1111}) when $(O_{\mu},x_{\mu})$, $\mu\in\Lambda$, are as in the proposition. In particular, if $M$ is an open subset of $\RR^m$, the topology on $\EE'^r_{L;K}(M)$ is the same as in the Euclidean case (employ this with the global trivial chart).
\end{remark}

In view of Proposition \ref{pro-for-esp-closincinimbinprdsp}, we immediately deduce the following result.

\begin{corollary}
The space $\EE'^r_{L;K}(M;E)$ is a reflexive Fr\'echet space.
\end{corollary}

We define locally convex topologies on $\EE'^r_W(M)$ and $\EE'^r_W(M;E)$ in the same way as in the Euclidean case. Namely, first we notice that $\EE'^r_W(M)=\bigcup_{(L,K)\in\mathfrak{W}} \EE'^r_{L;K}(M)$ and $\EE'^r_W(M;E)=\bigcup_{(L,K)\in\mathfrak{W}} \EE'^r_{L;K}(M;E)$, where $\mathfrak{W}$ is the set of all pairs $(L,K)$ with $K\subset\subset M$ and $L$ a closed conic subset of $T^*M\backslash 0$ satisfying $L\subseteq W$ and $\pi_{T^*M}(L)\subseteq K$ and then we define locally convex topologies on $\EE'^r_W(M)$ and $\EE'^r_W(M;E)$ by
$$
\EE'^r_W(M)=\lim_{\substack{\longrightarrow\\ (L,K)\in\mathfrak{W}}} \EE'^r_{L;K}(M)\quad \mbox{and}\quad \EE'^r_W(M;E)=\lim_{\substack{\longrightarrow\\ (L,K)\in\mathfrak{W}}} \EE'^r_{L;K}(M;E);
$$
as before, $\mathfrak{W}$ is a directed set by inclusion and the linking mappings in the inductive limits are the canonical inclusions. As in the Euclidean case, one can find a sequence $(L_j,K_j)_{j\in\ZZ_+}\subseteq \mathfrak{W}$ which satisfies \eqref{inc-set-com-exchscoses} with $M$ in place of $U$; whence
$$
\EE'^r_W(M)=\lim_{\substack{\longrightarrow\\ j\rightarrow \infty}} \EE'^r_{L_j;K_j}(M)\quad \mbox{and}\quad \EE'^r_W(M;E)=\lim_{\substack{\longrightarrow\\ j\rightarrow \infty}} \EE'^r_{L_j;K_j}(M;E)\quad\mbox{topologically}
$$
(arguing by contradiction, it is straightforward to check that if $L\subseteq W$ is a closed conic subset of $T^*M\backslash0$ such that $\pi_{T^*M}(L)$ is compact, then there is $j\in\ZZ_+$ such that $L\subseteq\operatorname{int} L_j$). Consequently, $\EE'^r_W(M)$ and $\EE'^r_W(M;E)$ are $(LF)$-spaces and thus both barrelled and bornological. Of course, if $M$ is an open subset of $\RR^m$, then the topology on $\EE'^r_W(M)$ is the same as in the Euclidean case.

\begin{remark}\label{rem-for-cof-casdlk}
Arguing as in Remark \ref{rem-for-ide-spawithorddspinrsetck}, one shows the following topological identities:
\begin{gather*}
\EE'^r_{\emptyset;K}(M;E)=\Gamma_K(E)\quad\mbox{and}\quad \EE'^r_{\pi^{-1}_{T^*M}(K)\backslash0;K}(M;E)=H^r_K(M;E),\quad\mbox{for any}\,\, K\subset\subset M;\\
\EE'^r_{\emptyset}(M;E)=\Gamma_c(E),\quad \EE'^r_{T^*M\backslash0}(M;E)=H^r_{\comp}(M;E).
\end{gather*}
\end{remark}

\begin{remark}
Analogously as in Remark \ref{rem-for-con-ofmulonespaceofonesk} one shows that $\mathcal{C}^{\infty}(M)\times\EE'^r_{L;K}(M;E)\rightarrow\EE'^r_{L;K}(M;E)$, $(\chi,u)\mapsto \chi u$, is continuous and $\mathcal{C}^{\infty}(M)\times\EE'^r_W(M;E)\rightarrow\EE'^r_W(M;E)$, $(\chi, u)\mapsto \chi u$, is hypocontinuous.
\end{remark}

As in the Euclidean setting, the identity operator on $\EE'^r_W(M;E)$ can be approximated by regularising operators in the topology of precompact convergence.

\begin{proposition}\label{den-ofd-ine-forcponlfns}
Let $\pi_E:E\rightarrow M$ be a vector bundle of rank $k$.
\begin{itemize}
\item[$(i)$] Let $K$ and $\widetilde{K}$ be compact sets in $M$ satisfying $K\subseteq \operatorname{int} \widetilde{K}$ and let $L$ be a closed conic subset of $T^*M\backslash0$ such that $\pi_{T^*M}(L)\subseteq K$. There are continuous operators $P_n:\EE'^r_{L;K}(M;E)\rightarrow \Gamma_c(E_{\operatorname{int} \widetilde{K}})$, $n\in\ZZ_+$, such that $P_n\rightarrow \operatorname{Id}$ in $\mathcal{L}_p(\EE'^r_{L;K}(M;E),\EE'^r_{L;\widetilde{K}}(M;E))$.
\item[$(ii)$] Let $W$ be an open conic subset of $T^*M\backslash0$. There are continuous operators $P_n:\EE'^r_W(M;E)\rightarrow \Gamma_c(E)$, $n\in\ZZ_+$, such that $P_n\rightarrow \operatorname{Id}$ in $\mathcal{L}_p(\EE'^r_W(M;E))$. In particular, $\Gamma_c(E)$ is sequentially dense in $\EE'^r_W(M;E)$.
\end{itemize}
\end{proposition}

\begin{proof} To prove $(i)$, pick finite number of charts $(O_{\mu},x_{\mu})$, $\mu=1,\ldots,l$, which cover $K$ such that $O_{\mu}\subseteq \operatorname{int} \widetilde{K}$ and $E$ locally trivialises over $O_{\mu}$ via $\Phi_{x_{\mu}}:\pi_E^{-1}(O_{\mu})\rightarrow O_{\mu}\times\CC^k$, $\mu=1,\ldots,l$. For each $\mu$, let $(s_{\mu,1},\ldots,s_{\mu,k})$ be the frame of $E$ over $O_{\mu}$ induced by $\Phi_{x_{\mu}}$. Choose nonnegative $\varphi_{\mu}\in \DD(O_{\mu})$, $\mu=1,\ldots,l$, such that $\sum_{\mu=1}^l\varphi_{\mu}=1$ on a neighbourhood of $K$. Choose $\{\chi_n\}_{n\in\ZZ_+}\subseteq\DD(\RR^m)$ as in Proposition \ref{seq-den-comsmf} and pick $n_0\in\ZZ_+$ such that $\supp\chi_n +x_{\mu}(\supp\varphi_{\mu})\subseteq x_{\mu}(O_{\mu})$, for all $n\geq n_0$, $\mu=1,\ldots,l$. For $n\geq n_0$, we define
$$
P_n:\EE'^r_{L;K}(M;E)\rightarrow \Gamma_c(E_{\operatorname{int}\widetilde{K}}),\quad P_n(u):=\sum_{\mu=1}^l (\chi_n*((\varphi_{\mu}\circ x_{\mu}^{-1})u^j_{\Phi_{x_{\mu}}}))\circ x_{\mu} s_{\mu,j}.
$$
It is straightforward to check that $P_n$ is well-defined and continuous. In view of the Banach-Steinhaus theorem, to verify that $P_n\rightarrow \operatorname{Id}$ in $\mathcal{L}_p(\EE'^r_{L;K}(M;E),\EE'^r_{L;\widetilde{K}}(M;E))$, it suffices to show that for each $u\in\EE'^r_{L;K}(M;E)$, $P_n(u)\rightarrow u$ in $\EE'^r_{L;\widetilde{K}}(M;E)$. This can be done similarly as in the proof of Proposition \ref{res-for-den-ope-map-thagivdes} (cf. Proposition \ref{lem-for-den-ofdine'lddd}).\\
\indent We now address $(ii)$. Let $(O_{\mu},x_{\mu})$, $\mu\in\ZZ_+$, be relatively compact charts on $M$ such that $E$ locally trivialises over $O_{\mu}$ via $\Phi_{x_{\mu}}:\pi_E^{-1}(O_{\mu})\rightarrow O_{\mu}\times \CC^k$, $x_{\mu}(O_{\mu})=B(0,2)$ and $O'_{\mu}:=x_{\mu}^{-1}(B(0,1))$, $\mu\in\ZZ_+$, cover $M$. Take a partition of unity $(\varphi_{\mu})_{\mu\in\ZZ_+}$ subordinated to $(O'_{\mu})_{\mu\in\ZZ_+}$. Pick $(L_j,K_j)\in \mathfrak{W}$, $j\in\ZZ_+$, which satisfy \eqref{inc-set-com-exchscoses} with $M$ in place of $U$. For each $j\in\ZZ_+$, pick $\psi_j\in\DD(\operatorname{int}K_{j+1})$ such that $0\leq \psi_j\leq 1$ and $\psi_j=1$ on a neighbourhood of $K_j$. Let $\{\chi_n\}_{n\in\ZZ_+}\subseteq\DD(\RR^m)$ be as in Proposition \ref{seq-den-comsmf} and notice that $\supp\chi_n +x_{\mu}(\supp\varphi_{\mu})\subseteq x_{\mu}(O_{\mu})$, for all $n,\mu\in\ZZ_+$. We define
$$
P_n:\DD'(M;E)\rightarrow\Gamma_c(E),\, P_n(u):=\sum_{\mu\in\ZZ_+} \left(\chi_n*\left(((\psi_n\varphi_{\mu})\circ x_{\mu}^{-1})u^l_{\Phi_{x_{\mu}}}\right)\right)\circ x_{\mu} s_{\mu,l},\,\, n\in\ZZ_+.
$$
It is straightforward to check that $P_n$, $n\in\ZZ_+$, are well-defined and continuous. Let $u\in\EE'^r_W(M;E)$ be arbitrary but fixed. There is $j_0\in\ZZ_+$ such that $u\in\EE'^r_{L_{j_0};K_{j_0}}(M;E)$. Set $\Lambda_{j_0}:=\{\mu\in\ZZ_+\,|\, \supp\varphi_{\mu}\cap K_{j_0}\neq \emptyset\}$; of course, $\Lambda_{j_0}$ is finite. For $n\geq j_0$, it holds that
$$
P_n(u)=\sum_{\mu\in\Lambda_{j_0}} (\chi_n*((\varphi_{\mu}\circ x_{\mu}^{-1})u^l_{\Phi_{x_{\mu}}}))\circ x_{\mu} s_{\mu,l},
$$
and, in view of the proof of $(i)$, the right-hand side tends to $u$ in $\EE'^r_{L_{j_0};K_{j'_0}}(M;E)$ as $n\rightarrow\infty$ where $j'_0>j_0$ is large enough so that $\bigcup_{\mu\in\Lambda_{j_0}}O_{\mu}\subseteq \operatorname{int} K_{j'_0}$ (such $j'_0$ exists since the $O_{\mu}$'s are relatively compact). We deduce that $P_n\rightarrow \operatorname{Id}$ in the topology of simple convergence on $\mathcal{L}(\EE'^r_W(M;E))$ and, as $\EE'^r_W(M;E)$ is barrelled, the Banach-Steinhaus theorem \cite[Theorem 4.5, p. 85]{Sch} verifies that the convergence also holds in $\mathcal{L}_p(\EE'^r_W(M;E))$.
\end{proof}

\begin{remark}\label{den-con-inc-forespaceindh}
Proposition \ref{den-ofd-ine-forcponlfns} immediately shows that the following continuous inclusions are also dense:
\begin{gather}
\DD(M)\subseteq \EE'^r_W(M)\subseteq H^r_{\comp}(M),\quad \Gamma_c(E)\subseteq \EE'^r_W(M;E)\subseteq H^r_{\comp}(M;E);\label{con-inc-for-espainsobcomspdes}\\
\EE'^{r_2}_W(M)\subseteq \EE'^{r_1}_W(M)\quad \mbox{and}\quad \EE'^{r_2}_W(M;E)\subseteq \EE'^{r_1}_W(M;E),\quad \mbox{when}\,\, r_2\geq r_1.
\end{gather}
\end{remark}

\begin{theorem}\label{the-for-dua-ofdwitheonmanwithvectbundd}
Let $\pi_E:E\rightarrow M$ be a vector bundle of rank $k$ and let $L$ be a closed conic subset of $T^*M\backslash0$. Then
$$
(\DD'^r_L(M))'_b=\EE'^{-r}_{\check{L}^c}(M;DM)\quad \mbox{and}\quad (\DD'^r_L(M;E))'_b=\EE'^{-r}_{\check{L}^c}(M;E^{\vee})\quad \mbox{topologically}.
$$
\end{theorem}

\begin{proof} We only consider the bundle-valued case. In view of \eqref{inc-for-dif-nincfduals}, $(\DD'^r_L(M;E))'\subseteq H^{-r}_{\comp}(M;E^{\vee})$. First we show that $(\DD'^r_L(M;E))'_b$ is an $(LF)$-space. Proposition \ref{lem-for-top-imebdofmapscs} together with Theorem \ref{the-for-dua-fordwithespacewithwfs} imply that $(\DD'^r_L(M;E))'_b$ is topologically isomorphic to a complemented subspace of a countable locally convex direct sum of $(LF)$-spaces (the strong dual of a topological product is the locally convex direct sum of the strong duals; see \cite[p. 287]{kothe1}). Since a countable direct sum of $(LF)$-spaces is an $(LF)$-space and a quotient of an $(LF)$-space by a closed subspace is again an $(LF)$-space, we deduce that $(\DD'^r_L(M;E))'_b$ is an $(LF)$-space. Thus, in view of the open mapping theorem for $(LF)$-spaces \cite[Theorem 4, p. 43]{kothe2}, to prove the desired result it suffices to show that $\EE'^{-r}_{\check{L}^c}(M;E^{\vee})=(\DD'^r_L(M;E))'$ as sets and the inclusion $\EE'^{-r}_{\check{L}^c}(M;E^{\vee})\subseteq (\DD'^r_L(M;E))'_b$ is continuous.\\
\indent First we show that $\EE'^{-r}_{\check{L}^c}(M;E^{\vee})\subseteq (\DD'^r_L(M;E))'_b$ continuously. Let $\widetilde{K}\subset\subset M$ and let $\widetilde{L}$ be a closed conic subset of $T^*M\backslash0$ satisfying $\pi_{T^*M}(\widetilde{L})\subseteq \widetilde{K}$ and $\widetilde{L}\subseteq \check{L}^c$. Pick a finite number of relatively compact coordinate charts $(O_{\mu},x_{\mu})$, $\mu=1,\ldots,l$, which cover $\widetilde{K}$ and such that $E$ trivialises over $O_{\mu}$ via $\Phi_{x_{\mu}}:\pi_E^{-1}(O_{\mu})\rightarrow O_{\mu}\times \CC^k$, $\mu=1,\ldots,l$. Let $(s_{\mu,1},\ldots,s_{\mu,k})$ be the frame for $E$ over $O_{\mu}$ induced by $\Phi_{x_{\mu}}$ and let $(\sigma^1_{\mu},\ldots,\sigma^k_{\mu})$ be the induced frame for $E^{\vee}$ over $O_{\mu}$. For $v\in\DD'(M;E^{\vee})$, let $v_{\mu,j}\in\DD'(x_{\mu}(O_{\mu}))$ be the distribution $\langle v_{\mu,j},\phi\rangle:=\langle v,\phi\circ x_{\mu}s_{\mu,j}\rangle$, $\phi\in\DD(x_{\mu}(O_{\mu}))$. Pick nonnegative $\varphi_{\mu}\in\DD(O_{\mu})$, $\mu=1,\ldots,l$, such that $\sum_{\mu=1}^l\varphi_{\mu}=1$ on a neighbourhood of $\widetilde{K}$. Let $B$ be a bounded subset of $\EE'^{-r}_{\widetilde{L};\widetilde{K}}(M;E^{\vee})$. Then,
\begin{equation}\label{equ-for-dua-ofuineds}
\langle v,\psi\rangle=\sum_{\mu=1}^l\langle v,\varphi_{\mu}\psi^j_{\mu}s_{\mu,j}\rangle=\sum_{\mu=1}^l\langle (\varphi_{\mu}\circ x_{\mu}^{-1})v_{\mu,j},\psi^j_{\mu}\circ x_{\mu}^{-1}\rangle,\quad \psi\in\Gamma_c(E),\, v\in B.
\end{equation}
In view of Proposition \ref{lem-for-top-imebdofmapscs11}, $B_{\mu,j}:=\{(\varphi_{\mu}\circ x_{\mu}^{-1})v_{\mu,j}\,|\, v\in B\}$ is a bounded subset of $\EE'^{-r}_{\widetilde{L}_{\mu}; \widetilde{K}_{\mu}}(x_{\mu}(O_{\mu}))$, for every $\mu$ and $j$ with $\widetilde{K}_{\mu}$ and $\widetilde{L}_{\mu}$ as in Proposition \ref{lem-for-top-imebdofmapscs11}. Theorem \ref{the-for-dua-fordwithespacewithwfs} and Lemma \ref{lem-for-inc-ofeinddualsis} $(ii)$ show that $B_{\mu,j}$ is an equicontinuous subset of $(\DD'^r_{L_{\mu}}(x_{\mu}(O_{\mu})))'$ with $L_{\mu}$ as in Proposition \ref{lem-for-top-imebdofmapscs}. Since $\Gamma_c(E)$ is dense in $\DD'^r_L(M;E)$ and $\DD'^r_L(M;E)\rightarrow \DD'^r_{L_{\mu}}(x_{\mu}(O_{\mu}))$, $u\mapsto u_{\Phi_{x_{\mu}}}^j$, is continuous (cf. Proposition \ref{lem-for-top-imebdofmapscs}), \eqref{equ-for-dua-ofuineds} verifies that $B$ is an equicontinuous subset of $(\DD'^r_L(M;E))'$. Hence $\EE'^{-r}_{\widetilde{L};\widetilde{K}}(M;E^{\vee})\subseteq  (\DD'^r_L(M;E))'_b$ and the inclusion maps bounded sets into bounded sets; whence, it is continuous since $\EE'^{-r}_{\widetilde{L};\widetilde{K}}(M;E^{\vee})$ is Fr\'echet. We deduce that $\EE'^{-r}_{\check{L}^c}(M;E^{\vee})\subseteq (\DD'^r_L(M;E))'_b$ continuously.\\
\indent It remains to show that $(\DD'^r_L(M;E))'\subseteq \EE'^{-r}_{\check{L}^c}(M;E^{\vee})$. For $v\in (\DD'^r_L(M;E))'\backslash\{0\}$, \eqref{inc-for-dif-nincfduals} implies $v\in H^{-r}_{\comp}(M;E^{\vee})$. For the compact set $\supp v$, we choose $(O_{\mu},x_{\mu})$ and $\varphi_{\mu}$, $\mu=1,\ldots,l$, as above. We make the following\\
\\
\noindent \textbf{Claim.} Let $(O,x)$ be a chart on $M$ over which $E$ locally trivialises via $\Phi:\pi_E^{-1}(O)\rightarrow O\times \CC^k$ and let $(s_1,\ldots,s_k)$ be the induced frame for $E$ over $O$. Set $L_{O,x}:=\kappa(\pi_{T^*M}^{-1}(O)\cap L)$ where $\kappa$ is the total local trivialisation of $T^*M$ over $O$ (cf. \eqref{tot-loc-tri-cotbunovercoordpa}). For $v\in\DD'(M;E^{\vee})$ and $j\in\{1,\ldots,k\}$, let $v_{\Phi,j}$ be the distribution on $x(O)$ defined by $\langle v_{\Phi,j},\phi\rangle:=\langle v,\phi\circ x \,s_j\rangle$, $\phi\in\DD(x(O))$. Then for each $j\in\{1,\ldots,k\}$ and $\varphi\in\DD(O)$, the map $(\DD'^r_L(M;E))'_b\rightarrow (\DD'^r_{L_{O,x}}(x(O)))'_b$, $v\mapsto (\varphi\circ x^{-1})v_{\Phi,j}$, is well-defined and continuous.\\
\\
\indent We first show how the Claim implies the desired result. The Claim together with Theorem \ref{the-for-dua-fordwithespacewithwfs} immediately imply that $(\varphi_{\mu}\circ x_{\mu}^{-1})v_{\mu,j}\in \EE'^{-r}_{\check{L}_{\mu}^c}(x_{\mu}(O_{\mu}))$, $\mu=1,\ldots,l$, $j=1,\ldots,k$, with $L_{\mu}$ as in Proposition \ref{lem-for-top-imebdofmapscs}. Set
$$
\widetilde{K}:=\bigcup_{\mu,j} x_{\mu}^{-1}(\supp ((\varphi_{\mu}\circ x_{\mu}^{-1})v_{\mu,j}))\quad \mbox{and}\quad \widetilde{L}:=\bigcup_{\mu,j} \kappa_{\mu}^{-1}(WF((\varphi_{\mu}\circ x_{\mu}^{-1})v_{\mu,j})).
$$
Then $\widetilde{K}=\supp v$ and $\widetilde{L}$ is a closed conic subset of $T^*M\backslash 0$ satisfying $\widetilde{L}\subseteq \check{L}^c$ and $\pi_{T^*M}(\widetilde{L})\subseteq \widetilde{K}$. Furthermore, $(\varphi_{\mu}\circ x_{\mu}^{-1})v_{\mu,j}\in \EE'^{-r}_{\widetilde{L}_{\mu};\widetilde{K}_{\mu}}(x_{\mu}(O_{\mu}))$ with $\widetilde{K}_{\mu}$ and $\widetilde{L}_{\mu}$ as in Proposition \ref{lem-for-top-imebdofmapscs11}. Notice that \eqref{equ-for-dua-ofuineds} is valid for $v$ and the right-hand side is exactly the map \eqref{map-lef-inv-fortheedualofdcomosl} (with $E^{\vee}$ in place of $E$); whence, the proof of Proposition \ref{lem-for-top-imebdofmapscs11} implies $v\in \EE'^{-r}_{\widetilde{L};\widetilde{K}}(M;E^{\vee})$ and the proof is complete.\\
\indent It remains to show the Claim. In the same way as in the proof of Proposition \ref{lem-for-top-imebdofmapscs}, one shows that for each $j_0\in\{1,\ldots,k\}$, the continuous map
$$
\DD'(x(O))\rightarrow \DD'(M;E),\,\, u\mapsto \widetilde{u}\,\, \mbox{with}\,\, \langle \widetilde{u},\psi\rangle:=\langle u, (\varphi\psi_{j_0})\circ x^{-1}\rangle,\, \psi\in\Gamma_c(E^{\vee}),\, \psi_{|O}=\psi_j\sigma^j,
$$
restricts to a well-defined and continuous map $\DD'^r_{L_{O,x}}(x(O))\rightarrow \DD'^r_L(M;E)$ ($(\sigma^1,\ldots,\sigma^k)$ is the frame for $E^{\vee}$ over $O$ induced by $\Phi$). It is straightforward to verify that its transpose is the map in the Claim which implies the validity of the Claim.
\end{proof}

Later, we are going to need the following technical result. Its proof follows immediately by applying Theorem \ref{the-for-dua-ofdwitheonmanwithvectbundd}, Lemma \ref{lem-for-inc-ofeinddualsis} $(ii)$, Proposition \ref{lem-for-top-imebdofmapscs11} and Proposition \ref{lem-for-top-imebdofmapscs} and we omit it.

\begin{lemma}\label{lem-for-equ-subofeddofboonsukls}
Suppose that $\widetilde{K}\subset\subset M$ and let the closed conic subsets $L$ and $\widetilde{L}$ of $T^*M\backslash0$ satisfy $\pi_{T^*M}(\widetilde{L})\subseteq \widetilde{K}$ and $\widetilde{L}\subseteq L^c$. Then every bounded subset of $\EE'^r_{\widetilde{L};\widetilde{K}}(M;E)$ is equicontinuous with respect to the duality $\langle\EE'^r_{L^c}(M;E),\DD'^{-r}_{\check{L}}(M;E^{\vee})\rangle$.
\end{lemma}

\subsection{The topology of \texorpdfstring{$\DD'^r_L(M;E)$}{D'rL(M;E)} is compatible with \texorpdfstring{$\DD'_L(M;E)$}{D'L(M;E)}}

We now show that the H\"ormander space $\DD'_L(M;E):=\{u\in\DD'(M;E)\,|\, WF(u)\subseteq L\}$ is the projective limit of the spaces $\DD'^r_L(M;E)$, $r\in\RR$. The topology of $\DD'_L(M;E)$ is given by the continuous seminorms on $\DD'(M;E)$ together with all seminorms $\mathfrak{q}^{\Phi_x}_{\nu;\varphi,V}$ (see \eqref{sem-for-bun-valdiswavsincsstrs1122}) for all $\nu>0$, all charts $(O,x)$ over which $E$ locally trivialises via $\Phi_x:\pi^{-1}_E(O)\rightarrow O\times \CC^k$ and all $\varphi\in \DD(O)$ and closed cones $V\subseteq \RR^m$ satisfying \eqref{equ-for-emp-intse}; we refer to \cite{BD,D1} for its topological properties.

\begin{proposition}\label{hor-spa-for-fixdsmwavfrsw}
Let $L$ be a closed conic subset of $T^*M\backslash 0$. Then $\DD'_L(M;E)=\bigcap_{r\in\RR}\DD'^r_L(M;E)$ and
$$
\DD'_L(M;E)=\lim_{\substack{\longleftarrow\\ r\rightarrow \infty}}\DD'^r_L(M;E)\quad \mbox{topologically},
$$
where the linking mappings in the projective limit are the canonical inclusions \eqref{map-inc-for-dspasobw}.
\end{proposition}

\begin{proof} Denote the projective limit by $\mathfrak{P}$. Clearly, $\DD'_L(M;E)\subseteq\mathfrak{P}$ continuously. To show the opposite inclusion, it suffices to show that each seminorm $\mathfrak{q}^{\Phi_x}_{\nu;\varphi,V}(u)$, $u\in\mathfrak{P}$, is bounded by a continuous seminorms on some $\DD'^r_L(M;E)$. Employing a standard compactness argument, we find relatively compact open sets $O_1,\ldots,O_q$ which cover $\supp\varphi$ and satisfy $\overline{O_j}\subseteq O$, $j=1,\ldots,q$, and for each $O_j$ we find open cones $V_{j,h}:=\RR_+B(\eta^{(j,h)},\varepsilon_{j,h})$, $h=1,\ldots,\mu_j$, for some $\eta^{(j,h)}\in\mathbb{S}^{m-1}$ and $\varepsilon_{j,h}>0$, such that $V\backslash\{0\}\subseteq \bigcup_{h=1}^{\mu_j}V_{j,h}$ and
$$
\{(p,\xi_ldx^l|_p)\in T^*O_j\,|\, (\xi_1,\ldots,\xi_m)\in \overline{V_{j,h}}\}\cap L=\emptyset,\quad h=1,\ldots,\mu_j,\, j=1,\ldots,q.
$$
Pick nonnegative $\psi_j\in\DD(O_j)$, $j=1,\ldots,q$, such that $\sum_{j=1}^q\psi_j=1$ on a neighbourhood of $\supp\varphi$. Notice that $\mathfrak{q}^{\Phi_x}_{\nu;\varphi,V}(u)\leq \sum_{j=1}^q \max_{1\leq h\leq\mu_j}\mathfrak{q}^{\Phi_x}_{\nu;\varphi\psi_j,V_{j,h}}(u)$. The Sobolev imbedding theorem \cite[Theorem 4.12, p. 85]{adams} implies that there is $C>0$ such that
\begin{align*}
\mathfrak{q}^{\Phi_x}_{\nu;\varphi\psi_j,V_{j,h}}(u)&\leq C\max_{1\leq l\leq k}\max_{|\alpha|\leq 1+m/2}\left\|\partial^{\alpha}\left(\langle\cdot\rangle^{\nu}\mathcal{F}\left(((\varphi\psi_j)\circ x^{-1})u^l_{\Phi_x}\right)\right)\right\|_{L^2(V_{j,h})}\\
&\leq C'\max_{1\leq l\leq k}\max_{|\alpha|\leq 1+m/2}\|\langle\cdot\rangle^{\nu}\mathcal{F}(\phi_{j,\alpha} u^l_{\Phi_x}))\|_{L^2(V_{j,h})},
\end{align*}
where we denoted $\phi_{j,\alpha}(t)=t^{\alpha}(\varphi\psi_j)(x^{-1}(t))$, $t\in x(O_j)$, $\alpha\in\NN^m$.\footnote{The choice of $V_{j,h}$ was so that they satisfy the domain conditions for the Sobolev imbedding theorem.} Hence, $\mathfrak{q}^{\Phi_x}_{\nu;\varphi\psi_j,V_{j,h}}(u)\leq C'\max_{|\alpha|\leq 1+m/2}\mathfrak{p}^{\Phi_x}_{\nu;\phi_{j,\alpha}\circ x, V_{j,h}}(u)$ and the proof is complete.
\end{proof}

\begin{remark}\label{rem-for-wfi-ofsobofsmhks}
The proposition implies the well-known identity $WF(u)=\overline{\bigcup_{r\in\RR} WF^r(u)}\backslash 0$, $u\in\DD'(M;E)$. Indeed, the inclusion ``$\supseteq$'' is trivial and the opposite inclusion follows by applying Proposition \ref{hor-spa-for-fixdsmwavfrsw} with $L:=\overline{\bigcup_{r\in\RR} WF^r(u)}\backslash 0$.
\end{remark}

\subsection{Pullback by smooth maps on vector bundles}

We are now ready to show our main result on the pullback.

\begin{theorem}\label{mai-the-pul-forvecbundm}
Let $f:M\rightarrow N$ be a smooth map between the manifolds $M$ and $N$ with dimensions $m$ and $n$ respectively and let $L$ be a closed conic subset of $T^*N\backslash0$ which satisfies $L\cap \mathcal{N}_f=\emptyset$.
\begin{itemize}
\item[$(i)$] The pullback $f^*:\mathcal{C}^{\infty}(N)\rightarrow \mathcal{C}^{\infty}(M)$, $f^*(u)=u\circ f$, uniquely extends to a well-defined and continuous mapping $f^*:\DD'^{r_2}_L(N)\rightarrow \DD'^{r_1}_{f^*L}(M)$ when $r_2-r_1>n/2$ and $r_2>n/2$. If $f$ has constant rank $l\geq 1$, then this is valid when $r_2-r_1\geq (n-l)/2$ and $r_2>(n-l)/2$. When $f$ is a submersion, $f^*:\DD'^{r_2}_L(N)\rightarrow \DD'^{r_1}_{f^*L}(M)$ is well-defined and continuous even when $r_2\geq r_1$. Consequently, if $f$ is a diffeomorphism, then $f^*:\DD'^r_L(N)\rightarrow \DD'^r_{f^*L}(M)$ is a topological isomorphism for each $r\in\RR$.
\item[$(ii)$] Let $(E,\pi_E,N)$ be a vector bundle of rank $k$. The pullback $f^*:\Gamma(E)\rightarrow \Gamma(f^*E)$, $f^*(u)(p)=(p,u\circ f(p))$, $p\in M$, uniquely extends to a well-defined and continuous mapping $f^*:\DD'^{r_2}_L(N;E)\rightarrow \DD'^{r_1}_{f^*L}(M;f^*E)$ when $r_2-r_1>n/2$ and $r_2>n/2$. If $f$ has constant rank $l\geq 1$, then this is valid when $r_2-r_1\geq (n-l)/2$ and $r_2>(n-l)/2$. When $f$ is a submersion, $f^*:\DD'^{r_2}_L(N;E)\rightarrow \DD'^{r_1}_{f^*L}(M;f^*E)$ is well-defined and continuous even when $r_2\geq r_1$. Consequently, if $f$ is a diffeomorphism, then $f^*:\DD'^r_L(N;E)\rightarrow \DD'^r_{f^*L}(M;f^*E)$ is a topological isomorphism for each $r\in\RR$.
\end{itemize}
\end{theorem}

\begin{proof} We only show $(ii)$ as the proof of $(i)$ is similar. The uniqueness of the extension follows from Proposition \ref{res-for-den-ope-map-thagivdes}. To show the existence, pick a cover of charts $\{(O_{\mu},x_{\mu})\}_{\mu\in\Lambda}$ of $M$ and corresponding charts $(U_{\mu},y_{\mu})$, $\mu\in\Lambda$, in $N$ such that $f(O_{\mu})\subseteq U_{\mu}$, $\mu\in\Lambda$, and $E$ locally trivialises over $U_{\mu}$ via $\Phi_{\mu}:\pi_E^{-1}(U_{\mu})\rightarrow U_{\mu}\times \CC^k$, $\mu\in \Lambda$ (of course $U_{\mu}$, $\mu\in\Lambda$, only cover $f(M)$). Denote, $\hat{f}_{\mu}:=y_{\mu}\circ f_{|O_{\mu}}\circ x_{\mu}^{-1}: x_{\mu}(O_{\mu})\rightarrow y_{\mu}(U_{\mu})$, $\mu\in\Lambda$; if $f$ has constant rank $l$ then the same holds for $\hat{f}_{\mu}$, $\mu\in\Lambda$, as well. For each $\mu\in\Lambda$, denote by $\kappa_{\mu}$ and $\iota_{\mu}$ the total local trivialisations of $T^*M$ and $T^*N$ over $O_{\mu}$ and $U_{\mu}$ induced by $x_{\mu}$ and $y_{\mu}$ respectively (cf. \eqref{tri-cot-bun-coordindms}). Set $(f^*L)_{\mu}:=\kappa_{\mu}(\pi_{T^*M}^{-1}(O_{\mu})\cap f^*L)$ and $L_{\mu}:=\iota_{\mu}(\pi_{T^*N}^{-1}(U_{\mu})\cap L)$ and notice that (cf. \eqref{equ-for-con-subwithdifandordmass})
\begin{equation}\label{equ-for-con-pulckfrfuckl}
\hat{f}_{\mu}^*L_{\mu}=(f^*L)_{\mu}\quad\mbox{and}\quad L_{\mu}\cap \mathcal{N}_{\hat{f}_{\mu}}=\emptyset, \quad \mu\in\Lambda.
\end{equation}
Let $(\widetilde{U}_{\nu},\widetilde{y}_{\nu})$, $\nu\in\Theta$, be charts on $N$ over which $E$ locally trivialises and which cover $N\backslash \bigcup_{\mu\in\Lambda} U_{\mu}$ ($\Theta=\emptyset$ if $N= \bigcup_{\mu\in\Lambda} U_{\mu}$). Let\footnote{Here and throughout the rest of the article we employ the principle of nullary Cartesian product $A\times \prod_{\alpha\in\emptyset}B_{\alpha}=A$.} $\mathcal{I}_N:\DD'^{r_2}_L(N;E)\rightarrow \prod_{\mu\in\Lambda}\DD'^{r_2}_{L_{\mu}}(y_{\mu}(U_{\mu}))^k\times \prod_{\nu\in\Theta}\DD'^{r_2}_{L_{\nu}}(\widetilde{y}_{\nu}(\widetilde{U}_{\nu}))^k$ be the map \eqref{map-imb-lcf-incprodick} for $N$ and the cover $\{(U_{\mu},y_{\mu})\}_{\mu\in\Lambda}\cup\{(\widetilde{U}_{\nu},\widetilde{y}_{\nu})\}_{\nu\in\Theta}$. Consider the map
\begin{gather*}
\mathbf{F}:\prod_{\mu\in\Lambda}\DD'^{r_2}_{L_{\mu}}(y_{\mu}(U_{\mu}))^k\times \prod_{\nu\in\Theta}\DD'^{r_2}_{L_{\nu}}(\widetilde{y}_{\nu}(\widetilde{U}_{\nu}))^k\rightarrow \prod_{\mu\in\Lambda}\DD'^{r_1}_{(f^*L)_{\mu}}(x_{\mu}(O_{\mu}))^k,\\
\mathbf{F}(\mathbf{f})(\mu)^j=\hat{f}^*_{\mu}(\mathbf{f}(\mu)^j),\quad \mu\in\Lambda,\, j\in\{1,\ldots,k\};
\end{gather*}
in view of Theorem \ref{the-pul-bac-for-smcrmdiff} and \eqref{equ-for-con-pulckfrfuckl}, the map is well-defined and continuous. Pick a partition of unity $(\varphi_{\mu})_{\mu\in\Lambda}$ subordinated to $(O_{\mu})_{\mu\in\Lambda}$ and denote by $\mathcal{R}_M$ the map \eqref{res-map-t-c-for-newresinals} with $f^*L$, $f^*E$ and $r_1$ in place of $L$, $E$ and $r$ respectively. We infer that $\mathcal{R}_M\circ\mathbf{F}\circ\mathcal{I}_N:\DD'^{r_2}_L(N;E)\rightarrow\DD'^{r_1}_{f^*L}(M;f^*E)$ is well-defined and continuous. It is straightforward to verify that $\mathcal{R}_M\circ\mathbf{F}\circ\mathcal{I}_N(u)(p)=(p,u\circ f(p))$, $p\in M$, for $u\in\Gamma(E)$; whence, the desired extension is $f^*=\mathcal{R}_M\circ\mathbf{F}\circ\mathcal{I}_N$.
\end{proof}

\begin{remark}
Assume the same as in Theorem \ref{mai-the-pul-forvecbundm} $(ii)$. Let $(O,x)$ and $(U,y)$ be charts on $M$ and $N$ such that $E$ trivialises over $U$ via $\Phi:\pi_E^{-1}(U)\rightarrow U\times \CC^k$ and $f(O)\subseteq U$ and let $\kappa$ and $\iota$ be the total local trivialisations of $T^*M$ and $T^*N$ over $O$ and $U$ induced by $x$ and $y$ respectively. Denoting $L_{U,y}:=\iota(\pi_{T^*N}^{-1}(U)\cap L)$, $(f^*L)_{O,x}:=\kappa(\pi_{T^*M}^{-1}(O)\cap f^*L)$ and $\hat{f}_{x,y}:=y\circ f_{|O} \circ x^{-1}:x(O)\rightarrow y(U)$, it holds that $\hat{f}_{x,y}^*L_{U,y}=(f^*L)_{O,x}$ and $L_{U,y}\cap \mathcal{N}_{\hat{f}_{x,y}}=\emptyset$. Consequently, Theorem \ref{the-pul-bac-for-smcrmdiff} implies that $\hat{f}_{x,y}:\DD'^{r_2}_{L_{U,y}}(y(U))\rightarrow \DD'^{r_1}_{(f^*L)_{O,x}}(x(O))$ is well defined and continuous. Furthermore
\begin{equation}\label{ide-for-cha-chasnf}
(f^*u)_{(f^*\Phi)_x}^j=\hat{f}^*_{x,y}(u_{\Phi_y}^j),\quad u\in \DD'^{r_2}_L(N;E),\,j\in\{1,\ldots,k\}.
\end{equation}
It is straightforward to verify \eqref{ide-for-cha-chasnf} for $u\in\Gamma(E)$ and the general case follows by density.\\
\indent Under the assumptions of Theorem \ref{mai-the-pul-forvecbundm} $(i)$, the analogous statement is that $(f^*u)_x=\hat{f}_{x,y}^*(u_y)$, $u\in \DD'^{r_2}_L(N)$.
\end{remark}

\begin{remark}
If $f:M\rightarrow N$ has constant rank $0$, then we can argue as in Remark \ref{rem-for-zer-rankmappulbc} to show that there are at most countably many pairwise disjoint open sets $O_j\subseteq M$, $j\in\Lambda$, whose union is $M$ and distinct points $q_j\in N$, $j\in \Lambda$, such that $f(p)=q_j$, $p\in O_j$, $j\in\Lambda$. Notice that $\mathcal{N}_f=\bigcup_{j\in \Lambda}(\{q_j\}\times T^*_{q_j}N)$. If $L\cap \mathcal{N}_f=\emptyset$ and $r>n/2$, we can reason similarly as in Remark \ref{rem-for-zer-rankmappulbc} to show that $f^*$ uniquely extends to a continuous mapping $f^*:\DD'^r_L(N;E)\rightarrow \Gamma(f^*E)$. Furthermore, each $u\in \DD'^r_L(N;E)$ is continuous on a neighbourhood around $q_j$ for every $j\in \Lambda$ and $f^*u(p)= (p, u(q_j))$, $p\in O_j$, $j\in\Lambda$, $u\in \DD'^r_L(N;E)$ (cf. Proposition \ref{res-for-den-ope-map-thagivdes} and Remark \ref{con-ofm-for-confasec}).
\end{remark}

As a direct consequence of Proposition \ref{res-for-den-ope-map-thagivdes} and Remark \ref{con-ofm-for-confasec}, we have the following useful result.

\begin{lemma}\label{lem-for-con-funcpulbmsmn}
Let $f:M\rightarrow N$ be a smooth map between the manifolds $M$ and $N$ with dimensions $m$ and $n$ respectively, let $L$
be a closed conic subset of $T^*N\backslash0$ and let $(E,\pi_E,N)$ be a vector bundle. If $f^*:\Gamma(E)\rightarrow \Gamma(f^*E)$ extends to a well-defined and continuous map $f^*:\DD'^r_L(N;E)\rightarrow \DD'(M;f^*E)$ for some $r\in\RR$, then for any $u\in\Gamma^0(E)\cap \DD'^r_L(N;E)$ it holds that $f^*u\in \Gamma^0(f^*E)$ and $f^*u(p)=(p,u\circ f(p))$, $p\in M$.
\end{lemma}

\section{Action of \texorpdfstring{$\Psi$DOs}{PsiDOs} on \texorpdfstring{$\DD'^r_L(M;E)$}{D'rL(M;E)} and its dual}\label{sec-psido-on-despacesanddual}

In this section, we study the continuity properties of pseudo-differential operators when acting on the spaces $\DD'^r_L(M;E)$, $r\in\RR$.

\begin{proposition}\label{pro-for-psudomapforcontonvectbac}
Let $r,r',r_0\in\RR$, $r'\geq r_0$. Let $L$ be a closed conic subset of $T^*M\backslash0$ and $W$ an open conic subset of $T^*M\backslash0$. Let $E$ and $F$ be two vector bundles over $M$. Let $A\in\Psi^{r'}(M;E,F)$ be properly supported, of order $r_0$ in $L^c$ and of order $-\infty$ in $W$. Then $A:\DD'^r_L(M;E)\rightarrow \DD'^{r-r_0}_{L\cap W^c}(M;F)$ is well-defined and continuous. Furthermore, $A$ maps bounded subsets of $\DD'^r_L(M;E)$ into relatively compact subsets of $\DD'^{\widetilde{r}}_{L\cap W^c}(M;F)$ for every $\widetilde{r}<r-r_0$.
\end{proposition}

\begin{proof} The case when $L=W^c=T^*M\backslash0$ is trivial. Assume this is not the case, i.e. $L\cap W^c$ is not the whole $T^*M\backslash0$. Let $E$ and $F$ have ranks $k'$ and $k$ respectively. We first prove the continuity of $A$. Since $\Gamma_c(E)$ is dense in $\DD'^r_L(M;E)$, it suffices to show that $A:\Gamma_c(E)\rightarrow \Gamma(F)$ is continuous when $\Gamma_c(E)$ and $\Gamma(F)$ are equipped with the topologies induced by $\DD'^r_L(M;E)$ and $\DD'^{r-r_0}_{L\cap W^c}(M;F)$ respectively. As $A:\DD'(M;E)\rightarrow \DD'(M;F)$ is continuous, to prove the latter, it suffices to show that every seminorm \eqref{sem-for-bun-valdiswavsincsstrs} of $Au\in \DD'^{r-r_0}_{L\cap W^c}(N;F)$, $u\in\Gamma_c(E)$, is bounded by a sum of continuous seminorms on $\DD'^r_L(M;E)$ of $u$. Let $(O,x)$ be a chart over which $F$ trivialises via $\Phi_x:\pi_F^{-1}(O)\rightarrow O\times\CC^k$ and let $\varphi\in\DD(O)\backslash\{0\}$ and the closed cone $V\subseteq\RR^m$, $V\backslash\{0\}\neq \emptyset$, be such that $\widetilde{L}:=\{(p,\xi_l dx^l|_p)\in T^*O\,|\, p\in\supp\varphi,\, (\xi_1,\ldots,\xi_m)\in V\}$ has empty intersection with $L\cap W^c$; in view of Remark \ref{char-for-topofwfsetsobdefofn}, we can assume that $E$ also trivialises over $O$ via some $\widetilde{\Phi}_x:\pi_E^{-1}(O)\rightarrow O\times\CC^{k'}$. We want to estimate $\mathfrak{p}^{\Phi_x}_{r-r_0;\varphi,V}(Au)$, for $u\in\Gamma_c(E)$. Notice that $\widetilde{L}\cap L\subseteq W$. Hence, we can employ a standard compactness argument to find relatively compact open sets $O_1,\ldots,O_n$ which cover $\supp\varphi$ and satisfy $\overline{O_q}\subseteq O$, $q=1,\ldots,n$, and for each $O_q$ we find open cones $V_{q,h},V'_{q,h},V''_{q,h}\subseteq \RR^m$, $h=1,\ldots\mu_q$, such that $\overline{V_{q,h}}\subseteq V'_{q,h}\cup\{0\}$, $\overline{V'_{q,h}}\subseteq V''_{q,h}\cup\{0\}$, $\RR^m\backslash\overline{V''_{q,h}}\neq\emptyset$, $V\backslash\{0\}\subseteq \bigcup_{h=1}^{\mu_q}V_{q,h}$ and at least one of the following holds
\begin{itemize}
\item[$(*)$] $\{(p,\xi_l dx^l|_p)\in T^*O\,|\, p\in \overline{O_q},\,(\xi_1,\ldots,\xi_m)\in \overline{V''_{q,h}}\backslash\{0\}\}\subseteq W$;
\item[$(**)$] $\{(p,\xi_l dx^l|_p)\in T^*O\,|\, p\in\overline{O_q},\,(\xi_1,\ldots,\xi_m)\in \overline{V''_{q,h}}\backslash\{0\}\}\subseteq L^c$.
\end{itemize}
Let $J_{q;W}$ be the set of all $h\in\{1,\ldots,\mu_q\}$ for which $(*)$ holds true, while $J_{q;L^c}$ be the set of all $h\in\{1,\ldots,\mu_q\}$ for which $(**)$ holds true; $J_{q;W}$ and $J_{q;L^c}$ may not be disjoint, also, for each $q\in\{1,\ldots,n\}$, one of these may be empty but can not be both. There are $\{\widetilde{a}^l_j\}_{l,j}\subseteq S^{r'}_{\operatorname{loc}}(x(O)\times \RR^m)$ and smoothing operators $\widetilde{T}^l_j:\EE'(x(O))\rightarrow \mathcal{C}^{\infty}(x(O))$, $j=1,\ldots,k'$, $l=1,\ldots,k$, such that
$$
(A\chi)_{|O}=\Op(\widetilde{a}^l_j)(\chi^j\circ x^{-1})\circ x\, s_l+\widetilde{T}^l_j(\chi^j\circ x^{-1})\circ x\, s_l,\quad \chi=\chi^je_j\in\Gamma_c(E_O),
$$
where $(e_1,\ldots,e_{k'})$ and $(s_1,\ldots,s_k)$ are the local frames over $O$ for $E$ and $F$ induced by $\widetilde{\Phi}_x$ and $\Phi_x$ respectively. Arguing by compactness, we see that $\{\widetilde{a}^l_j\}_{l,j}\subseteq S^{r_0}_{\loc}(x(O_q)\times V''_{q,h})$, $q=1,\ldots,n$, $h\in J_{q;L^c}$, and $\{\widetilde{a}^l_j\}_{l,j}\subseteq S^{-\infty}_{\loc}(x(O_q)\times V''_{q,h})$, $q=1,\ldots,n$, $h\in J_{q;W}$. Pick nonnegative $\psi_q,\chi_q\in\DD(O_q)$, $q=1,\ldots,n$, such that $\sum_{q=1}^n\psi_q=1$ on a neighbourhood of $\supp\varphi$ and $\chi_q=1$ on a neighbourhood of $\psi_q$, $q=1,\ldots,n$. Let $u\in\Gamma_c(E)$; $u_{|O}=u^je_j$. Denote $u^j_x:= u^j\circ x^{-1}\in \mathcal{C}^{\infty}(x(O))$, $\varphi_x:=\varphi\circ x^{-1}\in \DD(x(O))$, $\psi_{q,x}:=\psi_q\circ x^{-1}\in \DD(x(O_q))$, $\chi_{q,x}:=\chi_q\circ x^{-1}\in\DD(x(O_q))$. Notice that
$$
\mathfrak{p}^{\Phi_x}_{r-r_0;\varphi,V}(Au)\leq\max_{1\leq l\leq k}\tilde{I}^l(u)+\sum_{q=1}^n\max_{1\leq l\leq k}\tilde{I}^l_q(u)+\sum_{q=1}^n\sum_{h=1}^{\mu_q}(\max_{1\leq l\leq k}\tilde{I}^l_{q,h}(u)+\max_{1\leq l\leq k}I^l_{q,h}(u)),
$$
with
\begin{align}
\tilde{I}^l(u)&:=\left(\int_{\xi\in V,\, |\xi|<1} |\mathcal{F}(\varphi_x (Au)^l_{\Phi_x})(\xi)|^2\langle \xi\rangle^{2(r-r_0)}d\xi\right)^{1/2},\label{int-forboun2}\\
\tilde{I}^l_q(u)&:=\left(\int_{\xi\in V,\, |\xi|\geq 1} \left|\mathcal{F}\left(\varphi_x\psi_{q,x} \left(A((1-\chi_q)u)\right)^l_{\Phi_x}\right)(\xi)\right|^2\langle \xi\rangle^{2(r-r_0)}d\xi\right)^{1/2},\label{int-forboun3}\\
\tilde{I}^l_{q,h}(u)&:=\left(\int_{\xi\in V_{q,h},\, |\xi|\geq 1} |\mathcal{F}(\varphi_x\psi_{q,x} \widetilde{T}^l_j(\chi_{q,x}u^j_x))(\xi)|^2\langle \xi\rangle^{2(r-r_0)}d\xi\right)^{1/2},\label{int-forboun4}\\
I^l_{q,h}(u)&:=\left(\int_{\xi\in V_{q,h},\, |\xi|\geq 1} |\mathcal{F}(\varphi_x\psi_{q,x} \Op(\widetilde{a}^l_j)(\chi_{q,x}u^j_x))(\xi)|^2\langle \xi\rangle^{2(r-r_0)}d\xi\right)^{1/2}.\label{int-forboun1}
\end{align}
Since $A:\DD'(M;E)\rightarrow \DD'(N;F)$ is continuous, it is straightforward to show that $\DD'(M;E)\rightarrow [0,\infty)$, $f\mapsto \widetilde{I}^l(f)$, is a seminorm on $\DD'(M;E)$ which is bounded on bounded subsets (cf. the proof of the Claim in Proposition \ref{pro-for-top-imbedingthforc}) and hence it is continuous since $\DD'(M;E)$ is bornological. Since the kernel of $A$ is smooth outside of the diagonal, the operator $v\mapsto \psi_q A((1-\chi_q)v)$ has a smooth kernel and thus \eqref{int-forboun3} is a continuous seminorm on $\DD'(M;E)$ of $u$. The operators $\widetilde{T}^l_j$ are smoothing and consequently \eqref{int-forboun4} is also a continuous seminorm on $\DD'(M;E)$ of $u$. To estimate $I^l_{q,h}(u)$, we proceed as follows. Pick $b_{q,h},b'_{q,h},b''_{q,h}\in \mathcal{C}^{\infty}(\RR^m)$ such that $\mathbf{1}_{\RR^m}\otimes b_{q,h}\in S^{r-r_0}(\RR^{2m})$, $\mathbf{1}_{\RR^m}\otimes b'_{q,h}\in S^{-r}(\RR^{2m})$ and $\mathbf{1}_{\RR^m}\otimes b''_{q,h}\in S^r(\RR^{2m})$ and they satisfy the following:
\begin{itemize}
\item[$(i)$] $0\leq b_{q,h}\leq |\cdot|^{r-r_0}$, $0\leq b'_{q,h}\leq |\cdot|^{-r}$ and $0\leq b''_{q,h}\leq |\cdot|^r$ on $\RR^m\backslash\{0\}$;
\item[$(ii)$] $\supp b_{q,h}\subseteq  V'_{q,h}\backslash \overline{B(0,1/2)}$ and $b_{q,h}(\xi)=|\xi|^{r-r_0}$ when $\xi\in \overline{V_{q,h}}\backslash B(0,1)$;
\item[$(iii)$] $\supp b'_{q,h}\subseteq V''_{q,h}\backslash \overline{B(0,1/2)}$, $\supp b''_{q,h}\subseteq V''_{q,h}\backslash \overline{B(0,1/2)}$ and both $b'_{q,h}(\xi)=|\xi|^{-r}$ and $b''_{q,h}(\xi)=|\xi|^r$ when $\xi\in \overline{V'_{q,h}}\backslash B(0,1)$
\end{itemize}
(e.g., take $\chi,\widetilde{\chi}\in\DD(\RR^m)$ such that $0\leq \chi,\widetilde{\chi}\leq 1$, $\chi=1$ on $\overline{V_{q,h}}\cap \mathbb{S}^{m-1}$ and $\supp\chi\subseteq V'_{q,h}$, $\widetilde{\chi}=1$ on $\overline{B(0,1/2)}$ and $\supp\widetilde{\chi}\subseteq B(0,1)$, and define $b_{q,h}(\xi):=(1-\widetilde{\chi}(\xi))\chi(\xi/|\xi|)|\xi|^{r-r_0}$; $b'_{q,h}$ and $b''_{q,h}$ can be constructed analogously). Notice that
\begin{align}
I^l_{q,h}(u)&\leq (2\pi)^{m/2}2^{|r-r_0|} \|b_{q,h}(D)\Op(\varphi_x\psi_{q,x}\widetilde{a}^l_j (1-b'_{q,h}b''_{q,h}))(\chi_{q,x} u^j_x)\|_{L^2(\RR^m)}\label{the-sectermincontmpasik}\\
&{}\quad+(2\pi)^{m/2}2^{|r-r_0|} \|b_{q,h}(D)\Op(\varphi_x\psi_{q,x} \widetilde{a}^l_jb'_{q,h})b''_{q,h}(D)(\chi_{q,x} u^j_x)\|_{L^2(\RR^m)}.\label{verylastpart-spli-fornesss}
\end{align}
By construction, $b_{q,h}(D)\Op(\varphi_x\psi_{q,x}\widetilde{a}^l_j (1-b'_{q,h}b''_{q,h}))$ is an operator with symbol in $S^{-\infty}(\RR^{2m})$ and hence it is a continuous mapping from $\EE'(\RR^m)$ into $L^2(\RR^m)$. Whence, the term in \eqref{the-sectermincontmpasik} is a continuous seminorm on $\DD'(M;E)$ of $u$. When $h\in J_{q;W}$, $b_{q,h}(D)\Op(\varphi_x\psi_{q,x} \widetilde{a}^l_jb'_{q,h})$ is an operator with symbol in $S^{-\infty}(\RR^{2m})$. Consequently, $b_{q,h}(D)\Op(\varphi_x\psi_{q,x} \widetilde{a}^l_jb'_{q,h})b''_{q,h}(D)$ is a continuous mapping from $\EE'(\RR^m)$ into $L^2(\RR^m)$ which implies that \eqref{verylastpart-spli-fornesss} is a continuous seminorm on $\DD'(M;E)$ of $u$. When $h\in J_{q;L^c}$, the operator $b_{q,h}(D)\Op(\varphi_x\psi_{q,x} \widetilde{a}^l_jb'_{q,h})$ has symbol in $S^0(\RR^{2m})$, hence it is continuous on $L^2(\RR^m)$. This implies
\begin{equation*}
\|b_{q,h}(D)\Op(\varphi_x\psi_{q,x} \widetilde{a}^l_jb'_{q,h})b''_{q,h}(D)(\chi_{q,x} u^j_x)\|_{L^2(\RR^m)} \leq C\mathfrak{p}^{\widetilde{\Phi}_x}_{r;\chi_q,\overline{V''_{q,h}}}(u).
\end{equation*}
Since $h\in J_{q;L^c}$, the right-hand side is a continuous seminorm on $\DD'^r_L(M;E)$ and the proof is complete.\\
\indent To verify the second part, in view of the above, we only need to show that the inclusion $\DD'^{r-r_0}_{L\cap W^c}(M;F)\rightarrow \DD'^{\widetilde{r}}_{L\cap W^c}(M;F)$ maps bounded into relatively compact sets. For a bounded subset $B$ of $\DD'^{r-r_0}_{L\cap W^c}(M;F)$ it is straightforward to show that $WF^{\widetilde{r}}_c(B)\subseteq L\cap W^c$. Corollary \ref{cor-for-relcom-sub-wafe-fronchar} implies that $B$ is relatively compact in $\DD'^{\widetilde{r}}_{L\cap W^c}(M;F)$ and the proof of the proposition is complete.
\end{proof}

\begin{remark}\label{rem-for-bou-compsetforrellem}
If we only know that the $\Psi$DO $A$ is of order $r_0$ in $L^c$, then we can apply the proposition with $W=\emptyset$ to deduce that $A:\DD'^r_L(M;E)\rightarrow \DD'^{r-r_0}_L(M;F)$ is well-defined and continuous and $A:\DD'^r_L(M;E)\rightarrow \DD'^{\widetilde{r}}_L(M;F)$ maps bounded into relatively compact sets when $\widetilde{r}<r-r_0$.
\end{remark}

\begin{remark}
The second part is a generalisation of the Rellich's lemma to the spaces $\DD'^r_L(M;E)$, $r\in\RR$. Indeed, taking $L=W=\emptyset$ and $E=F$, one infers that the inclusion mapping $H^{r_2}_{\operatorname{loc}}(M;E)\rightarrow H^{r_1}_{\operatorname{loc}}(M;E)$, $r_2>r_1$, maps bounded into relatively compact sets. Since the Banach space $H^{r_2}_K(M;E)$, with $K\subset\subset M$, is a closed subspace of $H^{r_2}_{\operatorname{loc}}(M;E)$, its unit ball is bounded in $H^{r_2}_{\operatorname{loc}}(M;E)$, and the proposition implies that it is relatively compact in $H^{r_1}_{\operatorname{loc}}(M;E)$; this is exactly the Rellich’s lemma.
\end{remark}

\begin{corollary}\label{cor-for-hwf-spadtoplincs}
Let $A\in \Psi^r(M;E,F)$ be properly supported and of order $-\infty$ in the open conic subset $W$ of $T^*M\backslash0$. Then $A:\DD'(M;E)\rightarrow \DD'_{W^c}(M;F)$ is well-defined and continuous. Furthermore, if $L$ is a closed conic subset of $T^*M\backslash 0$, then $A:\DD'_L(M;E)\rightarrow \DD'_{L\cap W^c}(M;F)$ is well-defined and continuous.
\end{corollary}

\begin{proof} Since $\DD'(M;E)=\DD'^{\widetilde{r}}_{T^*M\backslash0}(M;E)$, $\widetilde{r}\in\RR$, to show the first part we employ Proposition \ref{pro-for-psudomapforcontonvectbac} with $L=T^*M\backslash0$ and we deduce that $A:\DD'(M;E)\rightarrow \DD'^{\widetilde{r}-r}_{W^c}(M;F)$ is well-defined and continuous for all $\widetilde{r}\in\RR$. Now, the claim follows from Proposition \ref{hor-spa-for-fixdsmwavfrsw}. The proof of the second part is analogous and we omit it.
\end{proof}

\begin{corollary}\label{cor-for-ope-defondinecomsetfort}
Let $A\in\Psi^{r_0}(M;E,F)$ be of order $-\infty$ in the open conic subset $W$ of $T^*M\backslash0$. Assume that the kernel of $A$ has compact support in $M\times M$ and let $K\subseteq M$ be its projection on the first component. Then for every closed conic subset $L$ of $T^*M\backslash0$ satisfying $L\subseteq W$ and every $r\in\RR$, $A:\DD'^r_L(M;E)\rightarrow \EE'^{r-r_0}_{W^c\cap \pi_{T^*M}^{-1}(K);K}(M;F)$ is well-defined and continuous. Furthermore, $A$ maps bounded subsets of $\DD'^r_L(M;E)$ into relatively compact subsets of $\EE'^{\widetilde{r}}_{W^c\cap \pi_{T^*M}^{-1}(K);K}(M;F)$ for all $\widetilde{r}<r-r_0$.
\end{corollary}

\begin{proof} Denote $\widetilde{L}:= W^c\cap \pi_{T^*M}^{-1}(K)$. Proposition \ref{pro-for-psudomapforcontonvectbac} verifies that $A:\DD'^r_L(M;E)\rightarrow H^{r-r_0}_{\loc}(M;F)$ is well-defined and continuous. Since $\supp Au\subseteq K$ and the topology of $H^{r-r_0}_K(M;F)$ is the same as the one induced by $H^{r-r_0}_{\loc}(M;F)$, we deduce that $A:\DD'^r_L(M;E)\rightarrow H^{r-r_0}_K(M;F)$ is well-defined and continuous. This completes the proof of the continuity of $A$ when $W=\emptyset$ (cf. Remark \ref{rem-for-con-csifd} and Remark \ref{rem-for-cof-casdlk}). Assume that $W\neq\emptyset$. As $\Gamma_c(E)$ is dense in $\DD'^r_L(M;E)$, to show the continuity of $A$ it remains to estimate $\mathfrak{q}^{\Phi_x}_{\nu;\varphi,V}(u)$, $u\in\Gamma_c(E)$, where $\varphi\in\DD(O)\backslash\{0\}$ and the closed cone $V\subseteq \RR^m$, $V\backslash\{0\}\neq\emptyset$, satisfy \eqref{equ-for-emp-intse} with $\widetilde{L}$ in place of $L$ and $(O,x)$ is a chart on $M$ over which $E$ and $F$ locally trivialise (we denoted by $\Phi_x$ the local trivialisation of $F$). We employ a standard compactness argument to find relatively compact open sets $O_1,\ldots,O_n$ which cover $\supp\varphi$ and satisfy $\overline{O_q}\subseteq O$, $q=1,\ldots,n$, and for each $O_q$ we find open cones $V_{q,h}\subseteq \RR^m$, $h=1,\ldots\mu_q$, such that $V\backslash\{0\}\subseteq \bigcup_{h=1}^{\mu_q}V_{q,h}$ and at least one of the following holds
\begin{itemize}
\item[$(*)$] $\{(p,\xi_l dx^l|_p)\in T^*O_q\,|\, (\xi_1,\ldots,\xi_m)\in \overline{V_{q,h}}\backslash\{0\}\}\subseteq W$;
\item[$(**)$] $\{(p,\xi_l dx^l|_p)\in T^*O_q\,|\, (\xi_1,\ldots,\xi_m)\in \overline{V_{q,h}}\}\cap \pi_{T^*M}^{-1}(K)=\emptyset$.
\end{itemize}
Denote by $J_{q;W}$ the set of all $h\in\{1,\ldots,\mu_q\}$ for which $(*)$ holds. Pick nonnegative $\psi_q\in\DD(O_q)$, $q=1,\ldots,n$, such that $\sum_{q=1}^n\psi_q=1$ on a neighbourhood of $\supp\varphi$ and notice that
$$
\mathfrak{q}^{\Phi_x}_{\nu;\varphi,V}(Au)\leq\sum_{q=1}^n\sum_{h\in J_{q;W}}\mathfrak{q}^{\Phi_x}_{\nu;\varphi\psi_q,\overline{V_{q,h}}}(Au)
$$
since $\mathfrak{q}^{\Phi_x}_{\nu;\varphi\psi_q,\overline{V_{q,h}}}(Au)=0$ when $h\not\in J_{q,W}$. Corollary \ref{cor-for-hwf-spadtoplincs} yields that $A:\DD'(M;E)\rightarrow \DD'_{W^c}(M;F)$ is well-defined and continuous, and thus each term $\mathfrak{q}^{\Phi_x}_{\nu;\varphi\psi_q,\overline{V_{q,h}}}(Au)$, $h\in J_{q,W}$, is bounded by a continuous seminorm on $\DD'^r_L(M;E)$ of $u$ which completes the proof of the first part of the corollary.\\
\indent To show the second part, let $B$ be a bounded subset of $\DD'^r_L(M;E)$. Proposition \ref{pro-for-psudomapforcontonvectbac} shows that it is relatively compact in $\DD'^{\widetilde{r}+r_0}_L(M;E)$ when $\widetilde{r}<r-r_0$ (cf. Remark \ref{rem-for-bou-compsetforrellem}). Hence $A(B)$ is relatively compact in $\EE'^{\widetilde{r}}_{\widetilde{L};K}(M;F)$ in view of the first part of the corollary and the proof is complete.
\end{proof}

Proposition \ref{pro-for-psudomapforcontonvectbac} allows us to show the following improvement of \cite[Theorem 18.1.31, p. 90]{hor2}.

\begin{proposition}\label{proposition-for-gentopology-ellipoperr}
Let $r,r_0\in\RR$ and $L$ a closed conic subset of $T^*M\backslash0$. Let $E$ and $F$ be vector bundles of rank $k$ over $M$ and let $A\in\Psi^{r_0}(M;E,F)$ be properly supported and satisfying $\Char A\subseteq L$. For each $u\in\DD'(M;E)$, $u\in\DD'^r_L(M;E)$ is equivalent to $Au\in\DD'^{r-r_0}_L(M;F)$. Furthermore, the topology of $\DD'^r_L(M;E)$ is generated by the continuous seminorms on $\DD'(M;E)$ together with all seminorms $u\mapsto \mathfrak{p}^{\Phi_x}_{r-r_0;\varphi, V}(Au)$, with $\varphi\in\DD(O)$ and $(O,x)$ a chart on $M$ over which $F$ trivialises via $\Phi_x:\pi_F^{-1}(O)\rightarrow O\times \CC^k$ and $V$ a closed cone in $\RR^m$ which satisfy \eqref{equ-for-emp-intse}.
\end{proposition}

\begin{remark}
When $A$ is as in the proposition, the fact $u\in\DD'^r_L(M;E)\Longleftrightarrow Au\in\DD'^{r-r_0}_L(M;F)$ follows from \cite[Theorem 18.1.31, p. 90]{hor2}; the novelty in the proposition is that one can generate the topology of $\DD'^r_L(M;E)$ as described. This is, in fact, an a priori estimate for $A$: for every continuous seminorm $\widetilde{\mathfrak{p}}$ on $\DD'^r_L(M;E)$ there are $\varphi_{\mu}\in\DD(O_{\mu})$ and $V_{\mu}\subseteq \RR^m$, $\mu=1,\ldots,l$, as in the proposition, a constant $C>0$ and a bounded subset $B$ of $\Gamma_c(E^{\vee})$ such that
$$
\widetilde{\mathfrak{p}}(u)\leq C\sup_{\psi\in B}|\langle u,\psi\rangle|+C\sum_{\mu=1}^l\mathfrak{p}^{\Phi_{x_{\mu}}}_{r-r_0;\varphi_{\mu}, V_{\mu}}(Au),\quad u\in \DD'^r_L(M;E).
$$
When $L=\emptyset$ and $A$ is elliptic, this boils down to the widely known fact about a priori estimate for $A:H^r_{\loc}(M;E)\rightarrow H^{r-r_0}_{\loc}(M;F)$.
\end{remark}

\begin{proof}[Proof of Proposition \ref{proposition-for-gentopology-ellipoperr}] The claim is trivial when $L=T^*M\backslash0$. Assume that $L^c\neq\emptyset$. In view of Proposition \ref{pro-for-psudomapforcontonvectbac}, if $u\in\DD'^r_L(M;E)$ then $Au\in\DD'^{r-r_0}_L(M;F)$ and the seminorms $u\mapsto \mathfrak{p}^{\Phi_x}_{r-r_0;\varphi, V}(Au)$ are well defined and continuous on $\DD'^r_L(M;E)$. Let $u\in\DD'(M;E)$ be such that $Au\in\DD'^{r-r_0}_L(M;F)$. Let $(O,x)$ be a chart on $M$ over which $E$ trivialises via $\widetilde{\Phi}_x:\pi_E^{-1}(O)\rightarrow O\times \CC^k$ and let $\varphi\in\DD(O)\backslash\{0\}$ and the closed cone $V\subseteq \RR^m$, $V\backslash\{0\}\neq\emptyset$, satisfy \eqref{equ-for-emp-intse}. Our goal is to bound $\mathfrak{p}^{\widetilde{\Phi}_x}_{r;\varphi,V}(u)$ by seminorms of $u$ in $\DD'(M;E)$ and seminorms of $Au$ as in the proposition; in view of Remark \ref{char-for-topofwfsetsobdefofn}, we can assume that $F$ also trivialises over $(O,x)$ via $\Phi_x:\pi_F^{-1}(O)\rightarrow O\times \CC^k$. Lemma \ref{lemma-for-parmetrixalsmforsonlyindirc} together with a standard compactness argument imply that there are open sets $O'_1,\ldots, O'_n$ with compact closure in $O$ which cover $\supp\varphi$ and, for each $j\in\{1,\ldots, n\}$, open cones $V_{j,h},V'_{j,h}\subseteq \RR^m\backslash\{0\}$, $h=1,\ldots,\mu_j$, and properly supported $\Psi$DOs $A'_{j,h}\in\Psi^{-r_0}(M;F,E)$, $h=1,\ldots,\mu_j$, such that $\overline{V_{j,h}}\subseteq V'_{j,h}\cup\{0\}$, $V\backslash\{0\}\subseteq \bigcup_{h=1}^{\mu_j} V_{j,h}$, $W'_{j,h}:=\{(p,\xi_ldx^l|_p)\in T^*O'_j\,|\, (\xi_1,\ldots, \xi_m)\in V'_{j,h}\}$ does not intersect $L$ and $R'_{j,h}:=A'_{j,h}A-\operatorname{Id}\in\Psi^0(M;E,E)$ is of order $-\infty$ in $W'_{j,h}$, $h=1,\ldots, \mu_j$. Pick $\varphi_j\in\DD(O'_j)$, $j=1,\ldots,n$, such that $0\leq \varphi_j\leq 1$ and $\sum_{j=1}^n\varphi_j=1$ on a neighbourhood of $\supp\varphi$. Notice that
\begin{equation}\label{bound-onseminorm-ofuinlemforcopsopsemnor}
\mathfrak{p}^{\widetilde{\Phi}_x}_{r;\varphi,V}(u)\leq \sum_{j=1}^n \sum_{h=1}^{\mu_j}\mathfrak{p}^{\widetilde{\Phi}_x}_{r;\varphi\varphi_j, \overline{V_{j,h}}}(R'_{j,h}u)+\sum_{j=1}^n \sum_{h=1}^{\mu_j}\mathfrak{p}^{\widetilde{\Phi}_x}_{r;\varphi\varphi_j, \overline{V_{j,h}}}(A'_{j,h}Au).
\end{equation}
Proposition \ref{pro-for-psudomapforcontonvectbac} yields that $R'_{j,h}:\DD'(M;E)\rightarrow \DD'^r_{W'^{c}_{j,h}}(M;E)$ is well-defined and continuous (apply it with $L=T^*M\backslash0$ and $W=W'_{j,h}$; cf. Remark \ref{rem-for-con-csifd}) and thus $\DD'(M;E)\rightarrow [0,\infty)$, $u\mapsto \mathfrak{p}^{\widetilde{\Phi}_x}_{r;\varphi\varphi_j, \overline{V_{j,h}}}(R'_{j,h}u)$, is a continuous seminorm on $\DD'(M;E)$. Proposition \ref{pro-for-psudomapforcontonvectbac} also shows that $\DD'^{r-r_0}_L(M;F)\rightarrow [0,\infty)$, $v\mapsto \mathfrak{p}^{\widetilde{\Phi}_x}_{r;\varphi\varphi_j, \overline{V_{j,h}}}(A'_{j,h}v)$, is a continuous seminorm on $\DD'^{r-r_0}_L(M;F)$. Hence $\mathfrak{p}^{\widetilde{\Phi}_x}_{r;\varphi\varphi_j,\overline{V_{j,h}}}(A'_{j,h}Au)$ is bounded from above by a finite sum of seminorms of $Au$ as in the proposition together with a continuous seminorm on $\DD'(M;E)$ of $u$. This shows that $u\in\DD'^r_L(M;E)$ and that the seminorms in the proposition generate the topology of $\DD'^r_L(M;E)$.
\end{proof}

As a consequence, we show that the Sobolev compactness wave front set satisfies analogous bounds as the Sobolev wave front set \cite[Theorem 18.1.31, p. 90]{hor2}.

\begin{corollary}\label{cor-for-ide-ofwafornes}
Let $r,r_0\in\RR$. Let $E$ and $F$ be vector bundles over $M$ of rank $k$, let $B$ be a bounded subset of $\DD'(M;E)$ and let $A\in\Psi^{r_0}(M;E,F)$ be properly supported. Then
\begin{equation}\label{wafe-frontincl-forcomp}
WF^{r-r_0}_c(A(B))\subseteq WF^r_c(B)\subseteq WF^{r-r_0}_c(A(B))\cup\Char A.
\end{equation}
\end{corollary}

\begin{proof} The first inclusion in \eqref{wafe-frontincl-forcomp} follows from Corollary \ref{cor-for-relcom-sub-wafe-fronchar} and Proposition \ref{pro-for-psudomapforcontonvectbac} by taking $L=WF^r_c(B)$. To show the second inclusion, set $L:=WF^{r-r_0}_c(A(B))\cup\Char A$. Corollary \ref{cor-for-relcom-sub-wafe-fronchar} verifies that $A(B)$ is a relatively compact subset of $\DD'^{r-r_0}_L(M;F)$ and Proposition \ref{proposition-for-gentopology-ellipoperr} yields that $B\subseteq \DD'^r_L(M;E)$. In view of Corollary \ref{cor-for-relcom-sub-wafe-fronchar}, it suffices to show that $B$ is relatively compact, i.e. totally bounded, in $\DD'^r_L(M;E)$. Let $U$ be a neighbourhood of zero in $\DD'^r_L(M;E)$; in view of Proposition \ref{proposition-for-gentopology-ellipoperr}, without loss in generality, we can assume that
$$
U=\{u\in\DD'^r_L(M;E)\,|\, \mathfrak{p}^{\Phi_1}_{r-r_0;\varphi_1,V_1}(Au)<\varepsilon,\ldots,\mathfrak{p}^{\Phi_n}_{r-r_0;\varphi_n,V_n}(Au)<\varepsilon, \mathfrak{p}(u)<\varepsilon\}
$$
for some $\varepsilon>0$, where $\mathfrak{p}^{\Phi_j}_{r-r_0;\varphi_j,V_j}$, $j=1,\ldots,n$, are seminorms on $\DD'^{r-r_0}_L(M;F)$ of the form \eqref{sem-for-bun-valdiswavsincsstrs} and $\mathfrak{p}$ is a continuous seminorm on $\DD'(M;E)$. Since $A(B)$ is a relatively compact subset of $\DD'^{r-r_0}_L(M;F)$, there are $u_1,\ldots,u_q\in B$ such that for every $u\in B$ there is $u_j$ so that $\max_{1\leq l\leq n}\mathfrak{p}^{\Phi_l}_{r-r_0;\varphi_l,V_l}(Au_j-Au)<\varepsilon/2$. Since $(u_j+U_0)\cap B$, $j=1,\ldots,q$, with $U_0:=\{u\in\DD'^r_L(M;E)\,|\, \max_{1\leq l\leq n}\mathfrak{p}^{\Phi_l}_{r-r_0;\varphi_l,V_l}(Au)<\varepsilon/2\}$, is a relatively compact subset of $\DD'(M;E)$ (as $\DD'(M;E)$ is Montel), there is $B_j:=\{u_j+u_{j,1},\ldots,u_j+u_{j,t_j}\}\subseteq (u_j+U_0)\cap B$ so that for each $u\in (u_j+U_0)\cap B$ there is $u_j+u_{j,l}\in B_j$ such that $\mathfrak{p}(u-u_j-u_{j,l})<\varepsilon$. Set $B_0:=\bigcup_{j=1}^q B_j$. It is straightforward to show that $B\subseteq B_0+U$ which completes the proof of the corollary.
\end{proof}

We end the section with the following consequence of the H\"ormander's construction of a distribution with prescribed wave front set \cite[Theorem 8.1.4, p. 255]{hor}.

\begin{lemma}
Let $O$ be an open set in $\RR^m$. For every $r\in\RR$ and every closed conic subset $L$ of $O\times(\RR^m\backslash\{0\})$, there exists $u\in\DD'(O)$ such that $WF^r(u)=WF(u)=L$.
\end{lemma}

\begin{proof} It suffices to show the claim when $O=\RR^m$ and $L\subseteq \RR^m\times(\RR^m\backslash\{0\})$ for otherwise we can apply this case to the closure of $L$ in $\RR^m\times(\RR^m\backslash\{0\})$. For such $L$, let $u$ be the distribution constructed in the proof of \cite[Theorem 8.1.4, p. 255]{hor} which satisfies $WF(u)=L$. We claim that $WF^{2m^2}(u)=L$. Clearly, is suffices to show that $L\subseteq WF^{2m^2}(u)$ (cf. Remark \ref{rem-for-wfi-ofsobofsmhks}). By carefully examining the proof of \cite[Theorem 8.1.4, p. 255]{hor}, one sees that for every $\psi\in\DD(\RR^m)$ the smooth function $\mathcal{F}(\psi u)$ satisfies $\mathcal{F}(\psi u)\in L^{\infty}(\RR^m)$ and for every $(x_0,\xi_0)\in L$ and $\chi\in\DD(\RR^m)$ which equals $1$ on a neighbourhood of $x_0$, there is a sequence $(x_j,\theta_j)\in L$, $j\in\ZZ_+$, with $\theta_j\in\mathbb{S}^{m-1}$, which converges to $(x_0,\xi_0/|\xi_0|)$ and
$$
|\mathcal{F}(\chi u)(k_j^3\theta_j)|\geq k_j^{-m-2}/2,\,\, j\in\ZZ_+,\quad \mbox{where}\quad k_j\in\ZZ_+\,\,\mbox{and}\,\, k_{j+1}>k_j.
$$
Notice that the first property implies $\partial_l\mathcal{F}(\psi u)\in L^{\infty}(\RR^m)$, $l=1,\ldots,m$, for all $\psi\in\DD(\RR^m)$. Let $(x_0,\xi_0)\in L$ and assume that $(x_0,\xi_0)\not\in WF^{2m^2}(u)$. There is $\chi\in\DD(\RR^m)$ which equals $1$ on a neighbourhood of $x_0$ and an open cone $V\subseteq \RR^m$ containing $\xi_0$ such that $\|\langle \cdot\rangle^{2m^2}\mathcal{F}(\chi u)\|_{L^2(V)}<\infty$. Let $(x_j,\theta_j)\in L$, $j\in\ZZ_+$, be as above. There is $j_1\in\ZZ_+$ such that $B(k_j^3\theta_j,k_j^{-m-3})\subseteq V$ when $j\geq j_1$. We Taylor expand $\mathcal{F}(\chi u)$ at $k_j^3\theta_j$ up to order $0$ and employ the above properties of $\mathcal{F}(\chi u)$ to deduce that
$$
|\mathcal{F}(\chi u)(\xi)-\mathcal{F}(\chi u)(k_j^3\theta_j)|\leq m|\xi-k_j^3\theta_j|\max_{1\leq l\leq m}\|\partial_l\mathcal{F}(\chi u)\|_{L^{\infty}(\RR^m)},\quad \xi\in\RR^m,\, j\in\ZZ_+.
$$
Hence, there is $j_2> j_1$ such that $|\mathcal{F}(\chi u)(\xi)|\geq k_j^{-m-2}/4$ when $\xi\in B(k_j^3\theta_j,k_j^{-m-3})$, for all $j\geq j_2$. For any $j\geq j_2$, we infer
$$
\infty>\|\langle \cdot\rangle^{2m^2}\mathcal{F}(\chi u)\|_{L^2(V)}^2\geq \frac{(k_j^3-1)^{4m^2}}{16k_j^{2m+4}}\int_{B(k_j^3\theta_j,k_j^{-m-3})}d\xi= c_0\frac{(k_j^3-1)^{4m^2}}{16k_j^{m^2+5m+4}},
$$
where $c_0$ is the volume of the unit ball in $\RR^m$. This is a contradiction since the right hand side tends to $\infty$ as $j\rightarrow \infty$ and the proof of $WF^{2m^2}(u)=L$ is complete. For general $r\in\RR$, pick properly supported and elliptic $A\in\Psi^{2m^2-r}(\RR^m)$ and apply \eqref{wafe-frontincl-forcomp} and \cite[Theorem 18.1.28, p. 89]{hor2} to deduce $WF^r(Au)=WF^{2m^2}(u)=L=WF(u)=WF(Au)$.
\end{proof}

\section{Applications}\label{application}

We now showcase the theory we developed so far on two important and related concepts in the theory of PDEs and microlocal analysis: the existence of microlocal defect measures and the compensated compactness theorem. Our improvement primarily lies in that we can consider distributions whose Sobolev regularity is known only in parts of the cotangent bundle; of course, the results will provide information only in those parts. Both results are intrinsically connected with the $L^2$-inner product and we will have to work with sesquilinear forms. Since we aim for geometric extensions, we need to recall facts about anti-dual bundles. Most of these are widely known or easy to check, but, to the best of our knowledge, there are no standard notations for hardly any of them; we use the opportunity to fix the notation for the main results.\\
\indent For a l.c.s. $X$, we denote by $X^*$ its anti-dual (i.e. conjugate dual): it is the space of all continuous anti-linear (i.e. conjugate linear) functionals on $X$. As before, $X^*_b$ will stand for $X^*$ equipped with the strong dual topology.\\
\indent Given two complex vector bundles $E$ and $F$ over $M$, we denote by $\overline{L}(E,F)$ the complex vector bundle whose fibre at $p$ is the space of anti-linear (conjugate linear) maps $\overline{\mathcal{L}}(E_p,F_p)$. We denote by $E^*$ the anti-dual bundle of $E$, i.e. $E^*:=\overline{L}(E,\CC_M)$, while $E^{\#}$ stands for the functional anti-dual bundle of $E$, i.e. $E^{\#}:=\overline{L}(E,DM)$. The bundle $E$ is canonically isomorphic to $E^{\#\,\#}$ via the isomorphism $e\mapsto (e^*\mapsto \overline{e^*(e)})$; we will always identify these two bundles via this isomorphism (similarly as we identify $E$ with $E^{\vee\,\vee}$). Given $T\in L(E,F)_p$, we denote by $T^{\#}$ the element of $L(F^{\#},E^{\#})_p$ defined by $T^{\#}(f^*)(e)=f^*(Te)$, $e\in E_p$, $f^*\in F^{\#}_p$. Employing the above identification, we have $T=T^{\#\,\#}$. For $T\in L(E,F)_p$, we define $T^*\in L(F^*,E^*)_p$ analogously, and, employing the identifications $E=E^{*\,*}$ and $F=F^{*\,*}$ defined in the same way as above, we have $T^{*\,*}=T$.\\
\indent There is a canonical anti-linear isomorphism $\iota_E:E\rightarrow E^{\#\,\vee}=L(E^{\#},DM)$ given by $\iota_E(e)(e^*):=e^*(e)$, $e\in E_p$, $e^*\in E^{\#}_p$. (Although we will never use it, we point out that the conjugate bundle $\overline{E}$ is canonically isomorphic to $E^{\#\,\vee}$ via $e\mapsto(e^*\mapsto e^*(e))$.) The map $\iota_E$ induces an anti-linear topological isomorphism $\Gamma_c(E)\rightarrow \Gamma_c(E^{\#\,\vee})$ by applying it pointwise, which we again denote by $\iota_E$. Its transpose\footnote{If $T:X\rightarrow Y$ is a continuous anti-linear map, then one can define two transposes: $T_1:Y'\rightarrow X^*$, $T_1(y')(x):=y'(Tx)$, and $T_2:Y^*\rightarrow X'$, $T_2(y^*)(x):=y^*(Tx)$; both of them are linear and continuous when the spaces are equipped with their respective strong topologies. Here, $\iota_{0,E}$ is the first of these transposes.} $\iota_{0,E}:\DD'(M;E^{\#})\rightarrow (\Gamma_c(E))^*_b$ is a linear topological isomorphism. We denote by $(\cdot,\cdot)$ the sesquilinear form that comes from the anti duality of $\Gamma_c(E)$ and $(\Gamma_c(E))^*_b$; i.e. $(\iota_{0,E}(u),\varphi)=\langle u,\iota_E(\varphi)\rangle$, $u\in\DD'(M;E^{\#})$, $\varphi\in\Gamma_c(E)$. It is straightforward to verify that
$$
(\iota_{0,E}(f),\varphi)=\int_M \sqsubset f,\varphi\sqsupset,\quad f\in L^1_{\loc}(M;E^{\#}),\, \varphi\in\Gamma_c(E),\,\, \mbox{where}\,\, \sqsubset f,\varphi\sqsupset_p:=f_p(\varphi_p).
$$
Furthermore, $\iota_E$ uniquely extends to an anti-linear topological isomorphism $\iota_E:\DD'(M;E)\rightarrow \DD'(M;E^{\#\,\vee})$: pick a countable family of charts $\{(O_{\mu},x_{\mu})\}_{\mu\in\ZZ_+}$ as in Proposition \ref{lem-for-top-imebdofmapscs} and a partition of unity $(\varphi_{\mu})_{\mu\in\ZZ_+}$ subordinated to it and define (after identifying $E^{\#}$ with $(E^{\#})^{\vee\,\vee}$ as standard)
$$
\langle \iota_E(u),\psi\rangle:=\sum_{\mu\in\ZZ_+} \langle \overline{u^j_{\Phi_{x_{\mu}}}},(\varphi_{\mu}\psi_{\mu,j})\circ x^{-1}_{\mu}\rangle,\quad u\in\DD'(M;E),\,\,\psi\in\Gamma_c(E^{\#}),\, \psi_{|O_{\mu}}=\psi_{\mu,j}\widetilde{\sigma}^j_{\mu},
$$
where $(\widetilde{\sigma}^1_{\mu},\ldots,\widetilde{\sigma}^k_{\mu})$ is the frame for $E^{\#}$ induced by the trivialisation $\Phi_{x_{\mu}}$ of $E$ over $O_{\mu}$, namely $(\widetilde{\sigma}^j_{\mu})_p:=\Psi_{x_{\mu},p}^{-1}\circ \overline{\epsilon}^j\circ \Phi_{x_{\mu},p}$ with $\Phi_{x_{\mu},p}$ and $\Psi_{x_{\mu},p}$ the linear isomorphisms induced by the local trivialisations of $E$ and $DM$ over $O_{\mu}$ and $\overline{\epsilon}^j$ the anti-linear map $\CC^k\rightarrow\CC$, $z\mapsto \overline{z^j}$.
A more elegant way to define the extension of $\iota_E$ is as follows. Notice that $\iota_{E^{\#}}: E^{\#}\rightarrow E^{\vee}$ (we identify $E^{\#\, \#\,\vee}$ with $E^{\vee}$) is given by $\iota_{E^{\#}}(e^*)(e)=\overline{e^*(e)}$, $e\in E_p$, $e^*\in E^{\#}_p$ (the conjugation comes from the identification of $E$ with $E^{\#\, \#}$) and it induces an anti-linear topological isomorphism $\iota_{E^{\#}}:\Gamma_c(E^{\#})\rightarrow \Gamma_c(E^{\vee})$. The desired extension of $\iota_E$ is $\langle \iota_E(u),\psi\rangle:=\overline{\langle u,\iota_{E^{\#}}(\psi)\rangle}$, $\psi\in\Gamma_c(E^{\#})$, $u\in\DD'(M;E)$ (after identifying $E^{\#}$ with $(E^{\#})^{\vee\,\vee}$). In any case, $\iota_E$ restricts to an anti-linear topological isomorphism $\iota_E:\DD'^r_L(M;E)\rightarrow \DD'^r_{\check{L}}(M;E^{\#\,\vee})$ for any closed conic subset $L$ of $T^*M\backslash0$. Hence, in view of Theorem \ref{the-for-dua-ofdwitheonmanwithvectbundd}, $\iota_{0,E}$ restricts to a linear topological isomorphism $\iota_{0,E}:\EE'^{-r}_{L^c}(M;E^{\#})\rightarrow (\DD'^r_L(M;E))^*_b$. When $E$ is a complexification of a real vector bundle, $E$ has a natural operation of conjugation which induces conjugation on $E^{\vee}$ and $E^{\#}$ by $\overline{e'}(e)=\overline{e'(\overline{e})}$ and $\overline{e^*}(e)=\overline{e^*(\overline{e})}$ respectively which, in turn, allows us to identify $E^{\vee}$ with $E^{\#}$ via $e'\mapsto (e\mapsto \overline{\overline{e'}(e)}=e'(\overline{e}))$ (this is not the same as $\iota_{E^{\#}}^{-1}$ since $\iota_{E^{\#}}^{-1}$ is anti-linear!). This yields that $\iota_E$ is just conjugation in this case (since $E=E^{\vee\,\vee}$); furthermore for $\iota_{0,E}$ we have $(\iota_{0,E}(u),\varphi)=\langle u,\overline{\varphi}\rangle$ which implies that $(\cdot,\cdot)$ is induced by the standard $L^2$-sesquilinear form. However, when $E$ is not a complexification of a real vector bundle, $E$ does not possess conjugation compatible with its almost complex structure (see \cite[Proposition 2, p. 39]{luke-mis}) and one can not make these identifications.\\
\indent If $A\in\Psi^r(M;E,F)$ is properly supported, then its adjoint map with respect to the sesquilinear forms from the respective anti-dualities is defined by $A_0:(\Gamma_c(F))^*_b\rightarrow (\Gamma_c(E))^*_b$, $(A_0u,\varphi):=(u,A\varphi)$, $u\in(\Gamma_c(F))^*_b$, $\varphi\in\Gamma_c(E)$. It is customary to intertwine $A_0$ with the $\iota_0$-maps so it becomes a map on the distribution spaces, i.e. $A^*:\DD'(M;F^{\#})\rightarrow \DD'(M;E^{\#})$, $A^*:=\iota_{0,E}^{-1}A_0\iota_{0,F}$, and call $A^*$ the adjoint of $A$ instead (the notation $A^*$ is also suggestive for this, although not precise since $A_0$ is the true adjoint); we will always employ this definition and notation for $A^*$ throughout the rest of the article. Then $A^*\in\Psi^r(M;F^{\#},E^{\#})$, $A^*$ is properly supported and $\boldsymbol{\sigma}^r(A^*)=\boldsymbol{\sigma}^r(A)^{\#}$, where for $v\in\Gamma(\pi_{T^*M}^*L(E,F))$, we denote $v^{\#}\in\Gamma(\pi_{T^*M}^*L(F^{\#},E^{\#}))$, $v^{\#}(p,\xi)(f^*)e:=f^*(v(p,\xi)e)$, $e\in E_p$, $f^*\in F^{\#}_p$; i.e. $v^{\#}(p,\xi)=v(p,\xi)^{\#}$ (of course, after identifying $\pi^*_{T^*M}L(F^{\#},E^{\#})_{(p,\xi)}=\{(p,\xi)\}\times \mathcal{L}(F^{\#}_p,E^{\#}_p)$ with $\mathcal{L}(F^{\#}_p,E^{\#}_p)$). If in addition $A\in\Psi^r_{\phg}(M;E,F)$ then $A^*\in\Psi^r_{\phg}(M;F^{\#},E^{\#})$.\\
\indent For $r\in\RR$ and $W$ an open conic subset of $T^*M\backslash0$, we denote by $\Gamma_{\hom,r}(\pi_{T^*M}^*L(E,F)_W)$ the space of smooth sections $a:W\rightarrow \pi_{T^*M}^*L(E,F)_W$ which are positively homogeneous of degree $r$ on $W$, i.e. $\mathfrak{pr}(a(p,t\xi))=t^r\mathfrak{pr}(a(p,\xi))$, $(p,\xi)\in W$, $t>0$, where $\mathfrak{pr}$ is the smooth bundle homomorphism $\mathfrak{pr}:\pi^*_{T^*M}L(E,F)\rightarrow L(E,F)$, $\mathfrak{pr}((p,\xi),T)=T$ for $((p,\xi),T)\in \pi^*_{T^*M}L(E,F)_{(p,\xi)}=\{(p,\xi)\}\times\mathcal{L}(E_p,F_p)$. The principal symbol map $\boldsymbol{\sigma}^r$ induces a surjective linear map $\widetilde{\sigma}^r:\Psi^r_{\phg}(M;E,F)\rightarrow \Gamma_{\hom,r}(\pi^*_{T^*M}L(E,F)_{T^*M\backslash0})$;
for $A\in\Psi^r_{\phg}(M;E,F)$, $\widetilde{\sigma}^r(A)$ is such that it becomes an element of $\boldsymbol{\sigma}^r(A)$ once we modify it on (any) open neighbourhood $\mathcal{U}$ of the zero section in $T^*M$ such that $\mathcal{U}\cap \pi^{-1}_{T^*M}(K)$ is relatively compact for every $K\subset\subset M$. The kernel of $\widetilde{\sigma}^r$ is $\Psi^{r-1}_{\phg}(M;E,F)$. We denote by $\Psi_{\phg,c}^r(M;E,F)$ the space of all $\Psi$DOs in $\Psi^r_{\phg}(M;E,F)$ which have compactly supported kernel. When $r=0$, $\widetilde{\sigma}^0$ induces a surjective linear map $\sigma^0:\Psi^0_{\phg}(M;E,F)\rightarrow \Gamma(\pi_{S^*M}^*L(E,F))$, where $\pi_{S^*M}:S^*M\rightarrow M$ is the cosphere bundle and the latter space is the space of smooth sections $S^*M\rightarrow \pi_{S^*M}^*L(E,F)$ of the pullback vector bundle $\pi_{S^*M}^*L(E,F)$. The kernel of $\sigma^0$ is $\Psi^{-1}_{\phg}(M;E,F)$. For properly supported $A\in\Psi^0_{\phg}(M;E,F)$, it holds that $\sigma^0(A^*)=\sigma^0(A)^{\#}$ with the operation ${}^{\#}$ defined as above but now on $\pi^*_{S^*M}L(E,F)$. Notice that $\sigma^0$ restricts to a surjective linear map $\sigma^0:\Psi^0_{\phg,c}(M;E,F)\rightarrow \Gamma_c(\pi_{S^*M}^*L(E,F))$ whose kernel is $\Psi^{-1}_{\phg,c}(M;E,F)$. If $A\in\Psi^0_{\phg,c}(M;E,F)$ and $B\in \Psi^0_{\phg}(M;F,H)$ is properly supported, then $BA$ also has kernel with compact support and $\sigma^0(BA)=\sigma^0(B)\sigma^0(A)$; an analogous statement holds true when composing operators in the other direction. When $W$ is an open subset of $S^*M$, we denote by $\Gamma_c(\pi_{S^*M}^*L(E,F)_W)$ and $\Gamma_c^0(\pi_{S^*M}^*L(E,F)_W)$ the spaces of all smooth and continuous compactly supported sections $W \rightarrow \pi_{S^*M}^*L(E,F)_W$ respectively. These spaces are endowed with their respective strict $(LF)$- and strict $(LB)$-space topologies. If $A\in\Psi^0_{\phg}(M;E,F)$ is of order $r'<0$ at $(p,\xi)\in T^*M\backslash0$ (hence of order $\lfloor r' \rfloor$ at $(p,\xi)$), then $\sigma^0(A)=0$ in a neighbourhood of $(p,[\xi])\in S^*M$ where $(p,[\xi])$ is the image of $(p,\xi)$ under the natural map $T^*M\backslash0\rightarrow S^*M$. For an open conic subset $W$ of $T^*M\backslash0$ and $r\in\RR$, we denote
$$
\Psi^r_{\phg,c,W}(M;E,F):=\{A\in\Psi^r_{\phg,c}(M;E,F)\,|\, A\,\, \mbox{is of order}\,\, -\infty\,\, \mbox{at every point of}\,\, W^c\}.
$$
If $A\in\Psi^r_{\phg,c,W}(M;E,F)$, then $A^*\in\Psi^r_{\phg,c,W}(M;F^{\#},E^{\#})$. Notice that $\sigma^0$ restricts to a well-defined linear map
\begin{equation}\label{equ-for-pri-maponpolsymcospbdd}
\sigma^0:\Psi^0_{\phg,c,W}(M;E,F)\rightarrow \Gamma_c(\pi_{S^*M}^*L(E,F)_{[W]});
\end{equation}
here and throughout the rest of the article we denote by $[L]$ the image of the conic subset $L\subseteq T^*M\backslash0$ under the natural map $T^*M\backslash0\rightarrow S^*M$. We point out that \eqref{equ-for-pri-maponpolsymcospbdd} is surjective. More generally, given $a\in \Gamma(\pi^*_{S^*M}L(E,F))$ with $\supp a\subseteq [W]$, there is a properly supported $A\in\Psi^0_{\phg}(M;E,F)$ such that $A$ is of order $-\infty$ on $W^c$ and $\sigma^0(A)=a$. The kernel of \eqref{equ-for-pri-maponpolsymcospbdd} is $\Psi^{-1}_{\phg,c,W}(M;E,F)$. If $A\in\Psi^0_{\phg,c,W}(M;E,F)$ and $B\in\Psi^0_{\phg}(M;F,H)$ is properly supported then $BA\in\Psi^0_{\phg,c,W}(M;E,H)$; an analogous statement holds true when composing operators in the other direction. Notice that all of the above holds even when $W=\emptyset$ if we interpret $\Gamma_c(\pi_{S^*M}^*L(E,F)_{[W]})$ as the space of smooth sections with compact support in $[W]=\emptyset$ which consists only of the zero section.

\subsection{Generalisation of the microlocal defect measures of \texorpdfstring{G\'erard}{Gerard} and Tartar}

We devote this subsection to generalising the concept of microlocal defect measures introduced by G\'erard \cite[Theorem 1]{Ger} and Tartar \cite[Theorem 1.1]{Tar} to sequences in the space $\DD'^0_L(M;E)$ where $L$ is a closed conic subset of $T^*M\backslash0$; i.e. sequences whose elements are only known to be $L^2_{\loc}$ outside of $L$. The two main ingredients that will allow us to show this fact is the duality from Theorem \ref{the-for-dua-ofdwitheonmanwithvectbundd}, which we will constantly tacitly apply it from now on, and the generalisations of the Rellich's lemma we showed in Section \ref{sec-psido-on-despacesanddual}.\\
\indent Before we start, we point out the following consequence of the above considerations. Given $A\in\Psi^0_{\phg,c,L^c}(M;E,E^{\#})$, Corollary \ref{cor-for-ope-defondinecomsetfort} implies that $A:\DD'^0_L(M;E)\rightarrow \EE'^0_{L^c}(M;E^{\#})$ is well-defined and continuous and thus $(\iota_{0,E}(Au),v)$, $u,v\in\DD'^0_L(M;E)$, is well-defined. Furthermore $(\iota_{0,E}(A^*u),v)=\overline{(\iota_{0,E}(A v),u)}$, $u,v\in\DD'^0_L(M;E)$.\\
\indent We need the following variant of \cite[Lemma 1.2]{Ger}; throughout the rest of the article, $\M_k(\CC)$ stands for the space of $k\times k$ complex matrices.

\begin{lemma}\label{lem-for-ope-whinonsymbbozerliminfst}
Let $O$ be an open subset of $\RR^m$ and $L$ a closed conic subsets of $O\times (\RR^m\backslash\{0\})$. Let $u,u_n\in\DD'^0_L(O;\CC^k)$, $n\in\ZZ_+$, be such that $\{u_n\}_{n\in\ZZ_+}$ is bounded in $\DD'^0_L(O;\CC^k)$ and $u_n\rightarrow u$ in $\DD'(O;\CC^k)$. If $A\in \Psi^0_{\phg,c,L^c}(O;\CC^k,\CC^k)$ is such that $\sigma^0(A)$ is positive semi-definite at every point, then
\begin{equation}
\lim_{n\rightarrow\infty}\operatorname{Im}(A(u_n-u),u_n-u)=0\quad \mbox{and}\quad \liminf_{n\rightarrow \infty}\operatorname{Re}(A(u_n-u),u_n-u)\geq0.
\end{equation}
\end{lemma}

\begin{proof} The claim is trivial when $L=O\times(\RR^m\backslash\{0\})$. Assume that $L^c\neq\emptyset$. The proof of the first equality is the same as in \cite[Lemma 1.2]{Ger}, but, instead of the Rellich lemma, one employs Corollary \ref{cor-for-ope-defondinecomsetfort}. The proof of the second inequality is also similar to the proof of \cite[Lemma 1.2]{Ger}; we point out only the notable differences. Denote $a:=\sigma^0(A)\in\DD([L^c];\M_k(\CC))$, and for arbitrary but fixed $\delta>0$ set $b:=(\delta I+a)^{1/2}$ and $b':=b-\delta^{1/2}I$ with $I$ being the identity matrix; notice that $\supp b'=\supp a$. Pick $B'\in\Psi^0_{\phg,c,L^c}(O;\CC^k,\CC^k)$ such that $\sigma^0(B')=b'$. The difference is that we define $B$ as $B:=A_0(\delta^{1/2}\operatorname{Id}+B')$ where we choose $A_0\in\Psi^0_{\phg,c,L^c}(O;\CC^k,\CC^k)$ such that $\DD([L^c];\M_k(\CC))\ni \sigma^0(A_0)=a_0I$ with $a_0$ smooth, nonnegative and equal to $1$ on a neighbourhood of $\supp a$. Of course $B\in\Psi^0_{\phg,c,L^c}(O;\CC^k,\CC^k)$ and $\sigma^0(B)=a_0b=a_0(\delta I+a)^{1/2}$ and consequently $\sigma^0(B^*B)=\delta a_0^2I+a_0^2a=\delta a_0^2I+a=\sigma^0(\delta A_0^*A_0+A)$. Hence, there is $R\in\Psi^{-1}_{\phg,c,L^c}(O;\CC^k,\CC^k)$ such that $B^*B=\delta A_0^*A_0+A+R$. In view of Corollary \ref{cor-for-ope-defondinecomsetfort}, $A_0^*A_0,B^*B:\DD'^0_L(O;\CC^k)\rightarrow \EE'^0_{L^c}(O;\CC^k)$ are well-defined and continuous. Notice that $(B^*B\widetilde{u},\widetilde{u})=(B\widetilde{u},B\widetilde{u})\geq 0$, $\widetilde{u}\in\DD'^0_L(O;\CC^k)$. This trivially holds when $\widetilde{u}\in\DD(O;\CC^k)$ and the general case follows by density (as $\EE'^0_{L^c}(O;\CC^k)\subseteq \DD'^0_L(O;\CC^k)$); of course the same holds with $A_0$ in place of $B$. Employing this, in the same way as in \cite[Lemma 1.2]{Ger}, we deduce the second inequality in the lemma.
\end{proof}

We denote by $\mathcal{K}(M;\M_k(\CC))$ the space of continuous functions with compact support with values in $\M_k(\CC)$ and equipped with its standard $(LB)$-space topology; when $k=1$ (the complex-valued case) we simply write $\mathcal{K}(M)$. A Radon measure is a continuous functional on $\mathcal{K}(M;\M_k(\CC))$ (cf. \cite[Chapter 3]{bourbaki}). The Radon measure $\vartheta\in (\mathcal{K}(M;\M_k(\CC)))'$ is said to be positive if $\langle \vartheta,\psi\rangle\geq 0$ for every $\psi\in\mathcal{K}(M;\M_k(\CC))$ which is positive semi-definite at every point. Every $\vartheta\in (\mathcal{K}(M;\M_k(\CC)))'$ is a matrix of elements of $(\mathcal{K}(M))'$, i.e $\vartheta=(\vartheta^{j,l})_{j,l}$, $\vartheta^{j,l}\in(\mathcal{K}(M))'$, and, for $\psi=(\psi_{j,l})_{j,l}\in \mathcal{K}(M;\M_k(\CC))$, the dual pairing is $\langle \vartheta,\psi\rangle=\langle\vartheta^{j,l},\psi_{j,l}\rangle$ (of course, it looks like the trace of the product of $\vartheta$ with the transpose of $\psi$ since that is the dual pairing on the matrices). Each $\vartheta\in(\mathcal{K}(M;\M_k(\CC)))'$ can be viewed as an element of $\mathcal{L}(\mathcal{K}(M),\M_k(\CC))$ by defining $\vartheta(\phi)=(\langle\vartheta^{j,l},\phi\rangle)_{j,l}$, $\phi\in\mathcal{K}(M)$, and this gives topological isomorphism between the strong dual $(\mathcal{K}(M;\M_k(\CC)))'_b$ and $\mathcal{L}_b(\mathcal{K}(M);\M_k(\CC))$; some author choose to define the matrix valued Radon measures as $\mathcal{L}(\mathcal{K}(M),\M_k(\CC))$ (see \cite{Ger}). In this case, the measure is said to be positive if for every nonnegative $\phi\in\mathcal{K}(M)$, the matrix $(\langle\vartheta^{j,l},\phi\rangle)_{j,l}$ is positive semi-definite (see \cite{dur-rod,Ger,rosenberg}). Employing similar technique as in the proof of \cite[Proposition A.1]{Ger}, one can show that this definition of positiveness coincides with the one we give above. We specifically choose to define $(\mathcal{K}(M;\M_k(\CC)))'$ as the space of matrix-valued Radon measures because it is better suited for a generalisation to the vector bundle case; we intentionally put the indices in $(\vartheta^{j,l})_{j,l}$ up to hint at the geometrical picture given in the main theorem (if one thinks of the elements of $E$ as vectors, $\vartheta$ will be a $2$-vector filed, instead of an $(1,1)$-tensor field).\\
\indent We first show our result on the existence of microlocal defect measures in the Euclidean setting. The extensive analysis we have done in the previous sections will now pay off by allowing us to mimic the main ideas of \cite{Ger}.

\begin{proposition}\label{pro-for-mai-resonrspt}
Let $O$ be an open set in $\RR^m$ and $L$ a closed conic subset of $O\times (\RR^m\backslash\{0\})$. Let $u,u_n\in\DD'^0_L(O;\CC^k)$, $n\in\ZZ_+$, be such that $\{u_n\}_{n\in\ZZ_+}$ is bounded in $\DD'^0_L(O;\CC^k)$ and $u_n\rightarrow u$ in $\DD'(O;\CC^k)$. Then there is a subsequence $(u_{n_j})_{j\in\ZZ_+}$ and a positive Radon measure $\vartheta\in (\mathcal{K}([L^c];\M_k(\CC)))'$ such that
\begin{equation}\label{equ-for-pro-resgercomsct}
\lim_{j\rightarrow \infty}(A(u_{n_j}-u),u_{n_j}-u)=\langle \vartheta,\sigma^0(A)\rangle,\quad A\in\Psi^0_{\phg,c,L^c}(O;\CC^k,\CC^k),
\end{equation}
and $\vartheta$ satisfies $\langle\vartheta,\psi^*\rangle=\overline{\langle\vartheta,\psi\rangle}$, $\psi\in\mathcal{K}([L^c];\M_k(\CC))$. Furthermore, $\vartheta$ is unique in the following sense: if $\vartheta'\in (\mathcal{K}([L^c];\M_k(\CC)))'$ is such that \eqref{equ-for-pro-resgercomsct} is valid for $\vartheta'$, then $\vartheta'=\vartheta$.
\end{proposition}

\begin{proof} The claim is trivial when $L=O\times(\RR^m\backslash\{0\})$. Assume that $L^c\neq\emptyset$. Since for every $a\in\DD([L^c];\M_k(\CC))$ there is $A\in\Psi^0_{\phg,c,L^c}(O;\CC^k,\CC^k)$ such that $\sigma^0(A)=a$, the uniqueness follows from density. We show the existence. For brevity in notation, set $v_n:=u_n-u$, $n\in\ZZ_+$. One can find closed conic subsets $\widetilde{L}_j\neq\emptyset$, $j\in\ZZ_+$, of $O\times (\RR^m\backslash\{0\})$ such that $\pr_1(\widetilde{L}_j)$ is compact and $\widetilde{L}_j\subseteq \operatorname{int}\widetilde{L}_{j+1}$, $j\in\ZZ_+$, and $\bigcup_{j\in\ZZ_+}\widetilde{L}_j=L^c$ (cf. \eqref{inc-set-com-exchscoses}). Hence $[\widetilde{L}_j]$, $j\in\ZZ_+$, are compact. Since the closure of $\DD_{[\widetilde{L}_{j+1}]}([L^c];\M_k(\CC))$ in $\mathcal{C}_{[\widetilde{L}_{j+1}]}([L^c];\M_k(\CC))$ contains $\mathcal{C}_{[\widetilde{L}_j]}([L^c];\M_k(\CC))$, there is a countable dense subset $D_j$ of $\DD_{[\widetilde{L}_{j+1}]}([L^c];\M_k(\CC))$ such that its closure in $\mathcal{C}_{[\widetilde{L}_{j+1}]}([L^c];\M_k(\CC))$ contains $\mathcal{C}_{[\widetilde{L}_j]}([L^c];\M_k(\CC))$. Set $D:=\bigcup_{j\in\ZZ_+} D_j$ and denote $\widetilde{D}:=\operatorname{span}(D\cup\{a\,|\, a^*\in D\})$. Let $a\in D$ and pick $A\in \Psi^0_{\phg,c,L^c}(O;\CC^k,\CC^k)$ such that $\sigma^0(A)=a$. Since $\{v_n\}_{n\in\ZZ_+}$ is bounded in $\DD'^0_L(O;\CC^k)$, $\sup_{n\in\ZZ_+}|(Av_n,v_n)|<\infty$ and hence there is a subsequence $(v_{n_l})_{l\in\ZZ_+}$ such that $(Av_{n_l},v_{n_l})$ converges to some $\vartheta(a)\in\CC$. In view of Corollary \ref{cor-for-ope-defondinecomsetfort}, the same holds for any other $\widetilde{A}\in\Psi^0_{\phg,c,L^c}(O;\CC^k,\CC^k)$ satisfying $\sigma^0(\widetilde{A})=a$. By employing diagonal extraction, we can assume the subsequence $(v_{n_l})_{l\in\ZZ_+}$ is the same for all $a\in D$. Employing linearity and involutions, we deduce the existence of a linear functional
$$
\vartheta:\widetilde{D}\rightarrow \CC,\quad \vartheta(a)=\lim_{l\rightarrow \infty} (Av_{n_l},v_{n_l}),\,\,\mbox{with}\,\, \sigma^0(A)=a,
$$
satisfying $\vartheta(a^*)=\overline{\vartheta(a)}$. We equip $\widetilde{D}$ with the topology induced by $\mathcal{K}([L^c];\M_k(\CC))$. We claim that $\vartheta$ is continuous. For each $j\in\ZZ_+$, set $\widetilde{D}_j:=\widetilde{D}\cap \mathcal{C}_{[\widetilde{L}_j]}([L^c];\M_k(\CC))$ and equip it with the topology induced by $\mathcal{C}_{[\widetilde{L}_j]}([L^c];\M_k(\CC))$. By construction, $\mathcal{C}_{[\widetilde{L}_j]}([L^c];\M_k(\CC))$ is contained in the closure of $\widetilde{D}_{j+1}$ in $\mathcal{C}_{[\widetilde{L}_{j+1}]}([L^c];\M_k(\CC))$ and \cite[Corollary 1, p. 164]{hus-lfspp} implies that $\displaystyle\lim_{\substack{\longrightarrow\\ j\rightarrow\infty}} \widetilde{D}_j=\widetilde{D}$ topologically; whence, it suffices to show that $\vartheta$ is continuous on $\widetilde{D}_j$ for each $j\in\ZZ_+$. For every $j\in\ZZ_+$, pick nonnegative $b_j\in\DD_{[\widetilde{L}_{j+1}]}([L^c])$ such that $b_j=1$ on a neighbourhood of $[\widetilde{L}_j]$ and choose $B_j\in\Psi^0_{\phg,c,L^c}(O;\CC^k,\CC^k)$ such that $\sigma^0(B_j)=b_jI$, where $I$ is the identity matrix. Let $a\in\widetilde{D}_j$ be self-adjoint at every point and denote $r_a:=\|a\|_{\mathcal{C}_{[\widetilde{L}_j]}([L^c];\M_k(\CC))}$. Pick $A\in\Psi^0_{\phg,c,L^c}(O;\CC^k,\CC^k)$ such that $\sigma^0(A)=a$. Lemma \ref{lem-for-ope-whinonsymbbozerliminfst} is applicable for $r_aB_j-A\in\Psi^0_{\phg,c,L^c}(O;\CC^k,\CC^k)$, so we infer $\vartheta(a)\leq C_jr_a$ with $C_j:=\sup_{l\in\ZZ_+}|(B_jv_{n_l},v_{n_l})|<\infty$ (of course, $\vartheta(a)\in\RR$). Applying this with $-a$ in place of $a$, we deduce $|\vartheta(a)|\leq C_jr_a$. When $a\in\widetilde{D}_j$ is arbitrary, we apply the above to $(a+a^*)/2\in\widetilde{D}_j$ and $(a-a^*)/(2i)\in\widetilde{D}_j$ to conclude $|\vartheta(a)|\leq 2C_j r_a$. Consequently, $\vartheta:\widetilde{D}\rightarrow \CC$ is continuous; whence $\vartheta\in(\mathcal{K}([L^c];\M_k(\CC)))'$.\\
\indent To show \eqref{equ-for-pro-resgercomsct}, let $A\in \Psi^0_{\phg,c,L^c}(O;\CC^k,\CC^k)$. There is $j\geq 2$ such that $a:=\sigma^0(A)\in\DD_{[\widetilde{L}_j]}([L^c];\M_k(\CC))$. By construction, there are $a_l\in D_{j-1}$, $l\in\ZZ_+$, such that $a_l\rightarrow a$ in $\DD_{[\widetilde{L}_j]}([L^c];\M_k(\CC))$. Pick $\chi\in\mathcal{C}^{\infty}(\RR^m)$ such that $0\leq \chi\leq 1$, $\chi(\xi)=0$ when $|\xi|\leq 1/4$ and $\chi(\xi)=1$ when $|\xi|\geq 1/2$. Define $\widetilde{a}(x,\xi)=\chi(\xi)a(x,\xi/|\xi|)$ and $\widetilde{a}_l(x,\xi)=\chi(\xi)a_l(x,\xi/|\xi|)$, $l\in\ZZ_+$, $x\in O$, $\xi\in\RR^m$. Of course, $\widetilde{a},\widetilde{a}_l\in S^0_c(O\times\RR^m;\CC^k,\CC^k)$, $l\in\ZZ_+$. Notice that both $\{B_jv_{n_h}\}_{h\in\ZZ_+}$ and $\{B^*_jv_{n_h}\}_{h\in\ZZ_+}$ are bounded subsets of $\EE'^0_{L^c}(O;\CC^k)$ and thus also of $L^2(\RR^m)$ (see \eqref{con-inc-for-espainsobcomspdes}). Consequently
\begin{multline*}
|(B_j\Op(\widetilde{a})B_jv_{n_h},v_{n_h})-(B_j\Op(\widetilde{a}_l)B_jv_{n_h},v_{n_h})|\\
=|(\Op(\widetilde{a}-\widetilde{a}_l)B_jv_{n_h},B^*_jv_{n_h})|\leq C\sup_{|\alpha|+|\beta|\leq q}\|\langle\cdot\rangle^{|\alpha|} \partial^{\alpha}_{\xi}\partial^{\beta}_x(\widetilde{a}-\widetilde{a}_l)\|_{L^{\infty}(\RR^{2m};\M_k(\CC))},\,\,\, h,l\in\ZZ_+,
\end{multline*}
and, as $\widetilde{a}_l\rightarrow \widetilde{a}$ in $S^0(\RR^{2m};\CC^k,\CC^k)$, we deduce
$$
\lim_{l\rightarrow\infty}\sup_{h\in\ZZ_+}|(B_j\Op(\widetilde{a})B_jv_{n_h},v_{n_h})-(B_j\Op(\widetilde{a}_l)B_jv_{n_h},v_{n_h})|=0.
$$
For $\varepsilon>0$ there is $l_{\varepsilon}\in\ZZ_+$ such that for all $h\in\ZZ_+$ and $l\geq l_{\varepsilon}$ it holds that
\begin{gather*}
\operatorname{Re}(B_j\Op(\widetilde{a}_l)B_jv_{n_h},v_{n_h})-\varepsilon\leq \operatorname{Re}(B_j\Op(\widetilde{a})B_jv_{n_h},v_{n_h})\leq \operatorname{Re}(B_j\Op(\widetilde{a}_l)B_jv_{n_h},v_{n_h})+\varepsilon,\\
\operatorname{Im}(B_j\Op(\widetilde{a}_l)B_jv_{n_h},v_{n_h})-\varepsilon\leq \operatorname{Im}(B_j\Op(\widetilde{a})B_jv_{n_h},v_{n_h})\leq \operatorname{Im}(B_j\Op(\widetilde{a}_l)B_jv_{n_h},v_{n_h})+\varepsilon.
\end{gather*}
Since $\sigma^0(B_j\Op(\widetilde{a}_l)B_j)=a_l$, by what we proved above, we infer $(B_j\Op(\widetilde{a}_l)B_jv_{n_h},v_{n_h})\rightarrow \langle\vartheta,a_l\rangle$, as $h\rightarrow \infty$, and consequently
\begin{align*}
\operatorname{Re}\langle\vartheta,a_l\rangle-\varepsilon&\leq \liminf_{h\rightarrow\infty}\operatorname{Re}(B_j\Op(\widetilde{a})B_jv_{n_h},v_{n_h})\\
&\leq \limsup_{h\rightarrow\infty}\operatorname{Re}(B_j\Op(\widetilde{a})B_jv_{n_h},v_{n_h})\leq \operatorname{Re}\langle\vartheta,a_l\rangle+\varepsilon,\quad l\geq l_{\varepsilon}\\
\operatorname{Im}\langle\vartheta,a_l\rangle-\varepsilon&\leq \liminf_{h\rightarrow\infty}\operatorname{Im}(B_j\Op(\widetilde{a})B_jv_{n_h},v_{n_h})\\
&\leq \limsup_{h\rightarrow\infty}\operatorname{Im}(B_j\Op(\widetilde{a})B_jv_{n_h},v_{n_h}) \leq \operatorname{Im}\langle\vartheta,a_l\rangle+\varepsilon,\quad l\geq l_{\varepsilon}.
\end{align*}
We take the limit as $l\rightarrow \infty$. Since $\langle \vartheta,a_l\rangle\rightarrow \langle \vartheta,a\rangle$ and as $\varepsilon>0$ was arbitrary, we deduce
$$
\lim_{h\rightarrow\infty}(B_j\Op(\widetilde{a})B_jv_{n_h},v_{n_h})=\langle\vartheta,a\rangle.
$$
As $\sigma^0(B_j\Op(\widetilde{a})B_j)=a=\sigma^0(A)$, we have $B_j\Op(\widetilde{a})B_j-A\in\Psi^{-1}_{\phg,c,L^c}(O;\CC^k,\CC^k)$ and Corollary \ref{cor-for-ope-defondinecomsetfort} yields that $(Av_{n_h},v_{n_h})\rightarrow (\vartheta,a)$. Now, Corollary \ref{cor-for-ope-defondinecomsetfort} also shows that $(A'v_{n_h},v_{n_h})\rightarrow (\vartheta,a)$ for any other $A'\in\Psi^0_{\phg,c,L^c}(O;\CC^k,\CC^k)$ satisfying $\sigma^0(A')=a$. This completes the proof of \eqref{equ-for-pro-resgercomsct}.\\
\indent Finally, $\langle\vartheta,\psi^*\rangle=\overline{\langle\vartheta,\psi\rangle}$, $\psi\in\mathcal{K}([L^c];\M_k(\CC))$, holds true since it is true for $\psi\in\widetilde{D}$. Similarly, $\langle \vartheta,\psi\rangle\geq0$ holds true for all $\psi\in\DD([L^c];\M_k(\CC))$ which are positive semi-definite at every point in view of \eqref{equ-for-pro-resgercomsct} and Lemma \ref{lem-for-ope-whinonsymbbozerliminfst} and the general case when $\psi\in\mathcal{K}([L^c];\M_k(\CC))$ and it is positive semi-definite at every point follows by density.
\end{proof}

We say that $T\in \pi^*_{S^*M}L(E,E^{\#})_{(p,[\xi])}$ is \textit{positive semi-definite} if $T=T^{\#}$ and $T(e)(e)$ is a nonnegative density on $T_pM$ for all $e\in E_p$ (of course, after identifying $\pi^*_{S^*M}L(E,E^{\#})_{(p,[\xi])}=\{(p,[\xi])\}\times \mathcal{L}(E_p,E^{\#}_p)$ with $\mathcal{L}(E_p,E^{\#}_p)$), i.e. $T(e)(e)(\mathrm{v}_1,\ldots,\mathrm{v}_m)\in[0,\infty)$ for all $\mathrm{v}_1,\ldots,\mathrm{v}_m\in T_pM$. The main result of the subsection is the following theorem.

\begin{theorem}\label{ger-the-com-onmanifol}
Let $L$ be a closed conic subset of $T^*M\backslash0$ and let $u,u_n\in\DD'^0_L(M;E)$, $n\in\ZZ_+$. If $\{u_n\}_{n\in\ZZ_+}$ is bounded in $\DD'^0_L(M;E)$ and $u_n\rightarrow u$ in $\DD'(M;E)$, then there is a subsequence $(u_{n_j})_{j\in\ZZ_+}$ and $\vartheta\in (\Gamma^0_c(\pi_{S^*M}^*L(E,E^{\#})_{[L^c]}))'$ such that
\begin{equation}\label{ide-for-the-ofgermars}
\lim_{j\rightarrow \infty}(\iota_{0,E}(A(u_{n_j}-u)),u_{n_j}-u)=\langle \vartheta,\sigma^0(A)\rangle,\quad A\in\Psi^0_{\phg,c,L^c}(M;E,E^{\#}),
\end{equation}
and $\vartheta$ satisfies $\langle\vartheta,\psi^{\#}\rangle=\overline{\langle\vartheta,\psi\rangle}$, $\psi\in\Gamma^0_c(\pi_{S^*M}^*L(E,E^{\#})_{[L^c]})$, and $\langle\vartheta,\psi\rangle\geq0$ if $\psi\in\Gamma^0_c(\pi_{S^*M}^*L(E,E^{\#})_{[L^c]})$ is positive semi-definite at every point. Furthermore, $\vartheta$ is unique in the following sense: if \eqref{ide-for-the-ofgermars} is valid for $\vartheta'\in (\Gamma^0_c(\pi_{S^*M}^*L(E,E^{\#})_{[L^c]}))'$ then $\vartheta'=\vartheta$.
\end{theorem}

\begin{proof} The proof of the uniqueness is the same is in Proposition \ref{pro-for-mai-resonrspt}. We show the existence. For brevity in notation, set $v_n:=u_n-u$, $n\in\ZZ_+$. Pick a countable locally finite family of relatively compact coordinate charts $(O_{\mu},x_{\mu})$, $\mu\in\ZZ_+$, that cover $M$ such that, for each $\mu\in\ZZ_+$, $E$ locally trivialises over $O_{\mu}$ via $\Phi_{\mu}:=\Phi_{x_{\mu}}:\pi_E^{-1}(O_{\mu})\rightarrow O_{\mu}\times \CC^k$. We denote by $(s_{\mu,1},\ldots,s_{\mu,k})$ and $(\sigma^1_{\mu},\ldots,\sigma^k_{\mu})$ the local frames induced by $\Phi_{\mu}$ for $E$ and $E^{\#}$ over $O_{\mu}$ respectively; of course $\sigma^j_{\mu}(zs_{\mu,l})=\overline{z}\delta^j_l\lambda^{x_{\mu}}$, $z\in\CC$. Denoting by $(s'^1_{\mu},\ldots,s'^k_{\mu})$ the frame for $E'$ dual to $(s_{\mu,1},\ldots,s_{\mu,k})$, we see that $(s'^j_{\mu}\otimes \sigma^l_{\mu})_{j,l}$ is a local frame for $L(E,E^{\#})$ over $O_{\mu}$; for $e\in E_p$, $(s'^j_{\mu}\otimes\sigma^l_{\mu})_p(e)=s'^j_{\mu,p}(e)\sigma^l_{\mu,p}\in E^{\#}_p$. As standard, $(\pi^*_{S^*M}(s'^j_{\mu}\otimes\sigma^l_{\mu}))_{j,l}$ is the pullback local frame for $\pi^*_{S^*M}L(E,E^{\#})$ over $\pi^{-1}_{S^*M}(O_{\mu})$. For each $\mu\in\ZZ_+$, set $L_{\mu}:=\kappa_{\mu}(\pi_{T^*M}^{-1}(O_{\mu})\cap L)$ with $\kappa_{\mu}$ as in \eqref{tri-cot-bun-coordindms}; $L_{\mu}$ is a closed conic subset of $x_{\mu}(O_{\mu})\times(\RR^m\backslash\{0\})$. Denote $v^j_{n,\mu}:=v^j_{n,\Phi_{x_\mu}}\in\DD'^0_{L_{\mu}}(x_{\mu}(O))$, $j=1,\ldots,k$, $\mu\in\ZZ_+$, $n\in\ZZ_+$; see Proposition \ref{lem-for-top-imebdofmapscs}. Proposition \ref{pro-for-mai-resonrspt} is applicable to any subsequence of $\widetilde{v}_{n,\mu}:=(v^1_{n,\mu},\ldots,v^k_{n,\mu})\in\DD'^0_{L_{\mu}}(x_{\mu}(O_{\mu});\CC^k)$, $n\in\ZZ_+$, for each $\mu\in\ZZ_+$. We apply the proposition together with a standard diagonal argument to find a subsequence $(v_{n_h})_{h\in\ZZ_+}$ and positive Radon measures $\widetilde{\vartheta}_{\mu}=(\widetilde{\vartheta}^{l,j}_{\mu})_{l,j}\in (\mathcal{K}([(L_{\mu})^c];\M_k(\CC)))'$, $\mu\in\ZZ_+$, such that
\begin{equation}\label{equ-for-lim-ofopexmeaofs}
\lim_{h\rightarrow \infty}(\widetilde{A} \widetilde{v}_{n_h,\mu},\widetilde{v}_{n_h,\mu})=\langle\widetilde{\vartheta}_{\mu},\sigma^0(\widetilde{A})\rangle,\quad \widetilde{A}\in\Psi^0_{\phg,c,(L_{\mu})^c}(x_{\mu}(O_{\mu});\CC^k,\CC^k).
\end{equation}
We are going to glue together the $\widetilde{\vartheta}_{\mu}$'s to a continuous functional on $\Gamma^0_c(\pi_{S^*M}^*L(E,E^{\#})_{[L^c]})$. Choose a smooth partition of unity $(\chi_{\mu})_{\mu\in\ZZ_+}$ subordinated to $(O_{\mu})_{\mu\in\ZZ_+}$ and, for each $\mu\in\ZZ_+$, pick nonnegative $\chi'_{\mu}\in\DD(O_{\mu})$ such that $\chi'_{\mu}=1$ on a neighbourhood of $\supp\chi_{\mu}$. For brevity in notation, we set $\widetilde{\chi}_{\mu}:=\chi_{\mu}\circ x_{\mu}^{-1}$ and $\widetilde{\chi}'_{\mu}:=\chi'_{\mu}\circ x_{\mu}^{-1}$. Let $\hat{\kappa}_{\mu}$ be the diffeomorphism
\begin{equation}\label{dif-for-cos-bunovermanc}
\hat{\kappa}_{\mu}:\pi_{S^*M}^{-1}(O_{\mu})\rightarrow x_{\mu}(O_{\mu})\times \mathbb{S}^{m-1},\, \hat{\kappa}_{\mu}(p,[\xi_j dx^j|_p]):=(x_{\mu}(p),\hat{\xi}_1,\ldots,\hat{\xi}_m),
\end{equation}
where $\hat{\xi}_jdx^j|_p\in[\xi_jdx^j|_p]$ is such that $(\hat{\xi}_1,\ldots,\hat{\xi}_m)\in\mathbb{S}^{m-1}$. We define
\begin{gather*}
\vartheta:\Gamma_c^0(\pi_{S^*M}^*L(E,E^{\#})_{[L^c]})\rightarrow \CC,\, \langle \vartheta,\psi\rangle:=\sum_{\mu\in\ZZ_+}\langle \widetilde{\vartheta}^{l,j}_{\mu},\widetilde{\chi}_{\mu}\,\psi_{\mu,l,j}\circ \hat{\kappa}^{-1}_{\mu}\rangle,\quad \mbox{where}\\
\psi\in\Gamma^0_c(\pi_{S^*M}^*L(E,E^{\#})_{[L^c]}),\, \psi_{|\pi_{S^*M}^{-1}(O_{\mu})}=\psi_{\mu,l,j}\pi^*_{S^*M}(s'^j_{\mu}\otimes \sigma^l_{\mu}).
\end{gather*}
It is straightforward to verify that $\vartheta$ is well-defined and continuous (notice that $\supp(\widetilde{\chi}_{\mu}\,\psi_{\mu,l,j}\circ\hat{\kappa}_{\mu}^{-1})\subset\subset [(L_{\mu})^c]$). The fact $\langle\vartheta,\psi^{\#}\rangle=\overline{\langle \vartheta,\psi\rangle}$ follows from the fact that $\psi^{\#}_{|\pi^{-1}_{S^*M}(O_{\mu})}=\overline{\psi_{\mu,j,l}}\pi^*_{S^*M}(s'^j_{\mu}\otimes \sigma^l_{\mu})$ and the corresponding fact about $\widetilde{\vartheta}_{\mu}$ from Proposition \ref{pro-for-mai-resonrspt}. Furthermore, if $\psi$ is positive semi-definite at every point, then $\langle\vartheta,\psi\rangle\geq 0$ since $\widetilde{\vartheta}_{\mu}$, $\mu\in \ZZ_+$, are positive Radon measures.\\
\indent It remains to show that $\vartheta$ satisfies \eqref{ide-for-the-ofgermars}. Let $A\in\Psi^0_{\phg,c,L^c}(M;E;E^{\#})$ and set $\hat{a}:=\sigma^0(A)\in\Gamma_c(\pi_{S^*M}^*L(E,E^{\#})_{[L^c]})$. Fix an open neighbourhood $\mathcal{U}$ of the zero section in $T^*M$ such that $\mathcal{U}\cap \pi^{-1}_{T^*M}(K)$ is relatively compact for any $K\subset\subset M$ and if $(p,\xi)\not\in\mathcal{U}$, then $(p,t\xi)\not\in\mathcal{U}$, for all $t>1$; e.g. $\mathcal{U}:=\bigcup_{\mu\in\ZZ_+}\{(p,\xi_l dx^l_{\mu}|_p)\in T^*O_{\mu}\,|\, \chi_{\mu}(p)>0,\, \sum_{j=1}^m \xi_j^2<1\}$. We modify the pullback of $\hat{a}$ under the natural map $T^*M\backslash0\rightarrow S^*M$ in $\mathcal{U}$ so as to become a smooth section $a:T^*M\rightarrow \pi^*_{T^*M}L(E,E^{\#})$. Then $a\in S^0_{\loc}(T^*M;E,E^{\#})$ and $a\in\boldsymbol{\sigma}^0(A)$. For each $\mu\in\ZZ_+$, we have $a_{|T^*O_{\mu}}=\kappa^*_{\mu}(a_{\mu,l,j})\pi^*_{T^*M}(s'^j_{\mu}\otimes \sigma^l_{\mu})$ with $a_{\mu,l,j}\in S^0_{\loc}(x_{\mu}(O_{\mu})\times \RR^m)$, $l,j=1,\ldots,k$, and $(\pi^*_{T^*M}(s'^j_{\mu}\otimes \sigma^l_{\mu}))_{j,l}$ is the pullback local frame for $\pi_{T^*M}^*L(E,E^{\#})$ over $T^*O_{\mu}$. Of course, $a_{\mu,l,j}$ is positively homogeneous of degree $0$ outside of $\kappa_{\mu}(\pi_{T^*M}^{-1}(O_{\mu})\cap \mathcal{U})$. For each $\mu\in\ZZ_+$, there are $\{\widetilde{a}_{\mu,l,j}\}_{l,j}\subseteq S^0_{\loc}(x_{\mu}(O_{\mu})\times \RR^m)$ such that $\widetilde{a}_{\mu,l,j}$, $l,j=1,\ldots,k$, are of order $-\infty$ in a conic neighbourhood of every point of $L_{\mu}$ and $\widetilde{T}_{\mu,l,j}\in\Psi^{-\infty}(x_{\mu}(O_{\mu}))$, $l,j=1,\ldots,k$, such that
$$
(A\varphi)_{|O_{\mu}}=\Op(\widetilde{a}_{\mu,l,j})(\varphi^j\circ x_{\mu}^{-1})\circ x_{\mu}\,\sigma^l_{\mu}+ \widetilde{T}_{\mu,l,j}(\varphi^j\circ x_{\mu}^{-1})\circ x_{\mu}\,\sigma^l_{\mu},\quad \varphi=\varphi^js_{\mu,j}\in\Gamma_c(E_{O_{\mu}}).
$$
Then $a_{\mu,l,j}-\widetilde{a}_{\mu,l,j}\in S^{-1}_{\loc}(x_{\mu}(O_{\mu})\times\RR^m)$, which, in view of the properties of $a_{\mu,l,j}$, implies that $a_{\mu,l,j}=0$ in $W_{\mu}\backslash\kappa_{\mu}(\pi_{T^*M}^{-1}(O_{\mu})\cap \mathcal{U})$ for some open conic neighbourhood $W_{\mu}$ of $L_{\mu}$ in $x_{\mu}(O_{\mu})\times(\RR^m\backslash\{0\})$. We infer
$$
(A\varphi)_{|O_{\mu}}=\Op(a_{\mu,l,j})(\varphi^j\circ x_{\mu}^{-1})\circ x_{\mu}\,\sigma^l_{\mu}+ \widetilde{Q}_{\mu,l,j}(\varphi^j\circ x_{\mu}^{-1})\circ x_{\mu}\,\sigma^l_{\mu},\quad \varphi=\varphi^js_{\mu,j}\in\Gamma_c(E_{O_{\mu}}),
$$
with $\widetilde{Q}_{\mu,l,j}\in\Psi^{-1}(x_{\mu}(O_{\mu}))$ and $\widetilde{Q}_{\mu,l,j}$ is of order $-\infty$ in $L_{\mu}$, $l,j=1,\ldots,k$. Denote
$$
\Lambda_0:=\{\mu\in\ZZ_+\,|\, O_{\mu}\cap \pi_{S^*M}(\supp\hat{a})\neq \emptyset\,\, \mbox{or}\,\, O_{\mu}\cap \pr_1(\supp(\mbox{kernel of}\, A))\neq\emptyset\};
$$
$\Lambda_0$ is finite. Notice that
$$
(\iota_{0,E}(A v_{n_h}),v_{n,h})=\sum_{\mu\in\Lambda_0} (\iota_{0,E}(\chi_{\mu}A(\chi'_{\mu} v_{n_h})),v_{n_h})+\sum_{\mu\in\Lambda_0}(\iota_{0,E}(\chi_{\mu}A((1-\chi'_{\mu})v_{n_h})),v_{n_h}).
$$
Since the kernel of $A$ is smooth outside of the diagonal, the operator $\psi\mapsto \chi_{\mu}A((1-\chi'_{\mu})\psi)$ has smooth compactly supported kernel and thus it is a continuous map $\DD'(M;E)\rightarrow \Gamma_c(E^{\#})$. Consequently, the second sum tends to $0$ as $h\rightarrow \infty$. For the first sum, we infer
\begin{multline*}
\sum_{\mu\in\Lambda_0} (\iota_{0,E}(\chi_{\mu}A(\chi'_{\mu} v_{n_h})),v_{n_h})\\
=\sum_{\mu\in\Lambda_0}(\Op(\widetilde{\chi}_{\mu}a_{\mu,l,j})(\widetilde{\chi}'_{\mu}v_{n_h,\mu}^j),v_{n_h,\mu}^l)+ \sum_{\mu\in\Lambda_0}(\widetilde{\chi}_{\mu}\widetilde{Q}_{\mu,l,j}(\widetilde{\chi}'_{\mu}v_{n_h,\mu}^j),v_{n_h,\mu}^l).
\end{multline*}
In view of Corollary \ref{cor-for-ope-defondinecomsetfort}, the second sum tends to zero as $h\rightarrow\infty$. There is $b_{\mu,l,j}\in S^0_c(x_{\mu}(O_{\mu})\times \RR^m)$ such that $\Op(\widetilde{\chi}_{\mu}a_{\mu,l,j})\widetilde{\chi}'_{\mu}=\Op(b_{\mu,l,j})$. Notice that $\Op(b_{\mu,l,j})\in\Psi^0_{\phg,c,(L_{\mu})^c}(x_{\mu}(O_{\mu}))$. Since $\widetilde{\chi}_{\mu}a_{\mu,l,j}=\widetilde{\chi}_{\mu}a_{\mu,l,j}\widetilde{\chi}'_{\mu}\in \boldsymbol{\sigma}^0(\Op(b_{\mu,l,j}))$, from the way we defined $a$, we infer $\sigma^0(\Op(b_{\mu,l,j}))=\widetilde{\chi}_{\mu} \hat{a}_{\mu,l,j}\circ\hat{\kappa}_{\mu}^{-1}$, where $\hat{a}_{\mu,l,j}\in\mathcal{C}^{\infty}(\pi_{S^*M}^{-1}(O_{\mu}))$ are defined by $\hat{a}_{|\pi_{S^*M}^{-1}(O_{\mu})}=\hat{a}_{\mu,l,j}\pi^*_{S^*M}(s'^j_{\mu}\otimes\sigma^l_{\mu})$. In view of \eqref{equ-for-lim-ofopexmeaofs}, we deduce
\begin{equation}
\lim_{h\rightarrow\infty}\sum_{\mu\in\Lambda_0} (\Op(\widetilde{\chi}_{\mu}a_{\mu,l,j})(\widetilde{\chi}'_{\mu}v_{n_h,\mu}^j),v_{n_h,\mu}^l)= \sum_{\mu\in\ZZ_+}\langle\widetilde{\vartheta}^{l,j}_{\mu},\widetilde{\chi}_{\mu} \hat{a}_{\mu,l,j}\circ\hat{\kappa}_{\mu}^{-1}\rangle=\langle\vartheta,\hat{a}\rangle,
\end{equation}
which verifies \eqref{ide-for-the-ofgermars}. This completes the proof of the theorem.
\end{proof}

\begin{remark}\label{rem-for-mea-posdefonskt}
Strictly speaking, $\vartheta$ is not a measure because it acts on sections. However, once we choose coordinates and a trivialisation of the bundle it becomes a measure. To wit, the fact $\langle \vartheta,\psi\rangle\geq 0$ for all $\psi$ positive semi-definite at every point implies that any coordinate representation of $\vartheta$ is a matrix-valued positive Radon measure: Let $(x,O)$ be any coordinate chart over which $E$ trivialises via $\Phi:\pi_E^{-1}(O)\rightarrow O\times \CC^k$, let $(\pi^*_{S^*M}(s'^j\otimes\sigma^l))_{j,l}$ be the local frame for $\pi_{S^*M}^*L(E,E^{\#})$ over $O$ induced by $\Phi$ (defined as in the proof of the theorem) and denote $L_O:=\kappa(\pi_{T^*M}^{-1}(O)\cap L)$ with $\kappa$ the total local trivialisation of $T^*M$ over $O$. Then, $\widetilde{\vartheta}:=(\widetilde{\vartheta}^{l,j})_{l,j}$ is a matrix-valued positive Radon measure on $[(L_O)^c]$, where $\widetilde{\vartheta}^{l,j}\in (\mathcal{K}([(L_O)^c]))'$ is defined by $\langle \widetilde{\vartheta}^{l,j},\phi\rangle:=\langle\vartheta,\phi\circ \hat{\kappa}\, \pi^*_{S^*M}(s'^j\otimes\sigma^l)\rangle$ with $\hat{\kappa}$ as in \eqref{dif-for-cos-bunovermanc}.
\end{remark}

The support of $\vartheta$ contains all information about the directions in $[L^c]$ in which the convergence $u_{n_j}\rightarrow u$ is not in $L^2$-sense.

\begin{corollary}\label{cor-for-wav-ofdissupdefmes}
Let $L$, $u$ and $(u_n)_{n\in\ZZ_+}$ satisfy the assumptions in Theorem \ref{ger-the-com-onmanifol} and let $\vartheta$ and the subsequence $(u_{n_j})_{j\in\ZZ_+}$ be as in the claims of the theorem. Let $L^{(\vartheta)}$ be the inverse image of $\supp\vartheta$ under the natural map $T^*M\backslash0\rightarrow S^*M$. Then
\begin{equation}\label{equ-for-inc-ofwavwefreontsesk}
L^{(\vartheta)}\subseteq WF^0_c(\{u_{n_j}\}_{j\in\ZZ_+})\subseteq L\cup L^{(\vartheta)},
\end{equation}
$L\cup L^{(\vartheta)}$ is a closed conic subset of $T^*M\backslash0$ and $u_{n_j}\rightarrow u$ in $\DD'^0_{L\cup L^{(\vartheta)}}(M;E)$.
\end{corollary}

\begin{remark}
Since (by definition!) $L^{(\vartheta)}\subseteq L^c$, \eqref{equ-for-inc-ofwavwefreontsesk} is equivalent to
\begin{equation}
L^{(\vartheta)}=WF^0_c(\{u_{n_j}\}_{j\in\ZZ_+})\backslash L.
\end{equation}
Although $\vartheta$ carries similar information as the Sobolev compactness wave front set $WF^0_c$, it can be a convenient tool in practice as we show in the next subsection.
\end{remark}

\begin{proof}[Proof of Corollary \ref{cor-for-wav-ofdissupdefmes}] Corollary \ref{cor-for-relcom-sub-wafe-fronchar} shows that $\{u_{n_j}\}_{j\in\ZZ_+}$ is a relatively compact subset of $\DD'^0_{WF^0_c(\{u_{n_j}\}_{j\in\ZZ_+})}(M;E)$ and, since $u_{n_j}\rightarrow u$ in $\DD'(M;E)$,
\begin{equation}\label{ide-for-con-ofseqinspacfsgs}
u\in\DD'^0_{WF^0_c(\{u_{n_j}\}_{j\in\ZZ_+})}(M;E)\quad \mbox{and}\quad u_{n_j}\rightarrow u\quad \mbox{in}\quad \DD'^0_{WF^0_c(\{u_{n_j}\}_{j\in\ZZ_+})}(M;E).
\end{equation}
Set $L':= L\cup L^{(\vartheta)}$. The proof that $L'$ is closed in $T^*M\backslash0$ is straightforward\footnote{This fact is not redundant since $L^{(\vartheta)}$ is closed in $L^c$ but not necessarily in $T^*M\backslash0$.}. Once we show \eqref{equ-for-inc-ofwavwefreontsesk}, \eqref{ide-for-con-ofseqinspacfsgs} would imply $u_{n_j}\rightarrow u$ in $\DD'^0_{L'}(M;E)$.\\
\indent First we show the second inclusion in \eqref{equ-for-inc-ofwavwefreontsesk}. We may assume $L'^c\neq\emptyset$. Let $(p_0,\xi_0)\not\in L'$, $\xi_0\neq0$. There is a chart $(O,x)$ about $p_0$ over which $E$ locally trivialises via $\Phi_x:\pi_E^{-1}(O)\rightarrow O\times \CC^k$ and open cones $V,V'\subseteq \RR^m$ such that $\overline{V}\subseteq V'\cup\{0\}$, $\RR^m\backslash\overline{V'}\neq \emptyset$,
\begin{gather}
(p_0,\xi_0)\in \{(p,\xi_l dx^l|_p)\in T^*O\,|\, (\xi_1,\ldots,\xi_m)\in V\}\quad\mbox{and}\nonumber\\
\{(p,\xi_l dx^l|_p)\in T^*O\,|\, (\xi_1,\ldots,\xi_m)\in \overline{V'}\}\subseteq L'^c.\label{set-for-ano-incincomfortheincs}
\end{gather}
Pick $\phi,\widetilde{\phi}\in\DD(\RR^m)$ such that $0\leq \phi,\widetilde{\phi}\leq1$, $\phi=1$ on $\overline{V}\cap \mathbb{S}^{m-1}$, $\supp\phi\subseteq V'$, $\widetilde{\phi}=1$ on $\overline{B(0,1/4)}$ and $\supp\widetilde{\phi}\subseteq B(0,1/2)$. Define $b(\xi):=(1-\widetilde{\phi}(\xi))\phi(\xi/|\xi|)$; $b$ is smooth with support in $V'$, positively homogeneous of order $0$ when $|\xi|\geq 1/2$ and equal to $1$ on $\overline{V}\backslash B(0,1/2)$. Pick nonnegative $\chi,\chi'\in\DD(O)$ such that $\chi=1$ on a neighbourhood of $p_0$ and $\chi'=1$ on a neighbourhood of $\supp\chi$ and define
$$
A:\Gamma_c(E)\rightarrow \Gamma_c(E^{\#}),\, A\varphi:= \sum_{l=1}^k\chi\Op(b)((\chi' \varphi^l)\circ x^{-1})\circ x\, \sigma^l,\quad \varphi\in\Gamma_c(E),\, \varphi_{|O}=\varphi^l s_l,
$$
where $(s_1,\ldots,s_k)$ and $(\sigma^1,\ldots,\sigma^k)$ are the local frames for $E$ and $E^{\#}$ over $O$ induced by $\Phi_x$. Clearly, $A\in\Psi^0_{\phg,c,L'^c}(M;E,E^{\#})$ and $(p_0,\xi_0)\not\in\Char A$. We also define
$$
A':\Gamma_c(E)\rightarrow \Gamma_c(E),\, A'\varphi:= \sum_{l=1}^k\chi\Op(b)((\chi' \varphi^l)\circ x^{-1})\circ x\, s_l,\quad \varphi\in\Gamma_c(E),\, \varphi_{|O}=\varphi^l s_l.
$$
Of course, $A'\in\Psi^0_{\phg,c,L'^c}(M;E,E)$ and $(p_0,\xi_0)\not\in\Char A'$. We claim that
\begin{equation}\label{equ-for-nor-inlfsdohsk}
(\iota_{0,E}(Av),A'v)=\sum_{l=1}^k \|\chi\circ x^{-1}\Op(b)((\chi'\circ x^{-1}) v^l_{\Phi_x})\|_{L^2(\RR^m)}^2,\quad v\in \DD'^0_L(M;E),
\end{equation}
with $v^l_{\Phi_x}\in \DD'^0_{L_O}(x(O))$ as in \eqref{ind-tri-loc-nestk} and $L_O$ as in Remark \ref{rem-for-mea-posdefonskt} (cf. Proposition \ref{lem-for-top-imebdofmapscs}). Both the left and the right hand side are well-defined in view of Corollary \ref{cor-for-ope-defondinecomsetfort} and the continuous inclusions $\EE'^0_{L'^c}(M;E)\subseteq L^2_{\comp}(M;E)\subseteq \DD'^0_{L'}(M;E)$; notice that the right-hand side is the square of a norm on $L^2_{\supp\chi}(M;E)$ of $A'v\in L^2_{\supp\chi}(M;E)$. By plugging in the definitions of $A$ and $A'$, it is easy to see that \eqref{equ-for-nor-inlfsdohsk} hods true for $v\in\Gamma_c(E)$. The general case follows from the fact that $\Gamma_c(E)$ is sequentially dense in $\DD'^0_L(M;E)$ (Proposition \ref{res-for-den-ope-map-thagivdes}) and Corollary \ref{cor-for-ope-defondinecomsetfort}. Denoting $v_{n_j}:=u_{n_j}-u$, $j\in\ZZ_+$, Theorem \ref{ger-the-com-onmanifol} gives
$$
(\iota_{0,E}(Av_{n_j}),A'v_{n_j})=(\iota_{0,E}(A'^*Av_{n_j}),v_{n_j})\rightarrow\langle \vartheta,\sigma^0(A'^*A)\rangle=\langle \vartheta,\sigma^0(A')^{\#}\sigma^0(A)\rangle=0
$$
since $\supp\sigma^0(A)\cap \supp\vartheta=\emptyset=\supp\sigma^0(A')\cap\supp\vartheta$. Employing \eqref{equ-for-nor-inlfsdohsk}, we deduce that $A'v_{n_j}\rightarrow 0$ in $L^2_{\comp}(M;E)$ and thus $WF^0_c(A'(\{u_{n_j}\}_{j\in\ZZ_+}))=\emptyset$ in view of Corollary \ref{cor-for-relcom-sub-wafe-fronchar}. Corollary \ref{cor-for-ide-ofwafornes} now yields $WF^0_c(\{u_{n_j}\}_{j\in\ZZ_+})\subseteq \Char A'$ and thus $(p_0,\xi_0)\not\in WF^0_c(\{u_{n_j}\}_{j\in\ZZ_+})$.\\
\indent To show the first inclusion in \eqref{equ-for-inc-ofwavwefreontsesk}, let $(p_0,\xi_0)\not\in WF^0_c(\{u_{n_j}\}_{j\in\ZZ_+})$, $\xi_0\neq 0$. If $(p_0,\xi_0)\in L$, then $(p_0,\xi_0)\not\in L^{(\vartheta)}$. Assume that $(p_0,\xi_0)\not\in L'':= L\cup WF^0_c(\{u_{n_j}\}_{j\in\ZZ_+})$. We claim that $\langle \vartheta,a\rangle=0$, $a\in\Gamma_c(\pi^*_{S^*M}L(E,E^{\#})_{[L''^c]})$; this would imply that $(p_0,\xi_0)\not\in L^{(\vartheta)}$ (by density). For $a\in\Gamma_c(\pi_{S^*M}^*L(E,E^{\#})_{[L''^c]})$, pick $A\in\Psi^0_{\phg,c,L''^c}(M;E,E^{\#})$ such that $\sigma^0(A)=a$. Corollary \ref{cor-for-ope-defondinecomsetfort} together with \eqref{ide-for-con-ofseqinspacfsgs} yield that $Au_{n_j}\rightarrow Au$ in $\EE'^0_{L''^c}(M;E^{\#})$ and Theorem \ref{ger-the-com-onmanifol} gives $\langle\vartheta,a\rangle=0$.
\end{proof}

\begin{remark}
When $L=\emptyset$, \eqref{equ-for-inc-ofwavwefreontsesk} boils down to $WF^0_c(\{u_{n_j}\}_{j\in\ZZ_+})= L^{(\vartheta)}$; this result was shown in \cite{G1}. Even in the case $L=\emptyset$, the corollary improves this result by identifying that $u_{n_j}\rightarrow u$ in $\DD'^0_{L^{(\vartheta)}}(M;E)$. In this case, the result is optimal in the following sense: if $u_{n_j}\rightarrow u$ in $\DD'^0_{L'}(M;E)$ for some closed conic subset $L'$ of $T^*M\backslash0$, then Corollary \ref{cor-for-relcom-sub-wafe-fronchar} together with this result yield that $L^{(\vartheta)}\subseteq L'$.
\end{remark}

\subsection{The compensated compactness theorem}

We are now going to employ the microlocal defect measures from Theorem \ref{ger-the-com-onmanifol} to generalise the G\'erard-Murat-Tartar theorem on compensated compactness \cite{mur1,mur2,Tar-1,Ger}. We start with a technical lemma.

\begin{lemma}\label{lem-for-seq-ofnschfortheks}
Let $L$, $u$ and $(u_n)_{n\in\ZZ_+}$ satisfy the assumptions in Theorem \ref{ger-the-com-onmanifol} and let $\vartheta$ and $(u_{n_j})_{j\in\ZZ_+}$ be as in the claims of the theorem. For every $a\in \Gamma_c(\pi^*_{S^*M}L(E,E^{\#})_{[L^c]})$, there exists a subsequence of $(u_{n_j})_{j\in\ZZ_+}$, denoted again by $(u_{n_j})_{j\in\ZZ_+}$, such that for every $A\in\Psi^0_{\phg,c,L^c}(M;E,E^{\#})$ satisfying $\sigma^0(A)=a$ it holds that
\begin{equation}\label{ide-exi-lim-forpspssk}
\lim_{j\rightarrow\infty} (\iota_{0,E}(Au),u_{n_j})=(\iota_{0,E}(Au),u)\quad\mbox{and}\quad \lim_{j\rightarrow\infty}(\iota_{0,E}(Au_{n_j}),u_{n_j})\in \CC.
\end{equation}
\end{lemma}

\begin{proof} Let $a\in \Gamma_c(\pi^*_{S^*M}L(E,E^{\#})_{[L^c]})$ and pick $A\in\Psi^0_{\phg,c,L^c}(M;E,E^{\#})$ so that $\sigma^0(A)=a$. In view of Corollary \ref{cor-for-ope-defondinecomsetfort}, $\sup_{n\in\ZZ_+}|(\iota_{0,E}(Au),u_n)|<\infty$ and $\sup_{n\in\ZZ_+}|(\iota_{0,E}(Au_n),u_n)|<\infty$ so there is a subsequence of $(u_{n_j})_{j\in\ZZ_+}$ from Theorem \ref{ger-the-com-onmanifol}, which we again denote by $(u_{n_j})_{j\in\ZZ_+}$, such that both $\lim_{j\rightarrow\infty} (\iota_{0,E}(Au),u_{n_j})$ and $\lim_{j\rightarrow\infty}(\iota_{0,E}(Au_{n_j}),u_{n_j})$ exist in $\CC$. If $A'\in\Psi^0_{\phg,c,L^c}(M;E,E^{\#})$ is any other $\Psi$DO satisfying $\sigma^0(A')=a$, then both of the limits exist for $A'$ with the same subsequence in view of Corollary \ref{cor-for-ope-defondinecomsetfort} employed with $A'-A\in\Psi^{-1}_{\phg,c,L^c}(M;E,E^{\#})$ (for $(\iota_{0,E}(A'u_{n_j}),u_{n_j})$, it is straightforward to show that $(\iota_{0,E}((A'-A)u_{n_j}),u_{n_j})\rightarrow (\iota_{0,E}((A'-A)u),u)$). It remains to show the first identity in \eqref{ide-exi-lim-forpspssk}. Take a sequence $\{\psi_l\}_{l\in\ZZ_+}\subseteq \Gamma_c(E)$ which converges to $u$ in $\DD'^0_L(M;E)$ (cf. Proposition \ref{res-for-den-ope-map-thagivdes}). Notice that
\begin{align*}
\operatorname{Re}(\iota_{0,E}(Au),u_{n_j})&=\operatorname{Re}(\iota_{0,E}(A(u-\psi_l)),u_{n_j})+ \operatorname{Re}(\iota_{0,E}(A\psi_l),u_{n_j})\\
&\leq \sup_{n\in\ZZ_+}|(\iota_{0,E}(A(u-\psi_l)),u_n)|+\operatorname{Re}(\iota_{0,E}(A\psi_l),u_{n_j}).
\end{align*}
Hence $\lim_{j\rightarrow\infty}\operatorname{Re}(\iota_{0,E}(Au),u_{n_j})\leq \sup_{n\in\ZZ_+}|(\iota_{0,E}(A(u-\psi_l)),u_n)|+\operatorname{Re}(\iota_{0,E}(A\psi_l),u)$ and consequently (cf. Corollary \ref{cor-for-ope-defondinecomsetfort}) $\lim_{j\rightarrow\infty}\operatorname{Re}(\iota_{0,E}(Au),u_{n_j})\leq \operatorname{Re}(\iota_{0,E}(Au),u)$. Doing the same for $-A$ and $\operatorname{Im}(\iota_{0,E}(\pm Au),u_{n_j})$, we deduce the first identity in \eqref{ide-exi-lim-forpspssk}.
\end{proof}

For a closed conic subset $L$ of $T^*M\backslash0$ and $r\in\RR$, we define the l.c.s. $H^r_{L;\loc}(M;E):=H^r_{\loc}(M;E)\cap \DD'_L(M;E)$ equipped with the topology induced by all continuous seminorms on $H^r_{\loc}(M;E)$ and $\DD'_L(M;E)$. Clearly, $H^r_{L;\loc}(M;E)$ is complete. It is continuously and densely included in both $H^r_{\loc}(M;E)$ and $\DD'_L(M;E)$ (since $\Gamma_c(E)$ is dense in both spaces) and $H^r_{T^*M\backslash0;\loc}(M;E)=H^r_{\loc}(M;E)$ and $H^r_{\emptyset;\loc}(M;E)=\Gamma(E)$. For $\chi\in \Gamma_c(DM)$ fixed, set $K:=\supp\chi$ and notice that the following map is well-defined and continuous:
\begin{equation}\label{equ-for-con-opeofmaponsespansk}
H^r_{L;\loc}(M;E^*)\rightarrow \EE'^r_{L\cap\pi^{-1}_{T^*M}(K);K}(M;E^{\#}),\quad u\mapsto \chi u.
\end{equation}

\begin{proposition}\label{pro-for-bul-mapextprodofdisonspk}
Let $r\in\RR$ and let $L$ and $\widetilde{L}$ be two closed conic subsets of $T^*M\backslash0$ which satisfy $\widetilde{L}\subseteq L^c$. Then
$$
\mathcal{P}:H^{-r}_{\widetilde{L};\loc}(M;E^*)\times \DD'^r_L(M;E)\rightarrow \DD'(M),\quad \langle\mathcal{P}(u,v),\chi\rangle:=\langle \chi u,\iota_E(v)\rangle,\, \chi\in\Gamma_c(DM),
$$
where the last duality is $\langle\EE'^{-r}_{L^c}(M;E^{\#}), \DD'^r_{\check{L}}(M;E^{\#\,\vee})\rangle$, is well-defined hypocontinuous sesquilinear map that restricts to the sesquilinear map
\begin{equation}\label{map-for-csm-foreqlsphkrcvk1}
\Gamma(E^*)\times\Gamma(E)\rightarrow \mathcal{C}^{\infty}(M),\quad (\varphi,\psi)\mapsto (p\mapsto \varphi_p(\psi_p)).
\end{equation}
\end{proposition}

\begin{proof} In view of \eqref{equ-for-con-opeofmaponsespansk}, $\chi u\in\EE'^{-r}_{L^c}(M;E^{\#})$ and hence $\langle \chi u,\iota_E(v)\rangle$ is well-defined since $\iota_E(v)\in\DD'^r_{\check{L}}(M;E^{\#\, \vee})$. To show the continuity of $\mathcal{P}(u,v):\Gamma_c(DM)\rightarrow \CC$ it suffices to show that it maps bounded sets into bounded sets since $\Gamma_c(DM)$ is bornological. Fix a bounded subset $B$ of $\Gamma_c(DM)$. There is $K\subset\subset M$ such that $B$ is a bounded subset of $\Gamma_K(DM)$. Pick relatively compact charts $(O_{\mu},x_{\mu})$, $\mu=1,\ldots,l$, which cover $K$ and nonnegative $\varphi_{\mu}\in\DD(O_{\mu})$ such that $\sum_{\mu=1}^l\varphi_{\mu}=1$ on a neighbourhood of $K$ and choose $\varphi'_{\mu}\in\DD(O_{\mu})$ so that $\varphi'_{\mu}=1$ on a neighbourhood of $\varphi_{\mu}$. For $\chi\in B$, write $\chi_{|O_{\mu}}=\chi_{\mu}\lambda^{x_{\mu}}$, $\chi_{\mu}\in\mathcal{C}^{\infty}(O_{\mu})$, and notice that
\begin{equation}\label{ine-for-psi-forsepinchaonscosk1}
\sup_{\chi\in B}|\langle \chi u,\iota_E(v)\rangle|\leq \sum_{\mu=1}^l\sup_{\chi\in B}|\langle \varphi_{\mu} \lambda^{x_{\mu}}u,\varphi'_{\mu}\chi_{\mu}\iota_E(v)\rangle|.
\end{equation}
Since $\{(\varphi'_{\mu}\chi_{\mu})\iota_E(v)\,|\, \chi\in B\}$ is a bounded subset of $\DD'^r_{\check{L}}(M;E^{\#\,\vee})$ (cf. Remark \ref{hyp-con-rem-formanbundlcaseofcon}) and $(\varphi_{\mu}\lambda^{x_{\mu}}) u\in \EE'^{-r}_{L^c}(M;E^{\#})$, we infer that the right hand side of \eqref{ine-for-psi-forsepinchaonscosk1} is finite. This shows that $\mathcal{P}(u,v)\in\DD'(M)$ and hence $\mathcal{P}$ is well-defined. To show that $\mathcal{P}$ is hypocontinuous with respect to the first variable, fix a bounded subset $B_2$ of $\DD'^r_L(M;E)$. Let $B$ be a bounded subset of $\Gamma_c(DM)$. We want to estimate $\sup_{v\in B_2}\sup_{\chi\in B}|\langle \chi u,\iota_E(v)\rangle|$ by a continuous seminorm of $u$ in $H^{-r}_{\widetilde{L};\loc}(M;E^*)$. With the above notation, \eqref{ine-for-psi-forsepinchaonscosk1} is valid for all $u\in H^{-r}_{\widetilde{L};\loc}(M;E^*)$ and $v\in B_2$. Since $\{(\varphi'_{\mu}\chi_{\mu})\iota_E(v)\,|\, \chi\in B,\, v\in B_2\}$ is bounded in $\DD'^r_{\check{L}}(M;E^{\#\,\vee})$ (cf. Remark \ref{hyp-con-rem-formanbundlcaseofcon}), $\sup_{v\in B_2} \sup_{\chi\in B}|\langle \varphi_{\mu} \lambda^{x_{\mu}}u,\varphi'_{\mu}\chi_{\mu}\iota_E(v)\rangle|$ is a continuous seminorm of $\varphi_{\mu} \lambda^{x_{\mu}}u$ in $\EE'^{-r}_{L^c}(M;E^{\#})$ and \eqref{equ-for-con-opeofmaponsespansk} implies that it is bounded by a continuous seminorm of $u$ in $H^{-r}_{\widetilde{L};\loc}(M;E^*)$. To show the hypocontinuity of $\mathcal{P}$ with respect to the second variable, fix a bounded subset $B_1$ of $H^{-r}_{\widetilde{L};\loc}(M;E^*)$. For a bounded subset $B$ of $\Gamma_c(DM)$, we again write \eqref{ine-for-psi-forsepinchaonscosk1}. We apply \eqref{equ-for-con-opeofmaponsespansk} and Lemma \ref{lem-for-equ-subofeddofboonsukls} to deduce that $\sup_{u\in B_1}\sup_{\chi\in B}|\langle \varphi_{\mu} \lambda^{x_{\mu}}u,\varphi'_{\mu}\chi_{\mu}\iota_E(v)\rangle|\leq C_{\mu}\sup_{\chi\in B}\mathfrak{p}_{\mu}(\varphi'_{\mu}\chi_{\mu}\iota_E(v))$ with $\mathfrak{p}_{\mu}$ a continuous seminorm on $\DD'^r_{\check{L}}(M;E^{\#\,\vee})$. Now, Remark \ref{hyp-con-rem-formanbundlcaseofcon} yields the desired hypocontinuity. By direct inspection, one verifies that $\mathcal{P}$ restricts to \eqref{map-for-csm-foreqlsphkrcvk1} and the proof is complete.
\end{proof}

\begin{remark}\label{rem-for-exi-ofprodofdisindwiespa}
If $Q\in\Psi^0(M;E,E^*)$ is properly supported and of order $-\infty$ in the closed conic subset $L$ of $T^*M\backslash0$, Proposition \ref{pro-for-psudomapforcontonvectbac} and Corollary \ref{cor-for-hwf-spadtoplincs} imply that $Q:\DD'^{-r}_L(M;E)\rightarrow H^{-r}_{\widetilde{L};\loc}(M;E^*)$ is well-defined and continuous for some closed conic subset $\widetilde{L}$ of $T^*M\backslash0$ satisfying $\widetilde{L}\subseteq L^c$. Hence, Proposition \ref{pro-for-bul-mapextprodofdisonspk} shows that
$$
\DD'^{-r}_L(M;E)\times\DD'^r_L(M;E)\rightarrow \DD'(M),\quad (u,v)\mapsto \mathcal{P}(Qu,v),
$$
is well-defined hypocontinuous sesquilinear map that restricts to $\Gamma(E)\times\Gamma(E)\rightarrow\mathcal{C}^{\infty}(M)$, $(\varphi,\psi)\mapsto (p\mapsto (Q\varphi)_p(\psi_p))$.
\end{remark}

\begin{remark}
Proposition \ref{pro-for-bul-mapextprodofdisonspk} can be viewed as an extension of the multiplication to distributions. To be precise, the proposition gives a sesquilinear product. The bilinear version is the following: for any $r\in\RR$ and $L$ and $\widetilde{L}$ two closed conic subsets of $T^*M\backslash0$ which satisfy $\widetilde{L}\subseteq \check{L}^c$, the bilinear mapping
$$
\mathcal{P}_0:H^{-r}_{\widetilde{L};\loc}(M;E')\times \DD'^r_L(M;E)\rightarrow \DD'(M),\quad \langle\mathcal{P}_0(u,v),\chi\rangle:=\langle \chi u,v\rangle,\, \chi\in\Gamma_c(DM),
$$
where the last duality is $\langle\EE'^{-r}_{\check{L}^c}(M;E^{\vee}), \DD'^r_L(M;E)\rangle$, is well-defined and hypocontinuous and it restricts to $\Gamma(E')\times\Gamma(E)\rightarrow \mathcal{C}^{\infty}(M)$, $(\varphi,\psi)\mapsto (p\mapsto \varphi_p(\psi_p))$. This can be shown in an analogous way as in the proof of Proposition \ref{pro-for-bul-mapextprodofdisonspk}. Notice that this result is neither weaker nor stronger than the H\"ormander theorem for product of distributions \cite[Theorem 8.2.10, p. 267]{hor}, \cite[Theorem 6.1]{D1}, since one of the spaces is larger and the other smaller when compared with the H\"ormander theorem. It is passible that a similar result can be obtained from our theorem on the pull-back (Theorem \ref{mai-the-pul-forvecbundm}) by applying similar technique as in \cite[Theorem 8.2.10, p. 267]{hor} and with that to obtain information on the Sobolev wave front set of the product (cf. \cite[Theorem 8.3.3, p. 190]{hor1}). We leave this to the reader, as we are not going to need such result.
\end{remark}

We are ready to show our generalisation of the compensated compactness theorem. Our result on the microlocal defect measures together with the theory we developed before, allows us to adjust the main idea from the proof of \cite[Theorem 2]{Ger} to our needs.

\begin{theorem}\label{com-cpc-the-formanbunflskslt}
Let $E$ and $F$ be two vector bundles over $M$ of rank $k$ and $k'$, $L$ a closed conic subset of $T^*M\backslash0$ satisfying $L\neq T^*M\backslash0$ and $a\in\Gamma_{\hom,r}(\pi^*_{T^*M}L(E,F)_{L^c})$ for some $r\in\RR$. Let $\{u_n\}_{n\in\ZZ_+}$ be a bounded subset of $\DD'^0_L(M;E)$ such that $(u_n)_{n\in\ZZ_+}$ converges in $\DD'(M;E)$ to some $u\in\DD'^0_L(M;E)$. Assume that for every $(p,\xi)\in L^c$ there is an open conic set $W\subseteq L^c$ containing it and a properly supported $A\in\Psi^r_{\phg}(M;E,F)$ such that $\{Au_n\}_{n\in\ZZ_+}$ is relatively compact in $\DD'^{-r}_L(M;F)$ and $\widetilde{\sigma}^r(A)_{|W}=a_{|W}$. Let $b\in\Gamma(\pi^*_{S^*M}L(E,E^*))$ satisfies $\supp b\subseteq [L^c]$ and let $B\in\Psi^0_{\phg}(M;E,E^*)$ be properly supported, of order $-\infty$ in $L$ and $\sigma^0(B)=b$.
\begin{itemize}
\item[$(i)$] Assume that $b=b^*$ and the following implication holds true:
$$
\forall (p,\xi)\in L^c,\, \forall e\in E_p,\quad a(p,\xi)(e)=0\,\,\Longrightarrow\,\, b(p,[\xi])(e)(e)\geq0.
$$
Then, for $\chi\in\Gamma_c(DM)$ which is nonnegative at every point, it holds that
$$
\liminf_{n\rightarrow \infty}\operatorname{Re}\langle \mathcal{P}(Bu_n,u_n),\chi\rangle\geq \operatorname{Re}\langle \mathcal{P}(Bu,u),\chi\rangle,\,\, \lim_{n\rightarrow \infty}\operatorname{Im}\langle \mathcal{P}(Bu_n,u_n),\chi\rangle= \operatorname{Im}\langle \mathcal{P}(Bu,u),\chi\rangle.
$$
\item[$(ii)$] Assume the following implication holds true:
$$
\forall (p,\xi)\in L^c,\, \forall e\in E_p,\quad a(p,\xi)(e)=0\,\,\Longrightarrow\,\, b(p,[\xi])(e)(e)=0.
$$
Then $\mathcal{P}(Bu_n,u_n)\rightarrow \mathcal{P}(Bu,u)$ in $\DD'(M)$.
\end{itemize}
\end{theorem}

\begin{proof} We start by pointing out that for all properly supported $Q\in\Psi^0_{\phg}(M;E,E^*)$ which are of order $-\infty$ in $L$ and all $\chi\in\Gamma_c(DM)$, we have $\chi Q\in\Psi^0_{\phg,c,L^c}(M;E,E^{\#})$.\\
\indent We claim that $(ii)$ follows from $(i)$. To see this, set $b_1:=(b+b^*)/2$ and $b_2:=(b-b^*)/(2i)$ and pick properly supported $B_1,B_2\in\Psi^0_{\phg}(M;E,E^*)$ which are of order $-\infty$ in $L$ and $\sigma^0(B_j)=b_j$, $j=1,2$. In view of the assumption in $(ii)$, we can apply $(i)$ with $\pm b_1$ and $\pm b_2$. Applying it with $b_1$ and $-b_1$ we infer that for every $\chi\in\Gamma_c(DM)$ which is nonnegative at every point it holds that
\begin{equation}\label{lim-for-ide-ofmeaweithcsvk}
\lim_{n\rightarrow \infty}\langle \mathcal{P}(B_1u_n,u_n),\chi\rangle=\langle \mathcal{P}(B_1u,u),\chi\rangle.
\end{equation}
When $\chi\in\Gamma_c(DM)$ is real-valued, pick $\chi'\in\Gamma_c(DM)$ which is nonnegative at every point and such that $\chi'-\chi$ is nonnegative at every point and apply \eqref{lim-for-ide-ofmeaweithcsvk} with $\chi'$ and $\chi'-\chi$ to deduce that \eqref{lim-for-ide-ofmeaweithcsvk} is valid also for $\chi$. Employing this to the real and imaginary part of general $\chi\in\Gamma_c(DM)$, we deduce $\mathcal{P}(B_1u_n,u_n)\rightarrow \mathcal{P}(B_1u,u)$ in $\DD'(M)$ since every weakly convergent sequence is strongly convergent in $\DD'(M)$. Arguing analogously for $b_2$, we deduce the claim in $(ii)$ for $B_1+iB_2$. Since for each $\chi\in\Gamma_c(DM)$, $\chi(B_1+iB_2-B)\in\Psi^{-1}_{\phg,c,L^c}(M;E,E^{\#})$, we can invoke Corollary \ref{cor-for-ope-defondinecomsetfort} to deduce the claim for $B$.\\
\indent We now show $(i)$. Since
$$
\operatorname{Im}\langle \mathcal{P}(Bu_n,u_n),\chi\rangle=\operatorname{Im}(\iota_{0,E}(\chi Bu_n),u_n)=(2i)^{-1}\left(\iota_{0,E}((\chi B-(\chi B)^*)u_n),u_n\right)
$$
and $\chi B-(\chi B)^*\in\Psi^{-1}_{\phg,c,L^c}(M;E,E^{\#})$, Corollary \ref{cor-for-ope-defondinecomsetfort} verifies the second identity in $(i)$. To show the first inequality, pick a subsequence $(u_{n_j})_{j\in\ZZ_+}$ such that
$$
\liminf_{n\rightarrow \infty}\operatorname{Re}\langle \mathcal{P}(Bu_n,u_n),\chi\rangle=\lim_{j\rightarrow \infty}\operatorname{Re}\langle \mathcal{P}(Bu_{n_j},u_{n_j}),\chi\rangle.
$$
We apply Theorem \ref{ger-the-com-onmanifol} to $(u_{n_j})_{j\in\ZZ_+}$ to find $\vartheta\in (\Gamma^0_c(\pi_{S^*M}^*L(E,E^{\#})_{[L^c]}))'$ and Lemma \ref{lem-for-seq-ofnschfortheks} to $(u_{n_j})_{j\in\ZZ_+}$ and $\chi b$ to extract a subsequence of it, still denoted by $(u_{n_j})_{j\in\ZZ_+}$, which satisfies all of the properties stated in the lemma\footnote{We point out that $\vartheta$ depends on the subsequence chosen for the $\liminf$ and hence it depends on $B$ and $\chi$, however, this will not matter since we will only work with the last extracted subsequence.}. Since $\sigma^0((\chi B)^*)=(\chi b)^{\#}=\chi b=\sigma^0(\chi B)$ and \eqref{ide-for-the-ofgermars} is satisfied with $\chi B$, we infer
\begin{align*}
\langle\vartheta,\chi b\rangle=\operatorname{Re}\langle\vartheta,\chi b\rangle&=\lim_{j\rightarrow\infty}\Big(\operatorname{Re}(\iota_{0,E}(\chi Bu_{n_j}),u_{n_j})-\operatorname{Re}(\iota_{0,E}(\chi Bu),u_{n_j})\\
&{}\quad-\operatorname{Re}\overline{\left(\iota_{0,E}((\chi B)^*u),u_{n_j}\right)}+\operatorname{Re}(\iota_{0,E}(\chi Bu),u)\Big)\\
&=\lim_{j\rightarrow\infty}\operatorname{Re}(\iota_{0,E}(\chi Bu_{n_j}),u_{n_j})-\operatorname{Re}(\iota_{0,E}(\chi Bu),u).
\end{align*}
Thus, it suffices to show that $\langle\vartheta,\chi b\rangle\geq 0$; we show that this holds for all $\chi\in\Gamma_c(DM)$ which are nonnegative at every point. By employing a partition of unity, we see that it is enough to prove this only for such $\chi$ which also satisfy $\supp\chi \subseteq O$ where $(O,x)$ is a relatively compact chart over which $E$ and $F$ locally trivialise via $\Phi$ and $\Phi'$. Let $L_O$ be as in Remark \ref{rem-for-mea-posdefonskt}. If $(L_O)^c=\emptyset$, then the claim is trivial since $\chi b=0$. Assume that $(L_O)^c\neq \emptyset$. Write $\chi:=\widetilde{\chi}\circ x\,\lambda^x$, $a_{|L^c\cap T^*O}=\kappa^*(\widetilde{a}^l_j)\pi^*_{T^*M}(e'^j\otimes s_l)$ and $b_{|\pi^{-1}_{S^*M}(O)}=\hat{\kappa}^*(\widetilde{b}_{l,j})\pi^*_{S^*M}(e'^j\otimes e^{*\,l})$ with $\widetilde{\chi}\in\DD(x(O))$, $\widetilde{a}^l_j\in\mathcal{C}^{\infty}((L_O)^c)$ and $\widetilde{b}_{l,j}\in\mathcal{C}^{\infty}(x(O)\times\mathbb{S}^{m-1})$ and $(e_1,\ldots,e_k)$ and $(s_1,\ldots,s_{k'})$ the local frames for $E$ and $F$ over $O$ induced by $\Phi$ and $\Phi'$ and $(e'^1,\ldots,e'^k)$ and $(e^{*\,1},\ldots,e^{*\,k})$ the corresponding dual and anti-dual frames for $E'$ and $E^*$ over $O$; of course, $\kappa$ and $\hat{\kappa}$ are defined as in \eqref{loc-tri-ofb-undimatotcotbusks} and \eqref{dif-for-cos-bunovermanc}. Set $\widetilde{a}:=(\widetilde{a}^l_j)_{l,j}$ and $\widetilde{b}:=(\widetilde{b}_{l,j})_{l,j}$ and let $\widetilde{\vartheta}:=(\widetilde{\vartheta}^{l,j})_{l,j}$ be as in Remark \ref{rem-for-mea-posdefonskt}. Pick nonnegative $\widetilde{\psi}\in\DD([(L_O)^c])$ such that $\widetilde{\psi}=1$ on a neighbourhood of $(\supp\widetilde{\chi}\times \mathbb{S}^{m-1})\cap \supp\widetilde{b}$ and set $\psi:=\hat{\kappa}^*(\widetilde{\psi})\in\DD([L^c]\cap\pi^{-1}_{S^*M}(O))$. By employing the assumption in $(i)$, in the same way as in the proof of \cite[Lemma 2.3]{Ger}, one can show that for every $\varepsilon>0$ there is $C_{\varepsilon}>0$ such that
$$
\widetilde{\varphi}_{\varepsilon}: [(L_O)^c]\rightarrow\M_k(\CC),\quad \widetilde{\varphi}_{\varepsilon}(t,\omega):= \widetilde{b}(t,\omega)+C_{\varepsilon}\widetilde{a}(t,\omega)^*\widetilde{a}(t,\omega)+\varepsilon I,
$$
with $I$ the identity $k\times k$ matrix, is smooth and positive semi-definite at every point of $\supp\widetilde{\psi}$. The properties of $\widetilde{\vartheta}$ imply
$$
0\leq \langle \widetilde{\vartheta},\widetilde{\chi}\widetilde{\psi}\widetilde{\varphi}_{\varepsilon}\rangle=\langle \vartheta,\chi b\rangle+C_{\varepsilon}\langle\widetilde{\vartheta}, \widetilde{a}^*\widetilde{a}\widetilde{\psi}\widetilde{\chi}\rangle+\varepsilon\langle\widetilde{\vartheta}, \widetilde{\chi}\widetilde{\psi}I\rangle.
$$
We claim that $\langle\widetilde{\vartheta}, \widetilde{a}^*\widetilde{a}\widetilde{\psi}\widetilde{\chi}\rangle=0$. Once we show this, the inequality in $(i)$ follows by letting $\varepsilon\rightarrow0^+$ in the above inequality. We employ the assumption in the theorem to find open conic sets $W_1,\ldots, W_d\subseteq L^c\cap T^*O$ satisfying $\supp\psi\subseteq \bigcup_{\mu=1}^d [W_{\mu}]$ and properly supported $A_1,\ldots,A_d\in\Psi^r_{\phg}(M;E,F)$ such that $a_{\mu}:=\widetilde{\sigma}^r(A_{\mu})$ coincides with $a$ on $W_{\mu}$, $\mu=1,\ldots,d$. Pick nonnegative $\psi_{\mu}\in\DD([W_{\mu}])$ such that $\sum_{\mu=1}^d\psi_{\mu}^2=1$ on $\supp\psi$. For each $\mu\in\{1,\ldots,d\}$, define $\phi_{\mu}(t,\xi):=|\xi|^{-r}\psi_{\mu}\circ\hat{\kappa}^{-1}(t,\xi/|\xi|)$, $(t,\xi)\in x(O)\times(\RR^m\backslash\{0\})$. Notice that $\sum_{j=1}^{k'}\kappa^*(\phi_{\mu})\pi^*_{T^*M}(s'^j\otimes s_j)\in\Gamma_{\hom,-r}(\pi^*_{T^*M}L(F,F)_{T^*M\backslash0})$ and pick properly supported $B_{\mu}\in\Psi^{-r}_{\phg}(M;F,F)$ such that $\widetilde{\sigma}^{-r}(B_{\mu})$ is this section; of course, $(s'^1,\ldots,s'^{k'})$ is the dual frame for $F'$ induced by $\Phi'$. Finally, choose $B_0\in\Psi^0_{\phg,c,L^c}(M;F,F^{\#})$ such that $\sigma^0(B_0)=\sum_{j=1}^{k'}\psi\pi^*_{S^*M}(s'^j\otimes \sigma^j)$ with $(\sigma^1,\ldots,\sigma^{k'})$ the frame for $F^{\#}$ induced by $\Phi'$. Then $B_{\mu}A_{\mu}\in\Psi^0_{\phg}(M;E,F)$ and $(B_{\mu}A_{\mu})^*B_0(B_{\mu}A_{\mu})\in\Psi^0_{\phg,c,L^c}(M;E,E^{\#})$. Notice that
$$
\sigma^0(B_{\mu}A_{\mu})=\psi_{\mu}a_{\mu,l}^j\circ\kappa^{-1}\circ \hat{\kappa}\,\pi^*_{S^*M}(e'^l\otimes s_j),\quad \mbox{where}\quad a_{\mu\,|T^*O\backslash0}=a_{\mu,l}^j\pi^*_{T^*M}(e'^l\otimes s_j).
$$
Hence, denoting by $(\epsilon^1,\ldots,\epsilon^k)$ the frame for $E^{\#}$ induced by $\Phi$, we infer
$$
\sigma^0((B_{\mu}A_{\mu})^*B_0(B_{\mu}A_{\mu}))=\psi\psi_{\mu}^2\sum_{h=1}^{k'} (\overline{a^h_{\mu,j}}a^h_{\mu,l})\circ\kappa^{-1}\circ \hat{\kappa}\, \pi^*_{S^*M}(e'^l\otimes \epsilon^j).
$$
The assumptions in the theorem together with Corollary \ref{cor-for-ope-defondinecomsetfort} imply that $\{(B_{\mu}A_{\mu})^*B_0(B_{\mu}A_{\mu})u_n\}_{n\in\ZZ_+}$ is relatively compact in $\EE'^0_{L^c}(M;E^{\#})$ and consequently
\begin{align*}
0&=\lim_{j\rightarrow\infty}\sum_{\mu=1}^d\left(\iota_{0,E}\left((\widetilde{\chi}\circ x)(B_{\mu}A_{\mu})^*B_0(B_{\mu}A_{\mu})(u_{n_j}-u)\right),u_{n_j}-u\right)\\
&=\sum_{\mu=1}^d\sum_{h=1}^{k'}\left\langle\vartheta,(\widetilde{\chi}\circ x)\psi\psi_{\mu}^2(\overline{\widetilde{a}^h_j}\widetilde{a}^h_l)\circ \hat{\kappa}\, \pi^*_{S^*M}(e'^l\otimes \epsilon^j)\right\rangle= \sum_{h=1}^{k'}\left\langle\widetilde{\vartheta}^{j,l},\widetilde{\chi}\widetilde{\psi}\overline{\widetilde{a}^h_j}\widetilde{a}^h_l \right\rangle=\langle\widetilde{\vartheta},\widetilde{\chi}\widetilde{\psi}\widetilde{a}^*\widetilde{a}\rangle.
\end{align*}
This completes the proof of the theorem.
\end{proof}

We end the article with a typical application of the compensated compactness theorem to $\DD'$-sequential continuity of quadratic forms on the set of solutions of second order PDEs; see \cite{G1,Ger,jik-koz-ole,Tar1}.

\begin{example}\label{exa-for-equ-withsecordvecfonsk}
As before, $M$ is an $m$-dimensional manifold. Let $\mathcal{V}_1,\ldots,\mathcal{V}_n\in\Gamma(TM)$ (they are real-valued by definition!). Let $A_{j,l}\in\Psi^0_{\phg}(M)$, $j,l=1,\ldots,n$, and $A_1\in\Psi^1(M)$ be properly supported. Let $v,v^{(k)},f^{(k)}\in \DD'(M)$, $k\in\ZZ_+$, satisfy
$$
\sum_{j,l=1}^n \mathcal{V}_jA_{j,l}\mathcal{V}_l v^{(k)}+A_1v^{(k)}=f^{(k)},\,\, k\in\ZZ_+,\quad\mbox{and}\quad v^{(k)}\rightarrow v\,\,\mbox{in}\,\, \DD'(M).
$$
Denote $L_0:=\{(p,\xi)\in T^*M\backslash0\,|\, \xi(\mathcal{V}_j(p))=0,\, j=1,\ldots,n\}$; clearly $L_0$ is a closed conic subset of $T^*M\backslash0$. Let $L_1$ and $L_2$ be closed conic subset of $T^*M\backslash0$ which satisfy $WF^{-1}_c(\{f^{(k)}\}_{k\in\ZZ_+})\subseteq L_1$, $v\in\DD'^1_{L_2}(M)$ and $\{v^{(k)}\}_{k\in\ZZ_+}$ is a bounded subset of $\DD'^1_{L_2}(M)$. Setting $L:=L_0\cup L_1\cup L_2$, we claim that for every properly supported $B_0\in\Psi^0_{\phg}(M)$ which is of order $-\infty$ in $L$, it holds that
\begin{align}
\sum_{j,l=1}^n\mathcal{P}(A_{j,l}B_0(\mathcal{V}_lv^{(k)}),\mathcal{V}_jv^{(k)})&\rightarrow \sum_{j,l=1}^n\mathcal{P}(A_{j,l}B_0(\mathcal{V}_lv),\mathcal{V}_jv),\,\, \mbox{as}\,\, k\rightarrow\infty\,\, \mbox{in}\,\, \DD'(M);\label{equ-for-con-ofgralikoffoks}\\
\sum_{j,l=1}^n\mathcal{P}(B_0A_{j,l}(\mathcal{V}_lv^{(k)}),\mathcal{V}_jv^{(k)})&\rightarrow \sum_{j,l=1}^n\mathcal{P}(B_0A_{j,l}(\mathcal{V}_lv),\mathcal{V}_jv),\,\, \mbox{as}\,\, k\rightarrow\infty\,\, \mbox{in}\,\, \DD'(M).\label{equ-for-con-ofgralikoffoks1111}
\end{align}
Before we show the claims, we point out several things. First, if $L=T^*M\backslash 0$, the claims are trivial since in this case $B_0$ is a properly supported $\Psi$DO in $\Psi^{-\infty}(M)$; of course this always holds since we can always take $L_1=L_2=T^*M\backslash0$, but the result is meaningless! When this is not the case, $L\neq T^*M\backslash 0$ imposes real restrictions on $\{f_k\}_{k\in\ZZ_+}$ and $\{v_k\}_{k\in\ZZ_+}$ and the claims state that the convergence will hold once we employ a $\Psi$DO that removes the bad part of $T^*M$: the annihilators of the spaces spanned by the vector fields at every point\footnote{If $M$ has a (pseudo-)Riemannian metric, these can be identified with the orthogonal complements of the spaces spanned by the vector fields.} (the set $L_0$), the place where $\{f_k\}_{k\in\ZZ_+}$ are too singular or do not behave as a relatively compact subset of the Sobolev- $-1$ distributions (the set $L_1$, cf. Corollary \ref{cor-for-relcom-sub-wafe-fronchar}) and the place where $v$ and $\{v_k\}_{k\in\ZZ_+}$ are too singular or do not behave like a bounded subset of the Sobolev- $1$ distributions (the set $L_2$). Finally, $A_1 v^{(k)}$ is irrelevant since it can be absorbed in the right hand side in view of Remark \ref{rem-for-bou-compsetforrellem}.\\
\indent Notice that \eqref{equ-for-con-ofgralikoffoks1111} follows from \eqref{equ-for-con-ofgralikoffoks} in view of Corollary \ref{cor-for-ope-defondinecomsetfort} and the definition of $\mathcal{P}$ since $\chi(B_0A_{j,l}-A_{j,l}B_0)\in\Psi^{-1}_{\phg,c,L^c}(M;\CC_M,DM)$, for all $\chi\in\Gamma_c(DM)$. To prove \eqref{equ-for-con-ofgralikoffoks}, as we pointed out, we may assume $L\neq T^*M\backslash0$. Set $a_{j,l}:=\sigma^0(A_{j,l})\in \mathcal{C}^{\infty}(S^*M)$. Let $B_0$ be as in the claim and denote $b_0:=\sigma^0(B_0)\in \mathcal{C}^{\infty}(S^*M)$. Notice that $\supp b_0\subseteq [L^c]$. Set $u^{(k)}:=(\mathcal{V}_1v^{(k)},\ldots,\mathcal{V}_nv^{(k)})\in\DD'^0_L(M;\CC_M^n)$, $k\in\ZZ_+$, and $u:=(\mathcal{V}_1v,\ldots,\mathcal{V}_nv)\in\DD'^0_L(M;\CC_M^n)$. Consider the properly supported $\Psi$DO $\widetilde{A}\in\Psi^1_{\phg}(M;\CC_M^n,\CC_M^{(1+n(n-1)/2)})$ defined as follows: the terms in the last row in the $(1+n(n-1)/2)\times n$ matrix of $\widetilde{A}$ are $\sum_{j=1}^n \mathcal{V}_jA_{j,1},\ldots,\sum_{j=1}^n\mathcal{V}_j A_{j,n}$ and all of the previous rows are $(0,\ldots,0,\mathcal{V}_j,0,\ldots,0,-\mathcal{V}_l,0,\ldots,0)$ where $\mathcal{V}_j$ is on $l$-th place and $-\mathcal{V}_l$ is on $j$-th place, $1\leq l<j\leq n$. We are going to apply Theorem \ref{com-cpc-the-formanbunflskslt} with this $\widetilde{A}$, $\widetilde{a}:=\widetilde{\sigma}^1(\widetilde{A})$ and $b:=(b_0a_{j,l})_{j,l}\in\mathcal{C}^{\infty}(S^*M;\M_n(\CC))$; clearly $\supp b\subseteq [L^c]$. We check that the conditions are satisfied. Notice that $\widetilde{A} u^{(k)}$ is given as follows: the last entry is $f^{(k)}-A_1v^{(k)}$, while the previous are $[\mathcal{V}_j,\mathcal{V}_l]v^{(k)}$, $1\leq l<j\leq n$. The assumptions together with Corollary \ref{cor-for-relcom-sub-wafe-fronchar} and Remark \ref{rem-for-bou-compsetforrellem} imply that $\{f^{(k)}-A_1v^{(k)}\}_{k\in\ZZ_+}$ is a relatively compact subset of $\DD'^{-1}_L(M)$. Since $[\mathcal{V}_j,\mathcal{V}_l]\in\Psi^1_{\phg}(M)$, Remark \ref{rem-for-bou-compsetforrellem} also implies that $\{[\mathcal{V}_j,\mathcal{V}_l]v^{(k)}\}_{k\in\ZZ_+}$ is relatively compact in $\DD'^{-1}_L(M)$. Consequently, $\{\widetilde{A}u^{(k)}\}_{k\in\ZZ_+}$ is relatively compact in $\DD'^{-1}_L(M;\CC_M^{1+n(n-1)/2})$. Let $(p,\xi)\in L^c$ and $z=(z^1,\ldots,z^n)\in\CC^n$ be such that $\widetilde{a}(p,\xi)z=0$. Pick a chart $(O,x)$ about $p$ and write $\mathcal{V}_j=\mathcal{V}^l_j\frac{\partial}{\partial x^l}$ and $\xi=\xi_ldx^l|_p$. Set $c_j(p,\xi):=\mathcal{V}^l_j(p)\xi_l$ and notice that $\widetilde{a}(p,\xi)z=0$ implies that $c_j(p,\xi)z^l=c_l(p,\xi)z^j$, $j\neq l$, and $\sum_{j,l=1}^n a_{j,l}(p,[\xi])c_j(p,\xi)z^l=0$. At least one $c_j(p,\xi)\neq 0$ since $(p,\xi)\not\in L_0$. Consequently, the first equations imply $z^j=\zeta c_j(p,\xi)$, $j=1,\ldots,n$, for some $\zeta\in\CC$. If $\zeta\neq0$, plugging this in the last equation, we infer $\sum_{j,l=1}^n a_{j,l}(p,[\xi])z^l\overline{z^j}=0$ (recall, $c_j(p,\xi)\in\RR$, $j=1,\ldots,n$), i.e. $b(p,[\xi])z\cdot \overline{z}=0$. When $\zeta=0$, we have $z=0$ and hence $b(p,[\xi])z\cdot \overline{z}=0$. Thus, we can apply Theorem \ref{com-cpc-the-formanbunflskslt} to deduce \eqref{equ-for-con-ofgralikoffoks}.
\end{example}

In a similar way, one can apply Theorem \ref{com-cpc-the-formanbunflskslt} to generalise other results where the standard compensated compactness theorem plays a key role. We leave such investigations for future research and conclude the article with the following consequence of Example \ref{exa-for-equ-withsecordvecfonsk}.

\begin{example}
Let $(M,g)$ be a pseudo-Riemannian manifold. Let $A\in\Psi^0_{\phg}(M;TM\otimes \CC,TM\otimes\CC)$ and $A_1\in\Psi^1(M)$ be properly supported and let $v,v^{(k)},f^{(k)}\in \DD'(M)$, $k\in\ZZ_+$, satisfy
\begin{equation}\label{equ-for-con-onquadratiformonsol}
\operatorname{div}_g(A(\operatorname{grad}_g v^{(k)}))+A_1v^{(k)}=f^{(k)},\,\, k\in\ZZ_+,\quad \mbox{and}\quad v^{(k)}\rightarrow v\,\, \mbox{in}\,\, \DD'(M).
\end{equation}
Let $L_1$ and $L_2$ be closed conic subsets of $T^*M\backslash0$ which satisfy $WF^{-1}_c(\{f^{(k)}\}_{k\in\ZZ_+})\subseteq L_1$, $v\in\DD'^1_{L_2}(M)$ and $\{v^{(k)}\}_{k\in\ZZ_+}$ is a bounded subset of $\DD'^1_{L_2}(M)$. Setting $L:=L_1\cup L_2$, we claim that for every properly supported $B\in\Psi^0_{\phg}(M;TM\otimes \CC,TM\otimes \CC)$ which is of order $-\infty$ in $L$ and such that $\sigma^0(B)=b\operatorname{I}_{TM\otimes \CC}$ with $b\in\mathcal{C}^{\infty}(S^*M)$, it holds that
\begin{equation}\label{equ-for-con-ofgralikoffoks11}
\mathcal{P}(BA(\operatorname{grad}_gv^{(k)}),dv^{(k)})\rightarrow \mathcal{P}(BA(\operatorname{grad}_gv),dv),\quad \mbox{as}\,\, k\rightarrow\infty\,\, \mbox{in}\,\, \DD'(M).
\end{equation}
Notice that all terms are well-defined elements of $\DD'(M)$ in view of Proposition \ref{pro-for-bul-mapextprodofdisonspk} (cf. Remark \ref{rem-for-exi-ofprodofdisindwiespa}) since the anti-dual bundle of $T^*M\otimes\CC$ is canonically identified with $TM\otimes\CC$. As before, we may assume that $L\neq T^*M\backslash0$ since the claim is trivial if $L=T^*M\backslash0$. Let $B$ and $b$ be as in the claim; clearly $\supp b\subseteq [L^c]$. Pick properly supported $B'\in\Psi^0_{\phg}(M)$ such that $\sigma^0(B')=b$ and is of order $-\infty$ in $L$. Let $\{(O_{\mu},x_{\mu})\}_{\mu\in\ZZ_+}$ be a locally finite cover of $M$ of relatively compact charts. Let $(\varphi_{\mu})_{\mu\in\ZZ_+}$ be a partition of unity subordinated to this cover and choose $\varphi'_{\mu},\varphi''_{\mu}\in\DD(O_{\mu})$ such that $\varphi'_{\mu}=1$ on a neighbourhood of $\supp\varphi_{\mu}$ and $\varphi''_{\mu}=1$ on a neighbourhood of $\supp\varphi'_{\mu}$. Notice that the operator $B'_{\mu}:\DD(O_{\mu})\rightarrow \DD(O_{\mu})$, $B'_{\mu}\psi:= \varphi_{\mu} B'(\varphi'_{\mu}\psi)$, belongs to $\Psi^0_{\phg,c}(O_{\mu})$ and is of order $-\infty$ in $T^*O_{\mu}\cap L$. Furthermore, the operator $A^j_{\mu,l}:\DD(O_{\mu})\rightarrow \DD(O_{\mu})$, $A^j_{\mu,l}(\psi):=dx^j_{\mu}(\varphi'_{\mu}A(\varphi''_{\mu}\psi\frac{\partial}{\partial x^l_{\mu}}))$, belongs to $\Psi^0_{\phg,c}(O_{\mu})$. Write $g_{|O_{\mu}}=g_{\mu,j,l}dx_{\mu}^j\otimes dx_{\mu}^l$, $|g_{\mu}|:=|\det (g_{\mu,j,l})_{j,l}|$ and denote by $(g^{j,l}_{\mu})_{j,l}$ the inverse of $(g_{\mu,j,l})_{j,l}$. Notice that
\begin{equation*}
\operatorname{div}_g\left(\varphi'_{\mu}A(\varphi''_{\mu}\operatorname{grad}_gv^{(k)})\right)= \operatorname{div}_g(\varphi'_{\mu}A(\operatorname{grad}_gv^{(k)})) - \operatorname{div}_g\left(\varphi'_{\mu}A((1-\varphi''_{\mu})\operatorname{grad}_gv^{(k)})\right).
\end{equation*}
The set $\{\operatorname{div}_g(\varphi'_{\mu}A((1-\varphi''_{\mu})\operatorname{grad}_gv^{(k)}))\}_{k\in\ZZ_+}$ is relatively compact in $\DD'^{-1}_L(M)$ since $\varphi'_{\mu}A(1-\varphi''_{\mu})$ has smooth compactly supported kernel. For $u\in \DD'(M;TM\otimes\CC)$, we employ the notation $u^j_{\mu}\in\DD'(O_{\mu})$, $\langle u^j_{\mu},\varphi\rangle:=\langle u,\varphi dx^j_{\mu}\rangle$, $\varphi\in\Gamma_c(DO_{\mu})$. Notice that
$$
\operatorname{div}_g(\varphi'_{\mu}A(\operatorname{grad}_gv^{(k)})) = \frac{\partial\varphi'_{\mu}}{\partial x_{\mu}^j} \left(A(\operatorname{grad}_g v^{(k)})\right)^j_{\mu} + \varphi'_{\mu}\operatorname{div}_g(A(\operatorname{grad}_gv^{(k)})),\,\, \mbox{on}\,\, O_{\mu}.
$$
Hence, in view of the assumptions and Remark \ref{rem-for-bou-compsetforrellem}, we conclude that $\{\operatorname{div}_g(\varphi'_{\mu}A(\varphi''_{\mu}\operatorname{grad}_gv^{(k)}))\}_{k\in\ZZ_+}$ is a relatively compact subset of $\DD'^{-1}_{T^*O_{\mu}\cap L}(O_{\mu})$. Since
$$
\operatorname{div}_g(\varphi'_{\mu}A(\varphi''_{\mu}\operatorname{grad}_gv^{(k)}))= |g_{\mu}|^{-1/2}\frac{\partial}{\partial x_{\mu}^j}\left(|g_{\mu}|^{1/2}A^j_{\mu,l}\left(g_{\mu}^{l,h}\frac{\partial}{\partial x_{\mu}^h}v^{(k)}\right)\right),\quad \mbox{on}\quad O_{\mu},
$$
we infer that $\frac{\partial}{\partial x_{\mu}^j}\left(A^j_{\mu,l}(g_{\mu}^{l,h}\frac{\partial}{\partial x_{\mu}^h}v^{(k)})\right)$ is a relatively compact subset of $\DD'^{-1}_{T^*O_{\mu}\cap L}(O_{\mu})$. Hence, we can apply the claim \eqref{equ-for-con-ofgralikoffoks1111} from Example \ref{exa-for-equ-withsecordvecfonsk} with $B'_{\mu}$ to deduce
$$
\mathcal{P}\left(B'_{\mu}A^j_{\mu,l}\left(g^{l,h}_{\mu}\frac{\partial}{\partial x_{\mu}^h}v^{(k)}\right),\frac{\partial}{\partial x_{\mu}^j}v^{(k)}\right)\rightarrow \mathcal{P}\left(B'_{\mu}A^j_{\mu,l}\left(g^{l,h}_{\mu}\frac{\partial}{\partial x_{\mu}^h}v\right),\frac{\partial}{\partial x_{\mu}^j}v\right),\,\, \mbox{in}\,\, \DD'(O_{\mu}).
$$
Consider the $\Psi$DO $\widetilde{B}_{\mu}:\Gamma_c(TO_{\mu}\otimes\CC)\rightarrow \Gamma_c(TO_{\mu}\otimes\CC)$, $\widetilde{B}_{\mu}(\chi^j\frac{\partial}{\partial x_{\mu}^j})=B'_{\mu}(\chi^j)\frac{\partial}{\partial x_{\mu}^j}$. Of course, it is in $\Psi^0_{\phg,c}(O_{\mu};TO_{\mu}\otimes\CC, TO_{\mu}\otimes \CC)$ and is of order $-\infty$ in $T^*O_{\mu}\cap L$. In view of the definition of $\mathcal{P}$, we infer
\begin{equation}\label{equ-for-con-ofpsonsmalsetbigsetkls}
\mathcal{P}\left(\widetilde{B}_{\mu}\left(\varphi'_{\mu}A(\varphi''_{\mu}\operatorname{grad}_gv^{(k)})\right), \varphi'_{\mu}dv^{(k)}\right)\rightarrow \mathcal{P}\left(\widetilde{B}_{\mu}\left(\varphi'_{\mu}A(\varphi''_{\mu}\operatorname{grad}_gv)\right), \varphi'_{\mu}dv\right),\,\,\mbox{in}\,\, \DD'(O_{\mu}).
\end{equation}
Since $\sigma^0(\varphi_{\mu} B)=\sigma^0(\widetilde{B}_{\mu})$ in $T^*O_{\mu}$, we have $\varphi_{\mu} B-\widetilde{B}_{\mu}\in\Psi^{-1}_{\phg,c}(O_{\mu};TO_{\mu}\otimes\CC,TO_{\mu}\otimes\CC)$ and is of order $-\infty$ in $T^*O_{\mu}\cap L$. Hence, Corollary \ref{cor-for-ope-defondinecomsetfort} implies that \eqref{equ-for-con-ofpsonsmalsetbigsetkls} holds true but with $\varphi_{\mu} B$ in place of $\widetilde{B}_{\mu}$. The operator $\varphi_{\mu}B(1-\varphi'_{\mu})$ has smooth compactly supported kernel, whence
$$
\mathcal{P}(\varphi_{\mu}BA(\varphi''_{\mu}\operatorname{grad}_gv^{(k)}),\varphi'_{\mu}dv^{(k)})\rightarrow \mathcal{P}(\varphi_{\mu}BA(\varphi''_{\mu}\operatorname{grad}_gv),\varphi'_{\mu}dv),\quad\mbox{in}\quad \DD'(M).
$$
Similarly, as $\varphi_{\mu}BA(1-\varphi''_{\mu})$ has a smooth compactly supported kernel, we infer
$$
\mathcal{P}(\varphi_{\mu}BA(\operatorname{grad}_gv^{(k)}),\varphi'_{\mu}dv^{(k)})\rightarrow \mathcal{P}(\varphi_{\mu}BA(\operatorname{grad}_gv),\varphi'_{\mu}dv),\quad\mbox{in}\quad \DD'(M).
$$
In view of the definition of $\mathcal{P}$, this immediately implies \eqref{equ-for-con-ofgralikoffoks11}.
\end{example}

In the special case when $A=\operatorname{Id}$, \eqref{equ-for-con-onquadratiformonsol} becomes
\begin{equation}\label{equ-for-con-onquadratiformonsol11}
\Delta_gv^{(k)}+A_1v^{(k)}=f^{(k)},\,\, k\in\ZZ_+,\quad \mbox{and}\quad v^{(k)}\rightarrow v\,\, \mbox{in}\,\, \DD'(M),
\end{equation}
where $\Delta_g=\operatorname{div}_g\operatorname{grad}_g$ is the geometric Laplacian. In this case the claim is
\begin{equation}\label{equ-for-con-ofgralikoffoks22}
\mathcal{P}(B(\operatorname{grad}_gv^{(k)}),dv^{(k)})\rightarrow \mathcal{P}(B(\operatorname{grad}_gv),dv),\quad \mbox{as}\,\, k\rightarrow\infty\,\, \mbox{in}\,\, \DD'(M).
\end{equation}
For real valued $\psi$, notice that $\mathcal{P}(\operatorname{grad}_g\psi,d\psi)=g(\operatorname{grad}_g\psi,\operatorname{grad}_g\psi)$. Hence, in the special case of a $1+3$ Lorentzian manifold, the quadratic form in \eqref{equ-for-con-ofgralikoffoks22} is a pseudo-differential modification of the Lagrangian of $v$ without an external potential.

\appendix
\section{The optimality of the conditions in Theorem \ref{the-pul-bac-for-smcrmdiff}}\label{app-for-cou-foroptthpulbacsmom}

\begin{proposition}
Let $n,m,k\in \ZZ_+$ be such that $k\leq m$ and $k<n$ and define the following linear map of rank $k$:
$$
f:\RR^m\rightarrow \RR^n,\quad f(x_1,\ldots,x_m)=(x_1,\ldots,x_k,0,\ldots,0).
$$
\begin{itemize}
\item[$(i)$] The map $f^*:\mathcal{C}^{\infty}(\RR^n)\rightarrow \mathcal{C}^{\infty}(\RR^m)$ does not extend to a continuous mapping $f^*:H^{r_2}_{\loc}(\RR^n)\rightarrow\DD'(\RR^m)$ for any $r_2<(n-k)/2$.
\item[$(ii)$] The map $f^*:\mathcal{C}^{\infty}(\RR^n)\rightarrow \mathcal{C}^{\infty}(\RR^m)$ does not extend to a continuous mapping $f^*:H^{r_2}_{\loc}(\RR^n)\rightarrow\DD'^{r_1}_{\widetilde{L}}(\RR^m)$ if $r_2-r_1<(n-k)/2$ for any closed conic subset $\widetilde{L}$ of $\RR^m\times (\RR^m\backslash\{0\})$ which satisfies
    \begin{equation}\label{con-for-non-exofesxonconsofr}
    \left(\{0_m\}\times ((\RR^k\times\{0_{m-k}\})\backslash\{0_m\})\right)\backslash \widetilde{L}\neq \emptyset.
    \end{equation}
    In particular, $f^*$ does not extend to a continuous mapping $f^*:\DD'^{r_2}_L(\RR^n)\rightarrow\DD'^{r_1}_{f^*L}(\RR^m)$ if $r_2-r_1<(n-k)/2$ for any closed conic subset $L$ of $\RR^n\times (\RR^n\backslash\{0\})$ which satisfies $L\cap \mathcal{N}_f=\emptyset$ and $\left(\{0_n\}\times (\{\eta'_0\}\times\RR^{n-k})\right)\cap L= \emptyset$ for some $\eta'_0\in\RR^k\backslash\{0\}$.
\end{itemize}
\end{proposition}

\begin{proof} Throughout the proof, for $q\in\ZZ_+$, we denote by $\delta_q$ the $\delta$-distribution on $\RR^q$. We first make the following preliminary observation. For $s>0$, we claim that the tempered distribution $\langle D\rangle^{-s}\delta_q=\mathcal{F}^{-1}(\langle \cdot\rangle^{-s})\in\SSS'(\RR^q)$ is given by
\begin{equation}\label{def-off-inv-ftrdedltns}
\mathcal{F}^{-1}(\langle \cdot\rangle^{-s})(y)=\frac{1}{\Gamma(s/2)2^q\pi^{q/2}}\int_0^{\infty}t^{(s-q-2)/2}e^{-t} e^{-|y|^2/(4t)} dt,\quad y\in\RR^q\backslash\{0\}.
\end{equation}
Denote by $u^{(q)}_s(y)$ the right-hand side of \eqref{def-off-inv-ftrdedltns}. We first show that $u^{(q)}_s$ is smooth outside of the origin. For $|y|\geq \varepsilon>0$, we have
\begin{equation}\label{ine-for-the-funinsofderd}
\frac{|y|^2}{4t}+t\geq \frac{|y|^2}{8t}+\frac{t}{2}+\frac{|y|^2}{8t}+\frac{t}{2}\geq \frac{\varepsilon^2}{8t}+\frac{t}{2}+\frac{|y|}{2}.
\end{equation}
Since the derivatives with respect to $y$ of the integrand in \eqref{def-off-inv-ftrdedltns} are finite sums of terms of the form $Cy^lt^{\lambda}e^{-t} e^{-|y|^2/(4t)}$ for some $l\in\NN$, $C,\lambda\in\RR$, \eqref{ine-for-the-funinsofderd} implies that $u^{(q)}_s\in\mathcal{C}^{\infty}(\RR^q\backslash\{0\})$; \eqref{ine-for-the-funinsofderd} also yields that $0<u^{(q)}_s(y)\leq C_{\varepsilon}e^{-|y|/2}$, $|y|\geq \varepsilon$. Notice that $(t,y)\mapsto t^{(s-q-2)/2}e^{-t} e^{-|y|^2/(4t)}$ belongs to $L^1(\RR_+\times\RR^q)$ (integrate first with respect to $y$). This yields $u^{(q)}_s\in L^1(\RR^q)$ and, in view of the fact $\mathcal{F}(e^{-|\cdot|^2})=\pi^{q/2}e^{-|\cdot|^2/4}$, we infer
$$
\mathcal{F}(u^{(q)}_s)(\xi)=\frac{1}{\Gamma(s/2)}\int_0^{\infty}t^{s/2-1} e^{-(1+|\xi|^2)t} dt=\langle \xi\rangle^{-s},\quad \xi\in\RR^q,
$$
which shows \eqref{def-off-inv-ftrdedltns}. Throughout the proof we continue to employ the notation $u^{(q)}_s$ for $\mathcal{F}^{-1}(\langle\cdot\rangle^{-s})$. We claim that
\begin{equation}\label{far-for-bel-insoff}
\varphi u_s^{(q)}\in\SSS(\RR^q)\,\, \mbox{for any}\,\, \varphi\in\DD_{L^{\infty}}(\RR^q)\,\, \mbox{satisfying}\,\, \varphi=0\,\, \mbox{on a neighbourhood of}\,\, 0.
\end{equation}
To verify it, notice that $\partial^{\alpha}(\varphi u_s^{(q)})$ is a finite sum of terms of the form
$$
Cy^l\partial^{\beta}\varphi(y)\int_0^{\infty}t^{\lambda}e^{-t} e^{-|y|^2/(4t)} dt,\quad \mbox{for some}\,\, \beta\leq \alpha,\, l\in\NN,\, C,\lambda\in\RR.
$$
Hence, \eqref{ine-for-the-funinsofderd} implies $\varphi u^{(q)}_s\in\SSS(\RR^q)$. Finally, we point out that $u^{(q)}_s\in\mathcal{C}(\RR^q)$ when $s>q$.\\
\indent Throughout the rest of the proof, for $x\in\RR^m$, we denote $x=(x',x'')$, with $x'\in\RR^k$ and $x''\in\RR^{m-k}$. Similarly, for $y\in\RR^n$, we denote $y=(y',y''')$, with $y'\in\RR^k$ and $y'''\in\RR^{n-k}$. We first address $(i)$. Pick any $\varphi\in\DD(\RR^k)\backslash\{0\}$. By construction, $\varphi\otimes u^{(n-k)}_{n-k}\in H^{r_2}(\RR^n)$ since $r_2<(n-k)/2$. Choose nonnegative $\psi\in\DD(\RR^{n-k})$ such that $\psi(0)>0$ and $\int_{\RR^{n-k}}\psi(y''')d y'''=1$ and define $\psi_j(y''')=j^{n-k}\psi(jy''')$, $y'''\in\RR^{n-k}$, $j\in\ZZ_+$. Clearly $\varphi\otimes (u^{(n-k)}_{n-k}*\psi_j)\rightarrow \varphi\otimes u^{(n-k)}_{n-k}$ as $j\rightarrow \infty$ in $H^{r_2}(\RR^n)$ and hence in $H^{r_2}_{\loc}(\RR^n)$ as well. Notice that $f^*(\varphi\otimes (u^{(n-k)}_{n-k}*\psi_j))=(u^{(n-k)}_{n-k}*\psi_j)(0)\varphi\otimes \mathbf{1}_{\RR^{m-k}}$. Since
$$
u^{(n-k)}_{n-k}*\psi_j(0)=\frac{1}{\Gamma((n-k)/2)2^{n-k}\pi^{(n-k)/2}}\iint_{\RR_+\times\RR^{n-k}}t^{-1}e^{-t} e^{-|y'''|^2/(4tj^2)} \psi(y''')dtdy'''
$$
and the sequence of functions in the integral is pointwise increasing with respect to $j$, we can apply monotone convergence to deduce $u^{(n-k)}_{n-k}*\psi_j(0)\rightarrow \infty$ as $j\rightarrow \infty$. Hence $f^*(\varphi\otimes (u^{(n-k)}_{n-k}*\psi_j))$ does not converge in $\DD'(\RR^m)$ and the proof of $(i)$ is complete.
\indent We turn our attention to $(ii)$. Assume that $f^*$ extends to a continuous mapping $f^*:H^{r_2}_{\loc}(\RR^n)\rightarrow\DD'^{r_1}_{\widetilde{L}}(\RR^m)$ for some $r_1$, $r_2$ and $\widetilde{L}$ as in $(ii)$. Then $(i)$ implies $r_2\geq (n-k)/2$, which, in view of $r_2-r_1<(n-k)/2$, implies $r_1>0$. Pick $r'_2>r_2$ such that $r'_2-r_1<(n-k)/2$. Clearly, $u^{(n)}_{r'_2+n/2}\in H^{r_2}(\RR^n)$. Let $(\psi_j)_{j\in\ZZ_+}$ be a $\delta$-sequence as before but defined on $\RR^n$ instead. Of course, $\psi_j* u^{(n)}_{r'_2+n/2}\rightarrow u^{(n)}_{r'_2+n/2}$ in $H^{r_2}(\RR^n)$. We claim that $f^*(\psi_j*u^{(n)}_{r'_2+n/2})\rightarrow c_0u^{(k)}_{r'_2+k-n/2}\otimes \mathbf{1}_{\RR^{m-k}}$ in $\DD'(\RR^m)$ with
$$
c_0:=\Gamma(r'_2/2+k/2-n/4)\Gamma(r'_2/2+n/4)^{-1}2^{k-n}\pi^{(k-n)/2}>0
$$
and, by assumption, in $\DD'^{r_1}_{\widetilde{L}}(\RR^m)$ as well. Let $\varphi\in\DD(\RR^m)$ be arbitrary and notice that
\begin{multline*}
\langle f^*(\psi_j* u^{(n)}_{r'_2+n/2}),\varphi\rangle=\frac{1}{\Gamma(r'_2/2+n/4)2^n\pi^{n/2}}\\
\cdot\iiint_{\RR_+\times\RR^m\times\RR^n}t^{(r'_2-\frac{n}{2}-2)/2}e^{-t} e^{-|x'|^2/(4t)}e^{-|y'''|^2/(4tj^2)} \psi(y)\varphi(x'+y'/j,x'')dtdxdy.
\end{multline*}
It is straightforward to check that the integrand is dominated pointwise by a function in $L^1(\RR_+\times \RR^{m+n})$ for all $j$ and hence we can apply dominated convergence to deduce
$$
\lim_{j\rightarrow\infty}\langle f^*(\psi_j* u^{(n)}_{r'_2+n/2}),\varphi\rangle=\frac{1}{\Gamma(r'_2/2+n/4)2^n\pi^{n/2}} \iint_{\RR_+\times\RR^m}t^{(r'_2-\frac{n}{2}-2)/2}e^{-t} e^{-|x'|^2/(4t)}\varphi(x)dtdx.
$$
This shows that $f^*(\psi_j* u^{(n)}_{r'_2+n/2})\rightarrow c_0u^{(k)}_{r'_2+k-n/2}\otimes \mathbf{1}_{\RR^{m-k}}$ weakly in $\DD'(\RR^m)$ and hence also in the strong topology (since $\DD(\RR^m)$ is Montel)\footnote{When $r'_2>n/2$, $u^{(n)}_{r'_2+n/2}$ is continuous and the fact $f^* u^{(n)}_{r'_2+n/2}= c_0u^{(k)}_{r'_2+k-n/2}\otimes \mathbf{1}_{\RR^{m-k}}$ immediately follows from  Lemma \ref{lem-for-con-funcpulbmsmn}.}. We deduce $u^{(k)}_{r'_2+k-n/2}\otimes \mathbf{1}_{\RR^{m-k}}\in\DD'^{r_1}_{\widetilde{L}}(\RR^m)$. We show that this leads to a contradiction. We only consider the case $k<m$, since the case $k=m$ can be treated analogously. There is $\xi'_0\in\mathbb{S}^{k-1}$ such that $(0_m,\xi'_0,0_{m-k})\not\in \widetilde{L}$. Choose an open neighbourhood $O$ of the origin in $\RR^m$ and an open cone $V=\RR_+(B(\xi'_0,\varepsilon)\times B(0_{m-k},\varepsilon))$ with $0<\varepsilon<1/2$ such that $(O\times \overline{V})\cap \widetilde{L}=\emptyset$. We can choose $O= O_1\times O_2$ with $O_1$ and $O_2$ open balls both with radii $\varepsilon_0>0$ and centres at the origins in $\RR^k$ and $\RR^{m-k}$. Take $\phi_0\in\DD(\RR)$ such that $0\leq \phi_0\leq 1$, $\phi_0(\lambda)=\phi_0(-\lambda)$, $\phi_0=1$ on $[-\varepsilon_0/2,\varepsilon_0/2]$ and $\supp\phi_0\subseteq (-\varepsilon_0,\varepsilon_0)$. Set $\varphi_1(x'):=\phi_0(|x'|)$, $x'\in\RR^k$; clearly $\varphi_1\in\DD(\RR^k)$ with $\supp\varphi_1\subseteq O_1$. Pick $\varphi_2\in\DD(\RR^{m-k})\backslash\{0\}$ with $\supp\varphi_2\subseteq O_2$. Let $V_1$ be the open cone $\RR_+ B(\xi'_0,\varepsilon/2)$ in $\RR^k$ and notice that $\{\xi'\in V_1\,|\, |\xi'|\geq 1\}\times B(0_{m-k},\varepsilon/2)\subseteq V$. Denoting $\varphi:=\varphi_1\otimes \varphi_2$, we infer (cf. Corollary \ref{lemma-for-sem-wav-fr-set-spa})
\begin{align*}
\infty&>\int_V |\mathcal{F}(\varphi(u^{(k)}_{r'_2+k-n/2}\otimes \mathbf{1}_{\RR^{m-k}}))(\xi)|^2\langle\xi\rangle^{2r_1}d\xi\\
&\geq \int_{V_1,\, |\xi'|\geq 1}\langle \xi'\rangle^{2r_1} |\mathcal{F}(\varphi_1 u^{(k)}_{r'_2+k-n/2})(\xi')|^2 d\xi' \int_{B(0,\varepsilon/2)}|\mathcal{F}\varphi_2(\xi'')|^2d\xi''.
\end{align*}
Since $\varphi_2$ has compact support, $\mathcal{F}\varphi_2$ is entire and hence the very last integral is strictly positive. As $\langle\cdot \rangle^{r_1}\mathcal{F}(\varphi_1 u^{(k)}_{r'_2+k-n/2})\in L^2(B(0_k,1))$, we deduce $\langle\cdot \rangle^{r_1} \mathcal{F}(\varphi_1 u^{(k)}_{r'_2+k-n/2})\in L^2(V_1)$. Notice that $\mathcal{F}(\varphi_1 u^{(k)}_{r'_2+k-n/2})(\Phi \xi')=\mathcal{F}(\varphi_1 u^{(k)}_{r'_2+k-n/2})(\xi')$, $\xi'\in\RR^k$, $\Phi\in\operatorname{O}(k)$. Hence, the above together with the compactness of the unit sphere yields that $\langle\cdot \rangle^{r_1} \mathcal{F}(\varphi_1 u^{(k)}_{r'_2+k-n/2})\in L^2(\RR^k)$. In view of \eqref{far-for-bel-insoff}, the latter implies that $\langle\cdot \rangle^{r_1} \mathcal{F} u^{(k)}_{r'_2+k-n/2}\in L^2(\RR^k)$ which is straightforward to check that it is not true.\\
\indent Since $H^{r_2}_{\loc}(\RR^n)\subseteq \DD'^{r_2}_L(\RR^n)$ continuously, the second part of $(ii)$ follows from applying the first part with $\widetilde{L}:=f^*L$; the conditions on $L$ imply that $f^*L$ satisfies \eqref{con-for-non-exofesxonconsofr} in view of \eqref{equ-for-puba-of-coni-setfr}. This completes the proof.
\end{proof}

\end{document}